\documentclass[a4paper]{article}

\usepackage[cyr]{aeguill}
\usepackage[english]{babel}
\usepackage[utf8]{inputenc}
\usepackage[T1]{fontenc}

\usepackage[a4paper,top=3cm,bottom=2cm,left=3cm,right=3cm,marginparwidth=1.75cm]{geometry}

\usepackage{amsmath}
\usepackage{graphicx}
\usepackage[colorinlistoftodos]{todonotes}
\usepackage[colorlinks=true, allcolors=blue]{hyperref}
\usepackage{amssymb}
\usepackage{amsthm}
\usepackage{bbold}
\usepackage[all]{xy}
\usepackage{dirtytalk}
\usepackage[titletoc,title]{appendix}
\usepackage{ stmaryrd }

\usepackage{comment}

\newtheorem{theorem}{Theorem}[section]
\newtheorem{corollary}[theorem]{Corollary}
\newtheorem{lemma}[theorem]{Lemma}
\newtheorem{proposition}[theorem]{Proposition}

\theoremstyle{definition}
\newtheorem{definition}{Definition}[section]
\newtheorem{example}{Example}[section]

\newtheorem{remark}[example]{Remark}
\newtheorem{notations}[example]{Notations}

\newcommand{\emeta}{\ensuremath{\raisebox{-1.1ex}{\includegraphics[height=2.1ex]{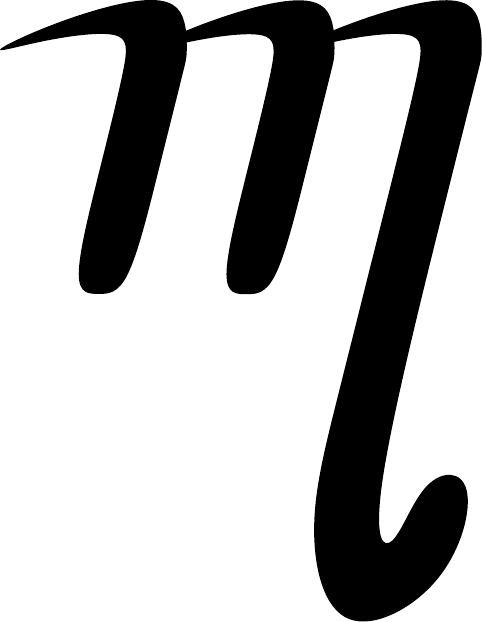}}}} 

\title{\textbf{Fourier decay of equilibrium states and \\ the Fibonacci Hamiltonian}}
\author{Gaétan Leclerc \footnote{Helsinki University (Finland) e-mail: \href{mailto:leclercg@ad.helsinki.fi}{gaetan.leclerc@imj-prg.fr}. }  }
\date{ \today }

\begin{document}

\maketitle

\begin{abstract}
We show positivity of the lower Fourier dimension for equilibrium states of nonlinear, area preserving, Axiom A diffeomorphisms on surfaces. To do so, we use the sum-product phenomenon to reduce Fourier decay to the study of some temporal distance function for a well chosen suspension flow, whose mixing properties reflects the nonlinearity of our base dynamics. We then generalize in an Axiom A setting the methods of Tsujii-Zhang, dealing with exponential mixing of three-dimensional Anosov flows \cite{TZ20}. The nonlinearity condition is generic and can be checked in concrete contexts. As a corollary, we prove power Fourier decay for the density of states measure of the Fibonacci Hamiltonian, which is related to the measure of maximal entropy of the Fibonacci trace map. This proves positivity of the lower Fourier dimension for the spectrum of the Fibonacci Hamiltonian, and suggest strong phase-averaged dispersive estimates in quasicrystals.   
\end{abstract}

\tableofcontents

\section{Introduction}\label{sec:Intro}

\subsection{The Fibonacci Hamiltonian, its Cantor spectrum, and transport}

One of the motivations to study the Fourier transform of measures comes from quantum mechanics. A common model to explore quantum transport in solids is to consider a bounded potential $V:\mathbb{Z} \rightarrow \mathbb{R}$ and to study the associated Hamiltonian in $\ell^2(\mathbb{Z})$. This Hamiltonian is then a bounded, self-adjoint operator acting on $\ell^2(\mathbb{Z})$ by the formula:
$$ (H u)(n) := u(n+1) + u(n-1) + V(n) u(n). $$

Schrödinger equation tells us that the quantum dynamics, starting from a state $u_0 \in \ell^2(\mathbb{Z})$, is given by $e^{-it H} u_0$. It is then natural to wonder about the possible spreading of this quantum state. \\

Say that we start from the inital state $\delta_0(n) = \delta_{0,n}$. Since $H$ is bounded and self-adjoint, there exists a finite, Borel, compactly supported \emph{spectral measure} $\mu$ in $\mathbb{R}$, such that for all continuous $g : \mathbb{R} \rightarrow \mathbb{C}$, we have:
$$ \langle \delta_0, g(H) \delta_0 \rangle_{\ell^2(\mathbb{Z})} = \int_{\sigma(H)} g(E) d\mu(E). $$
In particular, the correlation with $\delta_0$ can be written as
$$ \langle \delta_0, e^{-it H} \delta_0 \rangle_{\ell^2(\mathbb{Z})} = \int_{\sigma(H)} e^{-it E} d\mu(E); $$
we recognize the Fourier transform of the spectral measure $\mu$. \\

We usually model crystals by choosing a periodic potential $V(n)$. The spectrum of $H$ is then a union of intervals, and spectral measures are absolutely continuous with respect to the Lebesgue measure (see for example \cite{DF24}, or \cite{RS78} in the continuous case). This is the well known \say{electronic band structure} in solid state physics. In this setting, decay of the correlations is easy and is given by the Riemann-Lebesgue lemma. \\

A subtler setting is the one of quasi-crystals, modeled by \emph{quasi-periodic} potentials. A very studied model is the so-called \emph{Fibonacci Hamiltonian}, which is defined as follow. \\

Let $\alpha_0 := (\sqrt{5}-1)/2$ (the \emph{frequency}), let $\omega \in \mathbb{R}/\mathbb{Z}$ (the \emph{phase}), and let $V>0$ (the \emph{coupling constant}). Denote by $\chi_{[1-\alpha_0,1)}$ the characteristic function of $[1-\alpha_0,1) \text{ modulo } 1$. The associated Fibonacci Hamiltonian $H_{V,\omega}$ is a self-adjoint, bounded operator on $\ell^2(\mathbb{Z})$, acting by the formula:
$$ (H_{V,\omega} u)(n) := u(n+1) + u(n-1) + V \chi_{[1-\alpha_0,1)}(n \alpha_0 + \omega)u(n). $$
We know that the spectrum $\Sigma_V \subset \mathbb{R}$ of $H_{V,\omega}$ is independant of $\omega$ and is a \emph{Cantor set} (see for example \cite{DF22}, \cite{DGY16}, and Barry Simon's famous review \cite{RS78}). One can think of it as a nested intersection of bands.
Denote by $\mu_{V,\omega}$ the associated spectral measure, so that
$$ \langle \delta_0, e^{-itH_{V,\omega}} \delta_0\rangle_{\ell^2(\mathbb{Z})} = \int_{\Sigma_{V}} e^{-i t E} d\mu_{V,\omega}(E). $$
What can be said about the decay rate of those correlations ? \\
Very little is known on the structure of these measures, and studying directly the Fourier transform of $\mu_{V,\omega}$ seems out of reach. A common first step in the study of those spectral measures is to consider an average over $\omega$, constructing the \emph{density of states measure} $N_V$. It is defined as follow: for any continuous $g:\mathbb{R} \rightarrow \mathbb{C}$, set
$$ \int_{\Sigma_V} g(E) dN_V(E) := \int_{0}^1 \int_{\Sigma_V} g(E) d\mu_{V,\omega}(E) d\omega. $$
The density of state is well better understood than the individual spectral measures, especially for the Fibonnaci Hamiltonian. We know, for example, that this measure is a probability measure of support $\Sigma_V$ (this fact is more general), is exact dimensional, and satisfies some useful autosimilarity properties that we will describe later. One of our main results is the following: it applies in the weakly coupled regime.

\begin{theorem}\label{th:IDS1}
There exists  $V_0>0$ such that the following holds. For any $V \in (0,V_0)$, there exists $\rho(V)>0$ and $C(V) \geq 1$ such that:
$$ \forall t \in \mathbb{R}, \ \Big| \int_\omega \langle \delta_0, e^{-it H_{V,\omega}} \delta_0 \rangle_{\ell^2(\mathbb{Z})} d\omega \Big| \leq \frac{C}{1+|t|^\rho}. $$
\end{theorem}

It is the first time that such an estimate is proved without relying on an additionnal \emph{time-average} (as in \cite{La96} or \cite{DG12}) in an aperiodic ergodic setting. It seems natural to expect the following phase-averaged dispersive estimate to holds:
$$  \Big\| \int_\omega \Big(e^{-it H_{V,\omega}} \delta_0 \Big) \ d\omega \Big\|_{\ell^\infty(\mathbb{Z})} \leq \frac{C(V)}{1+|t|^{\rho(V)}} ,$$
but this does not immediatly follows from our work. Using Fourier decay results on fractal measures to prove dispersive estimates was also suggested in \cite{KPST25} (see problem 9.8 and problem 9.9). Our paper is a first step in this direction.
Even in the periodic setting, pointwise dispersive bounds \cite{DFY25} or transport properties \cite{GM21} are hard to understand and forms an active research area. Before explaining how we can prove Theorem \ref{th:IDS1}, let us underline why this bound is also saying something about the \say{additively chaotic} nature of the spectrum $\Sigma_V$.

\subsection{Fourier dimension and nonlinear autosimilarity}

Let $K \subset \mathbb{R}^d$ be any compact set (say, a Cantor set). Denote by $\mathcal{P}(K)$ the set of (Borel) probability measures supported on $K$. Frostman's lemma (\cite{Fro35,Ma15}) ensure that the Hausdorff dimension of $K$ can be computed as:
$$ \text{dim}_H(K) = \sup \Big\{\alpha \in [0,d] \ \Big{|} \ \exists \mu \in \mathcal{P}(K), \ \iint_{K^2} \frac{d\mu(x) d\mu(y)}{|x-y|^\alpha} < \infty \Big\} .$$
The energy integral $I_\alpha(\mu) :=  \iint_{K^2} |x-y|^{-\alpha} d\mu(x) d\mu(y)$ is finite when $\mu$ is regular enough: for example, any measure satisfying $\mu(B(x,r)) \leq C r^\beta$ for some $\beta>\alpha$ satisfies $I_\alpha(\mu)<\infty$. \\
Using Parseval's formula, $I_\alpha(\mu)$ can be rewritten in term of the \emph{Fourier transform} of $\mu$, as follow:
$$ \exists c_\alpha > 0, \ I_\alpha(\mu) = c_\alpha \int_{\mathbb{R}^d} |\widehat{\mu}(\xi)|^2 |\xi|^{\alpha-d} d\xi ,$$
where $\widehat{\mu}:\mathbb{R}^d \rightarrow \mathbb{C}$ is the oscillatory integral given by the formula:
$$ \widehat{\mu}(\xi) := \int_K e^{-2i\pi x \cdot \xi} d\mu(x). $$
In particular, if $\mu$ is regular enough, its Fourier Transform should exhibit power decay \emph{on average}, at least like $|\xi|^{-\alpha/2}$. \\

If decay on average of the Fourier transform is a soft matter, pointwise decay is way more subtle. As a matter of example, consider the triadic Cantor set $C \subset[0,1]$. Using Frostman's formula, one can show that $\text{dim}_H(C) = \ln(2)/\ln(3)>0$. Yet, there exists no probability measure $\mu$ supported on $C$ such that $\widehat{\mu}(\xi){\longrightarrow} 0$, $|\xi| \rightarrow \infty$. The key argument behind this fact is that $C$ is invariant by a \emph{linear} expanding map: the multiplication by $3$ (modulo $1$). As a matter of example, the Cantor distribution $\mu \in \mathcal{P}(C)$ satisfies $|\widehat{\mu}(3^n)| = 1$ for all $n \in \mathbb{N}$. \\

To study \say{additive chaos} of a set, a relevant quantity is the \emph{Fourier dimension}.
\begin{definition}\label{def:dimfourier}
Let $K \subset \mathbb{R}^d$ be a compact set. We call Fourier dimension of $K$ the following quantity:
$$ \text{dim}_F(K) := \sup\Big\{ \alpha \in [0,d] \ \Big| \ \exists \mu \in \mathcal{P}(K), \ \exists C \geq 1, \forall \xi \in \mathbb{R}^*, \ |\widehat{\mu}(\xi)| \leq C |\xi|^{-\alpha/2} \Big\}. $$
\end{definition}

In general, Frostman's lemma ensures the bounds $0 \leq \text{dim}_F(K) \leq \text{dim}_H(K) \leq d$. But as we saw with the case of the triadic Cantor set, it may be possible for a set $K$ to satisfy $\text{dim}_F(K)=0$ even if $\text{dim}_H(K)>0$. In fact, being given a deterministic set, it is hard to prove when its Fourier dimension is positive. \\

We can define in the obvious manner the Fourier dimension of a measure. A stronger information is given by the \emph{lower Fourier dimension} \cite{Le24}, which compute the infimum of the Fourier dimension of a measure under pushforwards.

\begin{definition}\label{def:lowerdimfourier}
Let $M$ be a manifold of dimension $d$. Let $\mu \in \mathcal{P}(M)$ be a probability measure with compact support. Let $\alpha \in (0,1)$. We denote $ \underline{\dim}_{F,C^{1+\alpha}}(\mu)>0 $ if there exists $\rho >0$ such that for all $C^{1+\alpha}$ local chart $\varphi : U \subset M \rightarrow \mathbb{R}^d$, for all smooth functions $\chi : M \rightarrow \mathbb{C}$ with $\text{supp}(\chi) \subset U$, we have
$$ \exists C \geq 1,  \forall \xi \in \mathbb{R}^d \setminus \{0\}, \quad   \Big| \int_{M} e^{i \xi \cdot \varphi(x)} \chi(x) d\mu(x) \Big| \leq C |\xi|^{-\rho}. $$
\end{definition}

A corollary of our work can then be stated as follow.

\begin{theorem}\label{th:IDS2}
Denote by $N_V \in \mathcal{P}(\Sigma_V)$ the density of states measure associated with the Fibonacci Hamiltonian. For $V>0$ small enough, and for any $\varepsilon>0$, we have:
$$ \underline{\text{dim}}_{F,C^{1+\varepsilon}}(N_V)>0 .$$
In particular, for any $C^{1+}$-diffeomorphism $\varphi:\mathbb{R} \rightarrow \mathbb{R}$, we have $\dim_F(\varphi(\Sigma_V)) > 0$.
\end{theorem}

The Fourier dimension was introduced at least since the early works of Salem and Kahane \cite{Sa43,Sa51,Ka66a,Ka66b}, and is relevant, among other topics, to study sets of multiplicity \cite{KS64,Kö92}, the restriction problem \cite{Ma15,ŁW16}, or the Fractal Uncertainty Principle \cite{BD17,BS23}. Very little is known about the behavior of the Fourier dimension in general. We know how to produce sets of positive Fourier dimension through random constructions (involving for example random Cantor sets \cite{Bl96} or the Brownian motion \cite{Ka85}), or through arithmetic constructions \cite{Ka81,Ha17,FH23}. \\

A relevant setting for us is the one of deterministic \emph{dynamical fractals}, where our sets exhibit autosimilarity. For example, let us consider a Cantor set $K \subset \mathbb{R}$, invariant by some analytic expanding map $T : K \rightarrow K$. Under what kind of conditions on $T$ can we ensure that $\text{dim}_F(K) > 0$? \\
In the case of the triadic Cantor set $C$, which is invariant by the \emph{linear} expanding map $T(x) := 3x \ \text{mod} \ 1$, we know that $\text{dim}_F(C)= 0$. Numerous recent results tends to show that, under a \emph{nonlinearity} condition on $T$, we should expect positivity of the Fourier dimension \cite{Le24}. More precisely, we can expect positivity of the lower Fourier dimension for equilibrium states, and in particular, for the measure of maximal entropy. Some works on this topics include \cite{Ka80,QR03,BD17,LNP19,SS20,Li20,Le21,Le22,Le23b,LPS25}, dealing with the Gauss map, Schottky groups, Cantor sets invariant under nonlinear expanding maps, Julia sets, and geodesic flow on hyperbolic surfaces. A recent review on works related to Fourier decay in a dynamical setting can be found in \cite{Sah23}. \\

Unfortunately, the density of states measure is not covered by one of those settings: its autosimilarity is better described by Axiom A diffeomorphisms on surfaces.

\subsection{Nonlinear Axiom A diffeomorphisms}

Let us introduce our dynamical setting. Some good references on the topic of hyperbolic dynamics includes \cite{Bo75,KH95,BS02}.  

\begin{definition}\label{def:hyperbolic}
Let $f: M \rightarrow M$ be a diffeomorphism of a complete  $C^\infty$ Riemannian surface $M$. A compact set $\Lambda \subset M$ is said to be hyperbolic for $f$ if $f(\Lambda) = \Lambda$, and if for each $x \in \Lambda$, the tangent plane $T_xM$ can be written as a direct sum $$ T_xM = E^s(x) \oplus E^u(x) $$
of one-dimensional subspaces such that
\begin{enumerate}
    \item $\forall x\in \Lambda$,  $(df)_x(E^s(x)) = E^s(f(x))$ and $(df)_x(E^u(x)) = E^u(f(x)) $
    \item $\exists C>0, \ \exists \kappa \in (0,1), \ \forall x \in \Lambda, $
    $$ \forall v \in E^s(x), \ \forall n \geq 0,  \ \|(df^n)_x(v)\| \leq C \kappa^n \| v \| $$
    and $$ \forall v \in E^u(x), \ \forall n \geq 0, \  \|(df^n)_x(v)\| \geq C^{-1} \kappa^{-n} \| v \| .$$
\end{enumerate} 
\end{definition}

The unstable and stable distributions $E^u,E^s$ on $\Lambda$ are always $C^{1+\alpha}$ in our setting (for some $\alpha>0$). Generically on the choice of $f$, the distributions are not $C^{2}$ at any point.

\begin{definition}\label{def:NW}
A point $x \in M$ is called non-wandering if, for any open neighborhood $U$ of $x$, $$ U \cap \bigcup_{n > 0} f^n(U)  \neq \emptyset.$$ We denote the set of non-wandering points by $\Omega(f)$.
\end{definition}

\begin{definition}\label{def:AxiomA}
A diffeomorphism $f: M \rightarrow M$ is said to be Axiom A if \begin{itemize}
    \item $\Omega(f)$ is a hyperbolic set for $f$ (in particular, it is compact),
    \item $ \Omega(f) = \overline{\{ x \in M \ | \ \exists n > 0, \ f^n(x) = x  \} } . $
\end{itemize} 
\end{definition}

\begin{proposition}\label{prop:basicsets}
One can write $\Omega(f) = \Omega_1 \cup \dots \cup \Omega_k$, where $\Omega_i$ are nonempty compact disjoint sets, such that 
\begin{itemize}
    \item $f(\Omega_i)=\Omega_i$, and $f_{|\Omega_i}$ is topologically transitive,
    \item $\Omega_i = X_{1,i} \cup \dots \cup X_{r_i,i}$ where the $X_{j,i}$ are disjoint compact sets,  $f(X_{j,i}) = X_{j+1,i}$ ($X_{r_i+1,i}=X_{1,i}$) and $f^{r_i}_{| X_{j,i}}$ are all topologically mixing.
\end{itemize}
The sets $\Omega_i$ are called basic sets.
\end{proposition}

Our main result is the following.

\begin{theorem}\label{th:main1}
Let $f:M \rightarrow M$ be a $C^\infty$ Axiom A diffeomorphim on a smooth and complete riemannian surface $M$. Let $\Omega$ be a basic set for $f$. Suppose that $|\det df| = 1$ on $\Omega$, and suppose that $E^s \notin C^2(\Omega,TM)$ or $E^u \notin C^2(\Omega,TM)$. Then the measure of maximal entropy $\mu \in \mathcal{P}(\Omega)$ satisfies
$$ \underline{\text{dim}}_{F,C^{1+\alpha}}(\mu) > 0,$$
for any $\alpha>0$.
\end{theorem}

In fact, our result applies to all equilibrium states, and to more general phases, see Theorem \ref{th:reduc1} (and Corollary \ref{co:reduc2}) for the full statement. Notice that our hypothesis $E^u \text{ or } E^s \notin C^2$ is indeed a \emph{nonlinearity} hypothesis. Indeed, if $f$ could be $C^2$ conjugated to a \say{linear map} such as the CAT map on the 2-torus, or such as Baker's map or Smale's horseshoe, then the stable and unstable distributions would be $C^2$. \\

The stable and unstable line distributions $E^s,E^u \subset TM$ can be integrated into local transverse stable and unstable laminations $W^s_{loc},W^u_{loc} \subset M$, along which the dynamics is respectively contracting or expanding. If we fix a line $L \subset M$ which is transverse to $W^s_{loc}$ and close enough to $\Omega$, one can then consider the stable \emph{holonomy} $\pi_L : \Omega \rightarrow L$, which \say{projects} points in $\Omega$ to points in $L$ by just following the local stable lamination. Restricted on some local unstable manifold, the map $\pi_L : W^u_{loc}(x_0)  \rightarrow L$ is a $C^{1+\alpha}$ diffeomorphism. In particular, we find the following corollary.

\begin{theorem}\label{th:main2}
Let $f:M \rightarrow M$ be a $C^\infty$ Axiom A diffeomorphim on a smooth and complete riemannian surface $M$. Let $\Omega$ be a basic set for $f$. Suppose that $|\det df| = 1$ on $\Omega$, and suppose that $E^s \notin C^2(\Omega,TM)$ or $E^u \notin C^2(\Omega,TM)$. Let $L \subset M$ be a line transverse to $W^s_{loc}$, and consider the associated stable holonomy $\pi_L : \Omega \rightarrow L$. Denote by $\mu \in \mathcal{P}(\Omega)$ the measure of maximal entropy, and $\nu := (\pi_L)_*(\mu) \in \mathcal{P}(L)$ the pushforward of $\mu$ by $\pi_L$. Then:
$$ \underline{\text{dim}}_{F,C^{1+\alpha}}( \nu ) > 0.$$
In particular, if we identify $\nu$ with a measure $\tilde{\nu}$ on the real line via some $C^{1+}$ parametrization of $L$, we find:
$$ \exists \rho >0, \exists C \geq 1, \forall \xi \in \mathbb{R}^*, \ \Big| \int_\mathbb{R} e^{i x \xi} d\tilde{\nu}(x) \Big| \leq C |\xi|^{-\rho}. $$
\end{theorem}

Let us conclude this section by recalling that there exists an explicit obstruction for the (un)stable distribution to be $C^2$ in our context: nonvanishing of the \emph{Anosov cocycle}. 

\begin{proposition}[\cite{An67}, Chapter 24; see also \cite{HK90}]\label{prop:Acocycle}
    Let $f:M\rightarrow M$ be a smooth Axiom A diffeomorphism with a basic set $\Omega$. Suppose that $E^u$ and $E^s \in C^2(\Omega,TM)$. Choose then a fixed point $p_0 \in \Omega$ for $f$, and suppose that there exists some smooth \say{adapted} coordinates centered at $p_0$, in which the dynamics can be written $$ \tilde{f}(x,y) = \begin{pmatrix} F(x,y) \\ G(x,y) \end{pmatrix}, $$
    with $\tilde{f}(0,0)=(0,0)$, $\partial_x F(0,0) = \lambda$, $\partial_y F(0,0) = \partial_x G(0,0)=0$, $\partial_x F(0,0) = \mu$, $\mu = \frac{1}{\lambda} \in (0,1)$. Then:
    $$ \frac{\partial^3_{xyy} G(0,0)}{\mu} - \frac{\partial_{xy}^2 F(0,0) \partial_{xy}^2 G(0,0)}{\lambda-1} - \frac{\partial_{xx}^2 G(0,0) \partial_{yy}^2 F(0,0)}{1-\mu^3} - \frac{\partial_{xy}^2 G(0,0) \partial_{yy}^2 G(0,0)}{\mu^2} = 0. $$
\end{proposition}

In particular, an explicit way of checking the hypothesis $E^u \text{ or }E^s \notin C^2$ in Theorem \ref{th:main1} and \ref{th:main2} is to compute this Anosov cocycle at some fixed point, and checking its non-vanishing. This condition is clearly \emph{generic} in the choice of the dynamics (in the set of $C^{3}$-diffeomorphisms). The Anosov cocycle will play a role later in our paper, see appendix \hyperref[ap:C]{C}.

\subsection{Application to the Fibonacci trace map}

It is known that the autosimilarity of the density of states measure for the Fibonacci Hamiltonian is described by a setting similar to the one of Theorem \ref{th:main2}. Let us quickly recall how: see for example \cite{Su87,DG09,DGY16,DF22} for more details. \\ 

Denote $V(n) := V \chi_{[1-\alpha_0,1)}(n \alpha_0)$ for $\alpha_0 := \frac{\sqrt{5}-1}{2}$. To find the spectrum of the Fibonacci Hamiltonian, we look at nonzero solutions in $\mathbb{R}^\mathbb{Z}$ of the eigenvalue equation:
$$ (H_{V,0} u)(n) = E u(n). $$
Introducing $X_n := \begin{pmatrix} u(n+1) \\ u(n) \end{pmatrix}  \in \mathbb{R}^2$, we find the relation $ X_{n} =  M(n) \dots M(1) X_0  $, where $$ M(j) := \begin{pmatrix} E-V(j) & -1 \\ 1  & 0 \end{pmatrix}.  $$
We can then always find a solution for the eigenvalue equation in $\mathbb{R}^\mathbb{Z}$. But to find an element of the spectrum, the eigensequence should grow sufficiently slowly so that we can find a quasimode for the $\ell^2(\mathbb{Z})$-spectrum of $H_{V,0}$ by taking cut-offs. This amount to understand the growth of the products $M(n) \dots M(1)$. Because of the special properties of the (inverse of the) golden ratio $\alpha$, one can show the relation $M_{k+1} = M_{k-1} M_{k}$, where $M_k = M(F_k) \dots M(1)$ is the product of the first $F_k$ matrices (where $F_k$ is the $k$-th Fibonacci number). Since $\det M_k = 1$, the spectrum of those matrices is determined by their trace $x_k := \frac{1}{2} \text{tr}(M_k)$. Using Cayley-Hamilton's theorem, we find the relation
$$ x_{k+1} = 2 x_k x_{k-1} - x_{k-2}.$$
It is then natural to introduce the \emph{Fibonacci trace map} $T:\mathbb{R}^3 \longrightarrow \mathbb{R}^3$:
$$ T(x,y,z) := (2xy-z,x,y). $$
This volume-preserving diffeomorphism gives us a dynamical system whose orbits detects the spectrum $\Sigma_V$ of $H_{V,\omega}$. More precisely, the following holds.

\begin{proposition}[\cite{Su87}]\label{spectrum}
Let $V>0$. Then $E \in \Sigma_V$ if and only if the forward orbit of $(\frac{E-V}{2},\frac{E}{2},1)$ under the trace map $T$ is bounded.
\end{proposition}

The trace map $T$ leaves invariant a family of cubic surfaces $S_V$, $V \geq 0$, defined by:
$$ S_V := \{ (x,y,z) \in \mathbb{R}^3, \ x^2+y^2+z^2-2xyz = 1+\frac{V^2}{4} \}. $$
For $V>0$, these are smooth surfaces. In the special case $V=0$, we recover Cayley's nodal cubic surface, which has 4 conic singularities, one of them standing at $p_0 = (1,1,1)$. The restriction of $T$ on each of theses surfaces is smooth, area-preserving and Axiom A for $V>0$ small enough \cite{DG09}. Its non-wandering set $\Omega(T_{V}) =: \Omega_V$ is its only basic set. Furthermore, for $V>0$ small enough, the line $\ell_V := \{ (\frac{E-V}{2},\frac{E}{2}, 1), \ E \in \mathbb{R} \} \subset S_V$ goes close enough to $\Omega_V$ so that we can study the intersection of the local stable lamination on $\Omega_V$ with $\ell_V$. This intersection is transverse. Denoting by $\pi_V$ the local stable holonomy projecting points in $\Omega_V$ to $\ell_V$ (technically, $\pi_V$ is defined on an element of a Markov partition of $\Omega_V$, see \cite{DG09} for details), we have the following result:

\begin{theorem}[\cite{DG09}]\label{th:IDSchara}
Let $V>0$ be small enough. Denote by $\mu_V \in \mathcal{P}(\Omega_V)$ the measure of maximal entropy for $(T_{|S_V},\Omega_V)$. The pushforward measure $(\pi_V)_* \mu_V \in \mathcal{P}(\ell_V)$ is identified with the density of states measure $N_V \in \mathcal{P}(\Sigma_V)$ through the parametrization $E \in \mathbb{R} \mapsto (\frac{E-V}{2},\frac{E}{2},1) \in \ell_V$.
\end{theorem}

We are ready to prove Fourier decay for the density of states measure in the weakly coupled regime, assuming Theorem \ref{th:main1}.

\begin{theorem}\label{th:IDS3}
Denote by $\mu_V \in \mathcal{P}(\Omega_V)$ the measure of maximal entropy for $(T_{|S_V},\Omega_V)$, and $N_V \in \mathcal{P}(\Sigma_V)$ the integrated density of sate for the Fibonacci Hamiltonian. There exists $V_0>0$ such that, for every $V \in (0,V_0)$, we have for all $\alpha \in (0,1)$:
$$   \underline{\text{dim}}_{F,C^{1+\alpha}}(\mu_V) > 0 \quad \text{and} \quad \underline{\text{dim}}_{F,C^{1+\alpha}}(N_V) > 0 . $$
\end{theorem}

\begin{proof}[Proof (of Theorem \ref{th:IDS1}, Theorem \ref{th:IDS2} and Theorem \ref{th:IDS3})] 
Iterating once the Trace map, we find
$$ T^2(x,y,z) = ((4x^2-1)y-2xz,2xy-z,z). $$
The set of $2$-periodic points of $T$ is then given by $\text{Per}_2(T) = \{ (t,\frac{t}{2t-1},t), \ t \in \mathbb{R} \setminus \{ 1/2\} \}$. If $t = 1+\varepsilon > 1$ is close to one, we have $p_t := (t,\frac{t}{2t-1},t) \simeq (1,1,1) + \varepsilon(1,-1,1)$. In particular, the $y$-coordinate is less than one if $t>1$. \\

Denote by $t_V>1$ the only real number close to $1^+$ such that $p_{t_V} \in S_V = \{ (x,y,z) \in \mathbb{R}^3, \  x^2+y^2+z^2 - 2xyz = 1 + V^2/4 \} $ for $V>0$. The relation 
$$ t_V^2 + \frac{t_V^2}{(2t_V-1)^2} +t_V^2 - 2 \frac{t_V^3}{2t_V-1} = 1+ \frac{V^2}{4} $$
can be rewritten as $$ \frac{(t_V-1)^2(4t_V^2+2t_V-1)}{(2t_V-1)^2} = \frac{V^2}{4}, $$
which reveals the useful estimate $$t_V = 1+\frac{V}{2\sqrt{5}} + O(V^2).$$
The point $p_V := p_{t_V} \in S_V$ is then a fixed point for the Axiom A, area-preserving, smooth map $T^2_{|S_V}$. Our goal now is to construct adapted coordinates centered at $p_V$, and to compute the Anosov cocycle in these coordinates. \\

Let us begin by introducing a first parametrisation of the surface $S_V$, $V>0$. These coordinates are not going to be adapted yet but we will modify them later. Let $(x,y,z) \in \mathbb{R}^3$ with $y<1$, close to $(1,1,1)$. Then $(x,y,z) \in S_V$ iff $$x^2+y^2+z^2 - 2x y z = 1+\frac{V^2}{4}. $$
Solving in $y$ yields
$$ y = xz - \sqrt{(x^2-1)(z^2-1) + \frac{V^2}{4}} := y_V(x,z). $$
We can then parametrize our surface $S_V$ close to the fixed point $p_V$ by 
$(x,z) \mapsto (x,y_V(x,z),z)$.
Notice that in these coordinates, $p_V$ is represented by $(t_V,t_V)$. For future use, we will need to derive asymptotics in $V>0$ of derivatives of $y_V$ at $(t_V,t_V)$, using the relation $t_V = 1 + \frac{V}{2 \sqrt{5}} + O(V^2)$. \\

Notice that $y_V(x,z)=y_V(z,x)$, so that on the diagonal, $\partial_x y_V(t,t) = \partial_z y_V(t,t),$ $\partial_x^2 y_V(t,t) = \partial_z^2 y_V(t,t)$, etc. By construction, we know that $$y_V(t_V,t_V) = \frac{t_V}{2t_V-1} = 1-\frac{V}{2\sqrt{5}} + O(V^2).$$
Then, for the first-order derivative, we have
$$ \partial_x y_V(x,z) = z - \frac{x(z^2-1)}{\sqrt{(x^2-1)(z^2-1)+\frac{V^2}{4}}}, $$
which gives, recalling $t_V=1+\frac{V}{2\sqrt{5}} + O(V^2)$: $$ \partial_x y_V(t_V,t_V) = \partial_z y_V(t_V,t_V) = t_V - \frac{t_V(t_V^2-1)}{\sqrt{(t_V^2-1)^2+\frac{V^2}{4}}} = 1 - \frac{V/\sqrt{5}}{\sqrt{\frac{V^2}{5}+\frac{V^2}{4}}} + O(V) = \frac{1}{3} + O(V). $$
Similar direct computations reveals the behavior of $y_V$ up to the third order in $(x,z)$, that we summarize here:
\begin{align*}
    y_V(t_V,t_V) &= 1-\frac{V}{2\sqrt{5}} + O(V^2)  & \partial_x y_V(t_V,t_V) &= \frac{1}{3} + O(V) \\
    \partial_{xx}^2 y_V(t_V,t_V) &= \frac{8\sqrt{5}}{27} V^{-1} + O(1) &
    \partial_{xz}^2 y_V(t_V,t_V) &= -\frac{28\sqrt{5}}{27} V^{-1} + O(1) \\
    \partial_{xxz}^3 y_V(t_V,t_V) &= \frac{320}{81} V^{-2} + O(V^{-1}) &
    \partial_{xxx}^3 y_V(t_V,t_V) &= -\frac{160}{81} V^{-2} + O(V^{-1})
\end{align*}
We are now ready to express the dynamics in our coordinate system. Denoting $f_V(x,z)$ the map $T^2_{|S_V}$ in our coordinates near $p_V$, we find the formula
$$ f_V\begin{pmatrix} x \\ z \end{pmatrix} = \begin{pmatrix} (4x^2-1) \ y_V(x,z)-2xz \\ x \end{pmatrix} .$$
A direct computation allows us to find the derivatives of $f_V$ at $(t_V,t_V)$. We can summarize these informations as a Taylor expansion like so:
$$ f_V\Big( \begin{pmatrix} t_V \\ t_V \end{pmatrix} + \begin{pmatrix} x \\ z \end{pmatrix} \Big)  =  \begin{pmatrix} t_V \\ t_V \end{pmatrix}+\begin{pmatrix} 7 + O(V) & -1 + O(V) \\ 1 & 0 \end{pmatrix}  \begin{pmatrix} x \\ z \end{pmatrix}  $$
$$ + \begin{pmatrix} \frac{4\sqrt{5}}{9} x^2 -\frac{28\sqrt{5}}{9}  xz + \frac{4\sqrt{5}}{9} z^2 \\ 0 \end{pmatrix} (V^{-1} + O(1)) $$ $$ +\begin{pmatrix} -\frac{80}{81} x^3 + \frac{160}{27} x^2 z + \frac{160}{27} x z^2 - \frac{80}{81} z^3 \\ 0 \end{pmatrix} (V^{-2} + O(V^{-1}))  + O_V(x^4+z^4). $$
For example, one can compute the Jacobian matrix at $(t_V,t_V)$ as follow:
$$ J_{f_V}(x,z) = \begin{pmatrix} 8x \cdot y_V(x,z) + (4x^2-1) \partial_xy_V(x,z) -2z & (4x^2-1) \partial_z y_V(x,z) - 2x \\ 1 & 0 \end{pmatrix} $$ 
and so
$$ J_{f_V}(t_V,t_V) = \begin{pmatrix} 8t_V \cdot y_V(t_V,t_V) + (4t_V^2-1)\partial_x y_V(t_V,t_V) - 2 t_V & (4t_V^2-1) \partial_z y_V(t_V,t_V) - 2t_V \\ 1 & 0 \end{pmatrix} $$ $$= \begin{pmatrix} 7  + O(V) & -1  + O(V) \\ 1 & 0 \end{pmatrix} . $$
Higher derivatives are computed similarly. Notice that $J_V := J_{f_V}(t_V,t_V)$ is analytic in $V$ and can be extended on a neighborhood of $V=0$. Notice also that we have
$$ P_0^{-1} J_0 P_0 = D_0, $$
where $$ D_0 := \begin{pmatrix} \frac{7+3\sqrt{5}}{2} & 0 \\ 0 & \frac{7-3\sqrt{5}}{2}  \end{pmatrix}  \quad , \quad P_0 := \frac{1}{2}\begin{pmatrix} 7+3\sqrt{5} & 7-3\sqrt{5} \\ 2 & 2 \end{pmatrix} \quad , \quad P_0^{-1} = \frac{1}{30}\begin{pmatrix} 2\sqrt{5} & 15-7 \sqrt{5} \\ -2\sqrt{5} & 15+7\sqrt{5} \end{pmatrix}.$$
Notice how $J_0$ has distinct real eigenvalues: by analyticity of $V \mapsto J_V$, the spectrum of $J_V$ is still real and distinct for small $V>0$. Moreover, since $T^2_{|S_V}$ is area preserving, we know that $\det J_V=1$. We can then find an analytic family of matrices $P_V,D_V$ such that, for all $V$ small enough,
$$ P_V^{-1} J_V P_V = D_V $$
where $D_V$ is diagonal (with coefficients $\lambda_V,\mu_V$). By analyticity, we then have $D_V = D_0 + O(V)$, $P_V=P_0 + O(V)$, and $P_V^{-1} = P_0^{-1} + O(V)$.
Now, we are ready to introduce our adapted coordinates: we let
$$ \tilde{f}_V(X,Y) := P_V^{-1} f_V\Big( \begin{pmatrix} t_V \\ t_V \end{pmatrix} + P_V \begin{pmatrix} X \\ Y \end{pmatrix} \Big). $$
These coordinates are adapted. To estimate the derivatives $\partial^\alpha \tilde{f}_V(0,0)$, one can proceed as follow. Let $(h_i)$ be some vector of the form $(1,0)$ or $(0,1)$. The $k$-th differential can be written as
$$ (d^k \tilde{f}_V)_{(0,0)}(h_1,\dots,h_k) = P_V^{-1} (d^{k}f_V)_{(t_V,t_V)}\Big( P_V h_1, \dots , P_V h_k \Big) $$
$$ = P_0^{-1} (d^{k}f_V)_{(t_V,t_V)}\Big( P_0 h_1, \dots , P_0 h_k \Big) + O(V \cdot \|(d^k f_V)_{(t_V,t_V)}\| ). $$
To find the derivatives of $\tilde{f}_V$, we then compute the Taylor coefficients of \\ $P_0^{-1} f_V\Big( \begin{pmatrix} t_V \\ t_V \end{pmatrix} + P_0 \begin{pmatrix} X \\ Y \end{pmatrix} \Big)$, by using the known Taylor estimate on $f_V$ at $(t_V,t_V)$, replacing the variables $(x,z)$ with $(X,Y)$. A direct computation yields:
$$ \tilde{f}_V(X,Y) = \begin{pmatrix} \frac{7+3\sqrt{5}}{2} + O(V) & 0 \\ 0 &  \frac{7-3\sqrt{5}}{2} + O(V) \end{pmatrix} \begin{pmatrix} X \\ Y \end{pmatrix} + \frac{20}{3} XY \begin{pmatrix} -1 \\ 1 \end{pmatrix} (V^{-1} + O(1)) + (X^2+Y^2) O(1) $$ $$+ \frac{40}{9} X Y \big( (5+3 \sqrt{5})X - (5-3 \sqrt{5})Y \big) \begin{pmatrix} 1 \\ -1 \end{pmatrix} (V^{-2} + O(V^{-1})) + (X^3+Y^3)O(V^{-1}) + O_V(X^4+Y^4) .$$
We can now estimate the Anosov cocycle in these coordinates. Denoting $\tilde{f}_V(X,Y) =  \begin{pmatrix} F_V(X,Y) \\ G_V(X,Y) \end{pmatrix}$, we have:
$$ \frac{\partial^3_{XYY} G_V(0,0)}{\mu_V} - \frac{\partial_{XY}^2 F_V(0,0) \partial_{XY}^2 G_V(0,0)}{\lambda_V-1} - \frac{\partial_{XX}^2 G_V(0,0) \partial_{YY}^2 F_V(0,0)}{1-\mu_V^3} - \frac{\partial_{XY}^2 G_V(0,0) \partial_{YY}^2 G_V(0,0)}{\mu_V^2}  $$
$$ = \frac{\partial^3_{XYY} G_V(0,0)}{\mu_0} - \frac{\partial_{XY}^2 F_V(0,0) \partial_{XY}^2 G_V(0,0)}{\lambda_0-1} + O(V^{-1}) $$
$$ =  \frac{160}{9}\frac{5-3\sqrt{5}}{7-3\sqrt{5}} V^{-2} - \frac{40}{9} \frac{2}{5+3\sqrt{5}} V^{-2} + O(V^{-1}) $$
$$ = - \frac{140+76\sqrt{5}}{3} \ V^{-2} + O(V^{-1}), $$
which is nonvanishing for $V>0$ small enough. It follows that $E^s \text{ or } E^u \notin C^2$ on $\Omega_V$ for small enough $V$, and so Theorem \ref{th:main1} and Theorem \ref{th:main2} apply. \end{proof}

\subsection{Strategy to prove Fourier decay and organization of the paper}

Let us now explain our strategy to prove Theorem \ref{th:main1}. We first fix some preliminary framework in section \hyperref[sec:preli]{2}, and then state the technical version of Theorem \ref{th:main1}. Our strategy can then be separated in two parts: first, in section \hyperref[sec:3]{3}, we reduce the Fourier decay problem to understanding the oscillations of a certain \emph{temporal distance function} $\Delta$ (using the \emph{sum-product phenomenon} as a tool, as in \cite{BD17}). Then, in section \hyperref[sec:4]{4}, we study the oscillations of this temporal distance function $\Delta$ by generalizing some methods introduced by M. Tsujii and Z. Zhang to study the rate of mixing of 3-dimensional Anosov flows \cite{TZ20}. 

\subsubsection{The sum-product phenomenon and a suspension flow}

The strategy that we follow in the first half of this paper was introduced by Jean Bourgain and Semyon Dyatlov in 2017, to study the Fourier transform of Patterson-Sullivan measures for Kleinian Schottky groups of $\text{PSL}(2,\mathbb{R})$ \cite{BD17}. Such measures are supported on Cantor sets in the real line and satisfy some strong autosimilarity properties. The main new tool was the \emph{sum-product phenomenon}. (The same strategy was generalized in \cite{LNP19,SS20,Le24,LPS25}, among others.) \\

An introduction to these ideas can be found in the easy-to-read paper of Ben Green \cite{Gr09}. The basic idea behind the sum-product phenomenon is that a multiplicative subgroup of a finite field have pseudorandom properties from the point of view of the addition. In other words, \say{multiplicative structure seem additively chaotic}. The second idea is that, in a sum of exponential, (pseudo)randomness in the phase produce cancellations. Combining those two ingredients gives:

\begin{theorem}[\cite{BD17},\cite{Bo10}]
For all $\gamma > 0$, there exist $\varepsilon > 0$ and $k \in \mathbb{N}$ such that the
following holds. Let $\mathcal{Z}$ be a finite set, let $\zeta:\mathcal{Z}\rightarrow [1/2,1]$, and let $\eta>0$ be large enough. Assume that:
$$ \forall a \in \mathbb{R}, \ \forall \sigma \in [|\eta|^{-1}, |\eta|^{- \varepsilon}], \quad \# \{\mathbf{b} \in \mathcal{Z} , \ \zeta(\mathbf{b}) \in [a-\sigma,a+\sigma] \} \leq \# \mathcal{Z} \ \sigma^{\gamma}. \quad (*)$$
Then 
$$ \left|\frac{1}{ \#\mathcal{Z}^{k}} \sum_{(\mathbf{b}_1, \dots, \mathbf{b}_k) \in \mathcal{Z}^k} \exp\left( i \eta \zeta(\mathbf{b}_1) \dots \zeta(\mathbf{b}_k) \right) \right| \leq |\eta|^{-{\varepsilon}}.$$
\end{theorem}

Since Bourgain in $\cite{Bo10}$, similar results were proved in increasingly general settings, such as in \cite{Li18,HdS22,OdSS24}. We will use a slightly more general version of this result during section \hyperref[sec:3]{3}, to reduce Fourier decay to checking a \emph{non-concentration condition} similar to $(*)$. Essentially, the idea is to use the autosimilarity of equilibrium states to approximate our Fourier transform with a sum of exponential as above. \\

The map $\zeta$ that we will then introduce will behave, essentially, like $\zeta(x) \sim \| (df)_{|E^u(x)} \|$. The non-concentration condition $(*)$ is then indeed a \emph{nonlinearity} assumption on the dynamics: if $f$ is linear, then $\| (df)_{|E^u(x)} \|$ takes a constant value. For example in the case of the triadic Cantor set, we would get a map $\zeta$ constant equal to $3$. \\

To be more precise, our nonconcentration hypothesis will be reformulated in terms of a function $\Delta$, defined as follow. We denote by $\tau_f(x) := \ln \| (df)_{|E^u(x)} \|$ the (opposite of the) geometric potential. Consider now two points $p,q \in \Omega$ that are close enough. We can then consider the points $r := W^u_{loc}(p) \cap W^s_{loc}(q)$ and $s := W^s_{loc}(p) \cap W^u_{loc}(q)$. The quantity $\Delta(p,q)$ is then defined by:
$$ \Delta(p,q) := \sum_{n \in \mathbb{Z}} \tau_f(f^n(p)) - \tau_f(f^n(r)) - \tau_f(f^n(s)) + \tau_f(f^n(q)). $$
Section \hyperref[sec:3]{3} is then devoted to show the following:

\begin{lemma}\label{lem:0}
Denote by $\mu \in \mathcal{P}(\Omega)$ the measure of maximal entropy.
Suppose that there exists $\Gamma>0$ such that
$$ \exists C \geq 1,  \forall \sigma \in (0,1), \ \mu^{\otimes 2}\big((p,q) \ , \ |\Delta(p,q)| \leq \sigma  \big) \leq C \sigma^\Gamma. $$
Then $\underline{dim}_{F,C^{1+\alpha}}(\mu)>0$.
\end{lemma}
More precise statement will be introduced in the core of the text, see in particular Theorem \ref{th:reduc1}. Similar (but not identical) kind of reductions can be found in \cite{BD17,LNP19,SS20,Le21,Le22,Le23b}. This precise formulation was introduced in the thesis \cite{Le24} and is also used in \cite{LPS25} (in the context of expanding maps with infinite countable Markov partitions). There is three important interpretations that can be given before explaining how we will study $\Delta$. 

\begin{enumerate}
\item \underline{The cohomology class interpretation:} Notice that the map $\Delta$ depends only on the \emph{cohomology class} of $\tau_f$. More precisely, if we choose a Hölder function $\tau:\Omega \rightarrow \mathbb{R}$ such that there exists $\theta:\Omega \rightarrow \mathbb{R}$ where $\tau_f = \tau + \theta \circ f - \theta$ (we denote $\tau_f \sim \tau$), then we can replace $\tau_f$ by $\tau$ in the definition of $\Delta$. Notice that if $\tau_f \sim \text{cst}$ (which suggest that $f$ is conjugated to a linear map), then $\Delta=0$. In some sense, the non-concentration of $\Delta$ near $0$ is a quantitative measure of the nonlinearity of $f$.
\item \underline{The geometric interpretation:} Fix $s,p \in \Omega$ on the same local unstable manifold. We can then consider the associated stable holonomy $\pi_{s,p} : W^u_{loc}(s) \cap \Omega \rightarrow W^u_{loc}(p) \cap \Omega $. For $q \in W^u_{loc}(s) \cap \Omega $, denote $r := \pi_{s,p}(q)$. The stable holonomy is $C^{1+\alpha}$ and its derivative along $W^u_{loc}$, $\partial_u \pi_{s,p}$, is explicit. We have:
$$ \ln \partial_u \pi_{s,p}(q) = \sum_{n=0}^\infty \tau_f(f^n(q)) - \tau_f(f^n(r)).$$
A similar formula exists for the unstable holonomies. This means that the regularity/behavior of $\Delta$ reflects the regularity/behavior of the holonomies, and of the (un)stable distributions. Our hypothesis that $E^s \text{ or } E^u \notin C^2$ is then reflected by the fact that we expect $\Delta$ to not be $C^1$, giving positivity of the Fourier dimension. Appendix \hyperref[ap:C]{C} make this discussion rigorous. One can interpret this by saying that the non-regularity of the distributions produces geometric chaos in the basic set $\Omega$. 
\item \underline{The suspension flow interpretation:} Consider the suspension space   $$\Omega^{\tau_f} := \{ (x,s) \ | \ s \in [0,\tau_f(x)] \}/((x,\tau_f(x)) \sim (f(x),0))$$ and the associated suspension flow $\Phi^t((x,s)) := (x,s+t)$. This flow behave like a 3-dimensional Axiom A flow. We know since Chernov and Dolgopyat \cite{Ch98,Do98,Do00,Do02} that the mixing properties of such flows is linked with the oscillations of the associated \emph{temporal distance function}, which measure the joint-non-integrability of the strong stable distribution and the strong unstable distribution. It happens that, for this suspension flow, the temporal distance function can be identified with $\Delta$. Our reduction can then be interpreted as: \emph{If the suspension flow with roof function $\tau_f$ is exponentially mixing, then $\underline{\text{dim}}_{F,C^{1+\alpha}}(\mu)>0$}. Of course, this is not the technical statement that we prove in this text, but more an intuitive interpretation. 
\end{enumerate}

\subsubsection{Tsujii-Zhang techniques to study $\Delta$}

There have been some recent breakthrough in the study of the rate of mixing of hyperbolic flows, especially in a high-dimensionnal, algebraic setting \cite{Kh23}, and in a low dimensionnal, Anosov setting \cite{TZ20}. In \cite{TZ20}, M. Tsujii and Z. Zhang study a setting similar to ours: they are interested in the mixing rate of topologically mixing three-dimensional Anosov flows. To do so, they conduct a precise study of the oscillations of their temporal distance function. In the second part of this paper, we generalize these methods in our \say{similar to 3D Axiom A flow} case. We can expect these ideas to be useful to the study of the rate of mixing of actual 3D Axiom A flows in future work. The idea can be summarized as follow. \\

The first important ingredient is the existence of \say{almost linearizing coordinates} $\iota_x : (-1,1)^2 \rightarrow M$ for each point $x \in \Omega$. The existence of these coordinates allows us to introduce \emph{templates}, objects that are inspired from the templates in $\cite{TZ20}$. Intuitively, these objects give us an identification $E^s \simeq TM/E^u$, which is useful because $E^s$ is not regular along $W^u_{loc}$, but $TM/E^u$ is. We can then consider \say{stable derivatives} along $W^u_{loc}$ in a way that is convenient for computations. In the end, using these adapted coordinates, we will be able to write the following estimate. Fix $p,s$ in the same stable manifold, think of $r \in W^u_{loc}(p)$ as the variable, and let $q := W^s_{loc}(r) \cap W^u_{loc}(s)$. We will find:
$$ \Delta(p,q) \simeq d(p,s) X_p(r) + \text{(Polynomial in r)} + (\text{small error}) $$
where the polynomial term is of bounded degree. The function $ r \mapsto X_p(r)$ that we will find will behave like a Weierstrass function, oscillating everywhere and at every scales. To conclude that $\Delta$ is oscillating enough to check the hypothesis of Lemma \ref{lem:0}, we will have to prove that these oscillations are preserved under adding polynomials of fixed degree. This is the content of section \hyperref[sec:4]{4}. Technical lemmas related to polynomials on fractals are proved in Appendix \hyperref[ap:B]{B}. Appendix \hyperref[ap:A]{A} is devoted to various technical asymptotics.

\section*{Acknowledgments}

This project was started during my PhD in 2022, under the supervision of \textbf{Frederic Naud}, who pointed out to me the work of M. Tsujii and Z. Zhang which is central for this work. I thank \textbf{Zhiyuan Zhang} for helping me understand the details of his work with M. Tsujii. I would also like to thank my postdoc advisor \textbf{Tuomas Sahlsten} for useful discussions on the topic of Fourier decay and the sum-product phenomenon. This paper would not exists without the help and support of \textbf{Semyon Dyatlov}, \textbf{Anton Gorodetksi}, \textbf{David Damanik} and \textbf{Jake Fillman}, who were working on the same topic and explained to me the potential application of Fourier decay results to the Fibonacci Hamiltonian. During this project, the author was supported by the \emph{Ecole Normale Superieure de Rennes} during 2021-2024, and then by the Research Council of Finland’s Academy Research Fellowship “Quantum chaos of large and many body systems”, grant Nos. 347365, 353738.

\section{Preliminaries: Thermodynamic formalism}\label{sec:preli}

\subsection{Axiom A diffeomorphisms on surfaces}

In this section, we recall basic facts about the Thermodynamic formalism in the setting of Axiom A diffeomorphisms on surfaces 
\cite{Bo75,KH95,BS02}. Let us first fix some notations and conventions. We will follow closely the presentation from \cite{Le24,Le23b} and we will skip the proof of some well known facts, to stay concise. More involved dynamical statements will be proved during Section \ref{sec:4}. \\

We fix for this section a smooth and complete riemannian surface $M$ and a smooth Axiom A diffeomorphism $f:M \rightarrow M$. We fix one of its basic set $\Omega \subset M$ and we make the hypothesis that:
$$ \forall x \in \Omega, \ |\det (df)_x| = 1 .$$
Recall that, by transitivity of $f:\Omega \rightarrow \Omega$ (there exists a dense orbit), a basic set is always a perfect set or a periodic orbit. We assume here that $\Omega$ is perfect. Typically, $\Omega=M=\mathbb{R}^2/\mathbb{Z}^2$, or $\Omega$ is a Cantor set (think of a product $K \times K$ where $K \subset [0,1]$ is a Cantor set). \\

We say that a function defined on a subset $K \subset M$ is $C^{r}$, for $r \in (0,\infty)$, if it can be extended to a $C^r$ function on some open neighborhood $\tilde{K}$ of $K$.
We know that the (un)stable distributions $E^u,E^s$ are $C^{1+}$ line bundle on $\Omega$ \cite{PR02}. In our setting, an interesting \emph{rigidity} phenomenon occurs: as soon as $E^{s/u} \in C^2$, then $E^{s/u} \in C^\infty$. This seems classical but precise statements are difficult to locate in the litterature (a related but weaker result for holonomies can be found in \cite{PR05}). A proof is joined in appendix \hyperref[ap:C]{C} to fill this gap.  \\

If $\varphi:\Omega \rightarrow N$ is a $C^{1}$ function from $\Omega$ to some riemanian manifold $N$, we denote $$\partial_u \varphi(x) := \| {(d\varphi)_x}(\vec{n}_x) \|_{\varphi(x)},$$ where $\text{Span}(\vec{n}_x) = E^u(x)$ and $\|\vec{n}_x\|_x=1$. This is \say{the derivative in the unstable direction}.
We define similarly $\partial_s \varphi(x)$ as the \say{derivative in the stable direction}. If $\varphi \in C^{\infty}$, then $(\partial_u \varphi, \partial_s \varphi) \in C^{1+}$ by $C^{1+}$regularity of the (un)stable distributions. Along an unstable manifold, $\partial_u \varphi \in C^\infty$. \\

Iterating $f$ if necessary, we can assume that there exists $\kappa,\kappa_- \in (0,1)$ such that $ \kappa_-^{-1} \geq \partial_u f(x) \geq \kappa^{-1} > 1$ and $0 < \kappa_- \leq \partial_s f(x) \leq \kappa < 1$ for all $x \in \Omega$. Notice that the area-preserving hypothesis ensure that there exists a Hölder function $\theta:\Omega\rightarrow \mathbb{R}$ such that
$$ \ln \partial_u f + \ln \partial_s f = \theta \circ f - \theta .$$
This can be seen as a consequence of Livsic's theorem \cite{Liv71}, or by a direct computation.
\subsection{Equilibrium states}

\begin{definition}{\cite{Bo75}, \cite{Ru78}}
Let $\psi : \Omega \rightarrow \mathbb{R}$ be a Hölder regular potential.
Define the pressure of $\psi$ by
$$ P(\psi) := \sup_{\mu \in \mathcal{P}_f(\Omega)} \left\{ h_f(\mu) + \int_\Omega \psi d\mu \right\}, $$
where $\mathcal{P}_f(\Omega)$ is the compact set of all probability measures supported on $\Omega$ that are $f$-invariant, and where $h_f(\mu)$ is the entropy of $\mu$. There exists a unique measure $\mu_\psi \in \mathcal{P}_f(\Omega)$ such that $$ P(\psi) = h_f(\mu_\psi) + \int_\Omega \psi d\mu_\psi .$$
This measure has support equal to $\Omega$, is ergodic on $(\Omega,f)$, and is called the equilibrium state associated to $\psi$.

\end{definition}

Some particular choices of potentials stands among the others. The first one is the constant potential $\psi := cst$. In this case, the equilibrium measure is the \emph{measure of maximal entropy}. Other natural choices are constant multiple of the \emph{geometric potential} $$ \psi := - t \ln \partial_u f(x). $$
A special case is the one of $t = \delta :=\dim_H(\Omega)/2$, for which we have $P(- \delta \ln \partial_u f)=0$. This is Bowen's formula \cite{Bo79,McM83}. In this case, the equilibrium state is absolutely continuous with respect to the Hausdorff measure of dimension $\dim_H(\Omega)$ restricted to $\Omega$. This measure is sometimes called the \emph{measure of maximal dimension}.\\

We fix from now on a Hölder potential $\psi:\Omega \rightarrow \mathbb{R}$ and its associated equilibrium state $\mu \in \mathcal{P}(\Omega)$.

\subsection{(Un)stable laminations, bracket and holonomies}

In this subsection, we will recall some results about the existence of stable/unstable laminations, holonomies, and their regularity.

\begin{definition}

Let $x \in \Omega$. For $\varepsilon>0$ small enough, we define the local stable and unstable manifold at $x$ by $$ W^s_{\epsilon}(x) := \{ y \in M \ | \ \forall n \geq 0, \ d(f^n(x),f^n(y)) \leq \varepsilon \}, $$
$$ W^u_\varepsilon(x) := \{ y \in M \ | \ \forall n \leq 0, \ d(f^n(x),f^n(y)) \leq \varepsilon \} .$$
We also define the global stable and unstable manifolds at $x$ by
$$ W^s(x) := \{ y \in M \ |  \ d(f^n(x),f^n(y)) \underset{n \rightarrow +\infty}{\longrightarrow} 0 \} ,$$
$$ W^u(x) := \{ y \in M \ | \ d(f^n(x),f^n(y)) \underset{n \rightarrow -\infty}{\longrightarrow} 0 \}.  $$

\end{definition}

\begin{theorem}[\cite{BS02}, \cite{KH95}, \cite{Bo75}]

Let $f$ be a $\mathcal{C}^\infty$ Axiom A diffeomorphism and let $\Omega$ be a basic set.
For $\varepsilon>0$ small enough and for $x \in \Omega$:
\begin{itemize}
    \item $W^s_\varepsilon(x)$ and $W^u_\varepsilon(x)$  are $C^\infty$ one-dimensional manifolds with boundary,
    \item $\forall y \in \Omega \cap W^s_\varepsilon(x), \ T_y W^s_\varepsilon (x) = E^s(y)$,
    \item $\forall z \in \Omega \cap W^u_\varepsilon(x), \ T_z W^u_\varepsilon (x) = E^u(z)$,
    \item $f(W^s_\varepsilon(x)) \subset \subset W^s_\varepsilon(f(x))$ and $ f(W^u_\varepsilon(x)) \supset \supset W^u_\varepsilon(f(x)) $ where $\subset \subset$ means \say{compactly included},
    \item $ \forall y \in W^s_\varepsilon(x), \ \forall n \geq 0,  \ d^s(f^n(x),f^n(y)) \leq  \kappa^n d^s(x,y) $ where $d^s$ is the arclenght along $W^s_\varepsilon$,
    \item $ \forall z \in W^u_\varepsilon(x), \ \forall n \geq 0, \ d^u(f^{-n}(x),f^{-n}(z)) \leq  \kappa^{n} d^u(x,z) $ where $d^u$ is the arclenght along $W^u_\varepsilon$.
\end{itemize}
Moreover
$$ \bigcup_{n \geq 0} f^{-n}(W^s_\varepsilon(f^n(x))) = W^s(x)  $$
and $$ \bigcup_{n \geq 0} f^{n}(W^u_\varepsilon(f^{-n}(x))) = W^u(x),$$
and so the global stable and unstable manifolds are injectively immersed manifolds in $M$.

\end{theorem}

We fix once and for all a small enough $\varepsilon_{dyn}>0$ and we will write, to simplify: $W^s_{loc}(x),W^u_{loc}(x) := W^s_{\varepsilon_{dyn}}(x),W^u_{\varepsilon_{dyn}}(x)$. The family $(W^s_{loc}(x),W^u_{loc}(x))_{x \in \Omega}$ forms two transverse continuous laminations, which allows us to define the so-called bracket and holonomies. \cite{Bo75,Ha89,KH95,BS02,PR02}

\begin{definition}
There exists $\widetilde{\text{Diag}} \subset \Omega \times \Omega$, a small open neighborhood of the diagonal, such that $W^s_{loc}(x) \cap W^u_{loc}(y)$ consists of a single point $[x,y]$ whenever $x,y \in \widetilde{\text{Diag}}$. In this case, $[x,y] \in \Omega$, and the map
$$ [ \ \cdot , \cdot \ ] : \widetilde{\text{Diag}} \longrightarrow \Omega $$
is $C^{1+\alpha}$. In particular, there exists $C>0$ such that
$$ d^s([x,y],x) \leq C d(x,y), \quad \text{and} \quad d^u([x,y],y) \leq C d(x,y).$$

\end{definition}

\begin{definition} Fix $x,y$  two close enough points in $\Omega$ lying in the same local stable manifold. Let $U^u \subset W^u_{loc}(x) \cap \Omega$ be a small open neighborhood of $x$ relatively to $ W^u_{loc}(x) \cap \Omega$. The map
$$ \pi_{x,y} : U^u \subset W^u_{loc}(x) \cap \Omega \longrightarrow W^u_{loc}(y) \cap \Omega $$
defined by $\pi_{x,y}(z) := [z,y]$ is called a stable holonomy map. One can define an unstable holonomy map similarly on pieces of local stable manifolds intersected with $\Omega$.
\end{definition}

The (un)stable lamination is $C^{1+\alpha}$ for some $\alpha>0$. In particular, the (un)stable holonomies maps are $C^{1 + \alpha}$ diffeomorphisms. More precisely, the maps $\pi_{x,y}$ extends to small curves $W^u_\delta(x) \rightarrow W^u_\varepsilon(y)$, and the extended map is a $C^{1+\alpha}$ diffeomorphism.

\subsection{Markov partitions}

In this subsection, we will construct a classical topological space on which we will work in the next section: we can approximate the iterated dynamics by a dynamical system on a (subset of a) finite disjoint union of smooth curves. For this we need to recall some results about Markov partitions.

\begin{definition}
A set $R \subset \Omega$ is called a \emph{rectangle} if 
$$ \forall x,y \in R, \ [x,y] \in R .$$
A rectangle is called proper if $\overline{\text{int}_\Omega(R)} = R$. If $x \in R$, and if $\text{diam}(R)$ is small enough, we define $$ W^s(x,R) := W^s_{loc}(x) \cap R  \quad \text{and} \quad W^u(x,R) := W^u_{loc}(x) \cap R .$$

\end{definition}

Notice that, if $\Omega$ is not connected (for example, when $\Omega$ is a Cantor set) then the rectangles are not connected either. 

\begin{definition}
A Markov partition of $\Omega$ is a finite covering $\{R_a\}_{a \in \mathcal{A}}$ of $\Omega$ by proper non-empty rectangles such that
\begin{itemize}
    \item $\text{int}_\Omega R_a \cap \text{int}_\Omega R_b = \emptyset$ if $a \neq b$
    \item $f\left( W^u(x,R_a) \right) \supset W^u(f(x),R_b)$ and $f\left( W^s(x,R_a) \right) \subset W^s(f(x),R_b)$ \\ when $x \in \text{int}_\Omega R_a \cap f^{-1}\left(\text{int}_\Omega R_b \right)$. 
\end{itemize}
\end{definition}

\begin{theorem}[\cite{Bo75},\cite{KH95}]
Let $\Omega$ be a basic set for an Axiom A diffeomorphism $f$. Then $\Omega$ has Markov partitions of arbitrary small diameter.
\end{theorem}

From now on, we fix once and for all a Markov partition $\{R_a\}_{a \in \mathcal{A}}$ of $\Omega$ with small enough diameter. Remember that $\Omega$ is a perfect set: in particular, $\text{diam } R_a >0$ for all $a \in \mathcal{A}$. 

\begin{definition}
We fix for the rest of this paper some periodic points $x_a \in \text{int}_\Omega R_a$ for all $a \in \mathcal{A}$ (this is possible by density of such points in $\Omega$). By periodicity, $x_a \notin W^u(x_b)$ when $a \neq b$. 
\end{definition}

\begin{definition}

We set, for all $a \in \mathcal{A}$, $$ S_a := W^s(x_a,R_a) \quad \text{and} \quad U_a := W^u(x_a,R_a) .$$
They are closed sets included in $R_a$, and are defined so that $[U_a,S_a] = R_a$. They will allow us to decompose the dynamics into a stable and an unstable part. Notice that the decomposition is unique:
$$ \forall x \in R_a, \exists ! (z,y) \in U_a \times S_a, \ x=[z,y] .$$

\end{definition}

The intuition of the construction to come is that, after a large enough number of iterates, $f^n$ can be approximated by a map that is only defined on $(U_a)_{a \in \mathcal{A}}$. Before heading into it, notice that the fractal nature of the sets $(U_a)_{a \in \mathcal{A}}$ might be a problem to do smooth analysis. It will be convenient for us to introduce some smooth curves that contains them and that are adapted to the dynamics.

\begin{definition}
For all $a \in \mathcal{A}$, $U_a \subset W^u_{loc}(x_a)$. Since $W^u_{loc}(x_a)$ is a smooth curve, it makes sense to consider the convex shell of $U_a$ seen as a subset of $W^u_{loc}(x_a)$. We denote it
$$ V_a := \text{Conv}_u(U_a) .$$
Each $V_a$ is then diffeomorphic to a compact interval, and contains $U_a$. Notice that $V_a \nsubseteq \Omega$: in particular, iterating forward the dynamics starting from a point $x \in V_a$ might sends it far away from our basic set.  
\end{definition}

\begin{remark}
By construction, the map $x \in V_a \mapsto |(df)_x(\vec{n}_x)| \in \mathbb{R}$, where $\vec{n}_x$ is a unitary vector tangent to $V_a$ at $x$, is a Hölder map that coincides with $|\partial_u f|$ when $x \in U_a$. Since in this case $|\partial_u f(x)| \geq \kappa^{-1} >1$, and since $f$ is $C^{\infty}$, we see that choosing our Markov partition with small enough diameters ensure that $f$ is still expanding along $V_a$.
\end{remark}

\subsection{A factor dynamics}

In this section we construct a one dimensional expanding map that will approximate our dynamics. The construction is inspired by the symbolic case and already appear in the work of Dolgopyat \cite{Do98}. 

\begin{notations}
Let $a$ and $b$ be two letters in $\mathcal{A}$.
We note $a \rightarrow b$ if $f(\text{int}_\Omega R_a) \cap \text{int}_\Omega R_b \neq \emptyset $.

\end{notations}

\begin{definition}

We define $$\mathcal{R} := \bigsqcup_{a \in A} R_a , \quad \mathcal{S} := \bigsqcup_{a \in A} S_a , \quad \mathcal{U} := \bigsqcup_{a \in A} U_a , \quad \mathcal{V} := \bigsqcup_{a \in \mathcal{A}} V_a ,$$
where $\bigsqcup$ denote a formal disjoint union. We also define $$ \mathcal{R}^{(0)} := \bigsqcup_{a \in \mathcal{A}} \text{int}_\Omega R_a \subset \mathcal{R} $$
and $$ \mathcal{R}^{(1)} := \bigsqcup_{a \rightarrow b} (\text{int}_{\Omega} R_a) \cap f^{-1}(\text{int}_{\Omega} R_b) \subset \mathcal{R}^{(0)}, $$
so that the map $f:\Omega \rightarrow \Omega$ may be naturally seen as a map $f : \mathcal{R}^{(1)} \longrightarrow \mathcal{R}^{(0)}$. We then define 
$$ \mathcal{R}^{(k)} := f^{-k}( \mathcal{R}^{(0)}) $$
and, finally, we denote the associated residual set by
$$ \widehat{\mathcal{R}} := \bigcap_{k \geq 0} \mathcal{R}^{(k)} $$
so that $f : \widehat{\mathcal{R}} \longrightarrow \widehat{\mathcal{R}}$. 
Seen as a subset of $\Omega$, $\widehat{\mathcal{R}}$ has full measure, by ergodicity of the equilibrium measure $\mu$. Hence $\mu$ can naturally be thought of as a probability measure on $\widehat{\mathcal{R}}$. 
\end{definition}

\begin{definition}
Let $\mathcal{R}/\mathcal{S}$ be the topological space defined by the equivalence relation $ x \sim y  \Leftrightarrow \exists a \in \mathcal{A}, \ y \in W^s(x,R_a) $ in $\mathcal{R}$. Let $\pi:\mathcal{R} \rightarrow \mathcal{R}/\mathcal{S}$ denote the natural projection. The map $f:\widehat{\mathcal{R}} \rightarrow \widehat{\mathcal{R}}$ induces a factor map $F : \widehat{\mathcal{R}}/\mathcal{S} \rightarrow \widehat{\mathcal{R}}/\mathcal{S}$. Moreover, the measure $\nu := \pi_* \mu$ is an $F$-invariant probability measure on $\mathcal{R}/\mathcal{S}$.
\end{definition}

\begin{remark}

There is a natural isomorphism $\mathcal{U} \simeq \mathcal{R}/\mathcal{S}$ that is induced by the inclusion $\mathcal{U} \hookrightarrow \mathcal{R}$. This allows us to identify all the precedent construction to a dynamical system on $\mathcal{U}$. Namely:

\begin{itemize}
    \item The projection $\pi: \mathcal{R} \rightarrow \mathcal{R}/\mathcal{S}$ is identified with
    $$ \begin{array}[t]{lrcl}
 \pi: & \mathcal{R} \quad & \longrightarrow & \quad \mathcal{U} \\
    & x \in R_a & \longmapsto &  [x,x_a] \in U_a  \end{array}  $$
    
    \item The factor map $F:\widehat{\mathcal{R}}/\mathcal{S} \rightarrow \widehat{\mathcal{R}}/\mathcal{S}$ is identified with
    $$ \begin{array}[t]{lrcl}
 F: & \widehat{\mathcal{U}} \quad \quad \quad & \longrightarrow & \quad \widehat{\mathcal{U}} \\
    & x \in R_a \cap f^{-1}(R_b)  & \longmapsto &  [f(x),x_b] \in U_b  \end{array}   $$
    where $\widehat{\mathcal{U}}$ is defined similarly to $\widehat{\mathcal{R}}$, but with $F$ replacing $f$ in the construction.
    
    \item The measure $\nu$ is identified to the unique measure on $\mathcal{U}$ such that:
    $$ \forall h \in C^0(\mathcal{R},\mathbb{R}) \ \text{S-constant}, \ \int_\mathcal{U} h d\nu \ = \ \int_\mathcal{R} h d\mu $$
    where $S$-constant means:  $\forall a \in \mathcal{A}, \forall x,y \in R_a, \ x \in W^s(y,R_a) \Rightarrow h(x)=h(y)$.
    
\end{itemize}
\end{remark}

\begin{remark}
Since our centers $x_a$ are periodic, $x_a \in \widehat{\mathcal{R}}$, and hence $x_a \in \widehat{\mathcal{U}}$.
\end{remark}

\begin{lemma}[\cite{Le24,Le23b}]
The set $\widehat{\mathcal{U}}$ is perfect. In particular, $\text{diam} \ U_a > 0$ for all $a \in \mathcal{A}$.

\end{lemma}

\begin{remark}
The fact that $\widehat{\mathcal{U}}$ is perfect, combined with the fact that holonomies extend to $C^{1+\alpha}$ maps, allows us to consider for $x \in \mathcal{R}$ the (absolute value of the) derivatives along the unstable direction of the holonomies in a meaningful way (without having to chose an extension). We denote them by $ \partial_u \pi(x) $.  \\

It is then known that those quantities are uniformly bounded \cite{PR02}. Even better, they can be chosen as close to $1$ as desired by taking the Markov partition small enough. In other words, for any $c < 1 < C$, one can take the Markov partition $\mathcal{R}$ so small that:
$$ \forall x \in \mathcal{R}, \ c < |\partial_u \pi(x)| \leq C .$$
This distortion bound holds even for the extensions of the holonomies on the local pieces of unstable manifolds.
\end{remark}

\begin{remark}
If the Markov partition is small enough, the map $F : \widehat{\mathcal{U}} \rightarrow \widehat{\mathcal{U}}$ is a $C^{1+}$ expanding map. Taking $\kappa \in (0,1)$ larger if necessary, one can write:
$$ \forall n \geq 0, \exists \delta>0, \forall x,y \in \widehat{\mathcal{U}}, \ d(x,y) \leq \delta \Rightarrow d^u(F^n x, F^n y) \geq \kappa^{-n} d^{u}(x,y). $$
\end{remark}

We will finally extend the dynamics on pieces of $\mathcal{V}$. 

\begin{definition}
Define, for $a \in \mathcal{A}$, $U_a^{(0)} := \text{int}_\Omega U_a$, $ V^{(0)}_a := \text{Conv}_u\left(  U_a^{(0)} \right) $, and $$ \mathcal{V}^{(0)} := \bigsqcup_{a \in \mathcal{A}} V_a^{(0)} \subset \mathcal{V}.$$ Define also, for $a,b \in \mathcal{A}$ such that $a \rightarrow b$, $U_{ab} := U_a \cap F^{-1}(U_b) $ and $U_{ab}^{(1)} := U_a^{(0)} \cap F^{-1}(U_b^{(0)})$. Finally, set $V_{ab} := \text{Conv}_u U_{ab} $, $ V_{ab}^{(1)} := \text{Conv}_u U_{ab}^{(1)}  $ and $$ \mathcal{V}^{(1)} := \bigsqcup_{a \rightarrow b} V_{ab}^{(1)} \subset \mathcal{V}^{(0)} .$$ We extends the dynamics on $\mathcal{V}$ by the formula:

$$ \begin{array}[t]{lrcl}
 F: &  \mathcal{V}^{(1)}  & \longrightarrow & \quad \mathcal{V}^{(0)} \\
    & x \in V_{ab}  & \longmapsto &  \tilde{\pi}_{f(x_a),x_b} (f(x)) \in V_b  \end{array}  $$
where $\tilde{\pi}_{f(x_a),x_b} : f(V_a) \rightarrow V_b$ is a $C^{1+\alpha}$ diffeomorphism that extends the stable holonomy between $f(U_a)$ and $U_b$. As long as the Markov partition have small enough diameter, this is a well defined (and $C^{1+\alpha}$) expanding map.
\end{definition}

\subsection{Transfer operator and symbolic formalism}

Recall that $\mu$ is the equilibrium state for some Hölder potential $\psi$ on $\Omega$. The measure $\nu := \pi_* \mu \in \mathcal{P}(\mathcal{U})$ is its pushforward by the stable holonomy map $\pi:\mathcal{R} \rightarrow \mathcal{U}$. An important tool to describe the structure of $\nu$ is the notion of transfer operator, that we recall:

\begin{definition}
For some Hölder potential $\varphi:\mathcal{U} \rightarrow \mathbb{R}$, define the associated \emph{transfer operator} $ \mathcal{L}_\varphi : C^0( \mathcal{U}^{(0)} ,\mathbb{C}) \longrightarrow  C^0( \mathcal{U}^{(1)} ,\mathbb{C}) $
by 
$$ \forall x \in \mathcal{U}^{(0)}, \  \mathcal{L}_\varphi h (x) := \sum_{y \in F^{-1}(x)} e^{\varphi(y)} h(y) .$$
Iterating $\mathcal{L}_\varphi$ gives
$$ \forall x \in \mathcal{U}^{(0)}, \ \mathcal{L}_\varphi^n h (x) = \sum_{y \in F^{-n}(x)} e^{S_n \varphi(y)} h(y) ,$$
where $S_n \varphi(z) := \sum_{k=0}^{n-1} \varphi(F^k(z))$ is a Birkhoff sum. \\

By duality, $\mathcal{L}_\varphi$ also acts on the set of measures on $\mathcal{U}$. If $m$ is a measure on $\mathcal{U}^{(1)}$, then $\mathcal{L}_\varphi^* m$ is the unique measure on $\mathcal{U}^{(0)}$ such that

$$ \forall h \in C^0_c( \mathcal{U}^{(1)}, \mathbb{C}), \ \int_{\mathcal{U}} h \ d \mathcal{L}_\varphi^* m = \int_{\mathcal{U}} \mathcal{L}_\varphi h \ dm  .$$

\end{definition}

\begin{remark}
We may rewrite the definition by highlighting the role of the inverse branches of $F$.
For some $a \rightarrow b$, we see that $F : V_{ab}^{(1)} \rightarrow V_{b}^{(0)}$ is a $C^{1+\alpha}$ diffeomorphism. We denote by $g_{ab} : V_b \rightarrow V_ {ab} \subset V_a$ its local inverse. Notice that, restricted on $\mathcal{U}$, we get an inverse branch $g_{ab} : \widehat{U}_{b} \rightarrow \widehat{U}_{ab} \subset U_a$ of $F : \widehat{\mathcal{U}} \rightarrow \widehat{\mathcal{U}}$ that satisfies the formula
$$ \forall x \in U_b, \ g_{ab}(x) := f^{-1}([ x , f(x_a) ]) .$$
With these notations, the transfer operator can be rewritten as follow:

$$ \forall x \in U_b, \ \mathcal{L}_\varphi h(x) = \sum_{a \rightarrow b} e^{\varphi(g_{ab}(x))} h(g_{ab}(x)) .$$
\end{remark}

\begin{lemma}[\cite{PP90,Cl20,Le23a}]
Denote by $\psi:\mathcal{R} \rightarrow \mathbb{R}$ the potential associated with $\mu$.
There exists a Hölder potential $\varphi:\mathcal{U} \rightarrow \mathbb{R}$ such that $\varphi < 0 $ and:
\begin{itemize}
\item $\mathcal{L}_\varphi^* \nu = \nu \quad ; \quad \mathcal{L}_\varphi(1) = 1$,
\item We have the cohomology relation $\varphi \circ \pi \sim \psi$.
\end{itemize}
\end{lemma}
A potential $\varphi$ satifying the previous lemma is called normalised, and we fix one from now on. Introducing some further symbolic notations, we can rewrite iterates of our transfer operator.

\begin{notations}
For $n \geq 1$, a word $\mathbf{a} = a_1 \dots a_n \in \mathcal{A}^n$ is said to be admissible if $a_1 \rightarrow a_2 \rightarrow \dots \rightarrow a_n$. We define:
\begin{itemize}
    \item $\mathcal{W}_n := \{ \mathbf{a} \in \mathcal{A}^n \ | \ \mathbf{a} \ \text{is admissible} \ \}$.
    \item For $\mathbf{a} \in \mathcal{W}_n$, define $g_{\mathbf{a}} := g_{a_1 a_2}  g_{a_2 a_3}  \dots g_{a_{n-1} a_{n}} : V_{a_n} \longrightarrow V_{a_{1}}$. 
    \item For $\mathbf{a} \in \mathcal{W}_n$, define $V_{\mathbf{a}} := g_{\mathbf{a}}\left( V_{a_n} \right)$, $ U_{\mathbf{a}} := g_{\mathbf{a}}\left( U_{a_n} \right) $, and $ \widehat{U}_{\mathbf{a}} := g_{\mathbf{a}}\left( \widehat{\mathcal{U}} \cap U_{a_n} \right) $ 
    \item For $\mathbf{a} \in \mathcal{W}_n$, set $x_{\mathbf{a}} := g_{\mathbf{a}}\left( x_{a_n} \right) \in \widehat{U}_{\mathbf{a}} $. 
    
    \item Let $\mathbf{a} \in \mathcal{W}_{n+1}$. We denote by $b(\mathbf{a})$ the last letter of $\mathbf{a}$, so that $g_\mathbf{a} : V_{b(\mathbf{a})} \rightarrow V_{\mathbf{a}} \subset V_{a_1}$.

\end{itemize}

\end{notations}

\begin{remark}

Since $F$ is expanding, the maps $g_{\mathbf{a}}$ are contracting as $n$ becomes large. As $\mathcal{U}$ is included in a finite union of unstable curve $\mathcal{V}$, that are one dimensional Riemannian manifolds, it makes sense to consider absolute value of the derivatives of $F$ and $g_{\mathbf{a}}$, and we will do it from now on. For points in $\mathcal{U}$, this correspond to the absolute value of the derivative in the unstable direction. We find:
$$ \forall \mathbf{a} \in \mathcal{W}_n, \ \forall x \in V_{a_n}, \ |g_{\mathbf{a}}'(x)| \leq \kappa^n  , $$
for some constant $\kappa \in (0,1)$, and it follows that
$$\forall \mathbf{a} \in \mathcal{W}_n,  \ \text{diam}(V_{\textbf{a}}) = \text{diam} \ g_{\mathbf{a}}\left( V_{a_n} \right) \leq \kappa^n .$$
A consequence for our potential is that it has $\emph{exponentially vanishing variations}$. Namely, since $\varphi$ is Hölder (with some Hölder exponent $\alpha$), the following holds:

$$ \exists C>0, \ \forall n \geq 1, \ \forall \mathbf{a} \in \mathcal{W}_n, \ \forall x,y \in U_{\mathbf{a}}, \ |\varphi(x) - \varphi(y)| \leq C \kappa^{\alpha n} .$$

\end{remark}

\begin{remark}

For a fixed $n$, the family $(U_{\mathbf{a}})_{\mathbf{a} \in W_n}$ is a partition of $\mathcal{U}$ (modulo a boundary set of zero measure). In particular, for any continuous map $g:\mathcal{U} \rightarrow \mathbb{C}$, we can write
$$ \int_{\mathcal{U}} g d \nu = \sum_{\mathbf{a} \in \mathcal{W}_n} \int_{U_{\mathbf{a}}} g d\nu.$$

\end{remark}

With these notations, we may rewrite the formula for the iterate of our transfer operator. For any continuous function $h : \mathcal{U} \longrightarrow \mathbb{C} $, we have

$$\forall b \in \mathcal{A}, \ \forall x \in U_b, \ \mathcal{L}_\varphi^n h(x) = \underset{\mathbf{a} \rightsquigarrow b}{\sum_{\mathbf{a} \in \mathcal{W}_{n+1}}} e^{S_n \varphi(g_{\mathbf{a}}(x))} h(g_{\mathbf{a}}(x)) = \underset{\mathbf{a} \rightsquigarrow b}{\sum_{\mathbf{a} \in \mathcal{W}_{n+1}}} w_{\mathbf{a}}(x) h(g_{\mathbf{a}}(x)),$$
where $$ w_{\mathbf{a}}(x) := e^{S_n\varphi(g_{\mathbf{a}}(x))} .$$

The $\mathcal{L}_\varphi$-invariance of $\nu$ yields:

\begin{lemma}[Gibbs estimates, \cite{PP90}]\label{Gibbs}
$$ \exists C_0>1, \ \forall n \geq 1, \ \forall \mathbf{a} \in \mathcal{W}_n, \  \forall x \in U_{b(\mathbf{a})}, \ \ C_0^{-1} w_\mathbf{a}(x) \leq \nu( U_{\mathbf{a}} ) \leq C_0 w_\mathbf{a}(x) .$$
\end{lemma}

The fact that $\varphi < 0$ and the one-dimensional setting allows us to deduce some regularity for our equilibrium state:

\begin{lemma}[\cite{Le23a}]\label{lem:upreg}
The measures $\nu \in \mathcal{P}(\mathcal{U})$ and $\mu \in \mathcal{P}(\Omega)$ are upper regular, that is, there exists $C\geq1$ and $\delta_{\text{reg}}>0$ such that
$$ \forall x \in \mathcal{R}, \ \forall r > 0, \quad \nu( B(x,r) ),\mu( B(x,r) ) \leq C r^{\delta_{\text{reg}}},$$
where $B(x,r)$ is a ball of center $x$ and radius $r$.
\end{lemma}

Let us finally introduce some notations to iterate $\mathcal{L}_\varphi^n$.  For $\mathbf{a}=a_1 \dots a_n a_{n+1} \in \mathcal{W}_{n+1}$, define $\mathbf{a}' := a_1 \dots a_{n} \in \mathcal{W}_n$.  For two words $\mathbf{a} \in \mathcal{W}_{n+1}$, $\mathbf{b} \in \mathcal{W}_{m+1}$, we will denote $\mathbf{a} \rightsquigarrow \mathbf{b} $ if $a_{n+1}=b_1$. In this case, $\mathbf{a}' \mathbf{b} \in \mathcal{W}_{n+m+1}$. \\

Notice that when $\mathbf{a} \rightsquigarrow \mathbf{b}$, we have $g_\mathbf{a} \circ g_\mathbf{b} = g_{\mathbf{a}'\mathbf{b}}$. In particular, iterating  $k$ times $\mathcal{L}_\varphi^n$ leads us to the formula
$$ \forall k \geq 1,  \forall x \in U_b, \quad  \mathcal{L}_\varphi^{nk} h(x)  = \sum_{\mathbf{a}_1 \rightsquigarrow \dots \rightsquigarrow \mathbf{a}_k \rightsquigarrow b} h(g_{\mathbf{a}_1 ' \dots \mathbf{a}_{k-1} ' \mathbf{a}_k}(x) ) w_{\mathbf{a}_1 ' \dots \mathbf{a}_{k-1} ' \mathbf{a}_k}(x) .$$
Gibbs estimates also allows us to see the almost-Bernouilli property of $\nu$:
\begin{lemma}
For all $k \geq 2$, there exists $C \geq 1$ such that, for any $n \geq 1$, for all words in $\mathcal{W}_{n+1}$ such that $\mathbf{a}_1 \rightsquigarrow \dots \rightsquigarrow \mathbf{a}_k$, we have:
$$ C^{-1} \nu( U_{\mathbf{a}_1' \mathbf{a}_2' \dots \mathbf{a}_k  } )  \leq \nu( U_{\mathbf{a}_1} )\nu( U_{\mathbf{a}_2} ) \dots  \nu( U_{\mathbf{a}_k} ) \leq C \nu( U_{\mathbf{a}_1' \mathbf{a}_2' \dots \mathbf{a}_k  } ) .$$
\end{lemma}

\subsection{Large deviations }

We finish this preliminary section by recalling some large deviation results. There exists a large bibliography on the subject, see for example \cite{Ki90}. Large deviations in the context of Fourier decay for some 1-dimensional shift was used for example in \cite{JS16, SS20,Le24,LPS25}. 

\begin{theorem}[\cite{Yo90,LQZ03}]\label{th:largedev}

Let $g : \mathcal{R} \rightarrow \mathbb{C}$ be any continuous map. Then, for all $\varepsilon>0$, there exists $n_0(\varepsilon)$ and $\delta_0(\varepsilon) > 0$ such that

$$ \forall n \geq n_0(\varepsilon), \ \mu \left( \left\{ x \in \Omega , \ \left| \frac{1}{n} \sum_{k=0}^{n-1} g(f^k(x)) - \int_{\Omega} g d\mu \right| \geq \varepsilon \right\} \right) \leq e^{-\delta_0(\varepsilon) n} $$

\end{theorem}

Notice that applying Theorem \ref{th:largedev} to a $S$-constant function immediately gives the same statement for $(\nu,\widehat{\mathcal{U}},F)$. This result allows us to derive order of magnitude for dynamically relevant quantities. For example, as $\partial_u(f^{-n})(x) = \exp( - S_n \tau_f(f^{-n}(x)) )$, it is natural to apply the large deviations to Birkhoff sums of $\tau_f$.

\begin{definition}\label{def:lyap}
Define $\tau_f := \ln |\partial_u f (x)| $  (so that $-\tau_f$ is the geometric potential) and $\tau_F := \ln |F'(x)|$. Beware that $\tau_F$ is only Hölder regular. Since $\pi \circ f = F \circ \pi$, we find
$$ \tau_f(x) = \tau_F(\pi(x)) - \ln |\partial_u \pi(f(x))| + \ln |\partial_u \pi(x)| .$$
In particular, $ \tau_f $ and $\tau_F \circ \pi$ are $f$-cohomologous. We can then define their associated \emph{Lyapunov exponent} by:
$$ \lambda = \int_{\Omega} \tau_f d\mu = \int_{\mathcal{U}} \tau_F d\nu > 0.$$
The \emph{dimension of $\nu$} is then defined by the formula
$$ \delta := -\frac{1}{\lambda} \int_\mathcal{U} \varphi d\nu > 0. $$
\end{definition}

Following similar computations done in \cite{SS20,Le23a,Le24, LPS25}, the large deviations gives us the following orders of magnitudes. In the next lemma, for two sequences $a_n,b_n$ depending on $\varepsilon$, the notation $a_n \sim b_n$ means that there exists constants $C,\beta \geq 1$ independent of $\varepsilon$ such that $$ C e^{- \varepsilon \beta n} b_n \leq a_n \leq C e^{\varepsilon \beta n} b_n .$$
We will also frequently denote $a_n \simeq b_n$ when $ C^{-1} b_n \leq a_n \leq C b_n $ , and $a_n \lesssim b_n$ when $a_n \leq C b_n$.

\begin{lemma}\label{lem:mag}
For all $\varepsilon>0$ small enough, for all $n \geq 1$ large enough, there exists a set of $\varepsilon$-regular words $\mathcal{R}_{n+1}(\varepsilon) \subset \mathcal{W}_{n+1}$ such that $$ \# \mathcal{R}_{n+1}(\varepsilon) \sim e^{\delta \lambda n} $$
and such that $$ \exists \delta_1(\varepsilon)>0, \ \sum_{\mathbf{a} \in \mathcal{W}_{n+1} \setminus \mathcal{R}_{n+1}(\varepsilon)} \nu(U_{\mathbf{a}}) \leq e^{-\delta_1(\varepsilon) n} .$$ Moreover, for all $\mathbf{a} \in \mathcal{R}_{n+1}(\varepsilon)$, the following estimates holds:
\begin{itemize}
    \item uniformly on $x \in V_{\mathbf{a}}$,  $|\partial_u f^{n}(x)| \sim e^{n \lambda}$,
    \item uniformly on $x \in V_{b(\mathbf{a})}, \  |g_\mathbf{a}'(x)| \sim e^{- n \lambda}$,
    \item $\mathrm{diam}(\widehat{U}_\mathbf{a}), \ \mathrm{diam}(U_\mathbf{a}), \ \mathrm{diam}(V_{\mathbf{a}})  \sim e^{- n \lambda} $,
    \item uniformly on $x \in V_{b(\mathbf{a})}, \ \nu(U_\textbf{a})  \simeq w_{\mathbf{a}}(x) \sim e^{- \delta \lambda n} $.
\end{itemize}
\end{lemma}

\begin{remark}
The definition of $\delta$ is chosen so that $ \nu(U_{\mathbf{a}}) \sim \text{diam}(U_\mathbf{a})^\delta $ for $\mathbf{a} \in \mathcal{R}_{n+1}(\varepsilon)$, which is why $\delta$ takes the role of a \say{dimension} for $\nu$.
\end{remark}

\begin{definition}
Define the set of $\varepsilon$-regular $k$-blocks by $$ \mathcal{R}_{n+1}^k(\varepsilon) = \left\{ \mathbf{A}=\mathbf{a}_1' \dots \mathbf{a}_{k-1}' \mathbf{a}_k \in \mathcal{W}_{nk+1} \ | \ \forall i \geq 1, \ \mathbf{a}_i \in \mathcal{R}_{n+1}(\varepsilon) \right\} .$$ 
\end{definition}

For a fixed $k \geq 1$ and $n$ large enough depending on $\varepsilon$, we have the order of magnitude \cite{SS20,Le23a}: $$ \# \mathcal{R}_{n+1}^k \sim e^{ k \delta \lambda n} .$$

\section{From Fourier Decay to the study of $\Delta$}\label{sec:3}

This section is devoted to the reduction of Fourier decay to the study of $\Delta$. We will follow closely the reduction presented in the thesis \cite{Le24}. We say that the \say{Quantitative Nonlinearity Condition} \hyperref[QNL]{(QNL)} is satisfied if:

\begin{definition}[QNL]\label{QNL}
For all Markov partition $(R_a)_{a \in \mathcal{A}}$ of small enough diameter, there exists $C \geq 1$ and $\Gamma>0$ such that:
$$ \sum_{a \in \mathcal{A}} \mu^{\otimes 2}\Big( (p,q) \in R_a \ \Big| \ |\Delta(p,q)| \leq \sigma \Big) \leq C \sigma^\Gamma ,$$
where
$$ \Delta(p,q) := \sum_{n \in \mathbb{Z}} \tau_f(f^n p) -  \tau_f(f^n [p,q]) -  \tau_f(f^n [q,p]) +  \tau_f(f^n q) .$$
\end{definition}

The goal of this section is to use the sum-product phenomenon to prove the following reduction:

\begin{theorem}\label{th:reduc1}
Let $f:M \rightarrow M$ be a smooth Axiom A map on a complete riemannian surface. Let $\Omega$ be a basic set. Let $\mu$ be an equilibrium state associated to some Hölder potential such that \hyperref[QNL]{(QNL)} holds.
Fix $\alpha \in (0,1)$. Then, there exists $\rho_1,\rho_2>0$ such that the following holds. \\

For any $\alpha$-Hölder map $\chi : \Omega 
\rightarrow \mathbb{C}$, there exists $C=C({f,\mu,\chi}) \geq 1$ such that, for any $\xi \geq 1$ and any $C^{1+\alpha}$ phase $\psi: \Omega \rightarrow \mathbb{R}$ satisfying $$\|\psi\|_{C^{1+\alpha}} + \big(\inf_{\text{supp}(\chi)} |\partial_u \psi| \ \big)^{-1} \leq \xi^{\rho_1}, $$
we have:
$$ \Big{|} \int_\Omega e^{i \xi \psi } \chi d\mu \Big{|} \leq C \xi^{-\rho_2} .$$
\end{theorem}

The previous result applies also when $f$ is not necessarily area-preserving. When $f$ is area-preserving and when \hyperref[QNL]{(QNL)} is satisfied, the fact that $\tau_f$ is cohomologous to $-\ln \partial_s f$ allows us see that \hyperref[QNL]{(QNL)} is also satisfied for the reversed dynamics $f^{-1}$. We then find Fourier decay for phases such that $\partial_s \psi \neq 0$ instead of $\partial_u \psi \neq 0$. In general, for any phase such that $\nabla \psi \neq 0$, one can show Fourier decay by using a partition of unity and reducing to the two previous cases. In particular, we find the following (see \cite{Le24} for details):

\begin{corollary}\label{co:reduc2}
Let $f:M \rightarrow M$ be a smooth Axiom A map on a complete riemannian surface. Let $\Omega$ be a basic set such that for $|\det df| = 1$ on $\Omega$. Let $\mu$ be an equilibrium state associated to some Hölder potential such that \hyperref[QNL]{(QNL)} holds. Let $\alpha \in (0,1)$.  Then there exists $\rho \in (0,1)$ such that the following hold. \\

For any $\alpha$-Hölder map $\chi : \Omega 
\rightarrow \mathbb{C}$, for any $C^{1+\alpha}$ local chart $\varphi: U \rightarrow \mathbb{R}^2$ with $U \supset \text{supp}(\chi)$, there exists $C=C({f,\mu,\chi,\varphi}) \geq 1$ such that:
$$\forall \xi \in \mathbb{R}^2 \setminus \{0\}, \ \Big{|} \int_\Omega e^{i \xi \cdot \varphi(x) } \chi(x) d\mu(x) \Big{|} \leq C |\xi|^{-\rho} .$$
\end{corollary}

\subsection{Reduction to sums of exponential}

We start the proof of Theorem \ref{th:reduc1} (which implies Theorem \ref{th:main1} and Theorem \ref{th:main2}). \\

We fix a $C^{3}$ Axiom A diffeomorphism $f:M \rightarrow M$ on a complete surface $M$ and we choose a (perfect) basic set $\Omega \subset M$. We have fixed a Hölder potential and its associated equilibrium state $\mu$. We fix a small enough Markov partition $(R_a)_{a \in \mathcal{A}}$. We denote by $(F,\mathcal{U},\nu)$ the expanding dynamical system in factor defined in section \hyperref[sec:preli]{2}. The measure $\nu$ is invariant by a transfer operator $\mathcal{L}_\varphi$, where $\varphi$ is some normalized and $\alpha$-Hölder potential on $\mathcal{U}$. We suppose \hyperref[QNL]{(QNL)}. The Hölder exponent $\alpha$ is fixed for the rest of the paper, small enough so that any Holder map that will appear in the text will be $C^\alpha$. \\

We let $\chi \in C^\alpha(M,\mathbb{C})$ be a Hölder function with compact support, let $\xi \geq 1$, and we let $\psi: \omega \subset M \rightarrow \mathbb{R}$ be a $C^{1+\alpha}$ phase on some open set $\omega \subset M$ containing the support of $\chi$, satisfying:
$$ \|\psi\|_{C^{1+\alpha}} + (\inf_{\omega} |\partial_u \psi|)^{-1} \leq \xi^{\rho_1} ,$$
for some $\rho_1>0$ that will be chosen small enough during the proof of Lemma \ref{lem:step1} and \ref{lem:reduc2}. Our goal is to prove power decay of the following oscillatory integral:
$$ \widehat{\psi_*(\chi d\mu)}(\xi) = \int_{\Omega} e^{- 2 i \pi \xi \psi(x)} \chi(x) d\mu(x) .$$

To do so, six quantities will be at play: $\xi$, $K$, $n$, $k$, $\varepsilon_0$ and $\varepsilon$. We will think of $k$, $\varepsilon_0$ and $\varepsilon$ as being fixed. The constant $k$ is fixed using Theorem \ref{th:sumprod} and Proposition \ref{prop:ncQNL}. The constant $\varepsilon_0>0$ will be fixed at the end of section \ref{sec:3}, during the proof of Proposition \ref{prop:ncQNL}. The constant $\varepsilon>0$ is chosen at the very end of the proof to be small in front of every other constant that might appear. 
The only variable in this section is $\xi$. We relate it to $n$ and $K$ by the formulas $$ n := \left\lfloor \frac{\ln |\xi|}{(2k+1) \lambda + \varepsilon_0} \right\rfloor \quad \text{and} \quad K := \left\lfloor \frac{((2k+1) \lambda + 2 \varepsilon_0)n}{\alpha |\ln \kappa|} \right\rfloor \geq n. $$

\begin{notations}
We recall a final set of notations, inspired from \cite{BD17}. For a fixed $n$ and $k$, denote:

\begin{itemize}
    \item $\textbf{A}=(\textbf{a}_0, \dots, \textbf{a}_k) \in \mathcal{W}_{n+1}^{k+1} \ , \ \textbf{B}=(\textbf{b}_1, \dots, \textbf{b}_k) \in \mathcal{W}_{n+1}^{k} $.
    \item We write $\textbf{A} \leftrightarrow \textbf{B}$ iff $\textbf{a}_{j-1} \rightsquigarrow \textbf{b}_j \rightsquigarrow \textbf{a}_j$ for all $j=1,\dots k$.
    \item If $\textbf{A} \leftrightarrow \textbf{B}$, then we define the words $\textbf{A} * \textbf{B} := \textbf{a}_0' \textbf{b}_1' \textbf{a}_1' \textbf{b}_2' \dots \textbf{a}_{k-1}' \textbf{b}_k' \textbf{a}_k$  \\ and  $\textbf{A} \# \textbf{B} :=  \textbf{a}_0' \textbf{b}_1' \textbf{a}_1' \textbf{b}_2' \dots \textbf{a}_{k-1}' \textbf{b}_k$.
    \item Denote by $b(\textbf{A}) \in \mathcal{A}$ the last letter of $\textbf{a}_k$.

\end{itemize}
\end{notations}

Our first goal is to prove the following reduction. This kind of bounds is now classical \cite{BD17,LNP19,SS20,Le21} and is proved in this form in \cite{Le24,Le23a}, but we will recall the proof for the reader's convenience.
\begin{proposition}\label{prop:reduc}
Define $J_n := \{ e^{\varepsilon_0 n/2} \leq |\eta| \leq  e^{2 \varepsilon_0 n} \}$ and, for $j \in \{1, \dots, k\}$, $\mathbf{A} \in \mathcal{W}_{n+1}^{k+1}$, and $\mathbf{b} \in \mathcal{W}_{n+1}$ such that $\textbf{a}_{j-1} \rightsquigarrow \textbf{b} \rightsquigarrow \textbf{a}_j$, $$  \zeta_{j,\mathbf{A}}(\mathbf{b}) := e^{2 \lambda n} |g_{\mathbf{a}_{j-1}' \mathbf{b}}'(x_{\mathbf{a}_j})| .$$
There exists some constant $\beta>1$ such that for $n$ large enough depending on $\varepsilon$:
$$ |\widehat{\psi_*(\chi d\mu)}(\xi)|^2 \lesssim e^{ \varepsilon \beta n} e^{-\lambda \delta (2k+1) n} \sum_{\mathbf{A} \in \mathcal{R}_{n+1}^{k+1}(\varepsilon)} \sup_{\eta \in J_n} \Bigg{|} \underset{\mathbf{A} \leftrightarrow \mathbf{B}}{\sum_{\mathbf{B} \in \mathcal{R}_{n+1}^k(\varepsilon)}} e^{i \eta \zeta_{1,\mathbf{A}}(\mathbf{b}_1) \dots \zeta_{k,\mathbf{A}}(\mathbf{b}_k) } \Bigg| $$
$$ \quad \quad  \quad \quad  \quad \quad  \quad \quad + e^{-  \varepsilon_0 n/2} +  e^{-\delta_1(\varepsilon) n} + e^{\varepsilon \beta n} \left( e^{- \lambda \alpha n} +  \kappa^{\alpha n} + e^{-(\alpha \lambda-\varepsilon_0) n/2} + e^{- \varepsilon_0 \delta_{reg}n/6} \right).$$
\end{proposition}

Once Proposition \ref{prop:reduc} is established, if we manage to prove that the sum of exponentials enjoys exponential decay in $n$ under \hyperref[QNL]{(QNL)}, then choosing $\varepsilon$ small enough will allow us to see that $|\widehat{\psi_*(\chi d\mu)}(\xi)|^2$ enjoys polynomial decay in $\xi$, and Theorem \ref{th:reduc1} will be proved. We prove Proposition \ref{prop:reduc} through a succession of lemmas. Our first step is to approximate our oscillatory integral over $\mu$ by an oscillatory integral over $\nu$.

\begin{lemma}\label{lem:step1}
If $\rho_1$ is chosen small enough (depending on $\varepsilon_0$ and on the dynamics), then there exists $C>0$ such that, for $n$ large enough:

$$ \left| \int_{\Omega} e^{i \xi \psi} \chi d\mu - \int_{\mathcal{U}} e^{ i \xi \psi \circ f^K} \chi \circ f^K d \nu \right| \leq C e^{- \varepsilon_0 n/2} .$$
\end{lemma}

\begin{proof}

Let $h(x) := e^{ i \xi \psi(x)} \chi(x) $. Then, since $\mu$ is $f$ invariant, $$ \widehat{\psi_*(\chi d\mu)}(-\xi/2\pi) = \int_\Omega h \circ f^K d\mu .$$

We then check that $h \circ f^K$ is close to a $S$-constant function as $K$ grows larger.
First of all, notice that since $\psi$ is $C^1$, $e^{i \xi \psi}$ also is. In particular it is Lipschitz with constant $|\xi| \|\psi\|_{C^1} \leq |\xi|^{1+\rho_1}$. Since $\Omega$ is compact, it follows that $e^{- i \xi \psi}$ is $\alpha$-Hölder on $\Omega$ with constant $|\xi|^{1+\rho_1} \text{diam}(\Omega)^{1-\alpha}$. Since $\chi$ is $\alpha$-Hölder too, with some constant $C_{\chi}$, the product $h$ is also locally $\alpha$-Hölder, with some constant $\lesssim |\xi|^{1+\rho_1} $. \\

Now, let $x \in \mathcal{R}_a$. By definition, $\pi(x) \in U_a$ is in $W^s_{loc}(x)$. Then $$ |h(f^K(x)) - h(f^K(\pi(x))) | \leq \|h\|_{C^\alpha} d( f^K(x),f^K(\pi(x)) )^\alpha $$ $$\leq \|h\|_{C^\alpha} \kappa^{\alpha K} \lesssim |\xi|^{1+\rho_1} \kappa^{\alpha K} .$$
The relation between $\xi$,$n$ and $K$ is chosen so that $|\xi| \kappa^{\alpha K} \simeq e^{-\varepsilon_0 n}$. If $\rho_1$ is taken small enough so that $|\xi|^{\rho_1}\lesssim e^{ \varepsilon_0 n/2}$, then there exists a constant $C>1$ such that 
$$ \| h \circ f^K - h \circ f^K \circ \pi \|_{\infty,\mathcal{R}} \leq C e^{- \varepsilon_0 n/2} .$$
The desired estimates follows from the definition of $\nu$. \end{proof}

Now we are ready to adapt the existing strategy for one dimensional expanding maps. From this point, we will follow \cite{SS20}. Notice however that we use an additional word $\mathbf{C}$ to cancel the distortions induced by $f^K$.

\begin{lemma}
$$ \left| \int_{\mathcal{U}} e^{i \xi \psi \circ f^K} \chi \circ f^K d \nu \right|^2 \lesssim \quad \quad \quad \quad \quad \quad \quad \quad \quad \quad \quad \quad \quad \quad \quad \quad \quad \quad \quad \quad \quad \quad \quad \quad \quad \quad \quad \quad \quad \quad $$ $$\quad \quad \quad  \Bigg{|} \sum_{\mathbf{C} \in \mathcal{R}_{K+1}(\varepsilon)} \underset{\mathbf{C} \rightsquigarrow \mathbf{A} \leftrightarrow \mathbf{B}}{ \underset{ \mathbf{B} \in \mathcal{R}_{n+1}^k(\varepsilon)}{ \underset{\mathbf{A} \in \mathcal{R}_{n+1}^{k+1}(\varepsilon) }{ \sum}} } \int_{U_{b(\mathbf{A})}} e^{i \xi \psi( f^K g_{\mathbf{C}'(\mathbf{A} * \mathbf{B})}(x))}  \chi(f^K g_{\mathbf{C}'(\mathbf{A} * \mathbf{B})}(x))  w_{\mathbf{C}'(\mathbf{A} * \mathbf{B})}(x)d\nu(x) \Bigg{|}^2 + e^{- \delta_1(\varepsilon) n }. $$

\end{lemma}

\begin{proof}

Since $\nu$ is invariant by $\mathcal{L}_\varphi$, we can write 
$$ \int_{\mathcal{U}} h \circ f^K d \nu = \sum_{\mathbf{C} \in \mathcal{W}_{K+1}}  \int_{U_{b(\mathbf{C})}} h\left(f^K(g_{ \mathbf{C}}(x))\right) w_{\mathbf{C}}(x) d\nu(x) .$$
 
If we look at the part of the sum where $\mathbf{C}$ is not a regular word, we get by \hyperref[Gibbs]{Gibbs} estimates and Lemma \ref{lem:mag}:
$$ \Bigg| \sum_{\mathbf{C} \notin \mathcal{R}_{K+1}(\varepsilon)} \int_{U_{b(\mathbf{A})}} h\left(f^K(g_{\mathbf{C}}(x))\right) w_{\mathbf{C}}(x) d\nu(x) \Bigg| \lesssim \sum_{\mathbf{C} \notin \mathcal{R}_{K+1}(\varepsilon)} \nu(U_{\mathbf{C}}) \lesssim e^{-\delta_1(\varepsilon) K}, $$
which decays exponentially in $K$, hence in $n$ since $K \geq n$.
Then, we iterate our transfer operator again on the main sum, to get the following term:
$$  \sum_{\mathbf{C} \in \mathcal{R}_{K+1}(\varepsilon)} \underset{\mathbf{C} \rightsquigarrow \mathbf{A} \leftrightarrow \mathbf{B}}{ \underset{ \mathbf{B} \in \mathcal{W}_{n+1}^k}{ \underset{\mathbf{A} \in \mathcal{W}_{n+1}^{k+1} }{ \sum}} } \int_{U_{b(\mathbf{A})}} e^{i \xi \psi( f^K g_{\mathbf{C}'(\mathbf{A} * \mathbf{B})}(x))}  \chi(f^K g_{\mathbf{C}'(\mathbf{A} * \mathbf{B})}(x))  w_{\mathbf{C}'(\mathbf{A} * \mathbf{B})}(x)d\nu(x) . $$
Looking at the part of the sum where words are not regular again, we see by \hyperref[Gibbs]{Gibbs} estimates that
$$ \Bigg| \sum_{\mathbf{C} \in \mathcal{R}_{K+1}(\varepsilon)} \underset{\text{or }\mathbf{B} \notin \mathcal{R}_{n+1}^k(\varepsilon)}{\underset{\mathbf{A} \notin \mathcal{R}_{n+1}^{k+1}(\varepsilon) }{\sum_{\mathbf{C} \rightsquigarrow \mathbf{A} \leftrightarrow \mathbf{B}}}} \int_{U_{b(\mathbf{A})}} h\left(f^K(g_{\mathbf{C}'(\mathbf{A} * \mathbf{B})}(x))\right) w_{\mathbf{C}' (\mathbf{A} * \mathbf{B})}(x) d\nu(x) \Bigg| \quad \quad  \quad \quad  \quad \quad  \quad \quad  \quad \quad $$ $$  \quad \quad  \quad \quad  \quad \quad  \quad \quad  \lesssim \sum_{\mathbf{C} \in \mathcal{R}_{K+1}(\varepsilon)} \underset{\text{or }\mathbf{B} \notin \mathcal{R}_{n+1}^k(\varepsilon)}{\underset{\mathbf{A} \notin \mathcal{R}_{n+1}^{k+1}(\varepsilon) }{\sum_{\mathbf{C} \rightsquigarrow \mathbf{A} \leftrightarrow \mathbf{B}}}} \nu(U_{\mathbf{C}}) \nu(U_{\mathbf{A} * \mathbf{B}}) \lesssim e^{-\delta_1(\varepsilon)n}, $$
by Lemma \ref{lem:mag}. \end{proof}

\begin{lemma}

Define $ \chi_{\mathbf{C}}(\mathbf{a}_0) := \chi(f^K g_{\mathbf{C}}(x_{\mathbf{a}_0}))$. There exists some constant $\beta>0$ such that, for $n$ large enough:

$$  \Bigg{|}\sum_{\mathbf{C} \in \mathcal{R}_{K+1}(\varepsilon)} \underset{\mathbf{C} \rightsquigarrow \mathbf{A} \leftrightarrow \mathbf{B}}{ \underset{ \mathbf{B} \in \mathcal{R}_{n+1}^k(\varepsilon)}{ \underset{\mathbf{A} \in \mathcal{R}_{n+1}^{k+1}(\varepsilon) }{ \sum}} } \int_{U_{b(\mathbf{A})}} e^{i \xi \psi( f^K g_{\mathbf{C}'(\mathbf{A} * \mathbf{B})}(x))} w_{\mathbf{C}'(\mathbf{A} * \mathbf{B})}(x) \chi(f^K g_{\mathbf{C}'(\mathbf{A} * \mathbf{B})}(x)) d\nu(x) \Bigg{|}^2  \quad \quad \quad \quad \quad \quad \quad $$ $$ \quad \quad \quad \quad  \lesssim  \Bigg{|} \sum_{\mathbf{C} \in \mathcal{R}_{K+1}(\varepsilon)} \underset{\chi_{\mathbf{C}}(\mathbf{a}_0) \neq 0}{\underset{\mathbf{C} \rightsquigarrow \mathbf{A} \leftrightarrow \mathbf{B}}{ \underset{ \mathbf{B} \in \mathcal{R}_{n+1}^k(\varepsilon)}{ \underset{\mathbf{A} \in \mathcal{R}_{n+1}^{k+1}(\varepsilon) }{ \sum}} }} \chi_{\mathbf{C}}(\mathbf{a}_0) \int_{U_{b(\mathbf{A})}} e^{i \xi \psi( f^K g_{\mathbf{C}'(\mathbf{A} * \mathbf{B})}(x)))} w_{\mathbf{C}'(\mathbf{A} * \mathbf{B})}(x) d\nu(x) \Bigg{|}^2 + e^{\varepsilon \beta n} e^{- \lambda \alpha n}. $$

\end{lemma}

\begin{proof}
Denote $ \chi_{\mathbf{C}}(\mathbf{a}_0) := \chi(f^K g_{\mathbf{C}}(x_{\mathbf{a}_0})). $ The orders of magnitude of Lemma \ref{lem:mag} combined gives us
$$ \left| \chi(f^K g_{\mathbf{C}} g_{\mathbf{A}*\mathbf{B}}(x)) -  \chi_{\mathbf{C}}(\mathbf{a}_0)  \right| \leq \|\chi\|_{C^\alpha} d\left(f^K g_{\mathbf{C}} g_{\mathbf{A}*\mathbf{B}}(x),f^K g_{\mathbf{C}} (x_{\mathbf{a}_0} ) \right)^\alpha $$ 
$$ \lesssim \|\partial_u (f^K)\|_{\infty,U_{\mathbf{C}}}^\alpha \| g_{\mathbf{C}}' \|_{\infty,U_{b(\mathbf{C})}}^\alpha  \text{diam}(U_{\mathbf{a}_0 })^\alpha  \lesssim e^{\varepsilon \beta n} e^{- \lambda \alpha n} .$$
Hence, by \hyperref[Gibbs]{Gibbs} estimates 
$$ { \Bigg{|} {\sum_{\mathbf{A},\mathbf{B},\mathbf{C}}} \int_{U_{b(\mathbf{A})}} e^{i \xi \psi( f^K g_{\mathbf{C}'(\mathbf{A} * \mathbf{B})} )} \chi( f^K g_{\mathbf{C}'(\mathbf{A} * \mathbf{B})})  w_{\mathbf{C}'(\mathbf{A} * \mathbf{B})} d\nu  -   \sum_{\mathbf{A},\mathbf{B},\mathbf{C}} \chi_{\mathbf{C}}(\mathbf{a}_0)  \int_{U_{b(\mathbf{A})}} e^{i \xi \psi(f^K g_{\mathbf{C}'(\mathbf{A} * \mathbf{B})})} w_{\mathbf{C}'(\mathbf{A} * \mathbf{B})} d\nu \Bigg{|} }$$
$$ \lesssim e^{\varepsilon \beta n} e^{- \lambda \alpha n}  \sum_{\mathbf{A},\mathbf{B},\mathbf{C}} \nu(U_{\mathbf{C}'(\mathbf{A} * \mathbf{B}) }) \lesssim e^{\varepsilon \beta n} e^{- \lambda \alpha n} .$$
The bound follow. \end{proof}

\begin{lemma}
There exists some constant $\beta>0$ such that, for $n$ large enough:
$$ \Bigg{|} \sum_{\mathbf{C} \in \mathcal{R}_{K+1}(\varepsilon)} \underset{\chi_{\mathbf{C}}(\mathbf{a}_0) \neq 0}{\underset{\mathbf{C} \rightsquigarrow \mathbf{A} \leftrightarrow \mathbf{B}}{ \underset{ \mathbf{B} \in \mathcal{R}_{n+1}^k(\varepsilon)}{ \underset{\mathbf{A} \in \mathcal{R}_{n+1}^{k+1}(\varepsilon) }{ \sum}} }} \chi_{\mathbf{C}}(\mathbf{a}_0) \int_{U_{b(\mathbf{A})}} e^{i \xi \psi( f^K g_{\mathbf{C}'(\mathbf{A} * \mathbf{B})}(x)))} w_{\mathbf{C}'(\mathbf{A} * \mathbf{B})}(x) d\nu(x) \Bigg{|}^2 \quad \quad \quad \quad \quad \quad \quad  $$
$$ \quad \quad \quad \quad \lesssim  e^{-(2k-1) \lambda \delta n} e^{- \lambda \delta K} \sum_{\mathbf{C} \in \mathcal{R}_{K+1}(\varepsilon)} \underset{\chi_{\mathbf{C}}(\mathbf{a}_0) \neq 0}{\underset{\mathbf{C} \rightsquigarrow \mathbf{A} \leftrightarrow \mathbf{B}}{ \underset{ \mathbf{B} \in \mathcal{R}_{n+1}^k(\varepsilon)}{ \underset{\mathbf{A} \in \mathcal{R}_{n+1}^{k+1}(\varepsilon) }{ \sum}} }} \Bigg{|} \int_{U_{b(\mathbf{A})}}  e^{i \xi \psi( f^K( g_{\mathbf{C}'(\mathbf{A} * \mathbf{B})}(x)))} w_{\mathbf{a}_k}(x) d\nu(x) \Bigg{|}^2 + e^{\varepsilon \beta n} \kappa^{\alpha n}. $$
\end{lemma}

\begin{proof}

Notice that $w_{\mathbf{C}'(\mathbf{A} * \mathbf{B})}(x)$ and $w_{\mathbf{a}_k}(x)$ are related by
$$ w_{\mathbf{C}'(\mathbf{A} * \mathbf{B})}(x) = w_{\mathbf{C}'(\mathbf{A} \# \mathbf{B})}(g_{\mathbf{a}_k}(x)) w_{\mathbf{a}_k}(x) .$$
Moreover:
$$  \frac{w_{\mathbf{C}'(\mathbf{A} \# \mathbf{B})}(g_{\mathbf{a}_k}(x))}{w_{\mathbf{C}'(\mathbf{A} \# \mathbf{B})}(x_{\mathbf{a}_k})} = \exp \left( S_{K+2kn}\varphi(g_{\mathbf{C}'(\mathbf{A} \# \mathbf{B})}(g_{\mathbf{a}_k}(x))) - S_{K+2kn}\varphi( g_{\mathbf{C}'(\mathbf{A} \# \mathbf{B})}(x_{\mathbf{a}_k})) \right) ,$$
with 
$$ \left| S_{K+2kn}\varphi(g_{\mathbf{C}'(\mathbf{A} \# \mathbf{B})}(g_{\mathbf{a}_k}(x)) - S_{K+2kn}\varphi( g_{\mathbf{C}'(\mathbf{A} \# \mathbf{B})}(x_{\mathbf{a}_k}))  \right| \lesssim \sum_{j=0}^{K+2kn-1} \kappa^{\alpha(K+n(2k+1) - j )} \lesssim \kappa^{\alpha n} $$
since $\varphi$ is $\alpha$-Hölder. Hence, there exists some constant $C>0$ such that
$$ e^{-C \kappa^{\alpha n}} w_{\mathbf{C}'(\mathbf{A} \# \mathbf{B})}(x_{\mathbf{a}_k}) \leq w_{\mathbf{C}'(\mathbf{A} \# \mathbf{B})}(g_{\mathbf{a}_k}(x)) \leq e^{C \kappa^{\alpha n}} w_{\mathbf{C}'(\mathbf{A} \# \mathbf{B})}(x_{\mathbf{a}_k}), $$
which gives
$$ \left| w_{\mathbf{C}'(\mathbf{A} \# \mathbf{B})}(g_{\mathbf{a}_k}(x)) -  w_{\mathbf{C}'(\mathbf{A} \# \mathbf{B})}(x_{\mathbf{a}_k})  \right| \leq \max\left| e^{\pm C \kappa^{\alpha n}} -1 \right|  w_{\mathbf{C}'(\mathbf{A} \# \mathbf{B})}(x_{\mathbf{a}_k}) \lesssim \kappa^{\alpha n} w_{\mathbf{C}'(\mathbf{A} \# \mathbf{B})}(x_{\mathbf{a}_k}) .$$
Hence
$$ \Bigg{|}  \sum_{\mathbf{C} \in \mathcal{R}_{K+1}(\varepsilon)} \underset{\chi_{\mathbf{C}}(\mathbf{a}_0) \neq 0}{\underset{\mathbf{C} \rightsquigarrow \mathbf{A} \leftrightarrow \mathbf{B}}{ \underset{ \mathbf{B} \in \mathcal{R}_{n+1}^k(\varepsilon)}{ \underset{\mathbf{A} \in \mathcal{R}_{n+1}^{k+1}(\varepsilon) }{ \sum}} }} \chi_{\mathbf{C}}(\mathbf{a}_0) \int_{U_{b(\mathbf{A})}} e^{i \xi \psi(f^K g_{\mathbf{C}'(\mathbf{A} * \mathbf{B})}(x))} \left(w_{\mathbf{C}'(\mathbf{A} * \mathbf{B})}(x) - w_{\mathbf{C}'(\mathbf{A} \# \mathbf{B})}(x_{\mathbf{a}_k}) w_{\mathbf{a}_k}(x) \right) d\nu(x) \Bigg{|}   $$

$$ \lesssim \kappa^{\alpha n}  \sum_{\mathbf{A},\mathbf{B},\mathbf{C}}\int_{U_{b(\mathbf{A})}}  w_{\mathbf{C}'(\mathbf{A} \# \mathbf{B})}(x_{\mathbf{a}_k}) w_{\mathbf{a}_k}(x) d\nu(x) \ \lesssim e^{\varepsilon \beta n} \kappa^{\alpha n} $$
by \hyperref[Gibbs]{Gibbs} estimates. It follows that $$ \Bigg{|} \sum_{\mathbf{C} \in \mathcal{R}_{K+1}(\varepsilon)} \underset{\chi_{\mathbf{C}}(\mathbf{a}_0) \neq 0}{\underset{\mathbf{C} \rightsquigarrow \mathbf{A} \leftrightarrow \mathbf{B}}{ \underset{ \mathbf{B} \in \mathcal{R}_{n+1}^k(\varepsilon)}{ \underset{\mathbf{A} \in \mathcal{R}_{n+1}^{k+1}(\varepsilon) }{ \sum}} }} \chi_{\mathbf{C}}(\mathbf{a}_0) \int_{U_{b(\mathbf{A})}} e^{i \xi \psi( f^K g_{\mathbf{C}'(\mathbf{A} * \mathbf{B})}(x)))} w_{\mathbf{C}'(\mathbf{A} * \mathbf{B})}(x) d\nu(x) \Bigg{|}^2 \quad \quad \quad \quad \quad \quad \quad  $$
$$ \quad \quad \lesssim  e^{\varepsilon \beta n} 
\kappa^{\alpha n} +  \Bigg{|}  \sum_{\mathbf{C} \in \mathcal{R}_{K+1}(\varepsilon)} \underset{\chi_{\mathbf{C}}(\mathbf{a}_0) \neq 0}{\underset{\mathbf{C} \rightsquigarrow \mathbf{A} \leftrightarrow \mathbf{B}}{ \underset{ \mathbf{B} \in \mathcal{R}_{n+1}^k(\varepsilon)}{ \underset{\mathbf{A} \in \mathcal{R}_{n+1}^{k+1}(\varepsilon) }{ \sum}} }} w_{\mathbf{C}'(\mathbf{A} \# \mathbf{B})}(x_{\mathbf{a}_k})  \chi_{\mathbf{C}}(\mathbf{a}_0) \int_{U_{b(\mathbf{A})}} e^{i \xi \psi(f^K g_{\mathbf{C}'(\mathbf{A} * \mathbf{B})}(x) ) }   w_{\mathbf{a}_k}(x) d\nu(x) \Bigg{|}^2 . $$
Moreover, by Cauchy-Schwarz and by the orders of magnitude of Lemma \ref{lem:mag},
$$ \Bigg{|}  \sum_{\mathbf{C} \in \mathcal{R}_{K+1}(\varepsilon)} \underset{\chi_{\mathbf{C}}(\mathbf{a}_0) \neq 0}{\underset{\mathbf{C} \rightsquigarrow \mathbf{A} \leftrightarrow \mathbf{B}}{ \underset{ \mathbf{B} \in \mathcal{R}_{n+1}^k(\varepsilon)}{ \underset{\mathbf{A} \in \mathcal{R}_{n+1}^{k+1}(\varepsilon) }{ \sum}} }} w_{\mathbf{C}'(\mathbf{A} \# \mathbf{B})}(x_{\mathbf{a}_k})  \chi_{\mathbf{C}}(\mathbf{a}_0) \int_{U_{b(\mathbf{A})}} e^{i \xi \psi(f^K g_{\mathbf{C}'(\mathbf{A} * \mathbf{B})}(x) ) }   w_{\mathbf{a}_k}(x) d\nu(x) \Bigg{|}^2   $$
$$  \lesssim e^{\varepsilon \beta n} e^{-\lambda \delta (2k-1) n} e^{- \lambda \delta K} \sum_{\mathbf{C} \in \mathcal{R}_{K+1}(\varepsilon)} \underset{\chi_{\mathbf{C}}(\mathbf{a}_0) \neq 0}{\underset{\mathbf{C} \rightsquigarrow \mathbf{A} \leftrightarrow \mathbf{B}}{ \underset{ \mathbf{B} \in \mathcal{R}_{n+1}^k(\varepsilon)}{ \underset{\mathbf{A} \in \mathcal{R}_{n+1}^{k+1}(\varepsilon) }{ \sum}} }}\left| \int_{U_{b(\mathbf{A})}} e^{i \xi \psi(f^K g_{\mathbf{C}' (\mathbf{A} * \mathbf{B})}(x) ) }   w_{\mathbf{a}_k}(x) d\nu(x) \right|^2 ,$$
where one could increase $\beta$ if necessary.\end{proof}

\begin{lemma}
Define $$ \Delta_{\mathbf{A},\mathbf{B},\mathbf{C}}(x,y) :=  \psi( f^K g_{\mathbf{C}'(\mathbf{A} * \mathbf{B})}(x)) - \psi( f^K g_{\mathbf{C}'(\mathbf{A} * \mathbf{B})}(y)). $$
There exists some constant $\beta>0$ such that, for $n$ large enough:
$$ e^{-\lambda \delta (2k-1) n} e^{- \lambda \delta K} \sum_{\mathbf{C} \in \mathcal{R}_{K+1}(\varepsilon)} \underset{\chi_{\mathbf{C}}(\mathbf{a}_0) \neq 0}{\underset{\mathbf{C} \rightsquigarrow \mathbf{A} \leftrightarrow \mathbf{B}}{ \underset{ \mathbf{B} \in \mathcal{R}_{n+1}^k(\varepsilon)}{ \underset{\mathbf{A} \in \mathcal{R}_{n+1}^{k+1}(\varepsilon) }{ \sum}} }} \Bigg{|} \int_{U_{b(\mathbf{A})}}  e^{i \xi \psi( f^K g_{\mathbf{C}'(\mathbf{A} * \mathbf{B})}(x))} w_{\mathbf{a}_k}(x) d\nu(x) \Bigg{|}^2 $$
$$ \lesssim e^{\varepsilon \beta n } e^{- \lambda \delta K} \sum_{\mathbf{C} \in \mathcal{R}_{K+1}(\varepsilon)}  e^{-\lambda \delta (2k+1) n} \underset{\mathbf{C} \rightsquigarrow \mathbf{A}}{\underset{\chi_\mathbf{C}(\mathbf{a}_0) \neq 0}{\underset{\mathbf{A} \in \mathcal{R}_{n+1}^{k+1}(\varepsilon)}{\sum}} }\iint_{U_{b(\mathbf{A})}^2 }  \Bigg{|} \underset{\mathbf{A} \leftrightarrow \mathbf{B}}{\sum_{\mathbf{B} \in \mathcal{R}_{n+1}^k(\varepsilon)}} e^{i \xi \left|\Delta_{\mathbf{A},\mathbf{B},\mathbf{C}}\right|(x,y) } \Bigg| d\nu(x) d\nu(y)  .$$

\end{lemma}

\begin{proof}

We first open up the modulus squared:

$$  \underset{\chi_\mathbf{C}(\mathbf{a}_0) \neq 0}{\underset{\mathbf{B} \in \mathcal{R}_{n+1}^k(\varepsilon)}{\underset{\mathbf{A} \in \mathcal{R}_{n+1}^{k+1}(\varepsilon) }{\sum_{\mathbf{A} \leftrightarrow \mathbf{B}}}}}  \Bigg{|} \int_{U_{b(\mathbf{A})}}  e^{i \xi \psi( f^K g_{\mathbf{C}'(\mathbf{A} * \mathbf{B})}(x))} w_{\mathbf{a}_k}(x) d\nu(x) \Bigg{|}^2  $$ $$= \underset{\chi_\mathbf{C}(\mathbf{a}_0) \neq 0}{\underset{\mathbf{B} \in \mathcal{R}_{n+1}^k(\varepsilon)}{\underset{\mathbf{A} \in \mathcal{R}_{n+1}^{k+1}(\varepsilon) }{\sum_{\mathbf{A} \leftrightarrow \mathbf{B}}}}} \iint_{U_{b(\mathbf{A})}^2} w_{\mathbf{a}_k}(x) w_{\mathbf{a}_k}(y) e^{i \xi \Delta_{\mathbf{A},\mathbf{B},\mathbf{C}}(x,y)  }  d\nu(x) d\nu(y).  $$
Since this quantity is real, we get:
$$ = \underset{\chi_\mathbf{C}(\mathbf{a}_0) \neq 0}{\underset{\mathbf{B} \in \mathcal{R}_{n+1}^k(\varepsilon)}{\underset{\mathbf{A} \in \mathcal{R}_{n+1}^{k+1}(\varepsilon) }{\sum_{\mathbf{A} \leftrightarrow \mathbf{B}}}}} \iint_{U_{b(\mathbf{A})}^2} w_{\mathbf{a}_k}(x) w_{\mathbf{a}_k}(y) \cos({ \xi \Delta_{\mathbf{A},\mathbf{B},\mathbf{C}}(x,y)  })  d\nu(x) d\nu(y) $$
$$ = \underset{\chi_\mathbf{C}(\mathbf{a}_0) \neq 0}{\sum_{\mathbf{A} \in \mathcal{R}_{n+1}^{k+1}(\varepsilon)}} \iint_{U_{b(\mathbf{A})}^2 }  w_{\mathbf{a}_k}(x) w_{\mathbf{a}_k}(y) \underset{\mathbf{A} \leftrightarrow \mathbf{B}}{\sum_{\mathbf{B} \in \mathcal{R}_{n+1}^k(\varepsilon)}}  \cos\left( \xi  |\Delta_{\mathbf{A},\mathbf{B},\mathbf{C}}|(x,y) \right) d\nu(x) d\nu(y),$$
and then we conclude using the triangle inequality and the estimates of Lemma \ref{lem:mag}, as follow:
$$ \lesssim e^{\varepsilon \beta n } e^{- 2 \lambda \delta n}  \underset{\chi_\mathbf{C}(\mathbf{a}_0) \neq 0}{\sum_{\mathbf{A} \in \mathcal{R}_{n+1}^{k+1}(\varepsilon)}} \iint_{U_{b(\mathbf{A})}^2 }  \Bigg{|} \underset{\mathbf{A} \leftrightarrow \mathbf{B}}{\sum_{\mathbf{B} \in \mathcal{R}_{n+1}^k(\varepsilon)}} e^{i \xi \left|\Delta_{\mathbf{A},\mathbf{B},\mathbf{C}}\right|(x,y) } \Bigg| d\nu(x) d\nu(y) . $$

\end{proof}

\begin{lemma}\label{lem:reduc2}
Define $$\zeta_{j,\mathbf{A}}(\mathbf{b}) := e^{2 \lambda n} |g_{\mathbf{a}_{j-1}' \mathbf{b}}'(x_{\mathbf{\mathbf{a}}_j})|$$
and $$ J_n := \{ e^{\varepsilon_0 n/2} \leq |\eta| \leq e^{2 \varepsilon_0 n}  \} .$$
There exists $\beta > 0$ such that, for $n$ large enough depending on $\varepsilon$,

$$  e^{- \lambda \delta K} \sum_{\mathbf{C} \in \mathcal{R}_{K+1}(\varepsilon)} e^{-\lambda \delta (2k+1) n}   \underset{\mathbf{C} \rightsquigarrow \mathbf{A}}{\underset{\chi_\mathbf{C}(\mathbf{a}_0) \neq 0}{\underset{\mathbf{A} \in \mathcal{R}_{n+1}^{k+1}(\varepsilon)}{\sum}}} \iint_{U_{b(\mathbf{A})}^2 }  \Bigg{|} \underset{\mathbf{A} \leftrightarrow \mathbf{B}}{\sum_{\mathbf{B} \in \mathcal{R}_{n+1}^k(\varepsilon)}} e^{i \xi \left|\Delta_{\mathbf{A},\mathbf{B},\mathbf{C}}\right|(x,y) } \Bigg| d\nu(x) d\nu(y)  \quad \quad \quad \quad \quad \quad   $$
$$ \quad \quad \quad \quad \quad \quad \lesssim  e^{-\lambda \delta (2k+1) n} \sum_{\mathbf{A} \in \mathcal{R}_{n+1}^{k+1}(\varepsilon)} \sup_{\eta \in J_n} \Bigg{|} \underset{\mathbf{A} \leftrightarrow \mathbf{B}}{\sum_{\mathbf{B} \in \mathcal{R}_{n+1}^k(\varepsilon)}} e^{i \eta \zeta_{1,\mathbf{A}}(\mathbf{b}_1) \dots \zeta_{k,\mathbf{A}}(\mathbf{b}_k) } \Bigg|  + e^{ \varepsilon \beta n} \left( e^{-  (\alpha \lambda - \varepsilon_0) n/2}  + e^{- \varepsilon_0 n \delta_{\text{reg}}/6} \right) .$$

\end{lemma}

\begin{proof}

Our goal is to carefully approximate $\Delta_{\mathbf{A},\mathbf{B},\mathbf{C}}$ by a product of derivatives of $g_{\mathbf{a}_{j-1}'\mathbf{b}_j}$, and then to renormalize the phase. Using arc length parameterization on $\mathcal{V}$, our 1-dimensional setting allows us to apply the mean value theorem: for all $x,y \in U_{b(\mathbf{A})}$, there exists $z \in V_{{\mathbf{a}_k}}$ such that
$$ |\psi( f^K g_{\mathbf{C}} g_{\mathbf{A}* \mathbf{B}} (x) ) -  \psi( f^K g_{\mathbf{C}} g_{\mathbf{A}* \mathbf{B}} (y) )| = |\psi( f^K g_{\mathbf{C}} g_{\mathbf{A} \# \mathbf{B}} (g_{\mathbf{a}_k} x) ) -  \psi( f^K g_{\mathbf{C}} g_{\mathbf{A} \# \mathbf{B}} (g_{\mathbf{a}_k} y) )|$$ $$ = |\partial_u \psi(  f^K g_{\mathbf{C}'(\mathbf{A} \# \mathbf{B})} z )| |\partial_u f^K(g_{\mathbf{C}'(\mathbf{A} \# \mathbf{B})} z )| |g_{\mathbf{C}}'(g_{\mathbf{A} \# \mathbf{B} } z)| \left[ \prod_{j=1}^{k} |g_{\mathbf{a}_{j-1}'\mathbf{b}_j}'(g_{\mathbf{a}_j' \mathbf{b}_{j+1}' \dots \mathbf{a}_{k-1}' \mathbf{b}_k}z) )| \right]  \ d^u(g_{\mathbf{a}_k} x,g_{\mathbf{a}_k} y) .$$
The estimates of Lemma \ref{lem:mag} gives
$$ \left| {\Delta_{\mathbf{A},\mathbf{B},\mathbf{C}}(x,y)} \right| \lesssim  |\psi|_{C^{1+\alpha}} e^{\varepsilon \beta n} e^{- (2k+1) \lambda n} \lesssim |\xi|^{\rho_1}  e^{\varepsilon \beta n} e^{- (2k+1) \lambda n}  .$$
We then relate $|\Delta_{\mathbf{A},\mathbf{B}, \mathbf{C}}|$ to the $\zeta_{j,\mathbf{A}}$. Set $$ |\tilde{\Delta}_{\mathbf{A},\mathbf{B},\mathbf{C}}(x,y)| := |\partial_u \psi(  f^K g_{\mathbf{C}}(x_{\mathbf{a}_0} )| |\partial_u f^K(g_{\mathbf{C}} (x_{\mathbf{a}_0}) )| |g_{\mathbf{C}}(x_{\mathbf{a}_0})| \left[ \prod_{j=1}^{k} |g_{\mathbf{a}_{j-1}'\mathbf{b}_j}'(x_{\mathbf{a}_j}) )| \right]  \ d^u(g_{\mathbf{a}_k}x,g_{\mathbf{a}_k}y).$$
Then, as before, $$ \left| {\tilde{\Delta}_{\mathbf{A},\mathbf{B},\mathbf{C}}(x,y)} \right| \lesssim |\xi|^{\rho_1}  e^{\varepsilon \beta n} e^{- (2k+1) \lambda n}  .$$ 
Hence, using the fact that $|e^{s} - e^{t}| \leq e^{\max(s,t)} |s-t| $, we get

$$ \left| |\Delta_{\mathbf{A},\mathbf{B},\mathbf{C}}(x,y)| - |\tilde{\Delta}_{\mathbf{A},\mathbf{B},\mathbf{C}}(x,y)| \right| \leq |\xi|^{\rho_1} e^{\varepsilon \beta n} e^{-(2k+1)\lambda n} \left| \ln |\Delta_{\mathbf{A},\mathbf{B},\mathbf{C}}|(x,y) - \ln |\tilde{\Delta}_{\mathbf{A},\mathbf{B},\mathbf{C}}|(x,y) \right|.$$

Moreover, using the estimates of Lemma \ref{lem:mag}, exponentially vanishing variations of Hölder maps, the lower bound $\partial_u \psi(f^K g_{\mathbf{C}}(x_{\mathbf{a}_0})) \gtrsim |\xi|^{-\rho_1} $ (which is valid since $\chi(f^K g_\mathbf{C}(x_{\mathbf{a}_0})) \neq 0$, and we supposed that the lower bound $|\partial_u \psi| \gtrsim |\xi|^{-\rho_1}$ holds on $\text{supp}(\chi)$), and using $|\ln a - \ln b| \leq \frac{|a-b|}{\min(a,b)}$, we get:
\begin{itemize}
    \item $ \left| \ln|\partial_u \psi(  f^K g_{\mathbf{C}'(\mathbf{A}\# \mathbf{B})} z )| - \ln |\partial_u \psi(f^K g_\mathbf{C}(x_{\mathbf{a}_0}))| \right| \lesssim |\xi|^{2 \rho_1} e^{\varepsilon \beta n} ( e^{\lambda K} e^{- \lambda K} e^{-\lambda n})^\alpha \lesssim |\xi|^{2 \rho_1} e^{\varepsilon \beta n }e^{- \alpha \lambda n} ,$
    \item $ \left| \ln |\partial_u f^K(g_{\mathbf{C}'(\mathbf{A} \# \mathbf{B})} z )| - \ln |\partial_u f^K(g_{\mathbf{C}} (x_{\mathbf{a}_0}) )| \right|  \lesssim e^{\varepsilon \beta n }e^{-\alpha \lambda n} $,
    \item $ \left| \ln |g_{\mathbf{C}}'(g_{\mathbf{A} \# \mathbf{B}} z )| - \ln |g_{\mathbf{C}}' (x_{\mathbf{a}_0}) | \right|  \lesssim e^{\varepsilon \beta n }e^{- \alpha \lambda n} $,
    \item $ \left| \ln|g_{\mathbf{a}_{j-1}'\mathbf{b}_j}'(g_{\mathbf{a}_j' \mathbf{b}_{j+1}' \dots \mathbf{a}_{k-1}' \mathbf{b}_k}z) )| - \ln|g_{\mathbf{a}_{j-1}' \mathbf{b}_j}'(x_{\mathbf{a}_j})| \right| \lesssim e^{\varepsilon \beta n} e^{-\alpha \lambda n} $,
\end{itemize}
so that summing every estimates gives us $$ \left| \ln |\Delta_{\mathbf{A},\mathbf{B},\mathbf{C}}|(x,y) - \ln |\tilde{\Delta}_{\mathbf{A},\mathbf{B},\mathbf{C}}|(x,y) \right| \lesssim |\xi|^{2 \rho_1} e^{\varepsilon \beta n} e^{- \alpha \lambda n}. $$
Hence, $$  \left| |\Delta_{\mathbf{A},\mathbf{B},\mathbf{C}}(x,y)| - |\tilde{\Delta}_{\mathbf{A},\mathbf{B},\mathbf{C}}(x,y)|\right| \lesssim |\xi|^{3 \rho_1} e^{\varepsilon \beta n} e^{-(2k+1+\alpha) \lambda n} ,$$
which allows us to approximate our main integral as follows:
$$e^{-\delta \lambda((2k+1)  n + K)} \Bigg{|}  \underset{\mathbf{C} \rightsquigarrow \mathbf{A}}{\underset{\chi_\mathbf{C}(\mathbf{a}_0) \neq 0}{\underset{\mathbf{A} \in \mathcal{R}_{n+1}^{k+1}(\varepsilon)}{\sum_{\mathbf{C} \in \mathcal{R}_{K+1}(\varepsilon)}}}} \iint_{U_{b(\mathbf{A})}^2 }  \Big{|} \underset{\mathbf{A} \leftrightarrow \mathbf{B}}{\sum_{\mathbf{B} \in \mathcal{R}_{n+1}^k(\varepsilon)}} e^{i \xi \left|\Delta_{\mathbf{A},\mathbf{B},\mathbf{C}}\right| } \Big| -  \Big{|} \underset{\mathbf{A} \leftrightarrow \mathbf{B}}{\sum_{\mathbf{B} \in \mathcal{R}_{n+1}^k(\varepsilon)}} e^{i \xi \left|\tilde{\Delta}_{\mathbf{A},\mathbf{B},\mathbf{C}}\right| } \Big| d\nu^{\otimes 2} \Bigg{|}  $$
$$ \lesssim |\xi|^{1+3\rho_1} e^{-(2k+1+\alpha) \lambda n} \lesssim e^{-(\alpha \lambda - \varepsilon_0)n/2},$$
since $|\xi| \simeq  e^{(2k+1) \lambda n} e^{\varepsilon_0 n} $, and for $\rho_1$ chosen small enough so that $|\xi|^{3 \rho_1} \leq e^{(\alpha \lambda - \varepsilon_0) n/2}$. 
To conclude, we notice that $ |\xi| |\tilde{\Delta}_{\mathbf{A},\mathbf{B},\mathbf{C}}|$ can be written as a product like so:
$$ |\xi| |\tilde{\Delta}_{\mathbf{A},\mathbf{B},\mathbf{C}}|(x,y) =  \eta_{\mathbf{A},\mathbf{C}}(x,y)  \zeta_{1,\mathbf{A}}(\mathbf{b}_1) \dots  \zeta_{k,\mathbf{A}}(\mathbf{b}_k),  $$
where $$ \eta_{\mathbf{A},\mathbf{C}}(x,y) := |\xi| |\partial_u \psi(  f^K g_{\mathbf{C}'}(x_{\mathbf{a}_0} )| |\partial_u f^K(g_{\mathbf{C}'} (x_{\mathbf{a}_0}) )| |g_{\mathbf{C}'}(x_{\mathbf{a}_0})| e^{-2 k \lambda n }d^u(g_{\mathbf{a}_k}x,g_{\mathbf{a}_0}y). $$
We estimate $\eta_{\mathbf{A},\mathbf{C}}$ using the hypothesis made on $\partial_u \psi$, the estimates of Lemma \ref{lem:mag}, and the mean value theorem, to get
$$ e^{-\varepsilon \beta n} e^{2 \varepsilon_0 n/3} d^u(x,y) \leq |\xi|^{-\rho_1} e^{-\varepsilon \beta n} e^{\varepsilon_0 n} d^u(x,y) \lesssim \eta_{\mathbf{A},\mathbf{C}}(x,y)  \lesssim |\xi|^{\rho_1} e^{\varepsilon \beta n} e^{\varepsilon_0 n} \leq e^{2 \varepsilon_0 n} ,$$
as long as $\rho_1>0$ is chosen so small that $|\xi|^{\rho_1} \leq e^{\varepsilon_0 n/3}$. \\

We then see that $\eta_{\mathbf{A},\mathbf{C}}(x,y) \in J_n$ as soon as $d^u(x,y) \geq e^{\varepsilon \beta n - \varepsilon_0 n/6}$. To get rid of the part of the integral where $d^u(x,y)$ is too small, we use the upper regularity of $\nu$ (recall Lemma \ref{lem:upreg}). For all $y \in \mathcal{U}$, the ball $B(y, e^{\varepsilon \beta n - \varepsilon_0 n /6} )$ has measure $ \lesssim e^{-( \varepsilon_0/6  - \beta \varepsilon) \delta_{reg} n} $, so that by integrating over $y$,

$$ \nu \otimes \nu \left( \{ (x,y) \in \mathcal{U}^2 \ , \ |x-y|<  e^{\varepsilon \beta n - \varepsilon_0 n/6} \} \right) \lesssim e^{-( \varepsilon_0/6  - \beta \varepsilon) \delta_{reg} n} $$
as well. Hence we can cut the double integral in two, the part near the diagonal which is controlled by the previous estimates, and the part far away from the diagonal where $\eta_{\mathbf{A},\mathbf{C}}(x,y) \in J_n$. Once this is done, the sum over $\mathbf{C}$ disappears, and the condition $\chi_{\mathbf{C}}(\mathbf{a}_0) \neq 0$ as well, as there is no longer a dependence over $\mathbf{C}$ and $\chi$ in the phase once taken the supremum for $\eta \in J_n$. \end{proof}

\subsection{The sum-product phenomenon}\label{sec:sumprod}

The precise version of the sum-product phenomenon that we will use is the following.

\begin{theorem}\label{th:sumprod}\cite{SS20}
Fix $\gamma \in (0,1)$. There exist an integer $k \geq 1$, $c>0$ and $\varepsilon_1 > 0$ depending only on $\gamma$ such that the following holds for $\eta \in \mathbb{R}$ with $|\eta|$ large enough. Let $1 < R < |\eta|^{\varepsilon_1/2}$ , $N > 1$ and $\mathcal{Z}_1,\dots , \mathcal{Z}_k$ be finite sets such that $ \# \mathcal{Z}_j \leq RN$. Consider some maps $\zeta_j : \mathcal{Z}_j \rightarrow \mathbb{R} $, $j \in \llbracket 1, k \rrbracket $, such that, for all $j$:
$$ \zeta_j ( \mathcal{Z}_j ) \subset [R^{-1},R] $$ and 
$$\forall \sigma \in [|\eta|^{-2}, |\eta|^{- \varepsilon_1}], \quad \# \{\mathbf{b} , \mathbf{c} \in \mathcal{Z}^2_j , \ |\zeta_j(\mathbf{b}) - \zeta_j(\mathbf{c})| \leq \sigma \} \leq N^2 \sigma^{\gamma}.$$
Then: 
$$ \left| N^{-k} \sum_{\mathbf{b}_1 \in \mathcal{Z}_1,\dots,\mathbf{b}_k \in \mathcal{Z}_k} e^{ i \eta \zeta_1(\mathbf{b}_1) \dots \zeta_k(\mathbf{b}_k)} \right| \leq c |\eta|^{-{\varepsilon_1}}$$
\end{theorem}

We will use Theorem \ref{th:sumprod} on the maps $\zeta_{j,\mathbf{A}}$.
Let's carefully define the framework. \\
For some fixed $\mathbf{A} \in \mathcal{R}_{n+1}^{k+1}(\varepsilon)$, define for $j \in \llbracket 1, k \rrbracket$  $$ \mathcal{Z}_j := \{ \mathbf{b} \in \mathcal{R}_{n+1}(\varepsilon) , \mathbf{a}_{j-1} \rightsquigarrow \mathbf{b} \rightsquigarrow \mathbf{a}_j \  \} ,$$
so that the maps $\zeta_{j,\mathbf{A}}(\mathbf{b}) := e^{2 \lambda n} |g_{\mathbf{a}_{j-1}' \mathbf{b}}'(x_{{\mathbf{a}}_j})|$ are defined on $\mathcal{Z}_j$. Recall from Lemma \ref{lem:mag} that there exists a constant $\beta>0$  such that:
$$\# \mathcal{Z}_j \leq e^{\varepsilon \beta n} e^{ \delta \lambda n} $$
and
$$ \zeta_{j,\mathbf{A}}( \mathcal{Z}_j ) \subset \left[ e^{- \varepsilon \beta n}, e^{\varepsilon \beta n} \right] .$$

Let $\gamma>0$ be small enough. Theorem \ref{th:sumprod} then fixes $k$ and some $\varepsilon_1$.
The goal is to apply Theorem \ref{th:sumprod} to the maps $\zeta_{j,\mathbf{A}}$, for $N := e^{\lambda \delta n}$, $R:=e^{\varepsilon \beta n}$ and $\eta \in J_n$. Notice that choosing $\varepsilon$ small enough ensures that $R<|\eta|^{\varepsilon_1/2}$, and taking $n$ large enough ensures that $|\eta|$ is large. 
If we are able to check the non-concentration hypothesis for most words $\mathbf{A}$ under \hyperref[QNL]{(QNL)}, then Theorem \ref{th:sumprod} can be applied and we would be able to conclude the proof of Theorem \ref{th:reduc1}.
Indeed, we already know that
$$ |\widehat{\psi_*(\chi d\mu)}(\xi)|^2 \lesssim e^{ \varepsilon \beta n} e^{-\lambda \delta (2k+1) n} \sum_{\mathbf{A} \in \mathcal{R}_{n+1}^{k+1}(\varepsilon)} \sup_{\eta \in J_n} \Bigg{|} \underset{\mathbf{A} \leftrightarrow \mathbf{B}}{\sum_{\mathbf{B} \in \mathcal{R}_{n+1}^k(\varepsilon)}} e^{i \eta \zeta_{1,\mathbf{A}}(\mathbf{b}_1) \dots \zeta_{k,\mathbf{A}}(\mathbf{b}_k) } \Bigg| $$
$$ \quad \quad  \quad \quad  \quad \quad  \quad \quad + e^{- \varepsilon_0 n/2} +  e^{-\delta_1(\varepsilon) n} + e^{\varepsilon \beta n} \left( e^{- \lambda \alpha n} +  \kappa^{\alpha n} + e^{-(\alpha \lambda-\varepsilon_0)n/2} +  e^{- \varepsilon_0 \delta_{reg}n/6} \right) $$
by Proposition \ref{prop:reduc}. Since every error term already enjoys exponential decay in $n$, we just have to deal with the sum of exponentials. Say that we are able to show an estimate like:
$$ e^{- \lambda \delta (k+1) n} \# \{ \mathbf{A} \in \mathcal{R}_{n+1}^{k+1} \ , \ \text{the sum-product phenomenon doesn't apply for } (\zeta_{j,\mathbf{A}})_j\} \leq \rho^n, $$
for some $\rho \in (0,1)$. Then, by Theorem \ref{th:sumprod}, we can write for all $\mathbf{A}$ such that the sum-product phenomenon applies:
$$ \sup_{\eta \in J_n} \Bigg{|} \underset{\mathbf{A} \leftrightarrow \mathbf{B}}{\sum_{\mathbf{B} \in \mathcal{R}_{n+1}^k(\varepsilon)}} e^{i \eta \zeta_{1,\mathbf{A}}(\mathbf{b}_1) \dots \zeta_{k,\mathbf{A}}(\mathbf{b}_k) } \Bigg| \leq c e^{\lambda k \delta n} e^{- \varepsilon_0 \varepsilon_1 n/2 } ,$$
and hence we get
$$ e^{-\lambda \delta (2k+1) n} \sum_{\mathbf{A} \in \mathcal{R}_{n+1}^{k+1}(\varepsilon)} \sup_{\eta \in J_n} \Bigg{|} \underset{\mathbf{A} \leftrightarrow \mathbf{B}}{\sum_{\mathbf{B} \in \mathcal{R}_{n+1}^k(\varepsilon)}} e^{ i  \eta \zeta_{1,\mathbf{A}}(\mathbf{b}_1) \dots \zeta_{k,\mathbf{A}}(\mathbf{b}_k)} \Bigg| $$ $$ \lesssim  e^{ \varepsilon \beta n} \rho^n + e^{\varepsilon \beta n} e^{-\lambda \delta  (2k+1) n} e^{\lambda \delta (k+1) n} e^{\lambda \delta k n} e^{- \varepsilon_0 \varepsilon_1 n/2} \lesssim  e^{\varepsilon \beta n} \Big(\rho^n + e^{-\varepsilon_0 \varepsilon_1 n/2} \Big) .$$
Now, we see that we can choose $\varepsilon$ small enough so that all terms enjoy exponential decay in $n$, and since $ |\xi| \simeq e^{\left( (2k+1) \lambda + \varepsilon_0 \right) n} $, we have proved polynomial decay of $|\widehat{\psi_* (\chi d\mu)}|^2$. So, to conclude the proof of Theorem \ref{th:reduc1}, we have to show that the sum-product phenomenon can be applied often under the condition \hyperref[QNL]{(QNL)}. This is the content of the next subsection and of Proposition \ref{prop:ncQNL}.

\subsection{The non-concentration estimates under (QNL)}

This section is devoted to the proof of the non-concentration hypothesis that we just used under \hyperref[QNL]{(QNL)}. We need to relate the non-concentration hypothesis to the behavior of $\Delta$.

\begin{definition}
For a given $\mathbf{A} \in \mathcal{R}_{n+1}^{k+1}(\varepsilon)$, define for $j \in \llbracket 1, k \rrbracket$,  $$ \mathcal{Z}_j := \{ \mathbf{b} \in \mathcal{R}_{n+1}(\varepsilon) , \ \mathbf{a}_{j-1} \rightsquigarrow \mathbf{b} \rightsquigarrow \mathbf{a}_j \  \}, $$
and then define $$\zeta_{j,\mathbf{A}}(\mathbf{b}) := e^{2 \lambda n} |g_{\mathbf{a}_{j-1}' \mathbf{b}}'(x_{{\mathbf{a}}_j})|$$ on $\mathcal{Z}_j$. The following is satisfied, for some fixed constant $\beta>0$:
 $$\# \mathcal{Z}_j \leq e^{\varepsilon \beta n} e^{ \delta \lambda n} $$
 and
$$ \zeta_{j,\mathbf{A}}( \mathcal{Z}_j ) \subset \left[ e^{- \varepsilon \beta n} , e^{\varepsilon \beta n} \right] .$$
\end{definition}
Denote further $R:=e^{\varepsilon \beta n}$ and $N:= e^{\lambda \delta n} $. The goal of this section is to prove the next proposition. Once done, Theorem \ref{th:reduc1} will follow by our previous discussion.

\begin{proposition}\label{prop:ncQNL}
Under \hyperref[QNL]{(QNL)}, there exists $\gamma>0$ such that, for $\varepsilon_0$ and $k$ given by Theorem \ref{th:sumprod}, the following hold. There exists $\rho_3 \in (0,1)$ such that:
$$ \#\Big\{ \mathbf{A} \in \mathcal{R}_{n+1}^{k+1}(\varepsilon) \Big|  \exists \sigma \in [e^{-4 \varepsilon_0 n}, e^{- \varepsilon_0 \varepsilon_1 n /2}], \exists j \in \llbracket 1,k \rrbracket, \ \# \{\mathbf{b} , \mathbf{c} \in \mathcal{Z}^2_j , \ |\zeta_{j,\mathbf{A}}(\mathbf{b}) - \zeta_{j,\mathbf{A}}(\mathbf{c})| \leq \sigma \} \geq N^2 \sigma^{\gamma} \Big\} $$ $$ \lesssim N^{k+1}  \rho_3^n. $$
\end{proposition}

This will be done in a succession of reductions. The first reduction follow an idea from \cite{BD17}: using Markov's inequality to reduce this estimate to a bound on some expected value. 

\begin{lemma}\label{lem:nc1}
Suppose that there exists $\gamma>0$ such that, for $k\geq 0$, $\varepsilon_0>0$ given by Theorem \ref{th:sumprod}, the following hold: for all $ \sigma \in [e^{-4\varepsilon_0 n-1}, e^{-\varepsilon_0 \varepsilon_1 n/2 +1}]$,
$$ \ \#\Big\{ (\mathbf{a},\mathbf{b},\mathbf{c},\mathbf{d}) \in \widehat{\mathcal{R}}_{n+1}^4 \ \Big| \  \ \big|e^{-2 \lambda n} |g_{\mathbf{a}'\mathbf{b}}'(x_{\mathbf{d}})| - e^{-2 \lambda n} |g_{\mathbf{a}'\mathbf{c}}'(x_{\mathbf{d}})|\big| \leq \sigma \Big\} \leq N^4 \sigma^{2 \gamma}, $$
where $\widehat{\mathcal{R}}_{n+1}^4 := \{ (\mathbf{a},\mathbf{b},\mathbf{c},\mathbf{d}) \in \mathcal{R}_{n+1}^4(\varepsilon) \ | \ \mathbf{a} \rightsquigarrow \mathbf{b} \rightsquigarrow \mathbf{d}, \ \mathbf{a} \rightsquigarrow \mathbf{c} \rightsquigarrow \mathbf{d}\}$.\\

Then the conclusion of Proposition \ref{prop:ncQNL} holds.
\end{lemma}

\begin{proof}
We will denote $\mathcal{Z}_{\mathbf{a},\mathbf{d}} := \{ \mathbf{b} \in \mathcal{R}_{n+1}(\varepsilon) \ | \ \mathbf{a} \rightsquigarrow \mathbf{b} \rightsquigarrow \mathbf{d} \}$. To prove our lemma, we use a dyadic decomposition. For each integer $l \geq 0$ such that $\varepsilon_0 \varepsilon_1 n/2 - 1 \leq l \leq 4 \varepsilon_0 n +1$, notice that we can write, using Markov's inequality:
$$ N^{-2} \#\Big\{ (\mathbf{a},\mathbf{d}) \in \mathcal{R}_{n+1}^2(\varepsilon) \Big| \ \#\{ (\mathbf{b},\mathbf{c}) \in \mathcal{Z}_{\mathbf{a},\mathbf{d}}^2, \ \big| e^{-2 \lambda n} |g_{\mathbf{a}'\mathbf{b}}'(x_{\mathbf{d}})| - e^{-2 \lambda n} |g_{\mathbf{a}'\mathbf{c}}'(x_{\mathbf{d}})| \big| \leq e^{-l} \} \geq N^2 e^{-\gamma (l+1)} \Big\}  $$
$$ \leq N^{-4} e^{\gamma (l+1)} \#\Big\{ (\mathbf{a},\mathbf{b},\mathbf{c},\mathbf{d}) \in \widehat{\mathcal{R}}_{n+1}^4 \ \Big| \  \ \big|e^{-2 \lambda n} |g_{\mathbf{a}'\mathbf{b}}'(x_{\mathbf{d}})| - e^{-2 \lambda n} |g_{\mathbf{a}'\mathbf{c}}'(x_{\mathbf{d}})| \big| \leq e^{-l} \Big\} \leq e^{- \gamma (l-1)}. $$
Now, using the fact that for all $\sigma \in [e^{-4 \varepsilon_0 n}, e^{-\varepsilon_0 \varepsilon_1 n/2}]$, there exists $l$ such as above satisfying $ e^{-(l+1)} \leq \sigma \leq e^{-l}$ yields:
$$ N^{-(k+1)} \#\Big\{ \mathbf{A} \in \mathcal{R}_{n+1}^{k+1}(\varepsilon) \Big| \ \exists \sigma \in [e^{-4 \varepsilon_0 n}, e^{- \varepsilon_0 \varepsilon_1 n /2}], \ \exists j \in \llbracket 1,k \rrbracket, \ \# \{\mathbf{b} , \mathbf{c} \in \mathcal{Z}^2_j , \ |\zeta_{j,\mathbf{A}}(\mathbf{b}) - \zeta_{j,\mathbf{A}}(\mathbf{c})| \leq \sigma \} \geq N^2 \sigma^{\gamma} \Big\}$$
$$ \leq k N^{-2} \#\Big\{ (\mathbf{a},\mathbf{d}) \in \mathcal{R}_{n+1}^2(\varepsilon) \Big| \exists \sigma \in [e^{-4 \varepsilon_0 n}, e^{- \varepsilon_0 \varepsilon_1 n/2}] \ \#\{ (\mathbf{b},\mathbf{c}) \in \mathcal{Z}_{\mathbf{a},\mathbf{d}}^2, \ e^{-2 \lambda n} \big| |g_{\mathbf{a}'\mathbf{b}}'(x_{\mathbf{d}})| -|g_{\mathbf{a}'\mathbf{c}}'(x_{\mathbf{d}})| \big| \leq \sigma \} \geq N^2 \sigma^{\gamma} \Big\} $$
$$ \leq k \sum_{l = \lfloor \varepsilon_0 \varepsilon_1 n/2 \rfloor}^{ \lceil 4\varepsilon_0 n \rceil} N^{-2} \#\Big\{ (\mathbf{a},\mathbf{d}) \in \mathcal{R}_{n+1}^2(\varepsilon) \Big| \ \#\{ (\mathbf{b},\mathbf{c}) \in \mathcal{Z}_{\mathbf{a},\mathbf{d}}^2, \ \big|e^{-2 \lambda n} |g_{\mathbf{a}'\mathbf{b}}'(x_{\mathbf{d}})| - e^{-2 \lambda n} |g_{\mathbf{a}'\mathbf{c}}'(x_{\mathbf{d}})| \big| \leq e^{-l} \} \geq N^2 e^{-\gamma (l+1)} \Big\}   $$
$$ \lesssim  \sum_{l = \lfloor \varepsilon_0 \varepsilon_1 n/2 \rfloor}^{ \lceil 4\varepsilon_0 n \rceil} e^{-\gamma l} \lesssim n \ e^{ - \varepsilon_0 \varepsilon_1 \gamma n /2} \lesssim \rho_3^n, $$
for some $\rho_3 \in (0,1)$.
\end{proof}

\begin{lemma}\label{lem:nc2}
Suppose that there exists $\gamma>0$ such that, for $k\geq 0$, $\varepsilon_1>0$ given by Theorem \ref{th:sumprod}, the following hold: for all $\sigma \in [e^{-5 \varepsilon_0 n}, e^{-\varepsilon_1 \varepsilon_0 n/4}]$,
$$  \# \Big\{ (\mathbf{a},\mathbf{b},\mathbf{c},\mathbf{d}) \in \widehat{\mathcal{R}}_{n+1}^4  \Big| \  \big|S_{2n} \tau_F\big( g_{\mathbf{a}'\mathbf{b}}(x_\mathbf{d}) \big) - S_{2n} \tau_F\big( g_{\mathbf{a}'\mathbf{c}}(x_\mathbf{d}) \big) \big| \leq \sigma \Big\} \leq N^4 \sigma^{3 \gamma} .$$
(Recall that $\tau_F$ is defined in Definition \ref{def:lyap}.) Then the conclusion of Proposition \ref{prop:ncQNL} holds.
\end{lemma}

\begin{proof}
Suppose that the estimate is true.
Let $\sigma \in [e^{-4 \varepsilon_0 n-1}, e^{-\varepsilon_1 \varepsilon_0 n/2 +1}]$.
Our goal is to check the bound of Lemma \ref{lem:nc1}. Since $g_{\mathbf{a}'\mathbf{b}}'(x_\mathbf{c})e^{- 2 \lambda n} \in [R^{-1},R]$ (with $R = e^{\epsilon \beta n}$), we have:
$$  \#\Big\{ (\mathbf{a},\mathbf{b},\mathbf{c},\mathbf{d}) \in \widehat{\mathcal{R}}_{n+1}^4 \ \Big| \  \ \big|e^{-2 \lambda n} |g_{\mathbf{a}'\mathbf{b}}'(x_{\mathbf{d}})| - e^{-2 \lambda n} |g_{\mathbf{a}'\mathbf{c}}'(x_{\mathbf{d}})| \big| \leq \sigma \Big\}   $$
$$ \leq \#\Big\{ (\mathbf{a},\mathbf{b},\mathbf{c},\mathbf{d}) \in \widehat{\mathcal{R}}_{n+1}^4  , \  \ \Big| \frac{|g_{\mathbf{a}'\mathbf{b}}'(x_{\mathbf{d}})|}{|g_{\mathbf{a}'\mathbf{c}}'(x_{\mathbf{d}})|} - 1 \Big| \leq R \sigma \Big\} .$$
Now, notice that $\ln(|g_{\mathbf{a}'\mathbf{b}}'(x_\mathbf{c})|) = S_{2n} \tau_F(g_{\mathbf{a}'\mathbf{b}}(x_\mathbf{c}))$. It follows that, for all $(\mathbf{a},\mathbf{b},\mathbf{c},\mathbf{d}) \in \widehat{\mathcal{R}}_{n+1}^4$,
$$ \Big| S_{2n} \tau_F(g_{\mathbf{a}'\mathbf{b}}(x_\mathbf{d})) - S_{2n} \tau_F(g_{\mathbf{a}'\mathbf{c}}(x_\mathbf{d})) \Big| = \Big|\ln \Big( \frac{|g_{\mathbf{a}'\mathbf{b}}'(x_{\mathbf{d}})|}{|g_{\mathbf{a}'\mathbf{c}}'(x_{\mathbf{d}})|} \Big) \Big| \leq R^2 \Big| \frac{|g_{\mathbf{a}'\mathbf{b}}'(x_{\mathbf{d}})|}{|g_{\mathbf{a}'\mathbf{c}}'(x_{\mathbf{d}})|} - 1 \Big|. $$
Hence:
$$ \#\Big\{ (\mathbf{a},\mathbf{b},\mathbf{c},\mathbf{d}) \in \widehat{\mathcal{R}}_{n+1}^4 \ \Big| \  \ \big|e^{-2 \lambda n} |g_{\mathbf{a}'\mathbf{b}}'(x_{\mathbf{d}})| - e^{-2 \lambda n} |g_{\mathbf{a}'\mathbf{c}}'(x_{\mathbf{d}})| \big| \leq \sigma \Big\}   $$
$$ \leq \# \Big\{ (\mathbf{a},\mathbf{b},\mathbf{c},\mathbf{d}) \in \widehat{\mathcal{R}}_{n+1}^4  \Big| \  \big|S_{2n} \tau_F\big( g_{\mathbf{a}'\mathbf{b}}(x_\mathbf{d}) \big) - S_{2n} \tau_F\big( g_{\mathbf{a}'\mathbf{c}}(x_\mathbf{d}) \big) \big| \leq R \sigma \Big\}
\leq N^4 (R^3 \sigma)^{3 \gamma} \leq N^4 \sigma^{2 \gamma}$$
for $n$ large enough, since $\varepsilon$ is chosen small enough. The conclusion follows from Lemma \ref{lem:nc1}. \end{proof}

\begin{lemma}\label{lem:nc3}
Suppose that there exists $\gamma>0$ such that, for $k \geq 0$ and $\varepsilon_1$ given by Theorem \ref{th:sumprod}, the following hold: for all $\sigma \in [e^{-5 \varepsilon_0 n}, e^{- \varepsilon_0 \varepsilon_1 n/5}]$, 
$$ \sum_{a \in \mathcal{A}}  \underset{\mathbf{c} \rightsquigarrow a}{\underset{\mathbf{b} \rightsquigarrow a}{\sum_{(\mathbf{b},\mathbf{c}) \in \mathcal{R}_{n+1}^2(\varepsilon)}}} \nu^{\otimes 2}\Big( (x,y) \in U_a^2 \Big| \ |S_{n} \tau_F ( g_\mathbf{b}(x)) -  S_{n} \tau_F( g_\mathbf{b}(y)) -  S_{n} \tau_F( g_\mathbf{c}(x)) +  S_{n} \tau_F( g_\mathbf{c}(y))| \leq \sigma \Big) \leq N^2 \sigma^{7 \gamma} $$
Then the conclusion of Proposition \ref{prop:ncQNL} holds.
\end{lemma}

\begin{proof}
Suppose that the previous estimate holds. We will check that Lemma \ref{lem:nc2} applies. Let $\sigma \in [e^{-5 \varepsilon_0 n},e^{-\varepsilon_0 \varepsilon_1 n/4}]$. First of all, notice that the family of maps $S_{2n} \tau_F \circ g_{\mathbf{a}'\mathbf{b}}$ are uniformly Hölder regular. In particular, there exists $C \geq 1$ such that, for all $\mathbf{a},\mathbf{b},\mathbf{d}$:
$$ \forall x \in U_{\mathbf{d}}, \ |S_{2n} \tau_F(g_{\mathbf{a}'\mathbf{b}}(x_\mathbf{d})) - S_{2n} \tau_F(g_{\mathbf{a}'\mathbf{b}}(x)) | \leq C \text{diam}(U_\mathbf{d})^{\alpha n} \leq C \kappa^{\alpha n}. $$
If the constant $\varepsilon_0$ is chosen sufficiently small so that $\kappa^{\alpha n} \leq e^{- 5 \varepsilon_0 n}$, then we find
$$ \forall x \in U_{\mathbf{d}}, \ |S_{2n} \tau_F(g_{\mathbf{a}'\mathbf{b}}(x_\mathbf{d})) - S_{2n} \tau_F(g_{\mathbf{a}'\mathbf{b}}(x)) | \leq C \sigma ,$$
which implies
$$ \nu(U_{\mathbf{d}}) \mathbb{1}_{[-\sigma,\sigma]}\Big( S_{2n} \tau_F(g_{\mathbf{a}'\mathbf{b}}(x_\mathbf{d})) - S_{2n} \tau_F(g_{\mathbf{a}'\mathbf{c}}(x_\mathbf{d})) \Big) $$ $$ \leq R \int_{U_\mathbf{d}} \mathbb{1}_{[-3 C \sigma,3 C \sigma]}(S_{2n} \tau_F(g_{\mathbf{a}'\mathbf{b}}(x)) - S_{2n} \tau_F(g_{\mathbf{a}'\mathbf{c}}(x)) ) d\nu(x) .$$
Using the fact that $\nu(U_{\mathbf{d}}) \sim N^{-1}$ from Lemma \ref{lem:mag} and summing the previous estimates yields

$$ N^{-4} \# \Big\{ (\mathbf{a},\mathbf{b},\mathbf{c},\mathbf{d}) \in \widehat{\mathcal{R}}_{n+1}^4  \Big| \  \big|S_{2n} \tau_F\big( g_{\mathbf{a}'\mathbf{b}}(x_\mathbf{d}) \big) - S_{2n} \tau_F\big( g_{\mathbf{a}'\mathbf{c}}(x_\mathbf{d}) \big) \big| \leq \sigma \Big\} $$
$$ \leq R N^{-3} \sum_{\mathbf{a},\mathbf{b},\mathbf{c}} \sum_{\mathbf{d}} \nu(U_\mathbf{d})  \mathbb{1}_{[-\sigma,\sigma]}\Big( S_{2n} \tau_F(g_{\mathbf{a}'\mathbf{b}}(x_\mathbf{d})) - S_{2n} \tau_F(g_{\mathbf{a}'\mathbf{c}}(x_\mathbf{d})) \Big)   $$
$$ \leq R N^{-3} \sum_{d \in \mathcal{A}} \underset{\mathbf{a} \rightsquigarrow \mathbf{c} \rightsquigarrow d}{\underset{\mathbf{a} \rightsquigarrow \mathbf{b} \rightsquigarrow d}{\sum_{\mathbf{a},\mathbf{b},\mathbf{c}}}} \nu\Big( x \in U_d, \ |S_{2n} \tau_F(g_{\mathbf{a}'\mathbf{b}}(x)) - S_{2n} \tau_F(g_{\mathbf{a}'\mathbf{c}}(x))| \leq 3C \sigma \Big).$$
Now, the idea is to use Cauchy-Schwarz inequality to make the $y$-variable appear. We have:
$$ \Big( N^{-4} \# \Big\{ (\mathbf{a},\mathbf{b},\mathbf{c},\mathbf{d}) \in \widehat{\mathcal{R}}_{n+1}^4  \Big| \  \big|S_{2n} \tau_F\big( g_{\mathbf{a}'\mathbf{b}}(x_\mathbf{d}) \big) - S_{2n} \tau_F\big( g_{\mathbf{a}'\mathbf{c}}(x_\mathbf{d}) \big) \big| \leq \sigma \Big\} \Big)^2 $$
$$ \lesssim R^2 N^{-3} \sum_{d \in \mathcal{A}} \underset{\mathbf{a} \rightsquigarrow \mathbf{c} \rightsquigarrow d}{\underset{\mathbf{a} \rightsquigarrow \mathbf{b} \rightsquigarrow d}{\sum_{\mathbf{a},\mathbf{b},\mathbf{c}}}} \nu\Big( x \in U_d, \ |S_{2n} \tau_F(g_{\mathbf{a}'\mathbf{b}}(x)) - S_{2n} \tau_F(g_{\mathbf{a}'\mathbf{c}}(x))| \leq 3C \sigma \Big)^2 $$
Now, notice that for any measurable function $h : U_d \rightarrow \mathbb{R}$ and for any interval $I \subset \mathbb{R}$ of length $|I|$, we can write:
$$ \nu(x \in U_d, \ h(x) \in I)^2 = \iint_{U_d^2} \mathbb{1}_{I}(h(x)) \mathbb{1}_{I}(h(y)) d\nu(x) d\nu(y) $$
$$ \leq \iint_{U_d^2} \mathbb{1}_{[-2|I|,2|I|]}(h(x)-h(y)) d\nu(x) d\nu(y) $$ $$= \nu^{\otimes 2}( (x,y) \in U_d, \ |h(x)-h(y)| \leq 2 |I| ). $$
Applying this elementary estimate to our case yields
$$ \Big( N^{-4} \# \Big\{ (\mathbf{a},\mathbf{b},\mathbf{c},\mathbf{d}) \in \widehat{\mathcal{R}}_{n+1}^4  \Big| \  \big|S_{2n} \tau_F\big( g_{\mathbf{a}'\mathbf{b}}(x_\mathbf{d}) \big) - S_{2n} \tau_F\big( g_{\mathbf{a}'\mathbf{c}}(x_\mathbf{d}) \big) \big| \leq \sigma \Big\} \Big)^2  $$
$$  \lesssim R^2 N^{-3} \sum_{d \in \mathcal{A}} \underset{\mathbf{a} \rightsquigarrow \mathbf{c} \rightsquigarrow d}{\underset{\mathbf{a} \rightsquigarrow \mathbf{b} \rightsquigarrow d}{\sum_{\mathbf{a},\mathbf{b},\mathbf{c}}}} \nu^{\otimes 2}\Big( (x,y) \in U_d^2, \ |H_{2n}(\mathbf{a},\mathbf{b},\mathbf{c},x,y)| \leq 6C \sigma \Big) $$
where $H_{2n}(\mathbf{a},\mathbf{b},\mathbf{c},x,y) := S_{2n} \tau_F(g_{\mathbf{a}'\mathbf{b}}(x)) - S_{2n} \tau_F(g_{\mathbf{a}'\mathbf{c}}(x)) - S_{2n} \tau_F(g_{\mathbf{a}'\mathbf{b}}(y)) + S_{2n} \tau_F(g_{\mathbf{a}'\mathbf{c}}(y))$. To conclude, we will show that this expression is close to another one that doesn't depend on $\mathbf{a}$. Notice that we have:
$$ S_{2n} \tau_F(g_{\mathbf{a}'\mathbf{b}}(x)) - S_{2n} \tau_F ( g_{\mathbf{a}' \mathbf{b}}(y)) =  (S_{n} \tau_F \circ g_\mathbf{a})(g_\mathbf{b}(x)) - (S_{n} \tau_F \circ g_{\mathbf{a}}) (g_\mathbf{b}(y)) + S_n \tau_F(g_\mathbf{b}(x)) - S_n \tau_F(g_\mathbf{b}(y)). $$
$$ = S_n \tau_F(g_\mathbf{b}(x)) - S_n \tau_F(g_\mathbf{b}(y)) + O(\sigma) $$
since $S_n \tau_F \circ g_\mathbf{a}$ are uniformly Hölder maps, and $\varepsilon_0$ is chosen sufficiently small. Hence $$ \Big|  S_{n} \tau_F ( g_\mathbf{b}(x)) -  S_{n} \tau_F( g_\mathbf{b}(y)) -  S_{n} \tau_F( g_\mathbf{c}(x)) +  S_{n} \tau_F( g_\mathbf{c}(y)) - H_{2n}(\mathbf{a},\mathbf{b},\mathbf{c},x,y) \Big| \leq 2 C \sigma, $$
and we finally get:
$$ N^{-4} \# \Big\{ (\mathbf{a},\mathbf{b},\mathbf{c},\mathbf{d}) \in \widehat{\mathcal{R}}_{n+1}^4  \Big| \  \big|S_{2n} \tau_F\big( g_{\mathbf{a}'\mathbf{b}}(x_\mathbf{d}) \big) - S_{2n} \tau_F\big( g_{\mathbf{a}'\mathbf{c}}(x_\mathbf{d}) \big) \big| \leq \sigma \Big\} $$
$$ \lesssim R \Big( N^{-3} \sum_{d \in \mathcal{A}} \underset{\mathbf{a} \rightsquigarrow \mathbf{c} \rightsquigarrow d}{\underset{\mathbf{a} \rightsquigarrow \mathbf{b} \rightsquigarrow d}{\sum_{\mathbf{a},\mathbf{b},\mathbf{c}}}} \nu^{\otimes 2}\Big( (x,y) \in U_d^2, \ |H_{2n}(\mathbf{a},\mathbf{b},\mathbf{c},x,y)| \leq 6C \sigma \Big) \Big)^{1/2} $$
$$ \lesssim R \Big( N^{-2} \sum_{d \in \mathcal{A}} \underset{\mathbf{a} \rightsquigarrow \mathbf{c} \rightsquigarrow d}{\underset{\mathbf{a} \rightsquigarrow \mathbf{b} \rightsquigarrow d}{\sum_{\mathbf{a},\mathbf{b},\mathbf{c}}}} \nu^{\otimes 2}\Big( (x,y) \in U_d^2, \ |S_{n} \tau_F ( g_\mathbf{b}(x)) -  S_{n} \tau_F( g_\mathbf{b}(y)) -  S_{n} \tau_F( g_\mathbf{c}(x)) +  S_{n} \tau_F( g_\mathbf{c}(y))| \leq 8 C \sigma \Big) \Big)^{1/2} $$
$$ \leq R (8C \sigma)^{7\gamma/2} \leq \sigma^{3 \gamma} $$
for $n$ large enough. Hence Lemma \ref{lem:nc2} applies, and the conclusion of Proposition \ref{prop:ncQNL} holds.
\end{proof}

Our last step to establish Proposition \ref{prop:ncQNL} is to properly recognise that the expression that we just introduced is actually (close to) $\Delta$.

\begin{lemma}\label{lem:Delta1}
For $a_0 \in \mathcal{A}$, and $p,q \in \widehat{R}_{a_0}$, recall that $\Delta(p,q) \in \mathbb{R}$ is defined by
$$ \Delta(p,q) := \sum_{n \in \mathbb{Z}} \Big( \tau_f(f^n p) - \tau_f(f^n [p,q]) - \tau_f(f^n [q,p]) + \tau_f(f^n q) \Big). $$
Then, if we denote $x := \pi(p) \in U_{a_0}$, $y := \pi(q) \in U_{a_0}$, and if we let $(a_k)_{k \geq 0}$ and $(b_k) \in \mathcal{A}^\mathbb{N}$ be defined by $$\forall k \geq 0, \ f^{-k}(p) \in R_{a_k} , f^{-k}(q) \in R_{b_k} , $$
then:
$$ S_{n} \tau_F ( g_\mathbf{a}(x)) -  S_{n} \tau_F( g_\mathbf{a}(y)) -  S_{n} \tau_F( g_\mathbf{b}(x)) +  S_{n} \tau_F( g_\mathbf{b}(y)) = \Delta(p,q) + O(\sigma)  $$
for all $\sigma \in [e^{-5 \varepsilon_0 n}, e^{- \varepsilon_0 \varepsilon_1 n/5}]$.
\end{lemma}

\begin{proof}
First of all, recall that, by Definition \ref{def:lyap}, we know that $\tau_f$ and $\tau_F \circ \pi$ are cohomologous. Let $\theta : \mathcal{R} \rightarrow \mathbb{R}$ be a Hölder map such that $\tau_f = \tau_F \circ \pi + \theta \circ f - \theta$ on $\mathcal{R}$. It is then easy to check that, for all $p,q \in R_a$ for some $a \in \mathcal{A}$,
$$ \Delta(p,q) = 
 \sum_{n \in \mathbb{Z} } \Big( \tau_F( \pi f^n p) - \tau_F(\pi f^n [p,q]) - \tau_F(\pi f^n [q,p]) + \tau_F(\pi f^n q) \Big) .$$
We then notice that, for nonnegative $n$, we have $\pi(f^n p) = \pi(f^n([p,q]))$ and $\pi(f^n q) = \pi(f^n([q,p]))$. It follows that
$$ \Delta(p,q) = \sum_{k=1}^\infty \Big( \tau_F(\pi f^{-k} p) - \tau_F(\pi f^{-k} [p,q]) - \tau_F(\pi f^{-k} [q,p]) + \tau_F(\pi f^{-k} q) \Big). $$
Now, we notice that $F^k( \pi f^{-k} p ) = \pi(p) =x$, and since $f^{-j}(p) \in R_{a_j}$, we find $\pi f^{-k}(p) = g_{a_k \dots a_0}(x)$. A similar computation for the other points gives the expression
$$ \Delta(p,q) = \sum_{k=1}^\infty \Big( \tau_F(g_{a_k \dots a_0} x) - \tau_F(g_{a_k \dots a_0} y) - \tau_F(g_{b_k \dots b_0} x) + \tau_F(g_{b_k \dots b_0} y) \Big) .$$
Now, since $\tau_F$ is Hölder, we find again
$$ \Big| \tau_F(g_{a_k \dots a_0} x) - \tau_F(g_{a_k \dots a_0} y) - \tau_F(g_{b_k \dots b_0} x) + \tau_F(g_{b_k \dots b_0} y) \Big| \leq  \|\tau_F \|_{C^\alpha} \big( \text{diam}(U_{a_k \dots a_0})^\alpha + \text{diam}(U_{b_k \dots b_0})^\alpha  \big) \lesssim \kappa^{\alpha k} ,$$
which gives, summing those bounds for all $k \geq n+1$:
$$ \Big|\sum_{k=n+1}^\infty \Big(\tau_F(g_{a_k \dots a_0} x) - \tau_F(g_{a_k \dots a_0} y) - \tau_F(g_{b_k \dots b_0} x) + \tau_F(g_{b_k \dots b_0} y) \Big)\Big| \lesssim \kappa^{\alpha n} \leq \sigma $$
for large enough $n$. Now we conclude noticing that
$$ S_{n} \tau_F ( g_\mathbf{a}(x)) -  S_{n} \tau_F( g_\mathbf{a}(y)) -  S_{n} \tau_F( g_\mathbf{b}(x)) + S_n \tau_F (g_\mathbf{c}(y)) $$ $$ = \sum_{k=1}^n \Big(\tau_F(g_{a_k \dots a_0} x) - \tau_F(g_{a_k \dots a_0} y) - \tau_F(g_{b_k \dots b_0} x) + \tau_F(g_{b_k \dots b_0} y) \Big). $$
\end{proof}

We are ready to prove Proposition \ref{prop:ncQNL}.

\begin{proof}[Proof (of Proposition \ref{prop:ncQNL}).]
We suppose \hyperref[QNL]{(QNL)}. Our goal is to check that the estimates of Lemma \ref{lem:nc3} holds. Let us denote by $R^{a_1 \dots a_n} := f^{n}(R_{a_1}) \cap \dots f(R_{a_{n-1}}) \cap R_{a_n}$. We have, using Lemma \ref{lem:Delta1}:

$$\mathbb{1}_{[-\sigma,\sigma]}\Big( S_{n} \tau_F ( g_\mathbf{a}(x)) -  S_{n} \tau_F( g_\mathbf{a}(y)) -  S_{n} \tau_F( g_\mathbf{b}(x)) +  S_{n} \tau_F( g_\mathbf{b}(y)) \Big)$$ $$ \leq \frac{1}{\mu(R^{\mathbf{a}}) \mu(R^{\mathbf{b}})} \iint_{R^{\mathbf{a}} \times R^{\mathbf{b}}} \mathbb{1}_{[-2 \sigma, 2 \sigma] }\Big( \Delta(p,q) \Big) d\mu^2(p,q). $$
Since the measure $\mu$ is $f$-invariant, we find that $\mu(R^{\mathbf{a}}) = \mu(f^{-n}(R^{\mathbf{a}})) = \nu(U_{\mathbf{a}}) \sim N^{-1}$. Integrating in $d\nu^{\otimes 2}(x,y)$ yields
$$ \nu^{\otimes 2}\Big( (x,y) \in U_a^2 \Big| \ |S_{n} \tau_F ( g_\mathbf{b}(x)) -  S_{n} \tau_F( g_\mathbf{b}(y)) -  S_{n} \tau_F( g_\mathbf{c}(x)) +  S_{n} \tau_F( g_\mathbf{c}(y))| \leq \sigma \Big)  $$
$$ \leq R^2 N^2 \iint_{R^{\mathbf{a}} \times R^{\mathbf{b}}} \mathbb{1}_{[-2 \sigma, 2 \sigma] }\Big( \Delta(p,q) \Big) d\mu^2(p,q), $$
and then summing those estimates gives
$$ \sum_{a \in \mathcal{A}}  \underset{\mathbf{c} \rightsquigarrow a}{\underset{\mathbf{b} \rightsquigarrow a}{\sum_{(\mathbf{b},\mathbf{c}) \in \mathcal{R}_{n+1}^2}}} \nu^{\otimes 2}\Big( (x,y) \in U_a^2 \Big| \ |S_{n} \tau_F ( g_\mathbf{b}(x)) -  S_{n} \tau_F( g_\mathbf{b}(y)) -  S_{n} \tau_F( g_\mathbf{c}(x)) +  S_{n} \tau_F( g_\mathbf{c}(y))| \leq \sigma \Big)  $$
$$ \leq R^2 N^2 \sum_{a \in \mathcal{A}}  \underset{\mathbf{c} \rightsquigarrow a}{\underset{\mathbf{b} \rightsquigarrow a}{\sum_{(\mathbf{b},\mathbf{c}) \in \mathcal{R}_{n+1}^2}}} \iint_{R^{\mathbf{a}} \times R^{\mathbf{b}}} \mathbb{1}_{[-2 \sigma, 2 \sigma] }\Big( \Delta(p,q) \Big) d\mu^2(p,q) $$
$$ = R^2 N^2 \sum_{a \in \mathcal{A}} \mu^{\otimes 2}\Big( (p,q) \in (R_a)^2, \  |\Delta(p,q)| \leq 2 \sigma \Big) \leq R^2 N^2 \sigma^{\Gamma} . $$
where the last inequality comes from \hyperref[QNL]{(QNL)}. We can finally fix $\gamma := \Gamma/8$ and let $k,\varepsilon_1$ being given by Theorem \ref{th:sumprod}. 
We choose $\varepsilon_0 := {\alpha |\ln(\kappa)|}/{8} $, and
we see that the estimates of Lemma \ref{lem:nc3} are satisfied, thus concluding the proof of Proposition \ref{prop:ncQNL}. 
\end{proof}

We prove Theorem \ref{th:reduc1} by applying Proposition \ref{prop:ncQNL} to the discussion of subsection \ref{sec:sumprod}.

\section{The non-concentration of $\Delta$}\label{sec:4}

We begin the second half of this work. From now on, we will need to assume that $f$ is area-preserving on $\Omega$. This section is then devoted to show that \hyperref[QNL]{(QNL)} holds when $E^s  \text{ or } E^u \notin C^2$. \\

Reversing the dynamics if necessary, one can assume that $E^s \notin C^2$. To make the link with the map $\Delta$, recall that we have the following formula for the distortions of stable holonomies. Fixing $p \in \Omega$ and $s \in \Omega \cap W^s_{loc}(p)$, and denoting $\pi_{p,s}:W_{loc}^u(p) \cap \Omega \rightarrow W_{loc}^u(s) \cap \Omega$ the stable holonomy, we have
$$ \ln \partial_u \pi_{p,s}(q) = \sum_{n=0}^\infty \tau_f(f^n q) - \tau_f(f^n r).$$
where $r := \pi_{s,p}(q)$ (See Lemma \ref{lem:disthol} for a proof.) Cutting half of the sum defining $\Delta$, we can then consider  $$\Delta^+_p([r,s]) := \sum_{n=0}^\infty \tau_f(f^{n} p) - \tau_f(f^{n}r) - \tau_f(f^{n}s) + \tau_f(f^{n}[r,s]) = \ln \partial_u \pi_{p,s}(r) - \ln \partial_u \pi_{p,s}(p).$$
We see that the regularity of $\Delta^+_p$ in the $r$ variable is determined by the regularity of $\ln \partial_u \pi_{s,p}$, which is linked to the regularity of the stable foliation and of $E^s$. We postpone to Appendix \hyperref[ap:C]{C} the proof of the following technical fact:

\begin{proposition}\label{prop:apC}
Suppose that $E^s \notin C^\infty$. Then $E^s \notin C^2$. Moreover, there exists $p \in \Omega$ such that $r \in W^u_{loc}(p) \cap \Omega \mapsto \partial_s \Delta^+_p(r) \in \mathbb{R}$ is not $C^1$. 
\end{proposition}

This reformulated condition on $\Delta^+$ is the one that we will use in this section to understand the behavior of $\Delta$. Since we will be interested in doing an asymptotic expansion in $r$ for a fixed $s,p$, we will need some regularity results on $\tau_f$.

\begin{definition}
Let $\tau : \Omega \rightarrow \mathbb{R}$. We say that $\tau \in \text{Reg}_u^{1+\alpha}(\Omega)$ (\say{regular in the unstable direction}) for some $\alpha \in (0,1)$ if $\tau \in C^{1+\alpha}(\Omega,\mathbb{R})$,  if moreover, for any $p \in \Omega$, the map $\tau_{|W^u_{loc}(p)} : W^u_{loc}(p) \cap \Omega \rightarrow \mathbb{R}$ is $C^{N}$ for $N := 5 + \lceil 1/\alpha \rceil$ (uniformly in $p$), the maps
$$ r \in W^u_{loc}(p) \cap \Omega \longmapsto (\partial_s \tau(r),\partial_u \tau(r)) \in \mathbb{R}^2  $$
are uniformly $C^{1+\alpha}$, and if finally the map $(\partial_u \partial_s \tau, \partial_u \partial_u \tau)$ is $\alpha$-Hölder on $\Omega$.
\end{definition}

We also postpone to appendix \hyperref[ap:A]{A} the proof of the following fact.

\begin{proposition}\label{prop:REGu}
For some $\alpha \in (0,1)$, $\tau_f \in \text{Reg}_u^{1+\alpha}(\Omega)$.
\end{proposition}

We are ready to state our main technical Theorem for this section.

\begin{theorem}\label{th:QNL}
Let $f :M \rightarrow M$ be a smooth Axiom A diffeomorphim on a surface. Let $\Omega$ be one of its basic set, and suppose that $|\det df|=1$ on $\Omega$. Let $\mu \in \mathcal{P}(\Omega)$ be one of its equilibrium states. Let $\alpha \in (0,1)$ small enough and let $\tau \in \text{Reg}_u^{1+\alpha}(\Omega)$. Denote $$\Delta(p,q) := \sum_{n \in \mathbb{Z}} T_n(p,q) \quad , \quad \Delta^+_p(q) := \sum_{n=0}^\infty T_n(p,q)$$
for $$T_n(p,q) := \tau(f^n p) - \tau(f^n [p,q]) - \tau(f^n[q,p]) + \tau(f^n q).$$ Suppose that there exists $p \in \Omega$ such that $r \in W^u_{loc}(p) \cap \Omega \mapsto \partial_s \Delta^+_p(r) \notin C^1$. Then \hyperref[QNL]{(QNL)} holds, in the sense that: \\

For any Markov Partition $(R_a)_{a \in \mathcal{A}}$ of $(\Omega,f)$, there exists $C \geq 1$, $\Gamma \in (0,1)$ such that
$$\forall \sigma \in (0,1), \ \sum_{a \in \mathcal{A}} \mu^{\otimes 2}\Big( (p,q) \in R_a \ \Big| \ |\Delta(p,q)| \leq \sigma \Big) \leq C \sigma^\Gamma .$$
\end{theorem}

We emphasize the fact that this lemma apply to more general $\tau$ than $\tau_f$, having in mind that the methods presented here might be in particular helpful to study suspension flows over Axiom A diffeomorphisms. \\
To prove this result, we will frequently use the \emph{local product structure} of the equilibrium state $\mu$. This will allow us to use more precise informations on the local behavior of $\mu$ than the structure given by the invariance under by a transfer operator. 

\begin{proposition}\cite{Le00,Cl20}\label{prop:locProdStruct}
Denote by $\psi:\Omega \rightarrow \mathbb{R}$ a Hölder potential such that $\mu$ is an equilibrium state for $\psi$. Define: $$\omega^+(x,y) := \sum_{n=0}^\infty \left( \psi(f^n(x)) - \psi(f^n(y)) \right) \quad , \quad \omega^-(x,y) := \sum_{n=0}^\infty \left( \psi(f^{-n}(x)) - \psi(f^{-n}(y)) \right). $$
For every small enough rectangle $R$, there exists two families of (nonzero and) finite measures $\mu_x^s$, $\mu_x^u$ indexed on $x \in R$ such that $\text{supp} ( \mu_x^u) \subset W^u(x,R) =: U_x$ and $\text{supp} ( \mu_x^s ) \subset W^s(x,R) =: S_x$ if $x \in R$, and that satisfies the following properties: 

\begin{itemize}
    \item If $\pi$ is a stable holonomy map between $W^u(y,R)$ and $W^u(\pi(y),R)$, then
    $$ \frac{d\left(\pi_*\mu_y^u\right)}{d \mu_{\pi(y)}^u}(\pi(x)) = e^{\omega^+(x,\pi x)}.$$
    and a similar formula holds for $\mu_y^s$ with an unstable holonomy map, replacing $\omega^+$ by $\omega^-$ (\cite{Cl20}, Theorem 3.7 and Theorem 3.9).
    \item The family of measures is $\psi$-conformal:
    $$ \frac{ d \left( f_* \mu_y^u \right)}{d \mu_{f(y)}^u}(f(x)) = e^{\psi(x) - P(\psi) } $$
    and a similar formula holds for $\mu^s_y$ (\cite{Cl20}, Theorem 3.4 and Remark 3.5).
\end{itemize}

Finally, this family of measures is related to $\mu$ in the following way (\cite{Cl20}, Theorem 3.10). There exists (uniformly) positive (and bounded) constants $c_x>0$ such that, for any measurable map $g:R \rightarrow \mathbb{C}$, the following formula holds:

$$ \forall x\in R, \ \int_R g d\mu = \int_{U_x} \int_{S_x} e^{\omega^+([z,y],z)+\omega^-([z,y],y)} g([z,y]) c_x d\mu_{x}^s(y) d\mu_{x}^u(z)  .$$
In particular, for a bounded familly of functions $\omega_x$, we have
$$ \int_R g d\mu = \iint_{U_x \times S_x} g([z,y]) e^{\omega_x(z,y)} d\mu^u(z) d\mu^s(y) .$$ 
\end{proposition}

Since the measures $\mu_x^u$ and $\mu_x^s$ are one-dimensionnal, it is not hard to see that they are \emph{doubling} \cite{Le24}. In particular, denoting $S_x(\sigma) := W^s_{loc}(x) \cap \Omega \cap B(x,\sigma)$ and $U_x(\sigma) := W^s_{loc}(x) \cap \Omega \cap B(x,\sigma)$, the following holds:\\

For all $D \geq 1$, there exists $\exists C \geq 1$, such that for all $x \in \Omega$, $\sigma \in (0,1)$:
$$  \mu_x^s( S_x(D \sigma) ) \leq C \mu_x^s(S_x(\sigma)) \quad ,  \quad \mu_x^u( U_x(D \sigma) ) \leq C \mu_x^u(U_x(\sigma)) . $$

\subsection{Localised non-concentration}
Our first step to prove Theorem \ref{th:QNL} is to first reduce it to checking a localized version of \hyperref[QNL]{(QNL)}. To do so, we will take advantage of the local product structure of $\mu$. \\

We let $\tau \in \text{Reg}_u^{1+\alpha}(\Omega)$ (think of $\tau$ as $\tau_f$) and we fix some Markov partition $(R_a)_{a \in \mathcal{A}}$. From now on, we fix a parameter $\beta_Z >1$. It will be chosen in the end of subsection \ref{sec:locdelta}: its role will be to make some terms negligeable in a future Taylor expansion. We will denote by $\text{Rect}_{\beta_Z}(\sigma)$ the set of rectangles of the form $R = \bigcap_{k = -k_1}^{k_2} f^k(R_{a_k})$ for some word $(a_k)_{k} \in \mathcal{A}^{\llbracket -k_1, k_2 \rrbracket}$ with unstable diameter $\simeq \sigma$ and stable diameter $\simeq \sigma^{\beta_Z}$. To be more precise, we choose our rectangles as follow. \\

Recall that for each $a$, the boundary of $R_a$ have zero measure. For each $a$, write $R_a = [U_a,S_a]$. Now, for each $n \geq 0$, define $({U^{(n)}_{\mathbf{a}} })_{\mathbf{a} \in J^n}$ as the partition:
$$ U^{(n)}_{\mathbf{a}} := U_{a_1} \cap f^{-1} (R_{a_2}) \cap \dots \cap f^{-n}(R_{a_n}).$$ For $\mu$-almost every $x \in R_a$, for each $n \geq 0$, there exists a unique $\mathbf{a}^{(n)} \in \mathcal{A}^n$ such that $[x,x_a] \in U_{\mathbf{a}^{(n)}}^{(n)}$. As $n$ grows, the diameter of those goes to zero exponentially quickly. Let $n(x)$ be the smallest integer $n \geq 0$ such that $x \in U_{\mathbf{a}^{(n(x))}}^{(n(x))}$ and $\text{diam}^u(U_{\mathbf{a}^{(n(x))}}^{(n(x))}) \leq \sigma$. The unstable part of the partition is then given by $(U(x))_{x \in \cup_a U_a}$, where $U(x) := U_{\mathbf{a}^{(n(x))}}^{(n(x))}$. This is a finite partition of $(U_a)$ with elements of diameter $\simeq \sigma$. We then do a similar construction $(S(x))_{x \in \cup_a S_a}$ in the stable direction, but for the scale $\sigma^{\beta_Z}$. We can then consider the partition of $\cup_a R_a$ given by the rectangles $R(x):=[U(x),S(x)]$. This partition will be called $\text{Rect}_{\beta_Z}(\sigma)$. \\

The point of this subsection is to prove the following reduction:

\begin{proposition}\label{prop:locQNL}
Suppose that there exists $C_0 \geq 1$ and $\gamma>0$ such that:
$$ \forall \sigma > 0, \ \forall R \in \text{Rect}_{\beta_Z}(\sigma), \ \mu^{\otimes 2}\Big( (p,q) \in R^2, \ |\Delta(p,q)| \leq 9 \sigma^{1+\beta_Z+\alpha} \Big) \leq C_0 \mu(R)^2 \sigma^\gamma. $$
Then, the conclusion of Theorem \ref{th:QNL} holds.
\end{proposition}

To this aim, we use two elementary localization lemma, each based on Cauchy-Schwarz inequality, to prove a similar statement for product measures. Localizing will be useful to be able to use local asymptotic estimates. We then use the local product structure of equilibrium states.

\begin{lemma}[First localization lemma]\label{lem:1loc}
Let $\lambda$ be a Borel probability measure on a topological space $X$ and $(X_i)_{i \in I}$ a finite partition with $\lambda(X_i) >0$. Let $h:X \rightarrow \mathbb{R}$ be a measurable function. Then, for all $\sigma>0$,
$$ \lambda^{\otimes 2}( (x,x') \in X^2, \ |h(x)-h(x')| \leq \sigma)^2 \leq \sum_{i \in I} \frac{1}{\lambda(X_i)} \lambda^{\otimes 2}( (x,\tilde{x}) \in X_i^2, \ |h(x)-h(\tilde{x})| \leq 2 \sigma).  $$
\end{lemma}

\begin{proof} The proof is a succession of Cauchy-Schwarz inequalities. We have:
$$\lambda^{\otimes 2}( (x,x') \in X^2, \ |h(x)-h(x')| \leq \sigma)^2 = \Big(\int_X \int_X \mathbb{1}_{[-\sigma,\sigma]}(h(x)-h(x')) d\lambda(x) d\lambda(x')\Big)^2 $$
$$ \leq \int_X \Big(\int_X \mathbb{1}_{[-\sigma,\sigma]}(h(x)-h(x')) d\lambda(x) \Big)^2 d\lambda(x') $$
$$ = \int_X \Big( \sum_{i \in I} \lambda(X_i) \frac{1}{\lambda(X_i)} \int_{X_i} \mathbb{1}_{[-\sigma,\sigma]}(h(x)-h(x')) d\lambda(x) \Big)^2 d\lambda(x')  $$
$$ \leq \int_X \sum_{i \in I} \lambda(X_i) \Big( \frac{1}{\lambda(X_i)} \int_{X_i} \mathbb{1}_{[-\sigma,\sigma]}(h(x)-h(x')) d\lambda(x) \Big)^2 d\lambda(x')  $$
$$ = \int_X \sum_{i \in I} \frac{1}{\lambda(X_i)} \iint_{X_i^2} \mathbb{1}_{[-\sigma,\sigma]}(h(x)-h(x')) \mathbb{1}_{[-\sigma,\sigma]}(h(\tilde{x})-h(x')) d\lambda^2(x,\tilde{x})  d\lambda(x') $$
$$ \leq  \sum_{i \in I} \frac{1}{\lambda(X_i)} \iint_{X_i^2} \mathbb{1}_{[-2 \sigma, 2 \sigma]}(h(x)-h(\tilde{x})) d\lambda^2(x,\tilde{x})  $$
$$ = \sum_{i \in I} \frac{1}{\lambda(X_i)} \lambda^{\otimes 2}( (x,\tilde{x}) \in X_i^2, \ |h(x)-h(\tilde{x})| \leq 2 \sigma). $$
\end{proof}

\begin{lemma}[Second localization lemma]\label{lem:2loc}
Let $\lambda_X$ (resp. $\lambda_Y$) be a borel measure on a topological space $X$ (resp. $Y$) and let $\lambda_Z := \lambda_X \otimes \lambda_Y$ be a probability measure on $Z := X \times Y$. Let $(X_i)$ (resp. $(Y_i)$) be a partition of $X$ (resp. $Y$) with $\lambda_X(X_i) > 0$ (resp. $\lambda_Y(Y_i)>0$). Denote $Z_{i,j} := X_i \times Y_j$. Let $h:Z \rightarrow \mathbb{R}$ be a measurable map. Then, for all $\sigma>0$, $$\lambda_Z^{\otimes 2}\Big( ((x,y),(x',y')) \in Z^2, \ |h(x,y)-h(x',y) - h(x,y') + h(x',y') | \leq \sigma \Big)^4 $$
$$ \leq \sum_{(i,j) \in I \times J} \frac{1}{\lambda_Z(Z_{i,j})} \lambda_Z^{\otimes 2}\Big( ((x,y),(x',y')) \in Z_{i,j}^2, \ | h(x,y) - h(x',y) - h(x,y') + h(x',y') | \leq 4\sigma ) \Big) $$
\end{lemma}

\begin{proof}
The proof uses successive Cauchy-Schwarz inequalities and the first localization lemma \ref{lem:1loc}. We have:
$$\lambda_Z^{\otimes 2}\Big( ((x,y),(x',y')) \in Z^2, \ |h(x,y)-h(x',y) - h(x,y') + h(x',y') | \leq \sigma \Big)^4$$
$$ = \Big( \iint_{Y^2} \lambda_X^{\otimes 2}( (x,x') \in X^2, \big| \big(h(x,y) - h(x,y')\big) - \big(h(x',y) - h(x',y')\big) \big| \leq \sigma ) d\lambda_Y^{\otimes 2}(y,y') \Big)^4 $$
$$ \leq \Big( \iint_{Y^2} \lambda_X^{\otimes 2}( (x,x') \in X^2, \big| \big(h(x,y) - h(x,y')\big) - \big(h(x',y) - h(x',y')\big) \big| \leq \sigma )^2  d\lambda_Y^{\otimes 2}(y,y') \Big)^2 $$
$$ \leq \Big( \iint_{Y^2} \sum_{i \in I} \frac{1}{\lambda(X_i)} \lambda_X^{\otimes 2}( (x,x') \in X_i^2, \big| \big(h(x,y) - h(x,y')\big) - \big(h(x',y) - h(x',y')\big) \big| \leq 2 \sigma  ) d\lambda_Y^{\otimes 2}(y,y') \Big)^2 $$
$$ = \Big( \sum_{i \in I} \lambda_X(X_i) \frac{1}{\lambda_X(X_i)^2} \iint_{X_i^2} \lambda_Y^{\otimes 2}( (y,y') \in Y^2, \ \big| \big(h(x,y) - h(x,y')\big) - \big(h(x',y) - h(x',y')\big) \big| \leq 2\sigma ) d\lambda_X^{\otimes 2}(x,x') \Big)^2 $$
$$ \leq  \sum_{i \in I} \lambda_X(X_i) \frac{1}{\lambda_X(X_i)^2} \iint_{X_i^2} \lambda_Y^{\otimes 2}( (y,y') \in Y^2, \ \big| \big(h(x,y) - h(x',y)\big) - \big(h(x,y') - h(x',y')\big) \big| \leq 2\sigma )^2 d\lambda_X^{\otimes 2}(x,x') $$
$$ \leq \sum_{(i,j) \in I \times J} \frac{1}{\lambda(Z_{i,j})} \lambda_Z^{\otimes 2}\Big( ((x,y),(x',y')) \in Z_{i,j}^2, \ | h(x,y) - h(x,y') - h(x',y) + h(x',y') | \leq 4\sigma ) \Big) .$$
\end{proof}

We are ready to prove Proposition \ref{prop:locQNL}, using the local product structure of $\mu$.

\begin{proof}[Proof (of Proposition \ref{prop:locQNL})]

Let $a \in \mathcal{A}$ and let $\sigma>0$. Denote $\Gamma := 1+\beta_Z+\alpha$. Since $\tau$ is Lipschitz, and by uniform hyperbolicity of the dynamics, there exists $n_0(\sigma)\geq 1$ such that
$$ \forall (p,q) \in R_a, \ |\Delta(p,q) - H_{n_0(\sigma)}(p,q)| \leq \sigma^{\Gamma}, $$
where $H_{n_0}(p,q) = \sum_{n=-n_0}^{n_0} T_n(p,q)$. It follows that
$$ \mu^{\otimes 2}\Big( (p,q) \in R_a^2, \ |\Delta(p,q)| \leq \sigma^{\Gamma} \Big) \leq \mu^{\otimes 2}\big( (p,q) \in R_a^2, \ |\sum_{n=-n_0}^{n_0} T_n(p,q)| \leq 2 \sigma^{\Gamma} \Big). $$
We fix some $x_a \in R_a$ and write, using the local product structure:
$$ \mu^{\otimes 2}\Big( (p,q) \in R_a^2, \ |H_{n_0}(p,q)| \leq 2 \sigma^{\Gamma} \Big) $$
$$ = \int_{U_{x_a}} \int_{S_{x_a}} \mathbb{1}_{[-2\sigma^{\Gamma},2\sigma^{\Gamma}]}( H_{n_0}([z,y],[z',y']) ) e^{\omega_{x_a}(z,y)} e^{\omega_{x_a}(z',y')} (d\mu_{x_a}^s)^{\otimes 2}(y) (d\mu_{x_a}^u)^{\otimes 2}(z) $$
$$ \leq e^{2 \| \omega_{a} \|_{\infty,U_a \times S_a}} (\mu_{x_a}^s \otimes \mu_{x_a}^u)^{\otimes 2}\Big( ((z,y),(z',y')), \ |h_{n_0}( z,y ) - h_{n_0}(z',y) - h_{n_0}(z,y') + h_{n_0}(z',y') | \leq 2 \sigma^{\Gamma} \Big),  $$
where $h_{n_0}(z,y) := \sum_{n=-n_0}^{n_0} \tau(f^n([z,y]))$. The second localisation Lemma \ref{lem:2loc} then gives the bound
$$ (\mu_{x_a}^s \otimes \mu_{x_a}^u)^{\otimes 2}\Big( ((z,y),(z',y')), \ |h_{n_0}( z,y ) - h_{n_0}(z',y) - h_{n_0}(z,y') + h_{n_0}(z',y') | \leq 2 \sigma^{\Gamma} \Big)^4 $$
$$ \leq \sum_{R \in \text{Rect}_{\beta_Z}(\sigma)} \frac{(\mu_{x_a}^s \otimes \mu_{x_a}^u)^{\otimes 2}\Big( ((z,y),(z',y')) \in \tilde{R}^2, \ |h_{n_0}( z,y ) - h_{n_0}(z',y) - h_{n_0}(z,y') + h_{n_0}(z',y') | \leq 8 \sigma^{\Gamma} \Big)}{\mu_{x_a}^s(S_{x_a}) \mu_{x_a}^u(U_{x_a})} . $$
Using  again the fact that $\omega_a \in L^\infty$ allows us to get the bound
$$ \mu^{\otimes 2}\Big( (p,q) \in R_a^2, \ |H_{n_0}(p,q)| \leq 2 \sigma^{\Gamma} \Big)^4  $$
$$ \leq \sum_{R \in \text{Rect}_{\beta_Z}(\sigma)} \frac{e^{C_1(\mu)}}{\mu(R)} \mu^{\otimes 2}\Big( (p,q \in R^2, \ |H_{n_0}(p,q)| \leq 8 \sigma^{\Gamma} \Big) ,$$
for some $C_1(\mu) \geq 1$. Now, using the fact that: $ \forall (p,q) \in R_a, \ |\Delta(p,q) - H_{n_0}(p,q)| \leq \sigma^{\Gamma} $, we finally find our last estimate:
$$ \mu^{\otimes 2}\Big( (p,q) \in R_a^2, \ |\Delta(p,q)| \leq  \sigma^{\Gamma} \Big)^4 $$
$$ \leq \sum_{R \in \text{Rect}_{\beta_Z}(\sigma)} \frac{e^{C_1(\mu)}}{\mu(R)} \mu^{\otimes 2}\Big( (p,q \in R^2, \ |\Delta(p,q)| \leq 9 \sigma^{\Gamma} \Big). $$
This is enough to conclude. Indeed, plugging the local nonconcentration estimate of the hypothesis of Proposition \ref{prop:locQNL} in the sum yields, for all $\sigma \in (0,1)$:
$$ \mu^{\otimes 2}\Big( (p,q) \in R_a^2, \ |\Delta(p,q)| \leq  9 \sigma^{\Gamma} \Big) \leq \Big(  \sum_{R \in \text{Rect}_{\beta_Z}(\sigma)} \frac{e^{C_1(\mu)}}{\mu(R)} C_0 \mu(R)^2 9 \sigma^\gamma \Big)^{1/4} \leq 2 C_0 e^{C_1(\mu)} \sigma^{\gamma/4} .$$
\end{proof}

We conclude this section with a final reformulation. For $p \in \Omega$ and $\sigma>0$ we will denote $S_p(\sigma) := W^s_{loc}(p) \cap \Omega \cap B(p,\sigma)$, $U_p(\sigma) := W^u_{loc}(p) \cap \Omega \cap B(p,\sigma)$ and $R_p^{\beta_Z}(\sigma) := [U_p(\sigma),S_p(\sigma^{\beta_Z})]$.

\begin{lemma}\label{lem:locQNL}
Suppose that there exists $\gamma>0$ and $C_0 \geq 1$ such that, for any $\sigma > 0$, for all $p \in \Omega$,  $$ \mu( q \in R_p^{\beta_Z}(\sigma) , \ |\Delta(p,q)| \leq \sigma^{1+\alpha+\beta_Z} ) \leq C_0 \mu(R_p^{\beta_Z}(\sigma))\sigma^\gamma . $$
Then the conclusion of Theorem  \ref{th:QNL} holds.
\end{lemma}

\begin{proof}
Suppose that the estimate of Lemma \ref{lem:locQNL} holds. We are going to check that Proposition \ref{prop:locQNL} applies. Let $R \in \text{Rect}_{\beta_Z}(\sigma)$. We have:
$$ \mu^{\otimes 2}\Big( (p,q) \in R^2, \ |\Delta(p,q)| \leq 9 \sigma^{1+\alpha+\beta_Z} \Big) = \int_{R} \mu(q \in R, \ |\Delta(p,q)| \leq 9 \sigma^{1+\alpha+\beta_Z}) d\mu(p). $$
The key fact to use is that the measure $\mu$ is doubling in our 2-dimensional setting. In fact, using the local product structure of $\mu$ and the fact that $\mu_p^u,\mu_p^s$ are doubling, it is easy to prove the following: there exists $C_1,C_2 \geq 10$ such that, for any $p \in R$, $$ R \subset R_p^{\beta_Z}(C_1 \sigma) \quad ; \quad \mu(R_p^{\beta_Z}(C_1 \sigma)) \leq C_2 \mu(R). $$
(See \cite{Le24}, section 5.9 for a proof.)
From this fact, it follows that
$$ \mu(q \in R, \ |\Delta(p,q)| \leq 9\sigma^{1+\alpha+\beta_Z}) \leq \mu(q \in R_p^{\beta_Z}(C_1 \sigma), \ |\Delta(p,q)| \leq (C_1 \sigma)^{1+\alpha+\beta_Z}) $$
$$ \leq (C_1 \sigma)^\gamma \mu(R_p^{\beta_Z}(C_1 \sigma)) \leq C_1^\gamma C_2 \sigma^\gamma \mu(R). $$
which proves the hypothesis of Proposition \ref{prop:locQNL}.
\end{proof}

Now we know that, to conclude, it suffices to understand the oscillations of $\Delta(p,q)$, for any fixed $p$, when $q$ gets close to $p$. To do so, we will introduce some special coordinate systems associated to the dynamics.

\subsection{Adapted coordinates and templates}

In this section, we construct a family of adapted coordinates in which the dynamics is going to be (almost) linearized. We also define templates (linear form version, and vector field version). We need to introduce some notations first. For each $x \in \Omega$, we will denote
$$ \lambda_x := \partial_u f(x) \quad, \quad \mu_x := \partial_s f(x). $$
Recall that $\lambda_x \in (1,\infty)$ and $\mu_x \in (0,1)$. By compacity of $\Omega$ and continuity of $x \mapsto (\lambda_x,\mu_x)$, there exists $\lambda_-,\lambda_+,\mu_-,\mu+$ such that:
$$ \forall x \in \Omega, \ 0 < \mu_+ < \mu_x <\mu_- < 1 < \lambda_- < \lambda_x < \lambda_+ < \infty. $$
The area-preserving hypothesis ensure that, for any periodic point $x$ of period $n$, we have
$$ \lambda_x \dots \lambda_{f^{n-1}(x)} = \mu_x^{-1} \dots \mu_{f^{n-1}(x)}^{-1} .$$
In other words, for any periodic point $x$ of period $n$, the Birkhoff sums $\sum_{k=0}^{n-1} \ln(\mu_{f^k(x)} \lambda_{f^{k}(x)})$ vanishes. Livsic's Theorem \cite{Liv71} then ensure the cohomology relation $\ln(\mu_x \lambda_x) \sim 0$. In particular, we can always write:
$$\forall x \in \Omega, \forall n \geq 0, \  \lambda_x \dots \lambda_{f^{n-1}(x)} = \mu_x^{-1} \dots \mu_{f^{n-1}(x)}^{-1} e^{O(1)} . $$
Our goal is to introduce some local coordinates along stable and unstable manifolds, for some base point $x \in \Omega$. There is a small technicality coming from the possible non-orientability of the stable and unstable laminations that we will solve by adding the data of some local orientation with our base-point. This is the meaning of the following definition:

\begin{definition}
We define: $$\widehat{\Omega} := \Big\{ (x,e_u,e_s) \in \Omega \times TM^2 \Big|  (e_u,e_s) \in E^u(x) \times E^s(x), \ |e_u|=|e_s|=1 \Big\}.$$
There is a natural projection $p: \widehat{\Omega} \rightarrow \Omega$. The fibers are isomorphic to $\{-1,1\}^2$. Notice also that there is a natural dynamics $\widehat{f} : \widehat{\Omega} \rightarrow \widehat{\Omega}$ such that $f \circ p = p \circ \widehat{f}$. A formula for $\widehat{f}$ is given by
$$ \widehat{f}(x,e_u,e_s) = \Big(f(x), \frac{(df)_x(e_u)}{|(df)_x(e_u)|}, \frac{(df)_x(e_s)}{|(df)_x(e_s)|}\Big). $$
Seeing $\widehat{\Omega}$ as a subset of the frame bundle naturally allows us to see it as a metric space. In the following, we will often denote by $\widehat{x}$ an element of the fiber $p^{-1}(x) \subset \widehat{\Omega}$.
\end{definition}

We extract from \cite{KK07} the existence of linearizing coordinates alongs $W^u$ and $W^s$.

\begin{lemma}
There exists a family of uniformly smooth maps $(\Phi^s_{\widehat{x}})_{\widehat{x} \in \widehat{\Omega}}$, such that for all $\widehat{x} \in \widehat{\Omega}$ $\Phi_{\widehat{x}}^s : \mathbb{R} \longrightarrow W^s(x)$
is a smooth parametrization of $W^s(x)$, $\Phi_{\widehat{x}}^s(0)=x$, $(\Phi_x^u)'(0)=e_s(\widehat{x})$, and:
$$ \forall x \in \widehat{\Omega}, \forall y \in \mathbb{R},  \ f\left( \Phi_{\widehat{x}}^s(y) \right) = \Phi_{\widehat{f}(\widehat{x})}^s( \mu_x y ),  $$
where ${\mu}_x := \partial_s f(x) \in (\mu_+,\mu_-) \subset (0,1)$. The dependence in $\widehat{x}$ of $(\Phi_{\widehat{x}}^s)_{\widehat{x} \in \widehat{\Omega}}$ is Hölder.
\end{lemma}

\begin{lemma}
There exists a family of uniformly smooth maps $(\Phi^u_{\widehat{x}})_{{\widehat{x}} \in \widehat{\Omega}}$, such that for all ${\widehat{x}} \in \widehat{\Omega}$ $\Phi_{\widehat{x}}^u : \mathbb{R} \longrightarrow W^u(x)$
is a smooth parametrization of $W^u(x)$, $\Phi_x^u(0)=x$, $(\Phi_x^u)'(0)=e_u({\widehat{x}})$, and:
$$ \forall {\widehat{x}} \in \widehat{\Omega}, \forall z \in \mathbb{R}, \ f\left( \Phi_{\widehat{x}}^u(z) \right) = \Phi_{\widehat{f}({\widehat{x}})}^u( {\lambda}_x z ),  $$
where $\lambda_x = \partial_u f(x) \in (\lambda_-,\lambda_+) \subset (1,\infty)$. The dependence in ${\widehat{x}}$ of $(\Phi_{\widehat{x}}^u)$ is Hölder.
\end{lemma}

\begin{definition}
These parametrizations often goes outside $\Omega$, but we are only interested by what is happening inside $\Omega$. So let us define:
$$ \Omega_{\widehat{x}}^u := (\Phi_{\widehat{x}}^u)^{-1}(\Omega) \subset \mathbb{R} \quad, \quad \Omega_{\widehat{x}}^s := (\Phi_{\widehat{x}}^s)^{-1}({\Omega}) \subset \mathbb{R} .$$
Notice that, for all ${\widehat{x}} \in \widehat{\Omega}$,  $ 0 \in \Omega_{\widehat{x}}^u $. Moreover:
$$ \forall {\widehat{x}} \in \widehat{\Omega}, \forall z \in \Omega_{{\widehat{x}}}^u, \ {\lambda}_x z \in \Omega_{\widehat{f}({\widehat{x}})}^u \subset \mathbb{R} .$$
Similar statements hold for $\Omega_{\widehat{x}}^s$.
\end{definition}

\begin{remark}
Let us define some further notations. Define, for $n \in \mathbb{Z}$ and $x \in \Omega$:
$$ \lambda_x^{\langle n \rangle} := \partial_u (f^n)(x) \quad; \quad \mu_x^{\langle n \rangle} := \partial_s (f^n)(x)  .$$
Notice that $\lambda_x^{\langle 0 \rangle} = \mu_x^{\langle 0 \rangle} = 1$, $ \lambda_x^{\langle -n \rangle} = (\lambda_{f^{-n}(x)}^{\langle n \rangle})^{-1} $ and $\mu_x^{\langle -n \rangle} = (\mu_{f^{-n}(x)}^{\langle n \rangle})^{-1}$. 
Moreover, we can write some relations involving $(\Phi_{\widehat{x}}^u)$ and $(\Phi_{\widehat{x}}^s)$. For all $n \in \mathbb{Z}$, ${\widehat{x}} \in \widehat{\Omega}$, $y \in \Omega_{\widehat{x}}^s$, $z \in \Omega_{\widehat{x}}^u$, we have:
$$ f^n (\Phi_{\widehat{x}}^u(z)) = \Phi_{\widehat{f}^n({\widehat{x}})}^u(\lambda_x^{\langle n \rangle} z) \quad ; \quad f^n (\Phi_{\widehat{x}}^s(y)) = \Phi_{{\widehat{f}}^n({\widehat{x}})}^s(\mu_x^{\langle n \rangle} y) .$$
\end{remark}

\begin{lemma}[change of parametrizations,\cite{KK07}]\label{lem:aff}
Let $x_0 \in \Omega$ and let $x_1 \in \Omega \cap W_{loc}^u(x)$. Then the real maps $\text{aff}_{\widehat{x_1},\widehat{x_0}} := (\Phi_{\widehat{x_1}}^u)^{-1} \circ \Phi_{\widehat{x_0}}^u : \mathbb{R} \longrightarrow \mathbb{R}$ are affine.
Moreover, there exists $C \geq 1$ and $\alpha>0$ such that $\ln |{\text{aff}_{\widehat{x_1},\widehat{x_0}}}'(0)| \leq C d(x_0,x_1)^\alpha$.
\end{lemma}

These coordinates are already convenient but only linearize the dynamics along the stable or unstable direction. Of course, we can not expect to fully linearize the dynamics in smooth coordinates, but we can still try to introduce coordinates that will linearize the dynamics in a weaker sense, in some particular places. This construction is directly taken from \cite{TZ20}, appendix B. We call them nonstationnary normal coordinates because the base point is not necessarily a fixed point, and because the dynamics can be written in a \say{normal form} in these coordinates.

\begin{lemma}[Nonstationary normal coordinates]\label{lem:coordinates}
There exists two small constants $\rho_1 < \rho < 1$, and a family of uniformly smooth coordinates charts $\{ \iota_{\widehat{x}}:(-\rho,\rho)^2 \rightarrow M \}_{\widehat{x} \in \widehat{\Omega}}$ such that:
\begin{itemize}
\item For every $\widehat{x} \in \Omega$, we have $$ \iota_{\widehat{x}}(0,0) = x, \quad  \iota_{\widehat{x}}(z,0) = \Phi^u_{\widehat{x}}(z), \quad \iota_{\widehat{x}}(0,y) = \Phi_{\widehat{x}}^s(y) ,$$
\item the map $f_{\widehat{x}} := \iota_{\widehat{f}({\widehat{x}})}^{-1} \circ f \circ \iota_{\widehat{x}} : (-\rho_1,\rho_1)^2 \longrightarrow (-\rho,\rho)^2$ is smooth (uniformly in ${\widehat{x}}$) and satisfies
$$ \pi_y( \partial_y f_{\widehat{x}}(z,0)) = \mu_x , \quad \pi_z(\partial_z f_{\widehat{x}}(0,y)) = {\lambda}_x,$$
where $\pi_z$ (resp. $\pi_y$) is the projection on the first (resp. second) coordinate.
\end{itemize}
Furthermore, one can assume the dependence in ${\widehat{x}}$ of $(\iota_{\widehat{x}})_{{\widehat{x}} \in \widehat{\Omega}}$ to be Hölder regular.
\end{lemma}

\begin{proof}
Since the stable/unstable manifolds are smooth, and since they intersect uniformly transversely, we know that we can construct a system of smooth coordinate charts $(\check{\iota}_{\widehat{x}})_{{\widehat{x}} \in \widehat{\Omega}}$ such that, for all ${\widehat{x}} \in \widehat{\Omega}$, $$ \check{\iota}_{\widehat{x}}(0,0) = x, \quad  \check{\iota}_{\widehat{x}}(z,0) = \Phi^u_{\widehat{x}}(z), \quad \check{\iota}_{\widehat{x}}(0,y) = \Phi_{\widehat{x}}^s(y) .$$
One can also assume the dependence in ${\widehat{x}}$ of these to be Hölder regular, since the stable/unstable laminations are Hölder (even $C^{1+\alpha}$). Define $\check{f}_{\widehat{x}} := \check{\iota}_{\widehat{f}({\widehat{x}})}^{-1} \circ f \circ \check{\iota}_{{\widehat{x}}}$. This is a smooth map defined on a neighborhood of zero, with a (hyperbolic) fixed point at zero. Notice also that $(d \check{f}_{\widehat{x}})_0$ is a diagonal map with coefficients $(\lambda_x,\mu_x)$. We will now modify these coordinates so that they satisfy the properties that we are searching for. Define
$$ \check{\rho}_{\widehat{x}}^u(z) := \sum_{n=1}^\infty \left( \ln | \pi_y \partial_y \check{f}_{\widehat{f}^{-n}({\widehat{x}})}(\lambda_x^{\langle -n \rangle} z,0)| - \ln |\mu_{f^{-n}({x})}| \right) $$
and $$ \check{\rho}_{\widehat{x}}^s(y) := \sum_{n=0}^\infty \left( \ln |\pi_z \partial_z \check{f}_{\widehat{f}^{n}({\widehat{x}})}(0,\mu_x^{\langle n \rangle} y)| - \ln |\lambda_{f^{n}(x)}| \right) .$$
Finally, set $\check{\mathcal{D}}_{\widehat{x}}^u(z,y) := (z, y  e^{\check{\rho}_{\widehat{x}}^u(z)} )$, $\check{\mathcal{D}}_{\widehat{x}}^s(z,y) := (z e^{-\check{\rho}_{\widehat{x}}^s(y)}, y)$, $\check{\mathcal{D}}_{\widehat{x}}:= \check{\mathcal{D}}_{\widehat{x}}^u \circ \check{\mathcal{D}}_{\widehat{x}}^s$ and $ \iota_{\widehat{x}} := \check{\iota}_{\widehat{x}} \circ \check{\mathcal{D}}_{\widehat{x}}$. Let us check that $f_{\widehat{x}} := \iota_{\widehat{f}({\widehat{x}})}^{-1} \circ f \circ \iota_{\widehat{x}}$ satisfies the desired relations.
First of all, notice that $\check{\rho}_{\widehat{x}}^u$ and $\check{\rho}_{\widehat{x}}^s$ are smooth and satisfy $\check{\rho}_{\widehat{x}}^u(0)=\check{\rho}_{\widehat{x}}^s(0)=0$. In particular, $\check{\mathcal{D}}_{\widehat{x}}, \check{\mathcal{D}}_{\widehat{x}}^u$ and $ \check{\mathcal{D}}_{\widehat{x}}^s$ are smooth, and coincide with the identity on $\{ (z,y) \ , \ z=0 \text{ or } y=0\}$.
Moreover, $$ \check{\rho}_{\widehat{f}({\widehat{x}})}^u({\lambda}_x z) = \ln |\pi_y \partial_y \check{f}_{{\widehat{x}}}(z,0)| - \ln |\mu_{x}| + \check{\rho}_{{\widehat{x}}}^{u}(z) $$
and $$ \check{\rho}_{\widehat{x}}^s(y) = \ln |\pi_z \partial_z \check{f}_{\widehat{x}}(0,y)| - \ln |\lambda_x| + \check{\rho}^s_{\widehat{f}({\widehat{x}})}(\mu_x y) .$$
Now let us write $f_{\widehat{x}}$ in terms of $\check{f}_{\widehat{x}}$: we have
$$ f_{\widehat{x}} = \iota_{\widehat{f}({\widehat{x}})}^{-1} \circ f \circ \iota_{\widehat{x}} = (\check{\mathcal{D}}_{\widehat{f}({\widehat{x}})}^s)^{-1} \circ (\check{\mathcal{D}}_{\widehat{f}({\widehat{x}})}^u)^{-1} \circ \check{f}_{\widehat{x}} \circ \check{\mathcal{D}}_{{\widehat{x}}}^u \circ \check{\mathcal{D}}_{{\widehat{x}}}^s. $$
Hence:
$$ (df_{\widehat{x}})_{(z,0)} = d((\mathcal{D}_{\widehat{f}({\widehat{x}})}^s)^{-1})_{(\lambda_x z,0)} \circ d((\mathcal{D}_{\widehat{f}({\widehat{x}})}^u)^{-1})_{(\lambda_x z,0)} \circ (d\check{f}_{\widehat{x}})_{(z,0)} \circ (d \mathcal{D}_{\widehat{x}}^u)_{(z,0)} \circ (d \mathcal{D}_{\widehat{x}}^s)_{(z,0)} $$
Abusing a bit notations, we can write in matrix form:
$$ (df_{\widehat{x}})_{(z,0)} = \begin{pmatrix} 1 & (*) \\ 0 & 1 \end{pmatrix} \begin{pmatrix} 1 & 0 \\ 0 & e^{-\check{\rho}_{\widehat{f}({\widehat{x}})}(\lambda_x z)} \end{pmatrix} \begin{pmatrix} \lambda_x & \pi_z \partial_y \check{f}_{\widehat{x}}(z,0) \\ 0 & \pi_y \partial_y \check{f}_{\widehat{x}}(z,0) \end{pmatrix} \begin{pmatrix} 1 & 0 \\ 0 & e^{\check{\rho}_{\widehat{x}}^u(z)} \end{pmatrix} \begin{pmatrix} 1 & (*) \\ 0 & 1 \end{pmatrix}  $$
$$ = \begin{pmatrix} \lambda_x & (*) \\ 0 & e^{-\check{\rho}_{\widehat{f}({\widehat{x}})}^u(\lambda_x z) + \check{\rho}_{\widehat{x}}^u(z)} \pi_y \partial_y \check{f}_{\widehat{x}}(z,0) \end{pmatrix} = \begin{pmatrix} \lambda_x & (*) \\ 0 & \mu_x \end{pmatrix} ,$$
which implies in particular that $\pi_y \partial_y f_{\widehat{x}}(z,0) = \mu_x$. A similar computation shows that
$$ (df_{\widehat{x}})_{(0,y)} = \begin{pmatrix} \lambda_x & 0 \\ (*) & \mu_x \end{pmatrix} .$$
In particular, $\pi_z \partial_z f_{\widehat{x}}(0,y) = \lambda_x$. 
\end{proof}

\begin{remark}\label{rem:temp}
Notice the following convenient property. As soon as the quantities written make sense, we have the identities:
$$ (df^{\langle n \rangle}_{\widehat{x}})_{(z,0)} = \begin{pmatrix} \lambda_x^{\langle n \rangle} & (*) \\ 0 & \mu_x^{\langle n \rangle} \end{pmatrix} \quad  ; \quad (df_x^{\langle n \rangle})_{(0,y)} = \begin{pmatrix} \lambda_x^{\langle n \rangle} & 0 \\ (*) & \mu_x^{\langle n \rangle} \end{pmatrix}$$
where $f_{\widehat{x}}^{\langle n \rangle} := \iota_{\widehat{f}^{n}({\widehat{x}})} \circ f^n \circ \iota_{\widehat{x}} : (-(\rho_1/\rho)^n \rho,(\rho_1/\rho)^n \rho)^2 \longrightarrow (-\rho, \rho)^2$.
\end{remark}

This coordinate system is not \say{canonically attached} to the dynamics, since the behavior of $\iota_x$ outside the \say{cross} $\{ (z,y) \ , \ zy=0 \}$ might be completely arbitrary. But the behavior of those coordinates near the cross seems to give rise to less arbitrary objects. Those objects will be called \say{templates} in this article. They are inspired from the \say{templates} appearing in \cite{TZ20}.

\begin{definition}[Templates, dual version]
Let $x \in \Omega$. A template based at $x$ is a continuous 1-form $\xi_x:W^u_{loc}(x) \rightarrow \Omega^1(M)$ such that:
$$ \forall z \in W^u_{loc}(x), \ \text{Ker}(\xi_x)_z \supset E^u(z). $$
We will denote by $\Xi(x)$ the space of templates based at $x$.
\end{definition}

\begin{remark}
Notice that, since $E^u(z)$ moves smoothly along the unstable local manifold $W^u_{loc}(x)$, it makes sense to consider \emph{smooth} templates. Notice further that, since $(df)_z E^u(z) = E^u(f(z))$, the diffeomorphism $f$ acts naturally on templates by taking the pullback. This yields a map $f^* : \Xi(f(x)) \rightarrow \Xi(\widehat{x})$.
\end{remark}

\begin{lemma}[Some important templates.]
There exists a family $(\xi_{\widehat{x}}^s)_{{\widehat{x}} \in \widehat{\Omega}}$ of smooth templates, where $\xi_{\widehat{x}}^s \in \Xi(x)$, that satifies the following invariance relation:
$$\forall z \in W^u_{loc}(x), \ (f^* (\xi_{\widehat{f}(\widehat{x})}^s))_{z} = \mu_x (\xi_{\widehat{x}}^s)_{z} .$$
Moreover, the dependence in $\widehat{x}$ of $\xi_{\widehat{x}}^s$ is Hölder.
\end{lemma}

\begin{proof}
For all ${\widehat{x}} \in \widehat{\Omega}$, define $\xi_{\widehat{x}}^s := (\iota_{\widehat{x}}^{-1})^*(dy)$. It is clear that this defines a smooth template at ${\widehat{x}}$. Furthermore, using remark \ref{rem:temp} for $n=1$:
$$ f^* \xi_{\widehat{f}({\widehat{x}})}^s = f^* ((\iota_{\widehat{f}({\widehat{x}})}^{-1})^* (dy)) = (\iota_{{\widehat{x}}}^{-1})^* ((f_{\widehat{x}})^*(dy)) = (\iota_{{\widehat{x}}}^{-1})^* (d (\pi_y f_{\widehat{x}})) = (\iota_{{\widehat{x}}}^{-1})^* (\mu_x dy) = \mu_x \xi_{{\widehat{x}}}^s .$$
The dependence is Hölder because $(\iota_{\widehat{x}})$ depends on ${\widehat{x}}$ in a Hölder manner.
\end{proof}

One could say that one of the reasons these templates are useful is because they give rise to an isomorphism $E^s(z) \simeq TM/E^u(z)$: the point being that one of those two line bundle is $C^{\infty}$ along local unstable manifolds, whereas the other is only $C^{1+\alpha}$. It is natural to try to find a \say{vector field} version of those templates. We suggest a way to proceed in the following.

\begin{definition}[Templates, vector version]
Let $x \in \Omega$. A (vector) template based at $x$ is a continuous section of the line bundle $TM/E^u$ along $W^u_{loc}(x)$. We will denote by $\Gamma(x)$ the space of (vector) templates at $x$. 
\end{definition}

\begin{remark}
If $X$ is some continuous vector field defined along $W^u_{loc}(x)$, we can take its class modulo $E^u$ to get a (vector) template $[X]$. Notice also that it makes sense to talk about smooth (vector) templates. Notice further that $f$ acts naturally on (vector) templates, since $(df)_z E^u(z) = E^u(f(z))$, by taking the pushforward. This define a map $f_* : \Gamma(x) \rightarrow \Gamma(f(x))$.
\end{remark}

\begin{lemma}[Some important templates.]
There exists a family $([\partial_s^{\widehat{x}}])_{{\widehat{x}} \in \widehat{\Omega}}$ of smooth (vector) templates, where $[\partial_s^{\widehat{x}}] \in \Gamma(x)$, such that:
$$ \forall z \in W^u_{loc}(x), \ (f_*[\partial_s^{\widehat{x}}])_{f(z)} = \mu_x [\partial_s^{\widehat{f}({\widehat{x}})}]_{f(z)}.$$
Moreover, the dependence in ${\widehat{x}}$ of 
$([\partial_s^{\widehat{x}}])$ is Hölder.
\end{lemma}

\begin{proof}
For all ${\widehat{x}} \in \widehat{\Omega}$, define $\partial_s^{\widehat{x}} := (\iota_{\widehat{x}})_*(\partial/\partial y)$ along $W^u_{loc}(x)$. This is smooth. Moreover, using remark \ref{rem:temp}, we find:
$$ f_* \partial_s^{\widehat{x}} = f_* (\iota_{\widehat{x}})_*(\partial_y)= (\iota_{\widehat{f}({\widehat{x}})})_* ((f_x)_* \partial_y) = (\iota_{\widehat{f}({\widehat{x}})})_* ( \mu_x \partial_y + (*) \partial_z ) = \mu_x \partial_s^{\widehat{f}({\widehat{x}})} + (*) \partial_u .$$
Hence, taking the class modulo $E^u$, we find
$$ f_*[\partial_s^{\widehat{x}}] = \mu_x [\partial_{s}^{\widehat{f}({\widehat{x}})}] ,$$
which is what we wanted. The dependence in ${\widehat{x}}$ is then Hölder because of the properties of $\iota_{\widehat{x}}$.
\end{proof}

\begin{remark}[A quick duality remark]
We can define a kind of \say{duality bracket} \newline $\Xi(x) \times \Gamma(x) \rightarrow C^0(W^u_{loc}(x),\mathbb{R})$ by the following formula:
$$ \langle \xi , [X] \rangle := \xi(X) .$$
Our special templates $(\xi_{\widehat{x}}^s)$ and $([\partial_s^{\widehat{x}}])$ can be chosen normalised so that $\langle \xi_{\widehat{x}}^s, [\partial_s^{\widehat{x}}] \rangle = 1$. This will not be useful for us, but can be good to keep in mind.
\end{remark}

\begin{remark}[Templates acting on a space of functions]\label{rem:functions}
It is natural to search for a space of functions on which (vector) templates could acts. A way to do it is as follow.
For each $x \in \Omega$, define $\mathcal{F}(x)$ as the set of functions $h_x$ defined on a neighborhood of $W^u_{loc}(x)$ that are $C^1$ along the stable direction and that vanish along $W^u_{loc}(x)$. In this case, for any point $z \in W^u_{loc}(x)$, we know that $\partial_s h_x(z)$ makes sense, and we know that $\partial_u h_x(z) = 0$ also makes sense. One can then make (vector) templates $[X]$ acts on $h_x$ by setting:
$$\forall z \in W^u_{loc}(x), \  ([X] \cdot h_x)(z) := (X \cdot h_x)(z) .$$
This is well defined.
In the particular case where $[X] = [\partial_s^{\widehat{x}}]$, we get the formula:
$$ \forall z \in (-\rho,\rho), \ ([\partial_s^{\widehat{x}}] \cdot h_x)(\Phi_{{\widehat{x}}}^u(z)) = \partial_y (h_x \circ \iota_{\widehat{x}})(z,0) .$$

Notice that $f$ acts naturally on these space of functions, by taking a pullback $f^* : \mathcal{F}(f(x)) \rightarrow \mathcal{F}(x)$. If we fix $h_{f(x)} \in \mathcal{F}(f(x))$, and if we set $h_x := h_{f(x)} \circ f \in \mathcal{F}(x)$, notice finally that one can write
$$ [\partial_s^{\widehat{x}}] \cdot h_{x} = [\partial_s^{\widehat{x}}] \cdot f^*h_{f(x)} = f_* [\partial_s^{\widehat{x}}] \cdot h_{f(x)} = \mu_x [\partial_s^{\widehat{f}({\widehat{x}})}] \cdot h_{f(x)}.$$
\end{remark}

\begin{lemma}[Changing basepoint]\label{lem:basepoint}
Let $x_0 \in \Omega$. For $x_1 \in \Omega \cap W^u_{loc}(x)$, let $$H(x_0,x_1) = \exp\Big{(}\sum_{n=0}^\infty \left(\ln \mu_{f^{-n}(x_0)} - \ln \mu_{f^{-n}(x_1)} \right)\Big{)} .$$ Then:
$$\forall z \in W^u_{loc}(x), \quad ([\partial_s^{\widehat{x_1}}])_{z} = \pm H(x_1,x_0) ([\partial_s^{\widehat{x_0}}])_{z},$$
where the sign only depends on the local orientation of $\widehat{x_0}$ and $\widehat{x_1}$.
\end{lemma}

\begin{proof}
First of all, notice that $[\partial_s^{\widehat{x}}]$ only change from a global sign when $\widehat{x}$ changes its local orientation (for $x$ fixed).
Then, remember that $TM/E^u$ is a line bundle, and that $[\partial_s^{\widehat{x_0}}]$ doesn't vanish. In particular, there exists a function $a_{\widehat{x_0},\widehat{x_1}}:W^u_{loc}(x_0) \rightarrow \mathbb{R}$ such that:
$$ \forall z \in W^u_{loc}(x_0), \quad ([\partial_s^{\widehat{x_1}}])_{z} = a_{\widehat{x_0},\widehat{x_1}}(z) ([\partial_s^{\widehat{x_0}}])_{z} .$$
The main point is to show that $a_{\widehat{x_0},\widehat{x_1}}$ is $z$-constant. Since the family $([\partial_s^{\widehat{x}}])$ depends in $\widehat{x}$ in a Hölder manner (and locally uniformly in $z$) we know that $a_{\widehat{x_0},\widehat{x_1}}(z) = \pm (1 + O(d(x_0,x_1)^{\alpha}))$ for some $\alpha$, where the sign depends on the orientations. The invariance properties of these (vector) templates yields an invariance property for $a_{\widehat{x_0},\widehat{x_1}}(z)$:
$$ \forall z \in W^u_{loc}(x_0), \  a_{\widehat{x_0},\widehat{x_1}}(z) = \frac{\mu_{x_0}^{\langle - n \rangle}}{\mu_{x_1}^{\langle -n \rangle}} a_{\widehat{f}^{-n}(\widehat{x_0}),\widehat{f}^{-n}(\widehat{x_1})}(f^{-n}(z)).$$
Taking the limit as $n \rightarrow + \infty$ gives the result.
\end{proof}


\subsection{The local behavior of $\Delta$}\label{sec:locdelta}

Now that we have the right tools, we come back to our study of $\Delta$. Recall that $\Delta : \widetilde{\text{Diag}} \rightarrow \mathbb{R}$ is defined as
$$\Delta(p,q) := \sum_{n \in \mathbb{Z}} T_n(p,q) ,$$
where $T_n(p,q) := \tau(f^n(p)) - \tau(f^n([p,q])) - \tau(f^n([q,p])) + \tau(f^n(q))$. The goal of this section is to understand the local behavior of $\Delta(p,q)$ when $q$ gets close to $p$. To this aim, we cut the sum in two, like so: 
$$ \Delta^+(p,q) := \sum_{n=0}^\infty T_n(p,q) \quad ; \quad \Delta^-(p,q) := \sum_{n=1}^\infty T_{-n}(p,q),$$
so that $\Delta = \Delta^- + \Delta^+$. Intuitively, $\Delta^-(p,\cdot)$ is smooth along the unstable direction but is Hölder in the stable direction. The behavior of $\Delta^+(p,\cdot)$ is the opposite: Hölder along the unstable direction and smooth in the stable direction. To catch oscillations for $\Delta$, we will fix an unstable manifold close to $W^u_{loc}(p)$ and try to understand the behavior of $\Delta(p,\cdot)$ along this manifold. The sum $\Delta^-(p,\cdot)$ will be approximated by a polynomial, and $\Delta^+(p,\cdot)$ will be approximated by an autosimilar function behaving like a Weierstrass function. We will thus reduce catching oscillations of $\Delta(p,\cdot)$ to understanding the oscillations of $\Delta^+(p,\cdot)$, along the unstable foliation, modulo polynomials. Let us begin by the local description of $\Delta^+(p,\cdot)$.\\

Let us fix some $p \in \Omega$, and set:
$$ \Delta_p^+(q) := \Delta^{+}(p,q) \quad , \quad T_{p,n}(q) := T_n(p,q) .$$
For each $p$ and $n$, $T_{p,n}$ is $C^{1+\alpha}$, and moreover taking the derivative along the local stable lamination yields: $|\partial_s T_{p,n}(q)| \leq |(\partial_s \tau)( f^n([p,q]))| \mu_{[p,q]}^{\langle n\rangle} |\partial_s \pi_p(q)| + |\partial_s \tau(f^n(q))| \mu_q^{\langle n \rangle} \lesssim \mu_-^n$, where $\pi_p(q) := [p,q]$. It follows that $\Delta_{p}^+$ is indeed $C^{1}$ along the local stable lamination. Moreover, $T_{p,n}$ vanish on $W^u_{loc}(p)$, and so does $\Delta_p^+$. It follows that
$$ \Delta_p^{+} \in \mathcal{F}(p) \quad, \quad T_{p,n} \in \mathcal{F}(p) ,$$
where $\mathcal{F}(p)$ denotes the space of function defined in remark \ref{rem:functions}. Now one of the main point of this subsection is to make the following intuition rigorous: since $\Delta_p^+ \in \mathcal{F}(p)$, we can do a Taylor expansion along the stable foliation. We get an expression like this (denoting $r=[q,p] \in W^u_{loc}(p) \cap W^s_{loc}(q)$):
$$ \Delta_p^+(q) \simeq  \Delta_p^+(r) \pm \partial_s \Delta_p^+(r) d^s(q,r) = \pm \partial_s \Delta_p^+(r) d^s(q,r). $$
It is then convenient to replace $\partial_s \Delta_p^+$ by $[\partial_s^{\widehat{p}}] \cdot \Delta_{p}^+$, to use more easily the autosimilarity properties of $\Delta_p^+$. We then expect all the oscillations of $\Delta_p^+$ to come from wild oscillations of $[\partial_s^{\widehat{p}}] \cdot \Delta_{p}^+$ (recall that, by hypothesis, there exists $p \in \Omega$ for which this is not $C^1$). Now let us make this rigorous: the fact that $\Delta_p^+ \in \mathcal{F}(p)$ ensure that the next definition makes sense.

\begin{definition}
For each ${\widehat{x}} \in \widehat{\Omega}$,  define $X_{\widehat{x}} \in C^{\alpha}((-\rho,\rho) \cap \Omega_{\widehat{x}}^u,\mathbb{R}) $ by:
$$ \forall z \in (-\rho,\rho) \cap \Omega_{\widehat{x}}^u, \quad X_{\widehat{x}}(z) := ([\partial_s^{\widehat{x}}] \cdot \Delta_{x}^+)(\Phi_{\widehat{x}}^u(z)).$$
The family $(X_{\widehat{x}})_{x \in \Omega}$ depends on ${\widehat{x}}$ in a Hölder manner. Moreover, changing the orientation of $\widehat{x}$ into $\widehat{x}'$ gives: $X_{\widehat{x}'}(z) = \pm X_{\widehat{x}}(\pm z)$ for some constant signs.
\end{definition}

\begin{lemma}[autosimilarity]
We have $X_{\widehat{x}}(0)=0$. Moreover, the family $(X_{\widehat{x}})_{\widehat{x} \in \widehat{\Omega}}$ satisfies the following autosimilarity relation:
$$\forall z \in \Omega_{\widehat{x}}^u \cap (-\rho \lambda_x^{-1},\rho \lambda_x^{-1}), \ X_{\widehat{x}}(z) = \widehat{\tau}_{\widehat{x}}(z) + \mu_x X_{\widehat{f}({\widehat{x}})}(\lambda_x z) ,$$
where $\widehat{\tau}_{\widehat{x}}(z) := ([\partial_{\widehat{x}}^s] \cdot T_{x,0})(\Phi_{\widehat{x}}^u(z)) \in C^{\alpha}$.
\end{lemma}

\begin{proof}
Notice that $T_{p,n+1}(q) = T_{f(p),n}(f(q))$. It follows that:
$$ \Delta_p^+(q) = T_{p,0}(q) + \sum_{n=0}^\infty T_{f(p),n}(f(q)) = T_{p,0}(q) + \Delta_{f(p)}^+(f(q)) .$$
Making the vector template $[\partial_s^{\widehat{p}}]$ acts on this along $W^u_{loc}(p)$ yields (using the invariance properties of the family $([\partial_s^{\widehat{x}}])_{\widehat{x} \in \widehat{\Omega}}$):
$$([\partial_s^{\widehat{p}}] \cdot \Delta_p^+) = ([\partial_s^{\widehat{p}}] \cdot T_{p,0}) + ([\partial_s^{\widehat{p}}] \cdot (\Delta_{f(p)}^+ \circ f)) = ([\partial_s^{\widehat{p}}] \cdot T_{p,0}) + \mu_p ([\partial_s^{\widehat{f}({\widehat{p}})}] \cdot \Delta_{f(p)}^+) \circ f  .$$
Testing this equality on $\Phi_{\widehat{p}}^u(z)$ gives the desired relation, since $f \circ \Phi_{\widehat{p}}^u = \Phi_{\widehat{f}({\widehat{p}})}^u \circ (\lambda_p Id )$.
\end{proof}

\begin{lemma}[regularity of $\widehat{\tau}_{\widehat{x}}$]\label{lem:reg1}
The function $\widehat{\tau}_{\widehat{x}}$ is $C^{1+\alpha}$. Moreover, $\tau_{\widehat{x}}(0)=0$.
\end{lemma}

\begin{proof}
Let us do an explicit computation of $\widehat{\tau}_{\widehat{x}}$. By definition of $[\partial_s^{\widehat{x}}]$:
$$ \widehat{\tau}_{\widehat{x}}(z) = ([\partial_s^{\widehat{x}}] \cdot T_{x,0})(\Phi_{\widehat{x}}^u(z)) = \partial_y (T_{x,0} \circ \iota_{\widehat{x}})(z,0) .$$
Recall that $T_{x,0}(\iota_{\widehat{x}}(z,y)) = \tau(\iota_{\widehat{x}}(z,y)) - \tau([x,\iota_{\widehat{x}}(z,y)])-\tau([\iota_{\widehat{x}}(z,y),x]) + \tau(x) \in C^{1+\alpha}$. Define $$\pi_{\widehat{x}}^s(z,y) := \iota_{\widehat{x}}^{-1}( [x,\iota_{\widehat{x}}(z,y)] ) \in \{ (0,y') , y' \in (-\rho,\rho) \}$$ and $$\pi_{\widehat{x}}^u(z,y) := \iota_{\widehat{x}}^{-1}( [\iota_{\widehat{x}}(z,y),x] ) \in \{ (z',0) , z' \in (-\rho,\rho) \}.$$
Define also $\tau_{\widehat{x}} := \tau \circ \iota_{\widehat{x}}$. Then:
$$ T_{x,0}(\iota_{\widehat{x}}(z,y)) = \tau_{\widehat{x}}(z,y) - \tau_{\widehat{x}}(\pi_{\widehat{x}}^u(z,y)) - \tau_{\widehat{x}}(\pi_{\widehat{x}}^s(z,y)) + \tau_{\widehat{x}}(0). $$
For each point $z \in (-\rho,\rho)$, let $\vec{N}_{\widehat{x}}(z)$ be a vector pointing along the direction $\iota_{\widehat{x}}^{-1}(E^s)$, and normalize it so that $\vec{N}_{\widehat{x}}(z) = \partial_y + n_{\widehat{x}}(z) \partial_z$. By regularity of $E^s$, we can choose $\vec{N}_{\widehat{x}}(z)$ to be $C^{1+\alpha}$ in $z$. We can then, for each ${\widehat{x}},z$, find a (smooth) path $ t \mapsto \gamma_{\widehat{x}}(z,t)$ such that $\pi_{\widehat{x}}^{u} \circ \gamma_{\widehat{x}}(z,t) = (z,0) $ and such that $\partial_t \gamma_{\widehat{x}}(z,0) = \vec{N}_{\widehat{x}}(z)$. (We just follow the stable lamination in coordinates.) Using this path, we can compute the derivative of $T_{x,0} \circ \iota_{\widehat{x}}$ as follow:
$$\partial_y (T_{x,0} \circ \iota_{\widehat{x}})(z,0) = \Big{(}(\partial_y + n_{\widehat{x}}(z) \partial_z) \cdot (T_{x,0} \circ \iota_{\widehat{x}}) \Big{)}(z,0) = \frac{d}{dt} T_{x,0}(\iota_{\widehat{x}}(\gamma_{\widehat{x}}(z,t)))_{|t=0}$$
$$ = \frac{d}{dt}_{|t=0} \Big{(} \tau_{\widehat{x}}( \gamma_{\widehat{x}}(z,t) ) - \tau_{\widehat{x}}(\pi_{\widehat{x}}^u( \gamma_{\widehat{x}}(z,t) )) - \tau_{\widehat{x}}(\pi_{\widehat{x}}^s( \gamma_{\widehat{x}}(z,t))) + \tau_{\widehat{x}}(0) \Big{)}  $$
$$ = \partial_y \tau_{\widehat{x}}(z,0) + n_{\widehat{x}}(z) \partial_z \tau_{\widehat{x}}(z,0) - (d\tau_{\widehat{x}})_{(0,0)} \circ (d\pi_{\widehat{x}}^s)_{(z,0)} ( \vec{N}_{\widehat{x}}(z) ) $$
$$ = \partial_y \tau_{\widehat{x}}(z,0) + n_{\widehat{x}}(z) \partial_z \tau_{\widehat{x}}(z,0) - \partial_y \tau_{\widehat{x}}(0,0) \pi_y( \partial_y \pi_{\widehat{x}}^s(z,0)),$$
since $(d\pi_{\widehat{x}}^s)_{(z,0)}(\partial_z) = 0$, as $\pi_{\widehat{x}}^s(z,0) = (0,0)$, and since $\pi_z \circ \pi_{\widehat{x}}^s = 0$.
In this expression, everything is $C^{1+\alpha}$ (since $\tau \in \text{Reg}_u^{1+\alpha}$); except eventually the last term $\eta_{\widehat{x}}(z) := \pi_y \partial_y \pi_{\widehat{x}}^s(z,0)$. Let us prove that $\eta_{\widehat{x}}(z) $ is, in fact, constant and equal to one. First of all, the maps $(\eta_{\widehat{x}})$ are at least continuous (and this, uniformly in ${\widehat{x}}$). Moreover, we have $\pi_y \pi_{\widehat{x}}^s(0,y) = y$, and hence $\eta_{\widehat{x}}(0)=1$. To conclude, let us use the fact that the stable lamination is $f$ invariant. This remark, written in coordinates, yields (as soon as the relation makes sense):
$$ \pi_{\widehat{x}}^s = f_{\widehat{f}^{-n}({\widehat{x}})}^{\langle n \rangle} \circ \pi_{\widehat{f}^{-n}({\widehat{x}})}^s \circ f_{\widehat{x}}^{\langle -n \rangle}. $$
Taking the differential at $(z,0)$ yields, using remark \ref{rem:functions}:
$$ \forall n \geq 0, \forall z \in (-\rho,\rho), \ (d\pi_{\widehat{x}}^s)_{(z,0)} = (df_{\widehat{f}^{-n}({\widehat{x}})}^{\langle n \rangle})_{(0,0)} \circ (d\pi^s_{\widehat{f}^{-n}({\widehat{x}})})_{(\lambda_x^{\langle -n \rangle}z,0)} \circ (df_{\widehat{x}}^{\langle -n \rangle})_{(z,0)} $$ $$=\begin{pmatrix} \lambda_{f^{-n}(x)}^{\langle n \rangle} & 0 \\ 0 & \mu_{f^{-n}(x)}^{\langle n \rangle} \end{pmatrix} \begin{pmatrix} 0 & 0 \\ 0 & \eta_{\widehat{f}^{-n}({\widehat{x}})}(\lambda_x^{\langle -n \rangle} z) \end{pmatrix}  \begin{pmatrix} \lambda_x^{\langle -n \rangle} & (*) \\ 0 & \mu_{x}^{\langle -n \rangle} \end{pmatrix} = \begin{pmatrix} 0 & 0 \\ 0 & \eta_{\widehat{f}^{-n}({\widehat{x}})}(\lambda_x^{\langle -n \rangle} z) \end{pmatrix} .$$
It follows that: 
$$ \forall z, \forall n \geq 0, \ \eta_{\widehat{x}}(z) = \eta_{\widehat{f}^{-n}({\widehat{x}})}(\lambda_{{x}}^{\langle -n \rangle} z) \underset{n \rightarrow \infty}{\longrightarrow} 1 .$$
In conclusion, we get the following expression for $\widehat{\tau}_{\widehat{x}}$:
$$ \widehat{\tau}_{\widehat{x}}(z) = \partial_y \tau_{\widehat{x}}(z,0) + n_{\widehat{x}}(z) \partial_z \tau_{\widehat{x}}(z,0) - \partial_y \tau_{\widehat{x}}(0,0), $$
which is a $C^{1+\alpha}$ function that vanish at zero. Let us compute its derivative at zero: we have
$$ (\widehat{\tau}_{\widehat{x}})'(0) = \partial_z \partial_y \tau_{\widehat{x}}(0,0) + n_{\widehat{x}}'(0) \partial_z \tau_{\widehat{x}}(0,0),$$
since $n_{\widehat{x}}(0)=0$.\end{proof}

\begin{remark}
Notice that, denoting by $\pi_p^S(q) := [p,q]$, our previous arguments prove that the map $\partial_s \pi_p^S : W^u_{loc}(p) \cap \Omega \rightarrow \mathbb{R}$ is $C^{1+\alpha}$ (and this, uniformly in $p$).
\end{remark}




\begin{lemma}[Change of basepoint]\label{lem:Aff}
Let $x_0 \in \Omega$. Let $x_1 \in W^u_{loc}(x)$ be close enough to $x_0$. Then, there exists $\text{Aff}_{\widehat{x_0},\widehat{x_1}}:\mathbb{R} \rightarrow \mathbb{R}$, an (invertible) affine map, such that: 
$$ \text{Aff}_{{\widehat{x_0}},{\widehat{x_1}}}\Big{(} X_{{\widehat{x_1}}}(z) \Big{)}=  X_{\widehat{x_0}}(\text{aff}_{{\widehat{x_0}},{\widehat{x_1}}}(z)), $$
where $\text{aff}_{{\widehat{x_0}},{\widehat{x_1}}} = (\Phi_{{\widehat{x_0}}}^u)^{-1} \circ \Phi_{{\widehat{x_1}}}^{u}$ is the affine change of charts defined in lemma \ref{lem:aff}.
Moreover, there exists $C \geq 1$ and $\alpha>0$ such that $\ln |{\text{Aff}_{{\widehat{x_0}},{\widehat{x_1}}}}'(0)| \leq C d(x_0,x_1)^\alpha$.
\end{lemma}

\begin{proof}
Let $x_0 \in \Omega$ and let $x_1 \in W^u_{loc}(x)$ be close enough to $x_0$. We have, for $q$ in a neighborhood of $x_1$:
$$ \Delta^+( x_1, q ) = \Delta^+( x_1, [x_0,q] ) + \Delta^+(x_0,q). $$
(Picture two rectangles with one shared side, the first one touching $x_1$, $x_0$, $[x_0,q]$, and the second one touching $x_0$, $[x_0,q]$ and $q$.) We differentiate (w.r.t. $q$) with the vector template $[\partial_s^{\widehat{x_1}}]$ along $W^u_{loc}(x_1)$ to find:
$$ X_{\widehat{x_1}}(z) = [\partial_s^ {\widehat{x_1}}] \cdot (\Delta_{x_1}^+ \circ [x_0,\cdot])(\Phi_{{\widehat{x_1}}}^u(z)) + ([\partial_s^{{\widehat{x_1}}}] \cdot \Delta_{x_0}^+)(\Phi_{{\widehat{x_1}}}^u(z)) .$$
The first thing to recall is that $\Phi_{{\widehat{x_1}}}^u = \Phi_{\widehat{x_0}}^u \circ (\Phi_{\widehat{x_0}}^{u})^{-1} \circ \Phi_{{\widehat{x_1}}}^u = \Phi_{\widehat{x_0}}^u \circ \text{aff}_{{\widehat{x_0}},{\widehat{x_1}}} $, and moreover, by Lemma \ref{lem:basepoint}, $[\partial_s^{{\widehat{x_1}}}] = \pm H(x_0,x_1) [\partial_s^{\widehat{x_0}}]$. From this, we see that the last term is $\pm H(x_0,x_1) X_{\widehat{x_0}}(\text{aff}_{{\widehat{x_0}},{\widehat{x_1}}}(z))$. To conclude, we only need to show that
$$ [\partial_s^ {{\widehat{x_1}}}] \cdot (\Delta_{{\widehat{x_1}}}^+ \circ [x_0,\cdot]) = \pm H(x_0,x_1) [\partial_s^{{\widehat{x_0}}}] \cdot (\Delta_{{\widehat{x_1}}}^+ \circ [x_0,\cdot])  $$
is constant along $W^u_{loc}({\widehat{x_0}})$. In coordinates, we see that:
$$ [\partial_s^{{\widehat{x_0}}}] \cdot (\Delta_{{\widehat{x_1}}}^+ \circ [x_0,\cdot])(\Phi_{\widehat{x_0}}(z))  = \partial_y (\Delta_{{\widehat{x_1}}}^+ \circ [x_0,\cdot] \circ \iota_{\widehat{x_0}})(z,0) = \partial_y( \Delta_{{\widehat{x_1}}}^+ \circ \iota_{\widehat{x_0}} \circ \pi_{\widehat{x_0}}^s)(z,0), $$
where $\pi_{\widehat{x_0}}^s$ is defined in the proof of Lemma \ref{lem:reg1}. Recall from this proof that
$$ (d\pi_{\widehat{x_0}}^s)_{(z,0)} = \begin{pmatrix} 0 & 0 \\ 0 & 1 \end{pmatrix} $$
is constant in $z$ (and in ${\widehat{x_0}}$).
It follows that:
$[\partial_s^{{\widehat{x_0}}}] \cdot (\Delta_{{\widehat{x_1}}}^+ \circ [x_0,\cdot])(\Phi_{\widehat{x_0}}(z)) = \partial_y(\Delta_{{\widehat{x_1}}}^+ \circ \iota_{\widehat{x_0}})(0,0) $, which is a constant expression in $z$. The proof is done.
\end{proof}

We conclude this section by showing that one can reduce the study of the oscillations of $\Delta(x,\cdot)$ to the study of the oscillations of $X_{\widehat{x}}$. The proof is in three parts: we first establish a proper asymptotic expansion for $\Delta_p^+$, then for $\Delta^-_p$, and finally we reduce \hyperref[QNL]{(QNL)}  to a \say{non-concentration modulo polynomials} statement about $(X_{\widehat{x}})_{{\widehat{x}} \in \widehat{\Omega}}$. 

\begin{theorem}
Let $p,q \in \Omega$ be close enough. We will denote $\pi_p^{S}(q) := [p,q] =: s \in \Omega \cap W^s_{loc}(p)$, and  $\pi_p^U(q) := [q,p] =: r \in \Omega \cap W^u_{loc}(p)$. We have $q = [r,s]$. These \say{coordinates} are $C^{1+\alpha}$. Suppose that $d^s(p,s) \leq \sigma^{\beta_Z}$ and $d^u(p,r) \leq \sigma$ for $\sigma>0$ small enough. If $\beta_Z>1$ is fixed large enough, then the following asymptotic expansion hold:
$$ \Delta_p^+([r,s]) = \pm  \partial_s \Delta_p^+(r) d^s(p,s) + O(\sigma^{1+\beta_Z+\alpha}) ,$$
where $\partial_s$ denotes the derivative in the stable direction.
\end{theorem}

\begin{proof}
Let us introduce some notations. Define, for any $p \in \Omega$, the $C^{1+\alpha}$ map $\nabla_p:W^s_{loc}(p) \rightarrow \mathbb{R}$ as $$ \nabla_p(s) := \sum_{n=0}^\infty \left(\tau( f^{n}(p) ) - \tau(f^n(s)) \right),$$
and notice that $$ \Delta_p^+(q) = \nabla_p(s) - \nabla_{r}(q) .$$
A Taylor expansion yields:
$$ \nabla_p(s) = \nabla_p(p) \pm \partial_s \nabla_p(p) d^s(p,s) + O(d^s(p,s)^{1+\alpha}) $$
$$ = \pm \partial_s \nabla_p(p) d^s(p,s) + O(\sigma^{(1+\alpha)\beta_Z}) $$
$$= \pm \partial_s \nabla_p(p) d^s(p,s) + O(\sigma^{1+\alpha+\beta_Z}) $$
as soon as $\beta_Z \alpha > 1 + \alpha$. Hence, $$ \Delta_p^+(q) = \pm\left( \partial_s \nabla_p(p) d^{s}(p,s) - \partial_s \nabla_r(r) d^{s}(r,q) \right) + O(\sigma^{1+\alpha+\beta_Z}) .$$
Now, we want to make $\partial_s \Delta_p^+(r)$ appear. We compute it and find out that
$$ \partial_s \Delta_p^+(r) = \partial_s \nabla_p( p ) \partial_s \pi_p^S(r) - \partial_s \nabla_r(r). $$
We can make this term appear in our asymptotic expansion like so:
$$ \Delta_p^+(q) = \pm\Big(  \partial_s \nabla_p(p) \frac{d^s(p,s)}{d^s(r,q)} - \partial_s \nabla_r(r) \Big) d^s(r,q) + O(\sigma^{1+\alpha+\beta_Z}) $$
$$ = \pm\Big(  \partial_s \Delta_p^+(r) + \varphi_{p,s}(r) \Big) d^s(r,q) + O(\sigma^{1+\alpha+\beta_Z}) $$
where $$\varphi_{p,s}(r) := \partial_s \nabla_p(p) \left(\frac{d^s(p,s)}{d^s(r,q)} - \partial_s\pi_p^S(r)\right).$$ Now, we would like to show that $\varphi_{p,s}(r) = O(\sigma^{1+\alpha})$.
(In this case $\varphi_{p,s}(r) d^s(r,q) = O(\sigma^{1+\alpha+\beta_Z})$.) To do so, notice the following. The holonomy is (uniformly in $r$) a $C^{1+\alpha}$ map $W^s_{loc}(r) \cap \Omega \rightarrow W^s_{loc}(p) \cap \Omega$. A Taylor expansion at $q=r$ yields:
$$ \frac{d^s(s,p)}{d^s(r,q)} = \frac{d^s(\pi_p^S(q),\pi_p^S(r))}{d^s(q,r)} = \partial_s \pi_p^S(r) + O(d^s(q,r)^\alpha) = \partial_s \pi_p^S(r) + O(\sigma^{\alpha \beta_Z}).$$
If $\beta_Z>2$ is chosen so large that $\alpha \beta_Z > 1+\alpha$ again, then we find the expansion:
$$ \frac{d^s(s,p)}{d^s(q,r)} = \partial_s \pi_p^S(r) + O(\sigma^{1+\alpha}),$$
which is what we wanted to find. This gives us the following expansion for $\Delta^+$:
$$ \Delta_p^+(q) = \pm \partial_s \Delta_p^+(r) d^s(r,q) + O(\sigma^{1+\alpha + \beta_Z}). $$
To conclude, we re-apply our estimate of $d^s(p,s)$ (but inverting the role of $(p,s)$ and $(r,q)$) to find
$$ \frac{d^s(r,q)}{d^s(p,s)} = \partial_s \pi_{s}^S([r,s]) + O(\sigma^{1+\alpha}) = 1 + \alpha_{p,s} d^u(p,r) + O(\sigma^{1+\alpha}) ,$$
(where $\alpha_{p,s}=O(1)$ is some constant that depends only on $p,s$), since $r \mapsto \partial_s \pi_{s}^S([r,s])$ is $C^{1+\alpha}$ (uniformly in $p,s$). It follows that
$$ \Delta_p^+(q) = \pm \partial_s \Delta_p^+(r) d^s(p,s) \pm \alpha_{p,s} \partial_s \Delta_p^+(r) d^u(p,r) d^s(p,s)  + O(\sigma^{1+\alpha+\beta_Z}) $$ $$= \pm \partial_s \Delta_p^+(r) d^s(p,s) +  O(\sigma^{1+\alpha+\beta_Z}) ,$$
since $\partial_s \Delta_p^+(r)$ is $\alpha$-Hölder and vanish at $r=p$. \end{proof}

Now that we have a local description of $\Delta^+$, we turn to a study of the local behavior of $\Delta^-$ along the unstable direction. To achieve a clean proof, we will use the following asymptotic expansion of the distance function, whose proof is postponed to Appendix \hyperref[ap:A]{A}.

\begin{proposition}\label{prop:dist}
Fix $p \in \Omega$ and $s \in W^s_{loc}(p) \cap \Omega$. There exists a $C^\infty$ (uniformly in $p$) function $\phi_{p}: W^u_{loc}(p) \cap \Omega \rightarrow \mathbb{R}$ such that
$$ d^u(s,[r,s]) = d^u(p,r) + \phi_{p}(r) d^s(p,s) + O\Big( 
d^s(p,s)(d^u(p,r)^{1+\alpha} + d^s(p,s)^{\alpha}) \Big) $$
\end{proposition}

\begin{theorem}
Let $p,q \in \Omega$ be close enough, and denote $[p,q] =: s \in W^s_{loc}(p) \cap \Omega$, and  $[q,p] =: r \in W^u_{loc}(p) \cap \Omega$. Suppose that $d^s(p,s) \leq \sigma^{\beta_Z}$ and $d^u(p,r) \leq \sigma$ for $\sigma>0$ small enough. If $\beta_Z>1$ is fixed large enough, then there exists a polynomial $P_{p,s} \in \mathbb{R}_{d_Z}[X]$ (where $d_Z :=\lfloor \beta_Z \rfloor + 2$) whose coefficients depends on $p,s$, and a (uniformly) $C^{1+\alpha}$ function $\psi_p : W^u_{loc}(p) \cap \Omega \rightarrow \mathbb{R}$ such that
$$ \Delta^-(p,q) = P_{p,s}\big( \pm(\Phi_{\widehat{p}}^u)^{-1}(r) \big) + 
\psi_p(r) d^s(p,s) + O(\sigma^{1+\beta_Z+\alpha}) .$$
\end{theorem}

\begin{proof}
Like before, let us decompose $\Delta^-(p,q)$ into two contributions:
$$ \Delta^-(p,q) = \overline{\nabla}_p(r) - \overline{\nabla}_s(q), $$
where $\overline{\nabla}_p: W^u(p) \rightarrow \mathbb{R}$ is defined by $$ \overline{\nabla}_p(r) := \sum_{n=1}^\infty \Big( \tau(f^{-n}(p)) - \tau(f^{-n}(r)) \Big) .$$
This map is clearly smooth along $W^u_{loc}(p)$. A Taylor expansion along $W^u_{loc}(p)$ of order $d_Z := \lceil \beta_Z \rceil + 2 \in \mathbb{N}$ yields
$$ \overline{\nabla}_p(r) = \sum_{k=1}^{d_Z} \alpha_k(p) d^u(p,r)^k + O(d^u(p,r)^{d_Z+1}) \quad ; \quad \overline{\nabla}_s(q) = \sum_{k=1}^{d_Z} \alpha_k(s) d^u(s,q)^k + O(d^u(s,q)^{d_Z+1}) ,$$
for some functions $\alpha_k : \Omega \rightarrow \mathbb{R}$, $k \in \llbracket 1,d_Z \rrbracket$. Notice that $d_Z$ is so large that $d^u(p,r)^{d_Z+1} , d^u(q,s)^{d_Z+1} = O(\sigma^{1+\alpha+\beta})$. This gives the expansion
$$ \Delta^-(p,q) = \sum_{k=1}^{d_Z} \Big( \alpha_k(p) d^u(p,r)^k - \alpha_k(s) d^u(s,q)^k \Big) + O(\sigma^{\beta_Z + \alpha + 1}) .$$
Now, our previous technical lemma ensure that there exists a (uniformly in $p$) $C^{1+\alpha}$ map $\phi_p : W^u(p) \cap \Omega \rightarrow \mathbb{R}$ such that
$$d^u(s,[r,s]) = d^u(p,r) + \phi_{p}(r) d^s(p,s) + O\Big( 
d^s(p,q)(d^u(p,r)^{1+\alpha} + d^s(p,s)^{\alpha}) \Big)$$
$$ = d^u(p,r) + \phi_{p}(r) d^s(p,s) + O\Big( 
\sigma^{1+\beta_Z+\alpha} \Big). $$
It follows that, for all $k \geq 1$, we can write
$$ d^u(s,q)^k = \Big( d^u(p,r) + \phi_{p}(r) d^s(p,s) + O\big( 
\sigma^{1+\beta_Z+\alpha} \big) \Big)^k $$
$$ = d^u(p,r)^k + k d^u(p,r)^{k-1} \phi_{p}(r) d^s(p,s) + O(\sigma^{2 \beta_Z} + \sigma^{1+\beta+\alpha}) $$
$$ = d^u(p,r)^k + k d^u(p,r)^{k-1} \phi_{p}(r) d^s(p,s) + O(\sigma^{1+\beta+\alpha})  $$
since $2 \beta_Z > 1 +\alpha + \beta_Z$, as $\beta_Z>2$ and $\alpha \in (0,1)$. Hence
$$ \Delta^-(p,q) = \sum_{k=1}^{d_Z} \Big( \alpha_k(p) d^u(p,r)^k - \alpha_k(s) (d^u(p,r)^k + k d^u(p,r)^{k-1} \phi_{p}(r) d^s(p,s)) \Big) + O(\sigma^{\beta_Z + \alpha + 1})  $$
$$ = \sum_{k=1}^{d_Z} \Big( (\alpha_k(p) - \alpha_k(s)) d^u(p,r)^k\Big) - d^s(p,s) \sum_{k=1}^{d_Z} \alpha_k(s)  k d^u(p,r)^{k-1} \phi_{p}(r)   + O(\sigma^{\beta_Z + \alpha + 1})  $$
$$ = P_{p,s}(r) + d^s(p,s) \psi_p(r) + O(\sigma^{1+\beta_Z+\alpha}) ,$$
where $P_{p,s}(r)$ is a polynomial in $d^u(p,r)$ of degree $d_Z$ whose coefficients depends on $p,s$, and where $\psi_p$ is a (uniformly in $p$) $C^{1+\alpha}$ function of $r$. To conclude, we write $P_{p,s}(r) = (P_{p,s} \circ \Phi_{\widehat{p}}^u)(z)$ for $z := (\Phi_{\widehat{p}}^u)^{-1}(r)$ and for some choice of orientation $\widehat{p}$. This is a (uniformly in $\widehat{p},s$) smooth function of $z$, and a Taylor expansion of order $d_Z$ allows us to write $P_{p,s}(r) = Q_{p,s}(\pm z) + O(\sigma^{1+\alpha+\beta_Z})$, for some $Q_{p,s} \in \mathbb{R}_{d_Z}[X]$. \end{proof}

\begin{remark}
Recall from Proposition \ref{prop:locProdStruct} that $\mu$ has a local product structure \cite{Cl20}, in the following sense. For all $x \in \Omega$, there exists $\mu_x^u$ and $\mu_x^s$ two measures supported on $U_x := \Omega \cap W^u_{loc}(x)$ and $S_x := \Omega \cap W^s_{loc}(x)$ such that, for all measurable $h:M \rightarrow \mathbb{C}$ supported in a small enough neighborhood of $x$, we have
$$ \int_\Omega h d\mu = \int_{U_x} \int_{S_x} h([z,y]) e^{\omega_x(z,y)} d\mu_x^s(y) d\mu_x^u(z) ,$$
where $\omega_x$ is some (uniformly in $x$) Hölder map. In particular, for $C(\mu) := \sup_x \|\omega_x\|_{L^\infty( W^u_{loc}(x) \times W^s_{loc}(x))}$, we find when $h$ is nonnegative:
$$ e^{-C(\mu)} \int_{U_x} \int_{S_x} h([z,y]) d\mu_x^s(y) d\mu_x^u(z) \leq  \int_\Omega h d\mu \leq e^{C(\mu)} \int_{U_x} \int_{S_x} h([z,y]) d\mu_x^s(y) d\mu_x^u(z).$$

Notice in particular that, for some rectangle $R_p = [U_p,S_p] \ni p$, and for any borel set $I \subset \mathbb{R}$ :
$$ \mu( q \in R_p , \ h([q,p]) \in I ) \simeq \mu_p^s(S_p) \mu_p^u(z \in U_p, h(z) \in I)). $$
\end{remark}

\begin{lemma}\label{lem:polyQNL}
Denote by $\emeta_{\widehat{x}} := (\Phi_{\widehat{x}}^u)^*\mu_x^u$, the measure $\mu_x^u$ seen in the coordinates $\Phi_{\widehat{x}}^u$. Suppose that the family $(X_{\widehat{x}})_{\widehat{x}}$ satisfies the following \say{non-concentration modulo polynomials} property: there exists $\alpha,\gamma,\sigma_0>0$ such that, for all ${\widehat{x}} \in \widehat{\Omega}$ and for all $0 < \sigma < \sigma_0$:
$$ \sup_{P \in \mathbb{R}_{d_Z}[X]} \emeta_{\widehat{x}}( z \in [-\sigma,\sigma], \ |X_{\widehat{x}}(z) - P(z)| \leq \sigma^{1+\alpha/2} ) \leq \sigma^{\gamma} \cdot \emeta_{\widehat{x}}([-\sigma,\sigma]). $$
Then the conclusion of Theorem \ref{th:QNL} holds.
\end{lemma}

\begin{proof}
Recall that, by Lemma \ref{lem:locQNL}, to check \hyperref[QNL]{(QNL)}, it suffice to establish the following bound:
$$ \mu \left( q \in R, |\Delta(p,q)| \leq \sigma^{1+\alpha+\beta_Z} \right) \lesssim \sigma^\gamma \mu(R) $$
where the bound is uniform in $p$ and $R$, for $R \in \text{Rect}_{\beta_Z}(\sigma)$ a rectangle containing $p$ with stable (resp. unstable) diameter $\sigma^{\beta_Z}$ (resp. $\sigma$). Let us check this estimate by using the Taylor expansion of $\Delta^+$ and $\Delta^-$. We can write, using the local product structure of $\mu$:
$$ \mu \left( q \in R, |\Delta(p,q)| \leq \sigma^{1+\alpha+\beta_Z} \right) \lesssim \int_{S} \mu_p^u \left( r \in U_p^\sigma \ , \ |\Delta_p^+([r,s]) + \Delta^-(p,[r,s])| \leq \sigma^{1+\alpha+\beta_Z}  \right) d\mu_p^s(s)  $$
$$ \lesssim \int_{S} \mu_p^u\left( r \in U_p^\sigma , \ |(\partial_s \Delta_p^+(r) + \psi_p(r)) d^s(p,s) + P_{p,s}(\pm (\Phi_{\widehat{p}}^u)^{-1}(r))| \leq C \sigma^{1+\beta_Z+\alpha} \right) d\mu_p^s(s) $$
$$ = \int_{S} \mu_p^u\left( r \in U_p^\sigma , \ |\partial_s \Delta_p^+(r) + \psi_p(r) + Q_{p,s}(\pm (\Phi_{\widehat{p}}^u)^{-1}(r)))| d^s(p,s) \leq C \sigma^{1+\beta_Z+\alpha} \right) d\mu_p^s(s) $$
where $Q_{p,s} := {d^s(p,s)}^{-1} {P_{p,s}}$. It is easy to see \cite{Le24}, using Gibbs estimates and the fact that $\text{diam}^s(S) \simeq \sigma^{\beta_Z}$, that there exists $\delta_{reg}>0$ such that $$\mu_p^s( B(p,\sigma^{\beta_Z+\alpha/2}) \cap S ) \leq \sigma^{\alpha \delta_{reg}/2} \mu_p^s(S) .$$
It follows that one can cut the integral over $S$ in two parts: the part where $s$ is $\sigma^{\beta_Z+\alpha/2}$-close to $p$, and the other part. We get, using the aforementioned regularity estimates:
$$ \mu \left( q \in R, |\Delta(p,q)|\leq \sigma^{1+\alpha+\beta_Z} \right) $$ $$ \lesssim \sigma^{\delta_{reg} \alpha/2} \mu(R) + \int_{S} \mu_p^u\left( r \in U_p^\sigma , \ |\partial_s \Delta_p^+(r) + \psi_p(r) + Q_{p,s}(\pm (\Phi_{\widehat{p}}^u)^{-1}(r))| \leq C \sigma^{1+\alpha/2} \right) d\mu_p^s(s). $$
We just have to control the integral term to conclude. To do so, notice that, for all $s$, we can write:
$$ \mu_p^u\left( r \in U_p^\sigma , \ |\partial_s \Delta_p^+(r) + \psi_p(r) + Q_{p,s}(\pm (\Phi_{\widehat{p}}^u)^{-1}(r))| \leq C \sigma^{1+\alpha/2} \right)  $$ $$= \left((\Phi_{\widehat{p}}^u)^*\mu_p^u \right)\Big( z \in (\Phi_{\widehat{p}}^u)^{-1}(U_p^\sigma), \ |\partial_s \Delta_p^+( \Phi_{\widehat{p}}^u(z) ) + \psi_p( \Phi_{\widehat{p}}^u(z) ) + Q_{p,s}(\pm z) | \leq  C \sigma^{1+\alpha/2} \Big) $$
$$ \leq \left((\Phi_{\widehat{p}}^u)^*\mu_p^u \right)\Big( z \in [-C \sigma,C \sigma], \ |\partial_s \Delta_p^+(\Phi_{\widehat{p}}^u(z)) + \psi_p(\Phi_{\widehat{p}}^u(z)) + Q_{p,s}(\pm z)| \leq C \sigma^{1+\alpha/2} \Big) . $$
Beware that, if $\psi_p \circ \Phi_{\widehat{p}}^u$ is uniformly in $p$ a $C^{1+\alpha}$ function of $z$, allowing us do approximate it by its $1+\alpha$-th order Taylor expansion, the same doesn't hold for $Q_{p,s}$, which may have unbounded coefficients as $s$ gets close to $p$. Still, writing $$ \psi_p(\Phi_{\widehat{p}}^u(z)) = \alpha_p^{(1)} \pm \alpha_p^{(2)} z + O(\sigma^{1+\alpha}), $$
we find, replacing $Q_{s,p}$ by $\tilde{Q}_{s,p}(z) := \alpha_p^{(0)} \pm \alpha_p^{(1)} z + Q_{s,p}(\pm z)$,
$$ \mu \left( q \in R, |\Delta(p,q)|\leq \sigma^{1+\alpha+\beta_Z} \right) $$
$$ \lesssim \sigma^{\delta_{reg} \alpha/2} \mu(R) + \int_{S} \emeta_p\left( z \in [-C\sigma,C\sigma] , \ |\partial_s \Delta_p^+(\Phi_{\widehat{p}}^u(z)) + \tilde{Q}_{p,s}(\pm z)| \leq C \sigma^{1+\alpha/2} \right) d\mu_p^s(s). \quad (*)$$
Now, since $TM/E^u$ is a line bundle, and since $[\partial_s]$ and $[\partial_s^{\widehat{p}}]$ are $C^{1+\alpha}$ sections of this line bundle ($[\partial_s]$ is $C^{1+\alpha}$ because $E^s$ is $C^{1+\alpha}$ in our 2-dimensional context, and $[\partial_s^{\widehat{p}}]$ is actually smooth by construction), there exists a nonvanishing $C^{1+\alpha}$ function $a_{\widehat{p}}(z)$ such that $[\partial_s]_{\Phi_{\widehat{p}}^u(z)} = a_{\widehat{p}}(z) [\partial_s^p]_{\Phi_{\widehat{p}}^u(z)} $. We have $a_{\widehat{p}}(z)= \pm(1 + O(\sigma))$. Hence, using the bound $X_{\widehat{x}}(z) = O(\sigma^\alpha)$ since $X_{{\widehat{x}}}$ is (uniformly) Hölder regular and vanish at zero:
$$\emeta_{\widehat{p}}\left( z \in [-C\sigma,C\sigma] , \ |\partial_s \Delta_p^+(\Phi_{\widehat{p}}^u(z)) + \tilde{Q}_{p,s}(\pm z)| \leq C \sigma^{1+\alpha/2} \right) $$
$$ = \emeta_{\widehat{p}}\left( z \in [-C\sigma,C\sigma] , \ |a_{\widehat{p}}(z) X_{\widehat{p}}(z) + \tilde{Q}_{p,s}(\pm z)| \leq C \sigma^{1+\alpha/2} \right) $$
$$ \leq \emeta_{\widehat{p}}\Big( z \in [-C \sigma,C \sigma], \ |X_{\widehat{p}}(z) \pm \tilde{Q}_{p,s}(\pm z)| \leq C' \sigma^{1+\alpha/2} \big) $$
$$ \leq C'' \emeta_{\widehat{p}}( [C' \sigma,C' \sigma]) \sigma^{\gamma}  ,$$
where the last control is given by the nonconcentration hypothesis made on $(X_{\widehat{x}})_{{\widehat{x}} \in \widehat{\Omega}}$. To conclude, notice that by regularity of the parametrizations $\Phi_{\widehat{p}}^u$, and since the measure $\mu_x^u$ is doubling, we get $\emeta_{\widehat{p}}\left( [-C \sigma, C\sigma] \right) \leq C \mu_p^u(U)$. Injecting this estimate in $(*)$ yields
$$ \mu \left( q \in R, |\Delta(p,q)|\leq \sigma^{1+\beta_Z+\alpha} \right) \leq C \left( \sigma^{\delta_{reg} \alpha/2} + \sigma^{\gamma} \right) \mu(R) ,$$
which is what we wanted. \end{proof}

Notice that the estimate required to apply Lemma \ref{lem:polyQNL} doesn't depend on the choice of the orientation $\widehat{x} \in \widehat{\Omega}$ in the fiber above $x$. We are reduced to understand oscillations of $ z \mapsto X_{\widehat{x}}(z)$ modulo polynomials of a fixed degree $d_{Z} \geq 1$. The next section will be devoted to proving a \say{blowup} result on the family $(X_{\widehat{x}})_{{\widehat{x}} \in \widehat{\Omega}}$, which will help us understand deeper the oscillations of those functions. This \say{blowup} result will allow us to exhibit a rigidity phenomenon. The final section is devoted to proving the non-concentration estimates in the hypothesis of Lemma \ref{lem:polyQNL}.

\subsection{Autosimilarity, polynomials, and rigidity}
Let us recall our setting. We are given a family of Hölder maps $(X_{\widehat{x}})_{{\widehat{x}} \in \widehat{\Omega}}$, where $X_{\widehat{x}} : \Omega_{\widehat{x}}^u \cap (-\rho,\rho) \longrightarrow \mathbb{R}$ is defined only on a (fractal) neighborhood of zero and vanish at $z=0$. Recall that $\Omega_{\widehat{x}}^u := (\Phi_{\widehat{x}}^u)^{-1}(\Omega) \ni 0$. Furthermore, we have an autosimilarity relation: for any ${\widehat{x}} \in \widehat{\Omega}$, and any $z \in \Omega_{\widehat{x}}^u \cap (-\rho,\rho)$, we have
$$ X_{\widehat{x}}(z) = \widehat{\tau}_{\widehat{x}}(z) + \mu_x X_{\widehat{f}({\widehat{x}})}(z \lambda_x) ,$$
where $\widehat{\tau}_{\widehat{x}} :  \Omega_{\widehat{x}}^u \cap (-\rho,\rho) \rightarrow \mathbb{R}$ is a $C^{1+\alpha}$ map that vanish at zero. Recall that, by hypothesis on $\Delta^+$, there exists $\widehat{x} \in \widehat{\Omega}$ such that $X_{\widehat{x}}$ is not $C^1$. To do so, we start by proving a \say{blowup} result, inspired from Appendix B in  \cite{TZ20}. The point of this lemma is to only keep, in the autosimilarity formula of $(X_{\widehat{x}})_{{\widehat{x}} \in \widehat{\Omega}}$, the germ of $\widehat{\tau}_{\widehat{x}}$ (in the form of its Taylor expansion at zero at some order). Depending on the contraction/dilatation rate of the dynamics, the order of this Taylor expansion is different: in our area-preserving case, it is enough to approximate $\widehat{\tau}_{\widehat{x}}(z)$ by $(\widehat{\tau}_{\widehat{x}})'(0)z$.
\begin{lemma}[Blowup]\label{lem:blow}
There exists two families of functions $(Y_{\widehat{x}})_{{\widehat{x}} \in \widehat{\Omega}}$, $(Z_{\widehat{x}})_{{\widehat{x}} \in \widehat{\Omega}}$ such that:
\begin{itemize}
\item For all ${\widehat{x}} \in \widehat{\Omega}$ and $z \in \Omega_{\widehat{x}}^u \cap (-\rho,\rho)$, $$X_{\widehat{x}}(z) = Y_{\widehat{x}}(z) + Z_{\widehat{x}}(z) .$$
\item The map $Y_{\widehat{x}} :  \Omega_{\widehat{x}}^u \cap (-\rho,\rho) \rightarrow \mathbb{R}$ is $C^{1+\alpha}$, and there exists $C \geq 1$ such that, for all ${\widehat{x}} \in \widehat{\Omega}$ and $z \in \Omega_{\widehat{x}}^u \cap (-\rho,\rho)$: $$|Y_{\widehat{x}}(z)| \leq C |z|^{1+\alpha}.$$
\item The family $(Z_{\widehat{x}})_{{\widehat{x}} \in \widehat{\Omega}}$ satisfies an autosimilarity formula: for any ${\widehat{x}} \in \widehat{\Omega}$, $z \in \Omega_{\widehat{x}}^u \cap (-\rho,\rho)$
$$ Z_{\widehat{x}}(z) = (\widehat{\tau}_{\widehat{x}})'(0) z + \mu_x Z_{\widehat{f}({\widehat{x}})}(z \lambda_x) .$$
Moreover, the dependence in ${\widehat{x}}$ of $(X_{\widehat{x}})_{{\widehat{x}} \in \widehat{\Omega}}$, $(Y_{\widehat{x}})_{{\widehat{x}} \in \widehat{\Omega}}$ and $(Z_{\widehat{x}})_{{\widehat{x}} \in \widehat{\Omega}}$ is Hölder.
\end{itemize}

\begin{proof}
In the original proof, there is an implicit argument used, which is the fact that polynomials (of order one, here) are maps with vanishing (second order) derivative. In our fractal context, this is not true, as $\Omega_{\widehat{x}}^u$ may not be connected: so we have to replace this derivative with a notion adapted to our fractal context. For $\beta<\alpha$, define a \say{$(1+\beta)$-order fractal derivative} as follows: if $h:\Omega_x^u \cap (-\rho,\rho) \rightarrow \mathbb{R}$ is ${C^{1+\alpha}}$, then its Taylor expansion at zero makes sense (since $\Omega_x^u$ is perfect, see appendix \hyperref[ap:B]{B}), and we can consider the function:
$$ \delta^{1+\beta}(h)(z) := \frac{h(z) - h(0) - h'(0) z}{|z|^{1+\beta}}. $$
This is a continuous function on $\Omega_{\widehat{x}}^u \cap (-\rho,\rho)$, which is bounded and vanish at zero at order $|z^{(\alpha-\beta)-}|$. Moreover, notice that $\delta^{1+\beta}(h)=0$ is equivalent to saying that $h$ is affine. Notice further that $$ \delta^{1+\beta}\Big{(} \mu h(\lambda \cdot) \Big{)}(z) = (\mu |\lambda|^{1+\beta}) \cdot \delta^{1+\beta}(h)(z \lambda ).$$
Now, let us begin the actual proof. Consider the autosimilarity equation of $(X_{\widehat{x}})$, and formally take the $(1+\beta)$-th fractal derivative. We search for a $C^{1+\alpha}$ solution $(Y_{\widehat{x}})$ of this equation:
$$ \delta^{1+\beta}(Y_{\widehat{x}})(z) = \delta^{1+\beta}(\widehat{\tau}_{\widehat{x}})(z) + \mu_x \lambda_x^{1+\beta} \cdot \delta^{1+\beta}(Y_{\widehat{x}})(z \lambda_x). $$
Notice that $\kappa_x := \mu_x \lambda_x^{1+\beta}$ behaves like a greater-than-one multiplier. Indeed, if we denote, for $x \in \Omega$ and $n \in \mathbb{Z}$, $$\kappa_x^{\langle n \rangle}  := \kappa_x \dots \kappa_{f^{n-1}(x)} $$
(if $n \geq 0$, and similarly if $n \leq 0$ as in the definition of $\lambda_x^{\langle n \rangle}) $,
we see that $\kappa_x^{\langle n \rangle} \geq (\lambda_-^n)^{\beta}$, where $\lambda_- > 1$. We can solve this equation by setting
$$ \delta^{1+\beta}(Y_{\widehat{x}})(z) := -\sum_{n=1}^\infty \delta^{1+\beta}(\widehat{\tau}_{\widehat{f}^{-n}({\widehat{x}})})(z \lambda_x^{\langle - n \rangle}) \cdot \kappa_x^{\langle - n \rangle} =: \tilde{Y}_{\widehat{x}}(z).$$
This defines a continuous function that vanishes at zero. We then define $Y_{\widehat{x}}$ as the only $C^{1+\alpha}$ function such that $ Y_{\widehat{x}}(0)=Y_{\widehat{x}}'(0)=0 $ and $\delta^{1+\beta}(Y_{\widehat{x}}) = \tilde{Y}_{\widehat{x}}$. In other words, $Y_{\widehat{x}}(z) := |z|^{1+\beta} \tilde{Y}_{\widehat{x}}(z)$. Using the sum formula of $\tilde{Y}_{\widehat{x}}$, we find the autosimilarity formula:
$$ Y_{\widehat{x}}(z) = \widehat{\tau}_{\widehat{x}}(z) - (\widehat{\tau}_{\widehat{x}})'(0)z + \mu_x Y_{\widehat{f}({\widehat{x}})}(z \lambda_x) .$$
We can then conclude by setting $Z_{\widehat{x}} := X_{\widehat{x}} - Y_{\widehat{x}}$.
\end{proof}
\end{lemma}

\begin{remark}\label{rem:continuous}
One could wonder what it means for $(X_{\widehat{x}})_{\widehat{x}}$ (or $(Y_{\widehat{x}})_{{\widehat{x}}}$ and $(Z_{\widehat{x}})_{\widehat{x}}$) to be Hölder in ${\widehat{x}}$, when the domain of definition of $X_{\widehat{x}}$ changes with ${\widehat{x}}$. We mean the following. If ${\widehat{x_1}} \in \Omega$ is close to ${\widehat{x_0}}$, then first of all they share the same local orientation. Denote $x_2 := [x_1,x_0] \in W^u_{loc}(x) \cap W^s_{loc}(x)$, and then consider $\widehat{x_2}$ in the fiber of $x_2$ with the same local orientation than $\widehat{x_1}$ and $\widehat{x_0}$. The holonomy between $W^u_{loc}(x_1)$ and $W^u_{loc}(x_2)$, written in coordinates $\Phi_{\widehat{x_1}}^u$ and $\Phi_{\widehat{x_2}}^u$, yields a map $\pi_{\widehat{x_2},\widehat{x_1}} : \Omega_{\widehat{x_2}}^u \cap (-\rho,\rho) \rightarrow \Omega_{\widehat{x_1}}^u \cap (-\rho',\rho')$. Then, the affine map $\text{aff}_{\widehat{x_0},\widehat{x_2}} = \Phi_{{\widehat{x_0}}}^u \circ (\Phi_{\widehat{x_2}}^u)^{-1} : \Omega_{\widehat{x_2}}^u \cap (-\rho',\rho') \rightarrow \Omega_{\widehat{x_0}}^u \cap (-\rho'',\rho'')$ allows us to compare functions on $\Omega_{{\widehat{x_2}}}$ with functions on $\Omega_{{\widehat{x_0}}}$. It is not difficult to see that the map  $\text{aff}_{{\widehat{x_0}},\widehat{x_2}} \circ \pi_{\widehat{x_2},{\widehat{x_1}}}$ is $C^{1+\alpha}$, and moreover that $\| \text{aff}_{{\widehat{x_0}},\widehat{x_2}} \circ \pi_{\widehat{x_2},{\widehat{x_1}}} - Id \|_{C^{1+\alpha}} \lesssim d(\widehat{x_0},\widehat{x_1})^\alpha$ \cite{PR02}.
What we mean for the Hölder dependence in $x$, then, is the following:
$$ \| X_{\widehat{x_1}} - X_{\widehat{x_0}}  \circ \text{aff}_{{\widehat{x_0}},\widehat{x_2}} \circ \pi_{\widehat{x_2},{\widehat{x_1}}}\|_{L^\infty(\Omega_{\widehat{x_1}} \cap (-\rho,\rho))} \lesssim d(\widehat{x_0},\widehat{x_1})^\alpha.$$ \end{remark}

To get from this construction a deeper understanding of the local behavior of $(X_{\widehat{x}})$, we need first to state some quantitative statements about equivalence of norms in our context. The proof of the next lemma will be done in Appendix \hyperref[ap:B]{B}.

\begin{lemma}\label{lem:norms}
There exists $\kappa>0$ such that, for all ${\widehat{x}}\in \widehat{\Omega}$, for all $P \in \mathbb{R}_{d_Z}[X]$, there exists $z_0 \in \Omega_{\widehat{x}}^u \cap (-\rho/2,\rho/2)$ such that
$$ \forall z \in [z_0 - \kappa, z_0 + \kappa], \ |P(z)| \geq \kappa \|P\|_{C^\alpha((-\rho,\rho))}. $$
If we denote $P = \sum_{k=0}^{d_Z} a_k z^k \in \mathbb{R}_{d_Z}[X]$, then moreover
$$\max_{0 \leq k \leq d_Z} |a_k| \leq \kappa^{-1} \|P\|_{L^\infty( (-\rho,\rho) \cap \Omega_{\widehat{x}}^u )}. $$
\end{lemma}

From this, we deduce a strong statement on the possible forms of $Z_{\widehat{x}}$.

\begin{lemma}\label{lem:polytolin}
Suppose that there exists $P \in \mathbb{R}_{d_Z}[X]$ such that $Z_{{\widehat{x_0}}} = P$ on $(-\rho,\rho) \cap \Omega_{{\widehat{x_0}}}^u$, for some ${\widehat{x_0}} \in \widehat{\Omega}$. Then $P$ is linear. 
\end{lemma}

\begin{proof}
Let $\kappa>0$ be such that $$ \forall P = \sum_{k=0}^{d_Z} a_k z^k \in \mathbb{R}_{d_Z}[X], \quad \max_{0 \leq k \leq d_Z} |a_k| \leq \kappa^{-1} \inf_x \|P\|_{L^\infty( (-\rho,\rho) \cap \Omega_{\widehat{x}}^u )}.$$
We suppose that there exists $P = \sum_k a_k z^k \in \mathbb{R}_{d_Z}[X]$ such that $Z_{{\widehat{x_0}}}(z) = P(z)$ for all $z \in (-\rho,\rho) \cap \Omega_{{\widehat{x_0}}}^u$. Since
$$ \forall n \geq 1, \forall z \in (-\rho,\rho) \cap \Omega_{\widehat{f}^{n}({\widehat{x_0}})}^u , \ Z_{\widehat{f}^n({\widehat{x_0}})}(z) = \text{linear} + \mu_x^{\langle - n \rangle} Z_{{\widehat{x_0}}}(\lambda_x^{\langle - n \rangle} z), $$
it follows that $Z_{\widehat{f}^{n}({\widehat{x_0}})}$ is polynomial on $\Omega_{\widehat{f}^n({\widehat{x_0}})}^u \cap (-\rho,\rho)$, of the form:
$$ Z_{\widehat{f}^{-n}({\widehat{x_0}})}(z) = \text{linear} + \sum_{k=2}^{d_Z} \mu_{x_0}^{\langle n \rangle} (\lambda_{x_0}^{\langle n \rangle})^k a_k z^k. $$
Equivalence of norms yields, for all $k \geq 2$,
$$ \forall n\geq 1, \ | \mu_{x_0}^{\langle n \rangle} (\lambda_{x_0}^{\langle n \rangle})^k a_k  | \leq \kappa^{-1} \sup_{{\widehat{x}} \in \widehat{\Omega}} \|Z_{\widehat{x}}\|_{L^\infty(\Omega_{\widehat{x}}^u \cap (-\rho,\rho))} < \infty ,$$
which implies that $a_k=0$ for $k \geq 2$. Hence, $Z_{{\widehat{x_0}}}$ is linear.
\end{proof}
 
The idea now is to consider the distance from $Z_{\widehat{x}}$ to the space of linear/polynomial maps. By the autosimilarity formula of $(Z_{\widehat{x}})$, there is going to be some invariance that will prove useful. 
\begin{definition}
For any $\rho>0$ small enough, consider the functions $\mathcal{D}_\rho^{lin}, \mathcal{D}_\rho^{poly}  : \Omega \longrightarrow \mathbb{R}_+$ defined as
$$ \mathcal{D}_\rho^{lin}(x) := \inf_{a \in \mathbb{R}} \sup_{z \in \Omega_{\widehat{x}}^u \cap (-\rho,\rho)} | Z_{\widehat{x}}(z) - a z | $$
and $$ \mathcal{D}_\rho^{poly}(x) := \inf_{P \in \mathbb{R}_{d_Z}[X]} \sup_{z \in \Omega_{\widehat{x}}^u \cap (-\rho,\rho)} | Z_{\widehat{x}}(z) - P(z) |. $$
where $\widehat{x} \in \widehat{\Omega}$ is any choice of element in the fiber of $x \in \Omega$. These functions are continuous, since $\widehat{x} \in \Omega \mapsto Z_{\widehat{x}} \in C^0$ is continuous in the sense of Remark \ref{rem:continuous}, and since we are computing a distance to a finite-dimensional vector space. Moreover,
$$ 0 \leq \mathcal{D}_{\rho}^{poly} \leq \mathcal{D}_{\rho}^{lin}. $$
\end{definition}
In the following, for some $C^{1+\alpha}$ function $Z$ on $\Omega_x^u \cap(-\rho,\rho)$, let us denote $$ \delta^{1+\alpha}(Z)(z) := \frac{Z(z)-Z(0)-Z'(0)z}{|z|^{1+\alpha}} .$$
\begin{lemma}
We have the following criterion. The following are equivalent, for some fixed $x \in \Omega$ and $\rho>0$:
\begin{itemize}
\item $\mathcal{D}_\rho^{lin}(x)=0$
\item For some $\widehat{x} \in \Omega$ in the fiber of $x$, for all $n \geq 0$, $Z_{f^{-n}(\widehat{x})} \in C^{1+\alpha}\big{(}\Omega_{f^{-n}(\widehat{x})}^u \cap (-\rho \lambda_x^{\langle -n \rangle } , \rho \lambda_x^{\langle -n \rangle}\big{)},\mathbb{R}) $, and there exists $C \geq 1$ such that, for all $n \geq 0$, $$\| \delta^{1+\alpha}(Z_{\widehat{f}^{-n}(\widehat{x})}) \|_{L^\infty( \Omega_{\widehat{f}^{-n}(\widehat{x})}^u \cap (-\rho \lambda_x^{\langle -n \rangle } , \rho \lambda_x^{\langle -n \rangle}) )} \leq C.$$
\end{itemize}
\end{lemma}

\begin{proof}
Suppose that $\mathcal{D}_\rho^{lin}(x)=0$. Choose any $\widehat{x} \in \widehat{\Omega}$ in the fiber of $x \in \Omega$. Since $Z_{\widehat{x}}(0)=0$, there exists $a \in \mathbb{R}$ such that $Z_{\widehat{x}}(z) = az$ on $(-\rho,\rho)$. The autosimilarity relation $Z_{\widehat{x}}(z) = (\widehat{\tau}_{\widehat{x}})'(0)z + \mu_x Z_{\widehat{f}({\widehat{x}})}(\lambda_x z)$ gives, with a change of variable,
$$ Z_{\widehat{f}^{-1}({\widehat{x}})}(z \lambda_x^{\langle -1 \rangle}) \mu_x^{\langle -1 \rangle} = (\widehat{\tau}_{\widehat{f}^{-1}({\widehat{x}})})'(0) \lambda_x^{\langle -1 \rangle} \mu_x^{\langle -1 \rangle} z + Z_{{\widehat{x}}}(z) .$$
Iterating this yields $$ Z_{\widehat{f}^{-n}({\widehat{x}})}(z \lambda_x^{\langle -n \rangle}) \mu_x^{\langle -n \rangle} = \text{linear} +  Z_{\widehat{x}}(z) .$$
In particular, if $Z_{\widehat{x}}$ is linear on $\Omega_{\widehat{x}}^u \cap (-\rho,\rho)$, then $Z_{\widehat{f}^{-n}({\widehat{x}})}$ is linear on $\Omega_{\widehat{f}^{-n}({\widehat{x}})}^u \cap (-\rho \lambda_x^{\langle -n \rangle } , \rho \lambda_x^{\langle -n \rangle}) $. In particular, it is $C^{1+\alpha}$ and the bound on $\delta^{1+\alpha}(Z_{\widehat{f}^{-n}({\widehat{x}})}) = 0$ holds. Reciprocally, if the second point holds, then we can write, on $\Omega_{\widehat{x}}^u \cap (-\rho,\rho)$:
$$ |\delta^{1+\alpha}(Z_{\widehat{x}})(z)| = |\mu_x^{ \langle - n \rangle} (\lambda_x^{ \langle -n \rangle })^{1+\alpha} \cdot \delta^{1+\alpha}(Z_{\widehat{f}^{-n}({\widehat{x}})})(z \lambda_x^{\langle -n \rangle})| \leq C' (|\lambda_x|^{\langle -n \rangle})^{\alpha} \underset{n \rightarrow \infty}{\longrightarrow} 0,  $$
where we used the fact that $\mu_x^{\langle n \rangle} \lambda_x^{\langle n \rangle} = e^{O(1)}$ by our area-preserving hypothesis made on the dynamics $f$. Hence $\delta^{1+\alpha}(Z_{\widehat{x}})=0$ on $\Omega_{\widehat{x}}^u \cap (-\rho,\rho)$, which means that $Z_{\widehat{x}}$ is linear on this set.
\end{proof}

\begin{lemma}[Rigidity lemma]
Suppose that there exists $x_0 \in \Omega$ such that $\mathcal{D}_\rho^{poly}(x_0) = 0$, then $\mathcal{D}_\rho^{lin} = 0$ on $\Omega$.
\end{lemma}

\begin{proof}
The proof is in three steps. Suppose that $\mathcal{D}_\rho^{poly}(x_0)=0$ for some $\rho>0$ and some $x_0 \in \Omega$.
\begin{itemize}
\item We first show that there exists $0 < \rho' < \rho$ and a set $\omega \subset W^u_{loc}(x_0) \cap \Omega$ which is an open neighborhood of $x$ for the topology of $W^u_{loc}(x_0) \cap \Omega$, such that $\mathcal{D}_{\rho'}^{lin}(\tilde{x})=0$ if $\tilde{x} \in \omega$. 
\end{itemize}
So suppose that $\mathcal{D}_\rho^{poly}(x_0)=0$. This means that $Z_{\widehat{x_0}}$ coincide with a polynomial of degree $d_Z$ on $(-\rho,\rho) \cap \Omega_{\widehat{x_0}}^u$. Lemma \ref{lem:polytolin} ensures that, in this case, $Z_{\widehat{x_0}}$ is in fact linear on $(-\rho,\rho) \cap \Omega_{\widehat{x_0}}^u$. Now, notice that since $Y_{\widehat{x_0}}$ is $C^{1+\alpha}$, we know that $X_{\widehat{x_0}}$ is $C^{1+\alpha}$ on $(-\rho,\rho) \cap \Omega_{\widehat{x_0}}^u$. Recall then that, by Lemma \ref{lem:Aff}, we know that if $x_1 \in W^u_{loc} \cap \Omega$ is close enough to $x_0$, (and if $\widehat{x_1}$ share the same local orientation than $\widehat{x_0}$) we can write
$$ \text{Aff}_{\widehat{x_0},\widehat{x_1}}\Big(X_{\widehat{x_1}}(z)\Big{)} = X_{\widehat{x_0}}(\text{aff}_{\widehat{x_0},\widehat{x_1}}(z)), $$
where $\text{Aff}_{\widehat{x_0},\widehat{x_1}}$ and $\text{aff}_{\widehat{x_0},\widehat{x_1}}$ are affine functions (that gets close to the identity as $x_1 \rightarrow x_0$). If follows that $X_{\widehat{x_1}}$ is $C^{1+\alpha}$ on some (smaller) open neighborhood of zero, $(-\rho',\rho') \cap \Omega_{\widehat{x_1}}^u$. In particular, $Z_{\widehat{x_1}}$ is also $C^{1+\alpha}$ on this set. Let us show that $\mathcal{D}_{\rho'}^{lin}(\widehat{x_1})=0$ by checking the criterion given in the previous lemma. We have:
$$ Z_{\widehat{x_1}}(z) = X_{\widehat{x_1}}(z) - Y_{\widehat{x_1}}(z) = \text{Aff}_{\widehat{x_0},\widehat{x_1}}^{-1} \Big( X_{\widehat{x_0}}(\text{aff}_{\widehat{x_0},\widehat{x_1}}(z))  \Big) - Y_{\widehat{x_1}}(z) $$
$$ = \text{Aff}_{\widehat{x_0},\widehat{x_1}}^{-1} \Big( Z_{\widehat{x_0}}(\text{aff}_{\widehat{x_0},\widehat{x_1}}(z))  \Big) + \text{Aff}_{\widehat{x_0},\widehat{x_1}}^{-1} \Big( Y_{\widehat{x_0}}(\text{aff}_{\widehat{x_0},\widehat{x_1}}(z))  \Big) - Y_{\widehat{x_1}}(z) $$
By hypothesis, $\text{Aff}_{\widehat{x_0},\widehat{x_1}}^{-1} \Big( Z_{\widehat{x_0}}(\text{aff}_{\widehat{x_0},\widehat{x_1}}(z))  \Big)$ is affine in $z$, and so its $(1+\alpha)$-th derivative is zero. We can then write, for all $n\geq 0$ and $z \in (-\rho' \lambda_x^{\langle -n \rangle},\rho' \lambda_x^{\langle -n \rangle}) \cap \Omega_{\widehat{f}^{-n}(\widehat{x_1})}^u$:
$$ \delta^{1+\alpha}(Z_{\widehat{f}^{-n}(\widehat{x_1})})(z) = \alpha_{\widehat{f}^{-n}(\widehat{x_0}),\widehat{f}^{-n}(\widehat{x_1})} \delta^{1+\alpha}(Y_{\widehat{f}^{-n}(\widehat{x_0})})(\text{aff}_{\widehat{f}^{-n}(\widehat{x_0}),\widehat{f}^{-n}(\widehat{x_1})}(z)) - \delta^{1+\alpha}(Y_{\widehat{f}^{-n}(\widehat{x_1})})(z) , $$
where $\alpha_{\widehat{x_0},\widehat{x_1}} := (\text{Aff}^{-1}_{\widehat{x_1},\widehat{x_0}})'(0) (\text{aff}_{\widehat{x_1},\widehat{x_0}})'(0)^{1+\alpha} = 1 + O(d^u(x_0,x_1))$.
Since $|\delta^{1+\alpha}(Y_{\widehat{x_0}})| \leq \|Y_{\widehat{x_0}}\|_{C^{1+\alpha}}$, the criterion applies.
\begin{itemize}
\item Second, we show that if $\mathcal{D}_{\rho'}^{lin}(x)=0$ for some $x$ and small $\rho'$, then $\mathcal{D}_{\min(\rho' \lambda_x,\rho)}^{lin}(f(x))=0$. \end{itemize}
This directly comes from the autosimilarity formula. We have, for $z \in (-\rho,\rho) \cap \Omega_{\widehat{x}}^u$:
$$ Z_{\widehat{x}}(z) = (\widehat{\tau}_{\widehat{x}})'(0)z + \mu_x Z_{\widehat{f}({\widehat{x}})}(z \lambda_x). $$
In particular, if $ Z_{\widehat{x}}$ is linear on $(-\rho',\rho') \cap \Omega_{\widehat{x}}^u$, then $Z_{\widehat{f}({\widehat{x}})}$ is linear on $(-\rho' \lambda_x,\rho' \lambda_x) \cap (-\rho,\rho) \cap \Omega_{\widehat{f}({\widehat{x}})}^u$.
\begin{itemize}
\item We conclude, using the transitivity of the dynamics and the continuity of $\mathcal{D}_\rho^{lin}$. \end{itemize}
We know that $\mathcal{D}_\rho^{lin}(x_0)=0$, by hypothesis. By step one, there exists $\omega$, some unstable neighborhood of $x_0$, and $\rho'<\rho$ such that $\mathcal{D}_{\rho'}^{lin}=0$ on $\omega$. Step 2 then ensures that $\mathcal{D}_{\min(\rho' \lambda_x^{\langle n \rangle} , \rho)}^{lin}=0$ on $f^n(\omega)$. Choosing $N$ large enough, we conclude that
$$ \forall x \in \bigcup_{n \geq N} f^n(\omega), \ \mathcal{D}_\rho^{lin}(x) = 0 .$$
Since the dynamics $f$ is transitive on $\Omega$, we know that $\bigcup_{n \geq N} f^n(\omega)$ is dense in $\Omega$. The function $\mathcal{D}_\rho^{lin}$ being continuous, it follows that $\mathcal{D}_\rho^{lin}=0$ on $\Omega$. \end{proof}

\begin{lemma}[Oscillations everywhere in $x$]
Since there exists ${\widehat{x_0}}$ such that $X_{\widehat{x_0}} \notin C^1$, the following hold. There exists $\kappa \in (0,\rho/10)$ such that, for all ${\widehat{x}} \in \widehat{\Omega}$, for all $P \in \mathbb{R}_{d_Z}[X]$, there exists $z_0 \in \Omega_{\widehat{x}}^u \cap (-\rho/2,\rho/2)$ such that
$$ \forall z \in \Omega_{\widehat{x}}^u \cap (z_0 - \kappa, z_0 + \kappa), \  |Z_{\widehat{x}}(z) - P(z)| \geq \kappa. $$
\end{lemma}

\begin{proof}
Our previous lemma gives us the following dichotomy: either $\mathcal{D}_\rho^{poly} > 0$ on $\Omega$, or either $\mathcal{D}_\rho^{lin} =0$ on $\Omega$. The second possibility can not occur since $X_{\widehat{x_0}} \notin C^1$ (hence $Z_{\widehat{x_0}} \notin C^1$, since $Y_{\widehat{x}} \in C^{1+}$). So $\mathcal{D}_{\rho}^{poly} > 0$ on $\Omega$. By continuity of $\mathcal{D}_\rho^{poly}$, and by compacity of $\Omega$, there exists some $\kappa>0$ such that $\mathcal{D}_\rho^{poly}(x) \geq \kappa$ for all $x \in \Omega$. One can do the same proof replacing $\rho$ with $\rho/2$, so we can directly says that $\mathcal{D}_{\rho/2}^{poly}(x) \geq \kappa$, taking $\kappa$ smaller if necessary. Now, this means the following: for every $P \in \mathbb{R}_{d_Z}[X]$, for every ${\widehat{x}} \in \widehat{\Omega}$, there exists $z_0(x,a,b) \in \Omega_{\widehat{x}}^u \cap (- \rho/2,\rho/2)$ such that
$$ |Z_{\widehat{x}}(z_0)-P(z)| \geq \kappa. $$
We still have to show that this doesn't only hold for some point $z_0$, but on a whole small interval. The proof uses the technical Lemma \ref{lem:norms} on equivalence of norms $\mathbb{R}_{d_Z}[X]$. The proof in different, depending if $\|P\|_{C^\alpha( \Omega_x^u \cap (-\rho,\rho) )}$ is small or large. \\

Let ${\widehat{x}} \in \widehat{\Omega}$ and $P \in \mathbb{R}_{d_Z}[X]$. There exists $z_0 \in \Omega_{\widehat{x}}^u \cap (-\rho/2,\rho/2)$ such that $|P(z_0)-Z_{\widehat{x}}(z_0)| \geq \kappa$. Hence:
$$ \forall z \in (z_0-\tilde{\kappa},z_0+\tilde{\kappa}) \cap \Omega_{\widehat{x}}^u, \ |Z_{\widehat{x}}(z)-P(z)| \geq \kappa - \|Z_{\widehat{x}} - P\|_{C^\alpha(\Omega_{\widehat{x}}^u \cap (-\rho,\rho))} \tilde{\kappa}^{\alpha} $$
$$ \geq \kappa - ( \sup_{x' \in \widehat{\Omega}} \|Z_{x'}\|_{C^\alpha(\Omega_{x'}^u \cap (-\rho,\rho))} + \|P\|_{C^\alpha(-\rho,\rho)} ) \tilde{\kappa}^\alpha. $$
This arguments gives us what we want if $\|P\|_{C^\alpha(-\rho,\rho)}$ is bounded by some constant $M \geq \sup_{x'} \|Z_{x'}\|_{C^\alpha}$. Indeed, in this case, choosing $\tilde{\kappa} := (\kappa/2M)^{1\alpha}$ yields
$$ \forall z \in (z_0-\tilde{\kappa},z_0+\tilde{\kappa}) \cap \Omega_{\widehat{x}}^u, \ |Z_{\widehat{x}}(z)-P(z)| \geq \kappa/2.$$
To conclude, we need to understand the case where $\|P\|_{C^\alpha((-\rho,\rho))}$ is large.
Recall that Lemma \ref{lem:norms} ensure that there exists $\kappa_{0} \in (0,1)$ and $z_1 \in (-\rho/2,\rho/2) \cap \Omega_{\widehat{x}}^u$ such that $$ \forall z \in (z_1 - \kappa_0, z_1 + \kappa_0), \ |P(z)| \geq \kappa_0 \|P\|_{C^\alpha((-\rho,\rho))} .$$
It follows that if $M$ is fixed so large that $M \geq \sup_{x'}\| Z_{x'} \|_{C^\alpha} \kappa_0^{-1} + \kappa$, then in the case where $\|P\|_{C^\alpha(-\rho,\rho)} \geq M$ we find
$$ \forall z \in (z_1 - \kappa_0, z_1 + \kappa_0), \ |P(z) - Z_{\widehat{x}}(z)| \geq |P(z)| - \|Z_{\widehat{x}}\|_{\infty} $$
$$ \geq \kappa_0 \|P\|_{C^{\alpha}} - \|Z_{\widehat{x}}\|_{\infty} \geq \kappa ,$$
which conclude the proof. \end{proof}

\begin{lemma}[Oscillations everywhere in $x$, at all scales in $z$]
There exists $\kappa>0$ such that, for all ${\widehat{x}} \in \widehat{\Omega}$, for all $P \in \mathbb{R}_{d_Z}[X]$, for all $n \geq 0$, there exists $z_0 \in \Omega_{\widehat{x}}^u \cap (-\rho/2,\rho/2)$ such that
$$ \forall z \in \Omega_{\widehat{x}}^u \cap (z_0 - \kappa, z_0 + \kappa), \  |Z_{\widehat{x}}^{\langle n \rangle}(z) - P(z)| \geq \kappa ,$$
where $Z_{\widehat{x}}^{\langle n \rangle}(z) := \mu_x^{\langle -n \rangle} Z_{\widehat{f}^{-n}({\widehat{x}})}( z \lambda_x^{\langle -n \rangle})$. The family $(Z_{\widehat{x}}^{\langle n \rangle})$ is a (n-th times) zoomed-in and rescaled version of $(Z_{\widehat{x}})$.
\end{lemma}
\begin{proof}
We know that $Z_{\widehat{x}}^{\langle n \rangle}(z) = \text{linear} + Z_{\widehat{x}}(z)$ on $(-\rho,\rho) \cap \Omega_{\widehat{x}}^u$. The result follows from the previous lemma.
\end{proof}

\begin{lemma}
There exists $\kappa>0$ and $n_0 \geq 0$ such that, for all $\widehat{x} \in \widehat{\Omega}$, for all $P \in \mathbb{R}_{d_Z}[X]$, for all $n \geq n_0$, there exists $z_0 \in \Omega_{{\widehat{x}}}^u \cap (-\rho/2,\rho/2)$ such that
$$\forall z \in \Omega_{\widehat{x}}^u \cap (z_0 - \kappa, z_0 + \kappa), \  |X_{\widehat{x}}^{\langle n \rangle}(z) - P(z)| \geq \kappa ,$$
where $X_{\widehat{x}}^{\langle n \rangle}(z) := \mu_x^{\langle -n \rangle} X_{\widehat{f}^{-n}({\widehat{x}})}( z \lambda_x^{\langle -n \rangle})$. 
\end{lemma}
\begin{proof}
Recall that $X_{\widehat{x}} = Y_{\widehat{x}} + Z_{\widehat{x}}$, and that $|Y_{\widehat{x}}(z)| \leq C|z|^{1+\alpha}$. Zooming in, we find, for all $n \geq 0$:
$$ X_{\widehat{x}}^{\langle n \rangle}(z) = Y_{\widehat{x}}^{\langle n \rangle}(z) + Z_{\widehat{x}}^{\langle n \rangle}(z) ,$$
where $Y_{\widehat{x}}^{\langle n \rangle}(z) := \mu_x^{\langle -n \rangle} Y_{\widehat{f}^{-n}({\widehat{x}})}(z \lambda_x^{\langle -n \rangle}) = O( (\lambda_x^{\langle - n \rangle})^\alpha ).$ Taking $n_0$ large enough so that this is less than $\kappa/2$ for the $\kappa$ given by the previous lemma allows us to conclude.
\end{proof}

We can rewrite this into a statement about $X_{\widehat{x}}$ only.

\begin{proposition}\label{prop:preUNI}
There exists $0<\kappa<1/10$ and $\sigma_0 > 0$ such that, for all ${\widehat{x}} \in \widehat{\Omega}$, for all $P \in \mathbb{R}_{d_Z}[X]$, for all $0<\sigma<\sigma_0$, there exists $z_0 \in \Omega_{{\widehat{x}}}^u \cap (- \sigma/2, \sigma/2)$ such that
$$\forall z \in \Omega_{\widehat{x}}^u \cap (z_0 - \kappa \sigma, z_0 + \kappa \sigma) \subset(- \sigma, \sigma), \  |X_{\widehat{x}}(z) - P(z)| \geq \kappa \sigma .$$
\end{proposition}

\begin{proof}
For each $\sigma$ small enough, define $n_x(\sigma)$ as the largest positive integer such that $ \sigma \leq \rho \lambda_x^{\langle -n_x(\sigma) \rangle}$.
We then have $\sigma \simeq \lambda_x^{\langle -n_x(\sigma) \rangle } \simeq (\mu_x^{\langle -n_x(\sigma) \rangle})^{-1}$, and we see that we can deduce our statement written with $\sigma$ by our statement written with $n_x(\sigma)$.
\end{proof}

We conclude this section by establishing what we will call the \say{Uniform Non Integrability condition} (UNI) in our context. This can be interpreted as a uniform-non-integrability bound for the temporal distance function of some suspension flow.

\begin{definition}[UNI]\label{UNI}
We say that the condition (UNI) is satisfied if there exists $0<\kappa<1/10$ and $\sigma_0 > 0$ such that, for all ${\widehat{x}} \in \widehat{\Omega}$, for all $P \in \mathbb{R}_{d_Z}[X]$, for all $0<\sigma<\sigma_0$, for any $z_0 \in \Omega_{\widehat{x}}^u \cap (-\rho,\rho)$, there exists $z_1 \in \Omega_{{\widehat{x}}}^u \cap (z_0 -  \sigma/2,z_0+ \sigma/2)$ such that
$$\forall z \in \Omega_{\widehat{x}}^u \cap (z_1 - \kappa \sigma, z_1 + \kappa \sigma) \subset(z_0 - \sigma,z_0+\sigma), \  |X_{\widehat{x}}(z) - P(z)| \geq \kappa \sigma .$$
\end{definition}

\begin{proposition}\label{prop:UNI}
Since there exists ${\widehat{x_0}}$ such that $X_{\widehat{x_0}} \notin C^1$, \hyperref[UNI]{(UNI)} holds.
\end{proposition}

\begin{proof}
Let us fix the $\kappa$ and $\sigma_0$ from Proposition \ref{prop:preUNI}. Let ${\widehat{x}} \in \widehat{\Omega}$, let $\sigma<\sigma_0$, and let $z_0 \in \Omega_{\widehat{x}}^u$. Let $P \in \mathbb{R}_{d_Z}[X]$. Define $x_1 := \Phi_{\widehat{x}}^u(z_0)$. Let $\widehat{x_1}$ be the element in the fiber of $x_1$ that share he same local orientation than $\widehat{x}$. Recall that, by Lemma \ref{lem:Aff}, there exists $\text{Aff}_{{\widehat{x}},{\widehat{x_1}}}$ and $\text{aff}_{{\widehat{x}},{\widehat{x_1}}}$, two affine functions with $e^{O(1)}$ linear coefficients, such that
$$ \text{Aff}_{{\widehat{x}},{\widehat{x_1}}}(X_{{\widehat{x_1}}}(z)) = X_{\widehat{x}}(\text{aff}_{{\widehat{x}},{\widehat{x_1}}}(z)).$$
Furthermore, $\text{aff}_{{\widehat{x}},{\widehat{x_1}}} = (\Phi_{\widehat{x}}^u)^{-1} \circ (\Phi_{{\widehat{x_1}}}^u)$. Notice that $\text{aff}_{{\widehat{x}},{\widehat{x_1}}}(0)=z_0$.
Since $(\text{Aff}_{{\widehat{x}},{\widehat{x_1}}})'(0) = e^{O(1)}$, the previous lemma applied to $X_{{\widehat{x_1}}}$ gives us some $z_2 \in (-\sigma/2,\sigma/2)$ such that:
$$ \forall z \in (z_2 - \kappa \sigma,z_2 + \kappa \sigma), \ |\text{Aff}_{{\widehat{x}},{\widehat{x_1}}}(X_{{\widehat{x_1}}}(z)) - P(z)| \geq \kappa \sigma ,$$
choosing $\kappa$ smaller if necessary.
Setting $z_1:=(\text{aff}_{{\widehat{x}},{\widehat{x_1}}})^{-1}(z_2)$ yields the desired result.
\end{proof}

\subsection{Checking (QNL)}

In this last subsection, we establish our last estimate, given by the following proposition, which takes the role of the \say{tree Lemma} in \cite{Le22,LPS25}.

\begin{proposition}\label{prop:UNItoQNL}
Under \hyperref[UNI]{(UNI)}, there exists $\gamma>0$ and $\alpha>0$ such that for all $\sigma>0$ small enough, for all ${\widehat{x}} \in \widehat{\Omega}$, for all $P \in \mathbb{R}_{d_Z}[X]$, we have
$$ \emeta_{\widehat{x}}\Big( z \in [-\sigma,\sigma], \ |X_{\widehat{x}}(z) - P(z)| \leq \sigma^{1+\alpha} \Big) \leq \sigma^\gamma \emeta_{\widehat{x}}([-\sigma,\sigma]) ,$$
where $\emeta_{\widehat{x}} := (\Phi_{\widehat{x}}^u)^*\mu_x^u$.
\end{proposition}
Once this proposition is proved, Lemma \ref{lem:polyQNL} will ensure that \hyperref[QNL]{(QNL)} holds, thus proving Theorem \ref{th:QNL}. Before heading into the proof, we will establish a cutting lemma. See Appendix \ref{ap:B} for related statements about the uniform perfection of the sets $\Omega_{\widehat{x}}^u \cap (-\rho,\rho)$.

\begin{lemma}
Let $\kappa>0$. There exists $\mathcal{C} \geq 1$, and $\eta_{cut,1},\eta_{cut,2} \in (0,1)$, with $\eta_{cut,1} \leq \kappa/100$, such that the following hold. Let $x \in \Omega$, $\sigma \in (0,\rho]$ and $n \geq 1$. Let $U^{(\sigma)}_{x} \subset W^u_{loc}(x) \cap \Omega$ be such that $B(x,\sigma) \cap W^u_{loc}(x) \cap \Omega \subset U^{(\sigma)}_{x} \subset B(x,10\sigma) \cap W^u_{loc}(x)$. There exists a finite set $\mathcal{A} := \mathcal{A}(x,\sigma,n)$ with $|\mathcal{A}| \leq \mathcal{C}$ and a family of intervals $(I_\mathbf{a})_{\mathbf{a} \in \bigcup_{k=0}^n\mathcal{W}_k}$, where $\mathcal{W}_k \subset \mathcal{A}^k$ is a set of words on the alphabet $\mathcal{A}$, and where $I_{\mathbf{a}} \subset W^u_{loc}(x)$ are such that: 
\begin{enumerate}
    \item $I_{\emptyset} \cap \Omega = U^{(\sigma)}_x$
    \item For all $\mathbf{a} \in \mathcal{W}_k$, $0 \leq k \leq n$, we have $I_{\mathbf{a}} \cap \Omega = \underset{\mathbf{a}b \in \mathcal{W}_{n+1}}{\bigcup_{b \in \mathcal{A}}} (I_{\mathbf{a}b} \cap \Omega)$, and this union is disjoint modulo a zero-measure set.
    \item For all $\mathbf{a} \in \mathcal{W}_k$, $0 \leq k \leq n$, there exists $x_{\mathbf{a}} \in \Omega \cap I_{\mathbf{a}}$ such that $$W^u_{loc}(x) \cap 
  B(x_{\mathbf{a}}, 10^{-1} \text{diam}^u(I_{\mathbf{a}}) ) \subset I_{\mathbf{a}}$$
    \item For all $\mathbf{a} \in \mathcal{W}_k$, $0 \leq k \leq n-1$, we have for any $b \in \mathcal{A}$ such that $\mathbf{a}b \in \mathcal{W}_{k+1}$: $$ \eta_{cut,1}^2 \text{diam}^u(I_\mathbf{a}) \leq \text{diam}^u(I_{\mathbf{a}b}) \leq \eta_{cut,1} \text{diam}^u(I_{\mathbf{a}}) \quad , \quad \mu_x^u(I_{\mathbf{a}b}) \geq \eta_{cut,2} \mu_x^u(I_{\mathbf{a}}).$$
\end{enumerate}
\end{lemma}

\begin{proof}
The proof is done in a couple of steps.
\begin{itemize}
\item \underline{Step 1}: We cut, once ($n=1$), the sets $U_x^{(\sigma)}$ when $\sigma \in [\rho \lambda_+^{-2},\rho]$. \end{itemize}
First of all, recall that for all $\varepsilon$, there exists $\delta>0$ such that for any $z \in W^u_{loc}(x) \cap \Omega$, $\mu_x^u(B(z,\varepsilon) \cap W_{loc}^u(x))>\delta$. Since, moreover, the measure $\mu_x^u$ is upper regular, that is:
$$ \forall U \subset W^u_{loc}(x), \ \mu_x^u(U) \leq C \text{diam}(U)^{\delta_{reg}}, $$
it follows that reciprocally, if $\mu_x^u(U) \geq \varepsilon$, then $\text{diam}(U) \geq \delta$. \newline

Consider a finite cover of $U_x^{(\sigma)}$ by a set of balls $E$ of the form $B(x_i,\varepsilon)$ with $x_i \in U_x^{(\sigma)}$, for $\varepsilon$ small enough (for example, $\varepsilon := \kappa \rho \lambda_+^{-2} \cdot 10^{-10}$ should be enough). By Vitali's lemma, there exists a subset $D \subset E$ such that the balls in $D$ are disjoints, and such that $\cup_{B \in D} 3B \supset U^{(\sigma)}_x$. Since $W^u_{loc}(x)$ is one-dimensional, we can set $\mathcal{A} := D$, and choose $B \subset I_a \subset 3B$ such that the $(I_a)$ have disjoint union and covers $U^{(\sigma)}$. By construction, the diameter of those intervals is small, but in a controlled way: the measure of them is greater than some constant, and they all contain a ball of radius $\varepsilon$ for some $\varepsilon$. This construct our $(I_a)_{a \in \mathcal{A}}$ in this case. Notice that the cardinal of $\mathcal{A}$ is uniformly bounded. We see that it suffice to take $\varepsilon$ small enough to find other constants $\eta_{cut,1},\eta_{cut,2}$ with $\eta_{cut,1} \leq \kappa/100$ that satisfies what we want in this macroscopic context. 

\begin{itemize}
\item \underline{Step 2:} We cut, once $(n=1)$, sets $U_x^{(\sigma)}$ when $\sigma>0$ is small.
\end{itemize}
The idea is to use the properties of $\mu_x^u$. If we denote by $\varphi:\Omega \rightarrow \mathbb{R}$ the potential defining our equilibrium state $\mu$, then we know \cite{Cl20} that $ f_* d\mu_x^u = e^{\varphi \circ f^{-1} - P(\varphi)} d\mu_{f(x)}^u $ (recall Proposition \ref{prop:locProdStruct}). Iterating yields $f_*^n d\mu_x^u = e^{S_n\varphi \circ f^{-n} - nP(\varphi)} d\mu_{f^n(x)}^u$. The Hölder regularity of $\varphi$ allows us to see that, for all intervals $J \subset I := W^u_{loc}(f^n(x)) \cap B(f^n(x),\rho) $, we have $$ \frac{\mu_x^u( f^{-n}(J))}{\mu_x^u(f^{-n}(I))} =  \frac{\mu_x^u(J)}{\mu_x^u(I)}\Big{(}1 + O(\rho^\alpha)\Big{)} \quad, \quad \frac{\text{diam}( f^{-n}(J))}{\text{diam}(f^{-n}(I))} =  \frac{\text{diam}(J)}{\text{diam}(I)}\Big{(}1 + O(\rho^\alpha)\Big{)} .$$
So we see that, using the dynamics, (and reducing $\rho$ if necessary) we can reduce our setting to the setting of step 1. This might deform a bit the constants though, which is why we chosed the $10^{-1}$ in point three instead of $6^{-1}$ (which was the constant given by Vitali's Lemma).
\begin{itemize}
\item \underline{Step 3:} We iterate this construction to the subintervals $I_b$. 
\end{itemize}
We consider $U^{(\sigma)}$ as in the statement of the lemma. We use step 2 to construct $(I_a)_{a \in \mathcal{A}}$. Then, we can iterate our construction on each of the $I_a$, since they satisfy the necessary hypothesis to do so. This gives us intervals $(I_{a_1 a_2})_{a_1 a_2 \in \mathcal{W}_2}$, with $\mathcal{W}_2 \subset \mathcal{A}^2$ (taking $\mathcal{A}$ with more letters if necessary). Doing this again and again yield the desired construction.
\end{proof}

Using this partition lemma, we can prove Proposition \ref{prop:UNItoQNL}. Notice that, looking at this lemma in coordinates $\Phi_{\widehat{x}}^u$, one can get the same construction  replacing subintervals of $W^u_{loc}$ by intervals of $\mathbb{R}$,  $\Omega$ by $\Omega_{\widehat{x}}^u$, $\mu_x^u$ by $\emeta_{\widehat{x}}$, etc.

\begin{proof}[Proof (of Proposition \ref{prop:UNItoQNL})]
Assume \hyperref[UNI]{(UNI)}, and denote $(\kappa,\sigma_0)$ the constants given by \hyperref[UNI]{(UNI)}. Let $0<\sigma<\sigma_0$ be small enough, let $P \in \mathbb{R}_{d_Z}[X]$. Let $x \in \Omega$. Define $k(\sigma) \in \mathbb{N}$ as the largest integer such that $$  \sigma^\alpha \leq (\kappa/20) \cdot \Big({\eta}_{cut,2}\Big)^{2 k(\sigma)}.$$ We have $k(\sigma) \simeq |\ln \sigma|$. By applying the previous construction to $\Phi_{\widehat{x}}^u([-\sigma,\sigma])$ (for the constant $\kappa$), with depth $k(\sigma)$, we find a family of intervals $(I_{\mathbf{a}})_{\mathbf{a} \in \cup_{j \leq k(\sigma)}\mathcal{W}_j}$ that satisfies some good partition properties. Then, the heart the proof is as follow: we want to construct (up to renaming some of the $I_{\mathbf{a}}$) a word $\mathbf{b} \in W_{k(\sigma)} \subset \mathcal{A}^{k(\sigma)}$ such that
$$ \{ z \in [-\sigma,\sigma] , \  |X_{\widehat{x}}(z) - P(z)| \leq \sigma^{1+\alpha} \} \subset \underset{\forall i, a_i \neq b_i}{\bigcup_{\mathbf{a} \in \mathcal{W}_{k(\sigma)}}} I_{\mathbf{a}} .$$
This is nothing more than a technical way of keeping track of the oscillations happening at all scales between $\sigma$ and $\sigma^{1+\alpha}$. Notice that, once this is proved, then the conclusion is easy: we will have
$$ \emeta_{\widehat{x}}\Big( \underset{\forall i, a_i \neq b_i}{\bigcup_{\mathbf{a} \in \mathcal{W}_{k(\sigma)}}} I_{\mathbf{a}} \Big) = \underset{\forall i, a_i \neq b_i}{\sum_{\mathbf{a} \in \mathcal{W}_{k(\sigma)}}} \emeta_{\widehat{x}}(I_{\mathbf{a}}) $$
$$ \leq (1-\eta_{cut,2})\underset{\forall i, a_i \neq b_i}{\sum_{\mathbf{a} \in \mathcal{W}_{k(\sigma)-1}}} \emeta_{\widehat{x}}(I_{\mathbf{a}}) $$
$$ \leq (1-\eta_{cut,2})^2 \underset{\forall i, a_i \neq b_i}{\sum_{\mathbf{a} \in \mathcal{W}_{k(\sigma)-2}}} \emeta_{\widehat{x}}(I_{\mathbf{a}}) $$
$$ \leq \dots \leq (1-\eta_{cut,2})^{k(\sigma)} \emeta_x([-\sigma,\sigma]) \simeq \sigma^\gamma \emeta_{\widehat{x}}([-\sigma, \sigma])  $$
for some $\gamma>0$, since $k(\sigma) \simeq \ln(\sigma^{-1})$. Now, let us construct this word $\mathbf{b}$: the idea is that since $2 \sigma^{1+\alpha} \simeq \sigma (\kappa/10) \cdot \Big({\eta}_{cut,2}\Big)^{2 k(\sigma)}$, for each $i$, we can find oscillations at scales $\sim \sigma \eta_{cut,1}^i$ of magnitude $\sim \sigma (\kappa/10) (\eta_{cut,1})^{2i} \geq 2 \sigma^{1+\alpha}$, and conclude. \\

Let us begin by the case $i=0$.
Using \hyperref[UNI]{(UNI)}, we know that there exists a point $z_{\emptyset} \in \Omega_{\widehat{x}}^u \cap(-\sigma/2,\sigma/2)$ such that:
$$ \forall z \in (z_{\emptyset} \pm \kappa\sigma), \ |X_{\widehat{x}}(z)-P(z)| \geq \kappa \sigma.$$
There exists $b_1 \in \mathcal{A}$ such that $z_{\emptyset} \in I_{b_1}$. Since $\text{diam}(I_{b_1}) \leq \eta_{cut,1} \sigma \leq \frac{\kappa}{100} \sigma$, we find 
$$ \forall z \in I_{b_1}, \ |X_{\widehat{x}}(z)-P(z)| \geq \kappa \sigma.$$
Hence: $$ \{ z \in [-\sigma,\sigma] , \ |X_{\widehat{x}}(z)-P(z)| \leq \sigma^{1+\alpha} \} \subset \{ z \in [-\sigma,\sigma] , \ |X_{\widehat{x}}(z)-P(z)| \leq \sigma \kappa/2 \} \subset \bigcup_{a_1 \in \mathcal{A}\setminus \{b_1\}} I_{a_1}.$$
Now, let $a_1 \in \mathcal{A} \setminus \{b_1\}$, and work on $I_{a_1}$. We know that there exists $z_{a_1} \in \Omega_{\widehat{x}}^u$ such that \\$(z_{a_1} \pm 10^{-1} \text{diam}(I_{a_1})) \subset I_{a_1}$. Now, \hyperref[UNI]{(UNI)} applied at this interval yields:
$$ \exists \tilde{z}_{a_1} \in I_{a_1} \cap \Omega_{\widehat{x}}^u, \forall z \in (\tilde{z}_{a_1} \pm (\kappa/10) \text{diam}(I_{a_1})), \ |X_{\widehat{x}}(z) - P(z)| \geq (\kappa/10) \text{diam}(I_{a_1}) \geq (\kappa/10) \eta_{cut,1}^2 \sigma. $$
Since $\text{diam}(I_{a_1 b}) \leq \eta_{cut,1} \text{diam}(I_{a_1})$, with $\eta_{cut,1} \leq \kappa/100$, we find, for some $b_2 \in \mathcal{A}$ (we can say that this is the same $b_2$ for any $a_1 \in \mathcal{A}$, up to renaming the intervals),
$$ \forall z \in I_{a_1 b_2}, \ |X_{\widehat{x}}(z) - P(z)| \geq (\kappa / 10) \eta_{cut,1}^2 \sigma. $$
Hence:
$$ I_{a_1} \cap \{ z \in [-\sigma,\sigma] , \ |X_{\widehat{x}}(z) - P(z)| \leq \sigma^{1+\alpha} \} $$ $$\subset I_{a_1} \cap \{ z \in [-\sigma,\sigma] , \ |X_{\widehat{x}}(z) - P(z)| \leq (\kappa/20) {\eta_{cut,1}^2} \sigma \} \subset \bigcup_{a_2 \neq b_2} I_{a_1 a_2}. $$
Hence $$ \{ z \in [-\sigma,\sigma] , \ |X_{\widehat{x}}(z) - P(z)| \leq \sigma^{1+\alpha} \}  \subset \bigcup_{a_1 \neq b_1} \bigcup_{a_2 \neq b_2} I_{a_1 a_2} . $$
Let us do the next step, and then we will stop here because the construction will be clear enough. We could formally conclude by induction.
Take $a_1 \neq b_1$ and $a_2 \neq b_2$. We know that there exists $z_{a_1 a_2} \in \Omega_{\widehat{x}}^u$ such that $(z_{a_1 a_2} \pm 10^{-1} \text{diam}(I_{a_1 a_2})) \subset I_{a_1 a_2}$.  Applying \hyperref[UNI]{(UNI)} to this interval yield
$$ \exists \tilde{z}_{a_1 a_2} \in I_{a_1 a_2} \cap \Omega_{\widehat{x}}^u, \forall z \in (\tilde{z}_{a_1 a_2} \pm (\kappa/10) \text{diam}(I_{a_1 a_2})), $$ $$ |X_{\widehat{x}}(z)-P(z)| \geq (\kappa/10) \text{diam}(I_{a_1 a_2}) \geq (\kappa/10) \eta_{cut,1}^4 \sigma.$$
There exists $b_3 \in \mathcal{A}$ such that $\tilde{z}_{a_1 a_2} \in I_{a_1 a_2 b_3}$. Since $\text{diam}(I_{a_1 a_2 b_3}) \leq \eta_{cut,1} \text{diam}(I_{a_1 a_2})$, with $\eta_{cut,1} \leq \kappa/100$, we have
$$ \forall z \in I_{a_1 a_2 b_3}, \ |X_{\widehat{x}}(z) - P(z)| \geq (\kappa/10) \eta_{cut,1}^4 \sigma \geq 2 \sigma^{1+\alpha} .$$
Hence,
$$ I_{a_1 a_2} \cap \{ z \in [-\sigma,\sigma], \ |X_{\widehat{x}}(z) - P(z)| \leq \sigma^{1+\alpha} \} \subset \bigcup_{a_3 \neq b_3} I_{a_1 a_2 a_3} ,$$
and so $$ \{ z \in [-\sigma,\sigma], \ |X_{\widehat{x}}(z) - P(z)| \leq \sigma^{1+\alpha} \} \subset \bigcup_{a_1 \neq b_1}  \bigcup_{a_2 \neq b_2}   \bigcup_{a_3 \neq b_3}  I_{a_1 a_2 a_3} .$$
This algorithm is done until the $k(\sigma)$-th step. This conclude this construction, hence the proof. \end{proof}

Theorem \ref{th:QNL} is then proved using Proposition \ref{prop:UNItoQNL}, Proposition \ref{prop:UNI}, and Lemma \ref{lem:polyQNL}.

\appendix

\section{Some technical regularity statements}\label{ap:A}

\subsection{Some regularity for the geometric potential}

In this section, we prove that (in our 2-dimensional, Axiom A, area-preserving context) even though the geometric potential $\tau_f(x) := \ln \| (df)_{|E^u(x)} \|$ is only $C^{1+\alpha}(\Omega,\mathbb{R})$,  for any $p \in \Omega$, the family of functions 
$$ \partial_s \tau_f : W^u_{loc}(p) \cap \Omega \longrightarrow \mathbb{R} $$
are $C^{1+\alpha}(W^u_{loc}(p),\mathbb{R})$ (and this, uniformly in $p$). Moreover, the differential of $\tau_f$ varies (at least) in a $C^{1+\alpha}$ manner along local unstable manifolds, and $\partial_u \partial_s \tau_f$ is Hölder on $\Omega$. In fact, $\tau_f \in \text{Reg}_u^{1+\alpha}(\Omega)$ (this is proposition \ref{prop:REGu}). Let us introduce some notations.

\begin{definition}
For any $\widehat{x} \in \widehat{\Omega}$ and $(z,y) \in (-\rho,\rho)^2 \cap \iota_{\widehat{x}}^{-1}(\Omega)$, denote by $\theta^u_{\widehat{x}}(z,y) \in \mathbb{R}$ the only real number such that
$$ (\iota_{\widehat{x}})^*(E^u(\iota_{\widehat{x}}(z,y))) = \text{Span}\Big{(}\partial_z + \theta^u_{\widehat{x}}(z,y) \partial_y\Big{)} ,$$
where $(\iota_{\widehat{x}})_{{\widehat{x}} \in \widehat{\Omega}}$ is the family of smooth coordinates defined in Lemma \ref{lem:coordinates}. 
The real valued map $\theta_{\widehat{x}}^u$ is $C^{1+\alpha}$, and this, uniformly in ${\widehat{x}} \in \widehat{\Omega}$.
\end{definition}

\begin{remark}
Notice that, for all $z$, $\theta_{\widehat{x}}^u(z,0) = 0$. In particular, $z \mapsto \partial_z \theta_{\widehat{x}}^u(z,0)=0$ is smooth.
\end{remark}

\begin{lemma}
The map $z \mapsto \partial_y \theta_{\widehat{x}}^u(z,0)$ is $C^{1+\alpha}((-\rho,\rho) \cap \Omega_{\widehat{x}},\mathbb{R})$, and this, uniformly in ${\widehat{x}}$.
\end{lemma}

\begin{proof}
Recall that $f_{\widehat{x}} := \iota_{\widehat{f}({\widehat{x}})}^{-1} \circ f \circ \iota_{\widehat{x}}$ denotes the dynamics in coordinates, and that $f_{\widehat{x}}(z,0) = (\lambda_x z,0)$.
Let us introduce four families of smooth functions $a_{\widehat{x}},b_{\widehat{x}},c_{\widehat{x}},d_{\widehat{x}} : (-\rho,\rho)^2 \rightarrow \mathbb{R}$ such that $$ (df_{\widehat{x}})_{(z,y)} = \begin{pmatrix} a_{\widehat{x}}(z,y) & b_{\widehat{x}}(z,y) \\ c_{\widehat{x}}(z,y) & d_{\widehat{x}}(z,y) \end{pmatrix} . $$
Recall that, by Remark \ref{rem:temp}, we have the following properties:
$$ a_{\widehat{x}}(z,0) = \lambda_x, \quad c_{\widehat{x}}(z,0) = 0, \quad d_{\widehat{x}}(z,0) = \mu_x.$$
Now, the invariance relation $(df)(E^u({\widehat{x}})) = E^u(f({\widehat{x}}))$, written in coordinates, yields
$$ \theta^u_{\widehat{f}({\widehat{x}})}(f_{\widehat{x}}(z,y)) = \frac{c_{\widehat{x}}(z,y)+d_{\widehat{x}}(z,y) \theta_{\widehat{x}}^u(z,y)}{a_{\widehat{x}}(z,y)+b_{\widehat{x}}(z,y) \theta_{\widehat{x}}^u(z,y)} .$$
Taking the derivative in $y$ and choosing $y=0$ gives, for the LHS:
$$ \frac{\partial}{\partial y}_{|y=0}\theta^u_{\widehat{f}({\widehat{x}})}(f_{\widehat{x}}(z,y)) = \partial_z \theta_{\widehat{f}({\widehat{x}})}^u(f_{\widehat{x}}(z,0)) b_{\widehat{x}}(z,0) + \partial_y \theta_{\widehat{f}({\widehat{x}})}^u(f_{\widehat{x}}(z,0)) d_{\widehat{x}}(z,0) = \partial_y \theta_{\widehat{f}({\widehat{x}})}^u(\lambda_x z,0) \mu_x. $$
Hence:
$$\partial_y \theta_{\widehat{f}({\widehat{x}})}^u(\lambda_x z,0) \mu_x = \frac{\Big{(} \partial_y c_{\widehat{x}}(z,0) + d_{\widehat{x}}(z,0) \partial_y \theta_{\widehat{x}}^u(z,0) \Big{)} a_{\widehat{x}}(z,0) - \Big{(} \partial_y a_{\widehat{x}}(z,0) + b_{\widehat{x}}(z,0) \partial_y \theta_{\widehat{x}}^u(z,0) \Big{)} c_{\widehat{x}}(z,0)}{a_{\widehat{x}}(z,0)^2} $$
$$ = \frac{\left(\partial_y c_{\widehat{x}}(z,0) + \mu_x \partial_y \theta_{\widehat{x}}(z,0)\right) \lambda_x}{\lambda_x^2}, $$
and so:
$$ \partial_y \theta_{\widehat{f}({\widehat{x}})}^u(\lambda_x z,0) = \frac{\partial_y c_{\widehat{x}}(z,0)}{\mu_x \lambda_x} + \lambda_x^{-1} \partial_y \theta_{\widehat{x}}(z,0) .$$
Now, set $\phi_{\widehat{x}}(z) := \frac{\partial_y c_{\widehat{x}}(z,0)}{\mu_x \lambda_x}$. This is a smooth function of $z$. A change of variable yields
$$ \partial_y \theta_{{\widehat{x}}}^u(z,0) = \phi_{\widehat{f}^{-1}({\widehat{x}})}(z \lambda_{f^{-1}(x)}^{-1}) + \lambda_{f^{-1}(x)}^{-1} \partial_y \theta_{\widehat{f}^{-1}({\widehat{x}})}(z \lambda_{f^{-1}(x)}^{-1},0)  ,$$
and so we find that
$$ \partial_y \theta_{\widehat{x}}^u(z,0) = \sum_{n=1}^\infty \lambda_x^{\langle -n \rangle} \phi_{\widehat{f}^{-n}({\widehat{x}})}(z \lambda_x^{\langle -n \rangle}) .$$
This expression is clearly smooth in $z$, and varies in a Hölder manner in ${\widehat{x}}$. Its derivative at zero,
$$ \partial_z \partial_y \theta_{\widehat{x}}^u(0,0) = \sum_{n=1}^\infty (\lambda_x^{\langle -n \rangle})^2 \phi_{\widehat{f}^{-n}({\widehat{x}})}'(0) ,$$
varies in a Hölder manner in ${\widehat{x}}$.
\end{proof}

\begin{definition}
Let $\mathbb{P}^1 M$ denote the projective tangent blundle of $M$, defined as the quotient of the unit tangent bundle $T^1 M$ under the action of the antipodal map $v \in T^1_x M \mapsto -v \in T^1_x M$. We consider the map:
$$ \mathbf{s}^u : \Omega \hookrightarrow \mathbb{P}^1 M $$
defined by $\mathbf{s}^u(x) := E^u(x) \in \mathbb{P}^1 M$. This is a section of the bundle $\mathbb{P}^1 M$: in other words, if $\pi_{\mathbb{P}^1 M} : \mathbb{P}^1 M \rightarrow M$ denotes the projection on the basepoint, we have $\pi_{\mathbb{P}^1M} \circ \mathbf{s}^u = Id_{\Omega}$.
\end{definition}

\begin{corollary}
The map $\mathbf{s}^u : \Omega \rightarrow \mathbb{P}^1 M$ is $C^{1+\alpha}$. Moreover, for any $p \in \Omega$, the map
$$ r \in W^u_{loc}(p) \cap \Omega \mapsto \partial_s \mathbf{s}^u(r) \in T(\mathbb{P}^1 M) $$
is $C^{1+\alpha}$ (and this, uniformly in $p$). Moreover, $\partial_u \partial_s \mathbf{s}^u$ and $\partial_u \partial_u \mathbf{s}^u$ are Hölder regular.
\end{corollary}

\begin{proof}
We use the coordinate chart $\iota_{\widehat{x}} : (-1,1)^2 \rightarrow M$ to create an associated coordinate chart
$ \overline{\iota}_{\widehat{x}} : (-1,1)^3 \rightarrow \mathbb{P}^1 M $ as follow:
$$ \overline{\iota}_{\widehat{x}}(z,y,\theta) := \text{Span} \Big( (\iota_{\widehat{x}})_* (\partial_z + \theta \partial_y) \Big) \in \mathbb{P}^1_{\iota_{\widehat{x}}(z,y)} M.$$
Using these coordinates, we see that
$$ \mathbf{s}^u( \iota_{\widehat{x}}(z,y) ) = E^u(\iota_{\widehat{x}}(z,y)) = (\iota_{\widehat{x}})_*(\partial_z + \theta_{\widehat{x}}^u(z,y) \partial_z ) = \overline{\iota}_{\widehat{x}}(z,y,\theta_{\widehat{x}}^u(z,y)).$$
In other words, $\mathbf{s}^u$ can be written, through the coordinates $\iota_{\widehat{x}}$ and $\overline{\iota}_{\widehat{x}}$, as $(z,y) \mapsto (z,y,\theta_{\widehat{x}}^u(z,y))$.
We know that $\theta_{\widehat{x}}(z,y)$ is $C^{1+\alpha}$, and that $z \mapsto \partial_y \theta_{\widehat{x}}^u(z,0) $ is smooth (and varies in a Hölder manner in ${\widehat{x}}$).
We also know that $z \mapsto \partial_z \theta_{\widehat{x}}^u(z,0) = 0$ is smooth. Hence, $z \mapsto (d\theta_{\widehat{x}}^u)_{(z,0)}$ is smooth, and varies in a Hölder manner in ${\widehat{x}}$. It follows that $r \in W^u_{loc}(p) \mapsto (d \mathbf{s}^u)_r $ is also smooth. Finally, $E^s$ is $C^{1+\alpha}$, and so 
$$r \in W^u_{loc}(p) \mapsto \partial_s \mathbf{s}^u(r) = (d \mathbf{s}^u)_r(\partial_s) \in T(\mathbb{P}^1 M) $$
is also $C^{1+\alpha}$. The Hölder property of $\partial_u \partial_s \mathbf{s}^u$ and $\partial_u \partial_u \mathbf{s}^u$ also follows from the behavior of $\theta_x^u$ in coordinates.
\end{proof}

\begin{theorem}
Define, for $p \in \Omega$, $\tau_f(p) := \ln \|{ (df)}_{|E^u(p)} \|$. In our area-preserving, Axiom A, 2-dimensional context, $\tau_f \in C^{1+\alpha}(\Omega,\mathbb{R})$. Moreover, for all $p \in \Omega$, the map $$ r \in W^u_{loc}(p) \cap \Omega \longrightarrow \partial_s \tau_f(r) \in \mathbb{R} $$
is $C^{1+\alpha}$, and this, uniformly in $p \in \Omega$. Finally, the map $\partial_u \partial_s \tau_f$ is Hölder regular on $\Omega$.
\end{theorem}

\begin{proof}
Let $\tilde{\Omega}$ denote a small enough open neighborhood of $s^u(\Omega) \subset \mathbb{P}^1 M$. Then $$ \overline{\tau}_f : V \in \tilde{\Omega} \longrightarrow \ln \| (df)_{|V} \| \in \mathbb{R} $$
is smooth. Moreover, $\tau_f = \overline{\tau}_f \circ \mathbf{s}^u$. It follows that $$ \partial_s \tau_f(r) = (d \overline{\tau}_f)_{s^u(r)}( \partial_s \mathbf{s}^u(r) ) $$ is $C^{1+\alpha}(W^u_{loc}(p),\mathbb{R})$. Moreover, $\partial_u \partial_s \tau_f(p)$ is Hölder regular for $p \in \Omega$, since $\overline{\tau}_f$ is smooth, and $\mathbf{s}^u, \partial_s \mathbf{s}^u,\partial_u \mathbf{s}^u, \partial_u \partial_s \mathbf{s}^u$ are Hölder on $\Omega$.
\end{proof}

\subsection{Some regularity for the distance function}

The goal of this subsection is to prove Proposition \ref{prop:dist}. Let us first define some notations.

\begin{definition}
Let ${\widehat{x}} \in \widehat{\Omega}$. For each $y \in \Omega_{\widehat{x}}^s \cap (-\rho,\rho)$, we parametrize the local unstable manifold $\widetilde{W}^u_{{\widehat{x}}}(y) := \iota_{\widehat{x}}^{-1}(W^u_{loc}(\Phi_{\widehat{x}}^s(y)))$ in the following way: $\widetilde{W}^u_{{\widehat{x}}}(y) = \{ (z,\mathcal{G}_{\widehat{x}}(z,y)), \ z \in (-\rho,\rho) \}$.
The function $z \mapsto \mathcal{G}_{\widehat{x}}(z,y)$ is smooth, and the function $\mathcal{G}_{\widehat{x}}$ is $C^{1+\alpha}$ by the properties of the local unstable foliation in our two-dimensional context.
\end{definition}

\begin{lemma}
For all $z \in (-\rho,\rho)$, $\partial_y \mathcal{G}_{\widehat{x}}(z,0) = 1$.
\end{lemma}

\begin{proof}
We will use the invariance of the unstable foliation under the dynamics. The formula $f_{\widehat{x}} \widetilde{W}_{\widehat{x}}^s(y) = \widetilde{W}_{\widehat{f}({\widehat{x}})}^s( \mu_x y )$ can be rewritten as
$$ f_{\widehat{x}}(z,\mathcal{G}_{\widehat{x}}(z,y)) = ( \psi_{\widehat{x}}(z,y), \mathcal{G}_{\widehat{f}({\widehat{x}})}(\psi_{\widehat{x}}(z,y),\mu_x y) ) ,$$
where $\psi_{\widehat{x}}(z,y)$ is some (uniformly) $C^{1+\alpha}$ function. We can compute its value when $y=0$: since $\mathcal{G}_{\widehat{x}}(z,0) = 0$, we find $f_{\widehat{x}}(z,\mathcal{G}_{\widehat{x}}(z,0)) = f_{\widehat{x}}(z,0) = (\lambda_x z,0)$ and so $\psi_{\widehat{x}}(z,0) = \lambda_x z$. We apply $\frac{\partial}{\partial y}_{|y=0}$ to our invariance relation and we find:
$$ \partial_y f_{\widehat{x}}(z,0) \partial_y \mathcal{G}_{\widehat{x}}(z,0) = ( \partial_y \psi_{\widehat{x}}(z,0), \partial_z \mathcal{G}_{\widehat{f}({\widehat{x}})}(\psi_{\widehat{x}}(z,0),0) ) \partial_y \psi_{\widehat{x}}(z,0) + \partial_y \mathcal{G}_{\widehat{f}({\widehat{x}})}(\psi_{\widehat{x}}(z,0),0)  \mu_x ). $$
Since $\partial_z \mathcal{G}_{\widehat{f}({\widehat{x}})}(\lambda_x z,0) = 0$, the equation given by the second coordinate can be rewritten as
$$ \mu_x \partial_y \mathcal{G}_{\widehat{x}}(z,0) = \partial_y \mathcal{G}_{\widehat{f}({\widehat{x}})}(\lambda_x z,0) \mu_x .$$
Hence, for all $z$, and for all $n \geq 0$, $$ \partial_y \mathcal{G}_{\widehat{x}}(z,0) = \partial_y \mathcal{G}_{\widehat{f}^{-n}({\widehat{x}})}( \lambda_x^{\langle -n \rangle} z,0) \underset{n \rightarrow \infty}{\longrightarrow} 1, $$ 
since $\partial_y \mathcal{G}_{\widehat{x}}(0,0) = 1$ (as $\partial_y \mathcal{G}_{\widehat{x}}(0,y) = y$), and since all the maps $\mathcal{G}_{\widehat{x}}(z,y)$ are uniformly $C^{1+\alpha}$.
\end{proof}

\begin{remark}
The path $\gamma_{{\widehat{x}},y}(z) := (z,\mathcal{G}_{\widehat{x}}(z,y))$ have normalised derivative
$$ \frac{\gamma_{{\widehat{x}},y}'(z)}{|\gamma_{{\widehat{x}},y}(z)|} = \Big( \frac{1}{\sqrt{1+\partial_z \mathcal{G}_{\widehat{x}}(z,y)^2}} , \frac{\partial_z \mathcal{G}_{\widehat{x}}(z,y)}{\sqrt{1+\partial_z \mathcal{G}_{\widehat{x}}(z,y)^2}} \Big) .$$
This function have the same regularity than $E^u$: that is, smooth in the $z$ variable, and $C^{1+\alpha}$ in the $(z,y)$ variable. It follows that $\partial_z \mathcal{G}_{\widehat{x}}$ is $C^{1+\alpha}$. Moreover, it follows from the previous subsection that $ z \mapsto \partial_y \partial_z \mathcal{G}_{\widehat{x}}(z,0)$ is smooth.
\end{remark}

\begin{definition}
Let us denote by $S_{\widehat{x}}(z,y)$ the symmetric $2\times 2$ matrix representing the Riemannian metric of $M$ through the coordinates $\iota_{\widehat{x}}$. In other words, for any path $\gamma(t)  \in (-\rho,\rho)^2$ in coordinates, its lenght seen as a path $\iota_{\widehat{x}} \circ \gamma$ on the manifold can be computed by the formula
$$ \text{Lenght}_{\widehat{x}}(\gamma) = \int_a^b \langle S_{\widehat{x}}( \gamma(t)) \dot{\gamma}(t), \dot{\gamma}(t) \rangle dt ,$$
where $\langle \cdot , \cdot \rangle$ is the usual scalar product in $\mathbb{R}^2$. The function $S_{\widehat{x}}$ is smooth (and this, uniformly in ${\widehat{x}}$).
\end{definition}

We are ready to state our main technical lemma.

\begin{lemma}
Denote by $\mathbf{L}_{\widehat{x}}(z,y) := \text{Lenght}_{\widehat{x}}((\gamma_{{\widehat{x}},y})_{|[0,z]})$. Then, $\mathbf{L}_{\widehat{x}}$ is smooth in $z$, $C^{1+\alpha}$ in $(z,y)$, and moreover the map $z \mapsto \partial_y \mathbf{L}_{\widehat{x}}(z,0)$ is smooth.
\end{lemma}

\begin{proof}
Denoting $\gamma_{\widehat{x}}(z,y) = (z,\mathcal{G}_{\widehat{x}}(z,y))$, we have the formula
$$ \mathbf{L}_{\widehat{x}}(z,y) = \int_0^z \langle S_{\widehat{x}}(\gamma_{\widehat{x}}(t,y)) \partial_z{\gamma}_{\widehat{x}}(t,y) , \partial_z{\gamma}_{\widehat{x}}(t,y) \rangle dt .$$
Since $\partial_z \mathcal{G}_{\widehat{x}}(\cdot,y)$ and $\mathcal{G}_{\widehat{x}}(\cdot,y)$ are smooth, the map $\mathbf{L}_{\widehat{x}}(\cdot,y)$ is also smooth. Moreover, since $\partial_z \mathcal{G}_{\widehat{x}}$ and $\mathcal{G}_{\widehat{x}}$ are $C^{1+\alpha}$, the map $\mathbf{L}_{\widehat{x}}$ also is. To compute $\partial_y \mathbf{L}_{\widehat{x}}(z,0)$, notice first that $$\frac{d}{dy}_{|y=0} S_{\widehat{x}}(\gamma_{\widehat{x}}(z,y)) = \partial_y S_{\widehat{x}}(\gamma_{\widehat{x}}(z,0)) \partial_y \mathcal{G}_{\widehat{x}}(z,0) = \partial_y S_{\widehat{x}}(z,0) \partial_y \mathcal{G}_{\widehat{x}}(z,0),$$ so that:
$$ \partial_y\mathbf{L}_{\widehat{x}}(z,0) = \int_0^z \big\langle  \partial_y S_{\widehat{x}}(t,0) \partial_y \mathcal{G}_{\widehat{x}}(t,0) \partial_z \gamma_{\widehat{x}}(t,0), \partial_z \gamma_{\widehat{x}}(t,0) \big\rangle dt $$ $$+ \int_0^z \langle S_{\widehat{x}}(\gamma_{\widehat{x}}(t,0)) \partial_z \gamma_{\widehat{x}}(t,0), \partial_y \partial_z \gamma_{\widehat{x}}(t,0) \rangle dt $$
$$ = \int_0^z \big\langle  \partial_y S_{\widehat{x}}(t,0) \mathbf{e}_z + S_{\widehat{x}}(t,0) \partial_y \partial_z \gamma_{\widehat{x}}(t,0), \ \mathbf{e}_z \big\rangle dt, $$
where $\mathbf{e}_z := (1,0)$, and $\partial_y \partial_z \gamma_{\widehat{x}}(t,0) = (0,\partial_y \partial_z \mathcal{G}_{\widehat{x}}(z,0))$. It follows that $\partial_y \mathbf{L}_{\widehat{x}}(\cdot,0)$ is smooth.
\end{proof}

Before getting to our main technical estimate, we need one last preliminary lemma.

\begin{lemma}
For ${\widehat{x}} \in \widehat{\Omega}$ and $z \in \Omega_{\widehat{x}}^u \cap (-\rho,\rho)$, $y \in \Omega_{\widehat{x}}^s \cap (-\rho,\rho)$, define
$ \mathbf{C}_{\widehat{x}}(z,y) := \iota_{\widehat{x}}^{-1}( [\Phi_{\widehat{x}}^u(z),\Phi_{\widehat{x}}^s(y)] ) $. The map $\mathbf{C}_{\widehat{x}}$ is (uniformly in ${\widehat{x}}$) $C^{1+\alpha}$, and moreover, denoting $\mathbf{C}_{\widehat{x}}(z,y) = ( \pi_z \mathbf{C}_{\widehat{x}}(z,y),\pi_y \mathbf{C}_{\widehat{x}}(z,y) )$, we have:
\begin{itemize}
\item $\forall z, \ \pi_y \partial_y \mathbf{C}_{\widehat{x}}(z,0) = 1$
\item There exists $\alpha_{\widehat{x}} \in \mathbb{R}$ (continuous in $x$) such that $$ \pi_z \partial_y \mathbf{C}_{\widehat{x}}(z,0) = \alpha_{\widehat{x}} z + O\big(|z|^{3/2}\big) .$$
\end{itemize}
\end{lemma}

\begin{proof}
The relation $f([p,q]) = [f(p),f(q)]$ yields, in coordinates, the formula
$$ f_{\widehat{x}}( \mathbf{C}_{\widehat{x}}(z,y) ) = \mathbf{C}_{\widehat{f}({\widehat{x}})}(\lambda_x z,\mu_x y). $$
Applying $\frac{d}{dy}_{|y=0}$ then gives the relation $$ (df_{\widehat{x}})_{(z,0)}( \partial_y \mathbf{C}_{\widehat{x}}(z,0) ) = \partial_y \mathbf{C}_{\widehat{f}({\widehat{x}})}(\lambda_x z,0) \mu_x .$$
Keeping only the $y$ coordinates gives
$$ \partial_y \pi_y \mathbf{C}_{\widehat{x}}(z,0) = \partial_y \pi_y \mathbf{C}_{\widehat{f}({\widehat{x}})}(\lambda_x z,0) , $$
and it follows that $$ \partial_y \pi_y \mathbf{C}_{\widehat{x}}(z,0) = \lim_{n \rightarrow \infty} \partial_y \pi_y \mathbf{C}_{\widehat{f}^{-n}({\widehat{x}})}(\lambda_x^{\langle-n\rangle} z,0) = 1 $$
since $\mathbf{C}_{\widehat{x}}(0,y) = (0,y)$. Now, denoting $b_{\widehat{x}}(z) := \pi_z \partial_y f_{\widehat{x}}(z,0)$ (a smooth map in $z$), we see by isolating the $z$ coordinate that
$$ \lambda_x  \partial_y \pi_z \mathbf{C}_{\widehat{x}}(z,0) + b_{\widehat{x}}(z)  \partial_y \pi_y \mathbf{C}_{\widehat{x}}(z,0) = \partial_y \pi_z \mathbf{C}_{\widehat{f}({\widehat{x}})}(\lambda_x z,0) \mu_x , $$
which can be rewritten as
$$ \partial_y \pi_z \mathbf{C}_{\widehat{x}}(z,0) = B_{\widehat{x}}(z) + \mu_x \lambda_x^{-1} \partial_y \pi_z \mathbf{C}_{\widehat{f}({\widehat{x}})}(\lambda_x z,0), $$
where $B_{\widehat{x}}(z) = - \lambda_x^{-1} b_{\widehat{x}}(z)$. The properties of our coordinate system ensure that $B_{\widehat{x}}(0)=0$ and $\partial_y \pi_z \mathbf{C}_{\widehat{x}}(0,0)=0$. It follows that $\widetilde{B}_{\widehat{x}}(z) := z^{-1} B_{\widehat{x}}(z)$ is a smooth map. If we denote further $\widetilde{\mathbf{C}}_{\widehat{x}}(z) := z^{-1}{\partial_y \pi_z \mathbf{C}_{\widehat{x}}(z,0)}$, our previous relation can be rewritten
$$ \widetilde{\mathbf{C}}_{\widehat{x}}(z) = \widetilde{B}_{\widehat{x}}(z) + \mu_x \widetilde{\mathbf{C}}_{\widehat{f}({\widehat{x}})}(\lambda_x z).$$
It would be convenient to be able to see $\widetilde{\mathbf{B}}_{\widehat{x}}(z)$ as an infinite sum involving $\tilde{B}_{\widehat{x}}$, but notice that $\lambda_x^{\langle n \rangle} z$ is eventually larger than $\rho$ and falls away from the domain of definition of our coordinate system. So, for each (small) $z$, define $N_x(z) \geq 1$ such that $\lambda_x^{\langle n \rangle} |z| \simeq \sqrt{|z|}$, ie $\mu_x^{\langle n \rangle} \sim \sqrt{|z|}$ since $f$ is area preserving. We find:
$$ \widetilde{\mathbf{C}}_{\widehat{x}}(z) = \sum_{n=0}^{N_x(z)-1} \mu_x^{\langle n \rangle} \widetilde{B}_{\widehat{f}^n({\widehat{x}})}(\lambda_x^{\langle n \rangle} z) + \mu_x^{\langle N_x(z) \rangle} \widetilde{\mathbf{C}}_{\widehat{f}^{N_x(z)}({\widehat{x}})}(\lambda_x^{\langle N_x(z) \rangle} z) = \sum_{n=0}^{N_x(z)-1} \mu_x^{\langle n \rangle} \widetilde{B}_{\widehat{f}^n({\widehat{x}})}(\lambda_x^{\langle n \rangle} z) + O(\sqrt{z}).  $$
Hence:
$$ \Big| \widetilde{\mathbf{C}}_{\widehat{x}}(z) - \sum_{n=0}^\infty \mu_x^{\langle n \rangle} \widetilde{B}_{\widehat{f}^n({\widehat{x}})}(0) \Big| \leq \sum_{n=0}^{N_x(z)-1} |\mu_x^{\langle n \rangle}| |\widetilde{B}_{\widehat{f}^n({\widehat{x}})}(\lambda_x^{\langle n \rangle} z) - \widetilde{B}_{\widehat{f}^n({\widehat{x}})}(0) | + O(\sqrt{|z|}) = O(\sqrt{|z|}), $$
which gives the desired Taylor expansion.
\end{proof}

\begin{corollary}
Fix $p \in \Omega$ and $s \in W^s_{loc}(p) \cap \Omega$. There exists a smooth (uniformly in $p$) function $\phi_{p}: W^u_{loc}(p) \cap \Omega \rightarrow \mathbb{R}$ such that
$$ d^u(s,[r,s]) = d^u(p,r) + \phi_{p}(r) d^s(p,s) + O\Big( 
d^s(p,s)(d^u(p,r)^{1+\alpha} + d^s(p,s)^{\alpha}) \Big) $$
\end{corollary}

\begin{proof}

Let $p=x \in \Omega$ and $s \in W^s_{loc}(x) \cap \Omega$. Choose some ${\widehat{x}} \in \widehat{\Omega}$ in the fiber of $x$. Denote $y := (\Phi_{\widehat{x}}^s)^{-1}(s)$, $z := \Phi_{\widehat{x}}^u(r)$.
 First of all, we can write, denoting $\mathbf{e}_y = (0,1)$:
$$ d^s(p,s) = \int_0^y \langle S_{\widehat{x}}(0,t) \mathbf{e}_y, \mathbf{e}_y \rangle dt = y \int_0^1  \langle S_{\widehat{x}}(y t,0) \mathbf{e}_y, \mathbf{e}_y \rangle dt = y \langle S_{\widehat{x}}(0,0) \mathbf{e}_y, \mathbf{e}_y \rangle + O(y^2). $$
Recall that $|(\Phi_{\widehat{x}}^s)'(0)| = 1$, so that $\langle S_{\widehat{x}}(0,0) \mathbf{e}_y, \mathbf{e}_y \rangle=1$.
It follows that $y = d^s(p,s) + O(d^s(p,s)^2)$. Similarly, $z = d^u(p,r) + O(d^u(p,r)^2)$. Then, denoting $\mathbf{C}_{\widehat{x}}(z,y) = (\iota_{\widehat{x}})^{-1}( [ \Phi_{\widehat{x}}^u(z), \Phi_{\widehat{x}}^s(y) ] )$, notice that we have:
$$ d^u(s,[r,s]) = \mathbf{L}_{\widehat{x}}(\pi_z \mathbf{C}_{\widehat{x}}(z,y),y).$$
This expression is $C^{1+\alpha}$, and a Taylor expansion in the $y$ variable gives:
$$ \mathbf{L}_{\widehat{x}}(\pi_z \mathbf{C}_{\widehat{x}}(z,y),y) = \mathbf{L}_{\widehat{x}}(\pi_z \mathbf{C}_{\widehat{x}}(z,0),0) + y \cdot \frac{d}{dy}  \mathbf{L}_{\widehat{x}}( \pi_z \mathbf{C}_{\widehat{x}}(z,y),y)_{|y=0} + O(y^{1+\alpha})  $$
$$ =  \mathbf{L}_{\widehat{x}}(z,0) + \big(\partial_z \mathbf{L}_{\widehat{x}}(z,0) \partial_y \pi_z \mathbf{C}_{\widehat{x}}(z,0) + \partial_y \mathbf{L}_{\widehat{x}}(z,0)\big)y + O(y^{1+\alpha})  $$
$$ = \mathbf{L}_{\widehat{x}}(z,0) + \big( \alpha_{\widehat{x}} z \partial_z \mathbf{L}_{\widehat{x}}(z,0) + \partial_y \mathbf{L}_{\widehat{x}}(z,0)\big)y + O( 
z^{1+\alpha} y + y^{1+\alpha}). $$
The expression $\big( \alpha_{\widehat{x}} z \partial_z \mathbf{L}_{\widehat{x}}(z,0) + \partial_y \mathbf{L}_{\widehat{x}}(z,0)\big)$ is smooth in $z$. Using $z=(\Phi_{\widehat{x}}^u)^{-1}(r)$, $y = d^s(p,s) + O(d^s(p,s)^2)$, and $\mathbf{L}_{\widehat{x}}(z,0) = d^u(p,r)$ finally gives us
$$ d^u(s,[r,s]) = d^u(p,r) + \phi_{p}(r) d^s(p,s) + O\Big( 
d^s(p,s)(d^u(p,r)^{1+\alpha} + d^s(p,s)^{\alpha}) \Big) ,$$
for some smooth function $\phi_p : W^u_{loc}(p) \cap \Omega \rightarrow \mathbb{R}$.
\end{proof}

\section{Polynomials on perfect sets}\label{ap:B}

The goal of this Appendix is to prove quantitative \say{equivalence of norms} lemmas for polynomials that we restrict to (uniformly) perfect sets. In particular, we prove Lemma \ref{lem:norms}.

\begin{lemma}\label{lem:lagrange}
Let $d \geq 1$. Let $a_0 \dots a_d \in [-1,1]$ be $(d+1)$ distinct points. Then, for all $P \in \mathbb{R}_d[X]$,
$$   d^{-1} 2^{-d} (\min_{i \neq j} |a_i-a_j|)^d \cdot \|P\|_{\infty,[-1,1]} \leq \max_i |P(a_i)| \leq \|P\|_{\infty,[-1,1]} . $$
\end{lemma}

\begin{proof}
One inequality is obvious. We prove the other one using Lagrange interpolating polynomials. Define
$$ L_i(x) := \frac{\prod_{j \neq i}(x-a_j)}{\prod_{j \neq i}(a_j-a_i)}. $$
From the formula $$\forall x \in \mathbb{R}, \ P(x) = \sum_{i=0}^d P(a_i) L_i(x),$$
we find
$$ \|P\|_{\infty,[-1,1]} \leq \sum_i |P(a_i)| \|L_i\|_{\infty,[-1,1]} \leq \frac{d \ 2^d}{(\min_{i \neq j} |a_i - a_j|)^d} \max_i |P(a_i)|.$$
\end{proof}

\begin{definition}
A metric space $(K,d)$ is $\kappa$-uniformly perfect if
$$ \forall \sigma \in (0,1),  \forall x \in K,  \exists y \in K, \ d(x,y) \in [\kappa \sigma, \sigma] .$$
We say that $K$ is uniformly perfect if $K$ is $\kappa$-uniformly perfect for some $\kappa \in (0,1)$.
\end{definition}

\begin{remark}
Notice that $K$ is uniformly perfect if and only if
$$\exists \kappa\in (0,1), \exists \sigma_0 \in (0,1), \forall \sigma \in (0,\sigma_0),  \forall x \in K,  \exists y \in K, \ d(x,y) \in [\kappa \sigma, \sigma] .$$
Indeed, in this case, $K$ would be $\kappa_0 \sigma_0$-uniformly perfect. Notice also that, if $K$ is uniformly perfect, then for any $x \in K$ and $r>0$, the space $K \cap B(x,r)$ is $\kappa r$-uniformly perfect.
\end{remark}

\begin{lemma}\label{lem:uperfect}
There exists $\kappa>0$ such that, for all ${\widehat{x}} \in \widehat{\Omega}$, the set $\Omega_{\widehat{x}}^u \cap (-\rho,\rho)$ is $\kappa$-uniformly perfect. 
\end{lemma}

\begin{proof}
Let $\rho$ be small enough so that, if $d(x,y) \leq \rho$, then $[x,y]$ and $[y,x]$ makes sense. Since $\Omega$ is a basic set (and since it is not a periodic orbit), it is a perfect set. Moreover, the periodic orbits are dense in $\Omega$. Let us then fix a finite family $(x_i)_{i \in I}$ of periodic points of $\Omega$ such that $\bigcup_{i \in I} B(x_i,\rho/2) = \Omega$. For each $i \in I$, we construct a companion for $x_i$ in the following way: by density of periodic points, there exists $\tilde{y}_i \in B(x_i,\rho/2) \setminus \{x_i\}$ which is periodic. It follows that $\tilde{y}_i \notin W^s_{loc}(x_i)$. We then set $y_i := [\tilde{y_i},x_i] \in W^u_{loc}(x_i) \setminus \{x_i\}$. \\

Let us define $\delta := \min_{i \in I} \inf\{ d^u( [x_i,z],[y_i,z] ) \ | \ z \in B(x_i,\rho)  \}> 0$. For any $z \in \Omega$, there exists $i \in I$ such that $x \in B(x_i,\rho)$. Now, since $d^u( [x_i,z],[y_i,z] ) \geq \delta$, we know that $d^u(z,[x_i,z]) \geq \delta/3$ or $d^u(z,[y_i,z]) \geq \delta/3$. It follows from this discussion the following property:
$$ \exists \kappa_0 \in (0,1), \forall x \in \Omega, \exists z \in W^u_{loc}(x), \ d^u(x,z) \in [\kappa_0 \rho, \rho]. \quad (*) $$
Now, we will use the invariance of the unstable foliation under the dynamics to extend the property $(*)$ to any scale.
Notice that there exists $C_0 \geq 1$ such that, if $x,\tilde{x}$ are in the same unstable leaf with $d^u(x,\tilde{x}) \leq \rho$, then
$$ \sum_{k=0}^\infty |\tau_f( f^{-k} x) - \tau_f( f^{-k} \tilde{x}) | \leq C_0 .$$
Now, let $x \in \Omega$ and let $\sigma \in (0,\rho)$. Let $n(x,\sigma) \geq 1$ be the largest integer such that $\partial_u f^{n(x,\sigma)}(x) \leq \sigma^{-1} \rho e^{-C_0}$. Define $\overline{x} := f^{n(x,\sigma)}(x)$. By the property $(*)$, there exists $\overline{z} \in W^u_{loc}(x)$ such that $d^u(\widehat{x},\widehat{z}) \in [\kappa_0 \rho,\rho]$. Define $z := f^{-n(x,\sigma)}(\overline{z})$. We have, by the mean value theorem, for some $t \in W^u_{loc}(x)$ between $x$ and $z$:
$$ d(x,z) = \partial_u f^{-n(x,\sigma)}(t) d^u(\overline{x},\overline{z}) \in [e^{-2C_0} \sigma \kappa_0 \rho, \sigma] .$$
We have thus proved the following property:
$$ \exists \kappa_1,\sigma_0 \in (0,1), \forall \sigma \in (0,\sigma_0), \forall x \in \Omega, \exists z \in W^u_{loc}(x), \ d^u(x,z) \in [\kappa_1 \sigma, \sigma]. \quad (**) $$
The conclusion follows from the fact that the maps $\Phi_{\widehat{x}}^u : (-\rho,\rho) \rightarrow W^u_{loc}(x)$ are uniformly $C^{1}$ (thus, uniformly Lipschitz) maps. 
\end{proof}

\begin{corollary}
Let $M$ be a complete Riemannian surface and let $f:M \rightarrow M$ be a smooth Axiom A diffeomorphism. Let $\Omega \subset M$ be a basic set. Then $\Omega$ is uniformly perfect.
\end{corollary}

\begin{proof}
Lemma \ref{lem:uperfect} ensure that there exists $\kappa$ such that, for any ${\widehat{x}} \in \widehat{\Omega}$, $\Omega_{\widehat{x}}^u \cap (-\rho,\rho)$ is $\kappa$-uniformly perfect. Symmetrically, (reducing $\kappa$ if necessary,) for any ${\widehat{x}} \in \widehat{\Omega}$, $\Omega_{\widehat{x}}^s \cap (-\rho,\rho)$ is $\kappa$-uniformly perfect. The local product structure of $\Omega$ and the $C^{1}$ regularity of the holonomies yields the result.
\end{proof}

\begin{lemma}\label{lem:cutperfect}
Let $K$ be a $\kappa$-uniformly perfect metric set. For all $k \geq 1$, there exists $\delta_{k} \in (0,1)$ that depends only on $\kappa$ and $k$ such that 
$$ \exists (x_i)_{i \in \llbracket 1, k \rrbracket} \in K^k,  \forall i \neq j, \ B(x_i,\delta_k) \cap B(x_j,\delta_k) = \emptyset. $$
\end{lemma}

\begin{proof}
The proof goes by induction on $k$. For $k=1$, this is trivial. Suppose the lemma true for some $k \geq 1$. There exists a family of points $x_i \in K$, for $i = 1, \dots ,k$, and there exists $\delta_{k} \in (0,1)$ that only depends on $k$ and $\kappa$, such that the balls $ B(x_i,\delta_k) $ are disjoint. Now, since $K$ is $\kappa$-uniformly perfect, there exists a point $x_{k+1} \in K$ such that $d(x_1,x_{k+1}) \in [\kappa \delta_k/3, \delta_k/3]$.
Setting $\delta_{k+1} := \kappa \delta_k/10$ then ensure that, for any $i \neq j \in \llbracket 1,k+1 \rrbracket$, one have
$$ B(x_i,\delta_{k+1}) \cap B(x_j,\delta_{k+1}) = \emptyset. $$
which concludes the proof.
\end{proof}

\begin{lemma}\label{lem:prenorm}
Let $K \subset (-\rho,\rho)$ be a $\kappa$-uniformly perfect set that contains zero (where $\kappa< \rho/100$). Let $d \geq 1$. Then there exists $\tilde{\kappa} \in (0,1)$ (that depends only on $\kappa$ and $d$) such that for all $P \in \mathbb{R}_d[X]$, there exists $z_0 \in K \cap (-\rho/2,\rho/2)$ such that
$$ \forall z \in [z_0-\tilde{\kappa},z_0+\tilde{\kappa}], \ |P(z)| \geq \tilde{\kappa} \|P\|_{\infty,[-1,1]}. $$
\end{lemma}

\begin{proof}
Using Lemma \ref{lem:cutperfect}, we first cut $(-\rho/2,\rho/2)$ into $2d$ subintervals $(I_i)$ such that for all $i$, there exists $x_i \in I_i \cap K$ satisfying $B(x_i,\delta) \subset I_i$. We know that we can choose $\delta$ to depend only on $\kappa$ and $d$.\\

Then, similarly, for each $i \in \llbracket 1, 2d \rrbracket$, we can choose two families $A_i = (a_k^{(i)})_k$ and $B_i = (b_k^{(i)})$ of $d+1$ points lying in $B(x_i,\delta) \cap K$, satisfying:
$$ \min_{j \neq k} |a_j^{(i)} - a_k^{(i)}| \geq \delta' ,\quad \min_{j \neq k} |b_j^{(i)} - b_k^{(i)}| \geq \delta' ,\quad d(A_i,B_i) \geq \delta'  $$
for some $\delta' < \delta$ that depends only on $\kappa$ and $d$. In fact, we can further suppose that for all $i \in I$, there exists a point $c^{(i)} \in K \cap I_i$ such that $$d(c^{(i)},A_i) \geq \delta' , \ d(c^{(i)},B_i) \geq \delta'.$$
The role of each point being symetric, we can further assume by renaming our points that we have $a_j^{(i)} \leq c^{(i)} \leq b_j^{(i)}$ for each $i,j$. Denoting by $\|P\|_{A_i} := \max_k |P(a^{(i)}_k)|$ and $\|P\|_{B_i} := \max_k |P(b^{(i)}_k)|$, Lemma \ref{lem:lagrange} ensures that
$$2^{-d} d^{-1} (\delta')^d \|P\|_{\infty,[-1,1]} \leq \|P\|_{A_i} \leq \|P\|_{\infty,[-1,1]} ,\quad 2^{-d} d^{-1} (\delta')^d \|P\|_{\infty,[-1,1]} \leq \|P\|_{B_i} \leq \|P\|_{\infty,[-1,1]} .$$
Now, let $P \in \mathbb{R}_d[X]$. Since $P'$ has degree $d-1$, it vanish at most $d-1$ times. Moreover, $P$ vanish at most $d$ times. It follows that there exists $i_{0} \in \llbracket 1,2d \rrbracket$ such that $P_{|I_{i_0}}$ is monotonous and doesn't change signs. The inequality
$$ \min( \|P\|_{A_{i_0}}, \|P\|_{B_{i_0}} )\geq \frac{1}{d} \Big(\frac{\delta'}{2}\Big)^d \ \|P\|_{\infty,[0,1]}$$
then ensures that there exists $j_0$ and $k_0$ such that
$ \min (|P(a_{j_0}^{(i_0)})|, |P(b_{k_0}^{(i_0)})| ) \geq \frac{1}{d} \big(\frac{\delta'}{2}\big)^d \|P\|_{\infty,[-1,1]} $.
The monotonicity of $P$ on $I_i$ and the fact that $P(a_{j_0}^{(i_0)})$ and $P(b_{k_0}^{(i_0)})$ have same sign ensure that
$$ \forall z \in [a_{j_0}^{(i_0)},b_{k_0}^{(i_0)}], \ |P(z)| \geq \frac{1}{d} \Big(\frac{\delta'}{2}\Big)^d \|P\|_{\infty,[0,1]} . $$
The fact that $[c^{(i_0)} - \delta',c^{(i_0)} + \delta'] \subset [a_{j_0}^{(i_0)},b_{k_0}^{(i_0)}]$ proves the desired estimate.
\end{proof}

\begin{corollary}
Let $f$ be an Axiom A diffeomorphism on a surface $M$. Denote by $\Omega$ one of its basic sets. There exists $\kappa>0$ such that, for all ${\widehat{x}} \in \widehat{\Omega}$, for all $P \in \mathbb{R}_{d_Z}[X]$, there exists $z_0 \in \Omega_{\widehat{x}}^u \cap (-\rho/2,\rho/2)$ such that
$$ \forall z \in [z_0 - \kappa, z_0 + \kappa], \ |P(z)| \geq \kappa \|P\|_{C^\alpha((-\rho,\rho))}. $$
If we denote $P = \sum_{k=0}^{d_Z} a_k z^k \in \mathbb{R}_{d_Z}[X]$, then
$$\max_{0 \leq k \leq d_Z} |a_k| \leq \kappa^{-1} \|P\|_{L^\infty( (-\rho,\rho) \cap \Omega_{\widehat{x}}^u )}. $$
\end{corollary}

\begin{proof}
This follows easily from Lemma \ref{lem:prenorm}, Lemma \ref{lem:uperfect}, and from the fact that the norms \newline
$\| \cdot \|_{C^\alpha((-\rho,\rho))}$, $\| \cdot \|_{\infty,[-1,1]}$ and $P \mapsto \max_k |a_k|$ are equivalent on $\mathbb{R}_{d_Z}[X]$.
\end{proof}

\begin{remark}
This Corollary holds even if $|\det df| \neq 1$ on $\Omega$. In this setting, the construction of the linearizing coordinates $\Phi_{\widehat{x}}^u$ follows through without difficulty and the same result applies. The only important geometric condition is that $M$ is a surface.
\end{remark}

\section{Regularity of $E^s$ and regularity of $\partial_s \Delta^+$.}\label{ap:C}

The goal of this appendix is to prove Proposition \ref{prop:apC}. In other words, we show that if $E^s \notin C^2$ then there exists $p \in \Omega$ such that $\partial_s \Delta^+_p$ is not $C^1$ along $W^u_{loc}(p) \cap \Omega$. To do so, we will suppose that $\partial_s \Delta^+_p$ is $C^1$ and prove that $E^s \in C^2$. In fact, we will prove that, in this case, $E^s \in C^\infty$. 

\subsection{From $\partial_s \Delta_p^+$ to $\partial_u E^s$.}

The fundamental link between regularity of foliations and regularity of $\Delta_p^+$ is given by the following lemma. We denote $\tau_f(x) := \ln \partial_u f(x)$ the (opposite of the) geometric potential.

\begin{lemma}\label{lem:disthol}
Let $p,q \in \Omega$. Let $\pi_{p,q}:W^u_{loc}(p) \cap \Omega \rightarrow W^u_{loc}(q) \cap \Omega $ be the holonomy along the stable lamination. Then
$$ \ln \partial_u \pi_{p,s}(r) = \sum_{n=0}^\infty \tau_f(f^n r) - \tau_f(f^n [s,r]). $$
\end{lemma}

\begin{proof}
Since $\pi_{p,s} = \pi_{r,[r,s]}$, one can suppose by changing the basepoints that $r=p$. Then, notice that we can write, on a neighborhood of $p$:
$$ \pi_{p,s} = f^{-n} \pi_{f^n(p),f^n(s)} f^{n}. $$
Taking the derivative along the unstable direction at $p$ then yields
$$ \ln \partial_u \pi_{p,s}(p) = \ln (\partial_u f^{-n})( \pi_{f^n(p),f^n(s)} f^{n}(p)) + \ln ((\partial_u \pi_{f^n(p),f^n(s)}) (f^{n}p )) + \ln (\partial_u f^n)(p) $$
$$ = \ln (\partial_u \pi_{f^n(p),f^n(s)}) (f^{n} p) + \sum_{k=0}^{n-1} \tau(f^k(p)) - \tau(f^k(s)) $$
$$ \underset{n \longrightarrow \infty}{\longrightarrow} \sum_{n=0}^\infty \tau_f(f^n p) - \tau_f(f^n s) .$$
\end{proof}

It follows that $$\Delta_p^+(q) = \ln \partial_u \pi_{p,s}(r) - \ln \partial_u \pi_{p,s}(p).$$
Now, intuitively, those stable holonomies are following a flow along a vector field defined by the stable distribution. So the regularity of $\Delta^+_p$ should reflect the regularity of the stable distribution. Our next step is to argue carefully that one can \say{exchange derivatives} to show that, intuitively, $\partial_s \Delta^+_p(r) \sim \partial_s \partial_u \pi_{p,s}(r) \sim \partial_u \partial_s \pi_{p,s}(r) \sim \partial_u E^s $. Which then should imply $C^1$ regularity of $\partial_u E^s$ along the unstable direction. This key gain of regularity will then activate a rigidity phenomenon. \\

To prove this properly, we introduce a good set of coordinates that straighten the unstable lamination (so that the unstable derivative $\partial_u$ in the sense of distortions become easier to study.)

\begin{lemma}
There exists a collection of local charts $\overline{\iota}_{\widehat{x}}:(-\rho,\rho)^2 \rightarrow M$ such that $\overline{\iota}_{\widehat{x}}(z,y) \in W^u(\Phi_x^s(y))$ for all $z \in (-\rho,\rho)$, and such that the following regularity holds:
\begin{itemize}
\item $(z,y) \mapsto \overline{\iota}_{\widehat{x}}(z,y) $ is $ C^{1+\alpha}$
\item $z \mapsto \overline{\iota}_{\widehat{x}}(z,y) $ is smooth.
\item $z \mapsto (d \overline{\iota}_{\widehat{x}})_{(z,y)}$ is smooth.
\end{itemize}
In particular, $(\overline{\iota}_{\widehat{x}}^{-1})_*(E^s)$ is $C^{1+\alpha}$ in $z$.
\end{lemma}

\begin{proof}
Denote by $\vec{e}^s(x) \in T_x S$ the unit vector field spanning $E^s(x)$ in the same local orientation than the one given by $\widehat{x} \in \widehat{\Omega}$. The map $z \mapsto \overline{\iota}_{\widehat{x}}(z,y)$ is then defined by integrating this vector field starting at the initial point $\Phi_{\widehat{x}}^s(y)$. That is:
$$ \frac{d}{dt} \overline{\iota}_{\widehat{x}}(t,y) = \vec{e}^s( \overline{\iota}_{\widehat{x}}(t,y) ) \quad , \quad \overline{\iota}_{\widehat{x}}(0,y) = \Phi_{\widehat{x}}^s(y). $$
Since the vector field that we integrate is $C^{1+\alpha}$, smooth along the unstable direction, and with smooth differential along the unstable direction (see Appendix \ref{ap:A}), general considerations on flows gives that $z \mapsto \overline{\iota}_{\widehat{x}}(z,y)$ satisfies the same regularity properties.
\end{proof}

\begin{lemma}
Assume that $\partial_s \Delta_p^+$ is $C^1$ along the unstable lamination. Then $\partial_u E^s$ is $C^1$ along the unstable lamination.
\end{lemma}

\begin{proof}
Fix some $\widehat{x} \in \widehat{\Omega}$. In the coordinates $\overline{\iota}_{\widehat{x}}$, the laminations takes the form
$$ \overline{W}_{\widehat{x}}^u(0,y) = \{ (z,y) \ | \ z \in (-\rho,\rho) \} \quad , \quad \overline{W}_{\widehat{x}}^s(\overline{z},0) = \{ (\gamma_{{\widehat{x}},\overline{z}}(y),y) \ | \ y \in (-\rho,\rho) \} $$
for some $C^{1+\alpha}$ function $y \in (-\rho,\rho) \mapsto \gamma_{{\widehat{x}},\overline{z}}(y)$ with $\gamma_{{\widehat{x}},\overline{z}}(0)=\overline{z}$. The dependency on $\overline{z}$ is also $C^{1+\alpha}$ since $(\gamma_{{\widehat{x}},\overline{z}}(t),t)$ identifies with the bracket \say{between $\overline{z}$ and $t$} in our coordinates.
For some point $(z,y) = (\gamma_{{\widehat{x}},\overline{z}}(y),y)$, the stable holonomy projecting on $\overline{W}^u_{\widehat{x}}((0,0))$ takes the form $\overline{\pi}_{\widehat{x}}^s(z,y) = (\overline{z},0)$. Denoting by $\overline{e}^s_{\widehat{x}}(z,y) := (\gamma_{{\widehat{x}},\overline{z}}'(y),1)$ (where $(z,y) = (\gamma_{{\widehat{x}},\overline{z}}(y),y)$) the natural vector field tangent to the stable foliation in this parametrization, our hypothesis on $\Delta_p^+$ implies that
$$ z \mapsto d(\ln |\partial_z \overline{\pi}_{\widehat{x}}^s|)_{(z,0)}(\overline{e}^s_{\widehat{x}}(z,0)) \in C^1 .$$
Notice that, since $\overline{\pi}_{{\widehat{x}}}^s(z,0)=(z,0)$, we have $|\partial_z \overline{\pi}_{\widehat{x}}^s(z,0)|=1$. It follows that $$ d(\ln |\partial_z \overline{\pi}_{\widehat{x}}^s|)_{(z,0)}(\overline{e}^s_{\widehat{x}}(z,0)) = d( |\partial_z \overline{\pi}_{\widehat{x}}^s|)_{(z,0)}(\overline{e}^s_{\widehat{x}}(z,0)). $$
The fact that $\overline{\pi}_{\widehat{x}}^s$ sends horizontal lines to horizontal lines then ensure that this quantity have the same regularity than $ d( \partial_z \overline{\pi}_{\widehat{x}}^s)_{(z,0)}(\overline{e}^s_{\widehat{x}}(z,0))$ (since those two quantities differ by a $C^{1}$ factor).\\

Let us then compute $d( \partial_z \overline{\pi}_{\widehat{x}}^s)_{(z,0)}(\overline{e}^s_{\widehat{x}}(z,0))$ in terms of $ \partial_z \overline{e}^s_{\widehat{x}}(z,0). $
First of all, we have (denoting  $\overline{\pi}_{\widehat{x}}^s(z,y) =: (\overline{z},0)$, so that $\overline{z}$ is a $C^{1+\alpha}$ function of $z,y$):
$$ (z,y) = \overline{\pi}_{\widehat{x}}^s(z,y) + \int_{0}^y \overline{e}^s_{\widehat{x}}(\gamma_{{\widehat{x}},\overline{z}}(t),t) dt .$$
Hence
$$ (1,0) = \partial_z \overline{\pi}_{\widehat{x}}^s(z,y) + \int_0^y (d \overline{e}_{\widehat{x}}^s)_{(\gamma_{{\widehat{x}},\overline{z}}(t),t)}\Big(  \frac{d}{dz} \gamma_{{\widehat{x}},\overline{z}}(t),0 \Big) dt $$
Then, following the path $y \mapsto (\gamma_{{\widehat{x}},\overline{z}}(y),y)$ and taking the derivative at $y=0$, we find:
$$ d(\partial_z \overline{\pi}_{\widehat{x}}^s)_{(\overline{z},0)}(\overline{e}^s_{\widehat{x}}(\overline{z},0)) = - (d \overline{e}_{\widehat{x}}^s)_{(\gamma_{{\widehat{x}},\overline{z}}(0),0)}\Big( \frac{d}{dz} \gamma_{{\widehat{x}},\overline{z}}(0),0 \Big) $$
$$ = - (d \overline{e}_{\widehat{x}}^s)_{(\overline{z},0)}( 1 ,0) = - \partial_z \overline{e}_{\widehat{x}}^s(\overline{z},0). $$
Which shows that $\partial_z \overline{e}_{\widehat{x}}^s$ is $C^{1}$, as announced.
\end{proof}

This conclude the proof that $\partial_u E^s \in C^1$. Notice that this is not enough to say that $E^u \in C^2$ yet.
Indeed, the objects we are considering are only defined on Cantor sets. One should keep in mind the example of the function $\varphi:K \rightarrow \mathbb{R}$ defined on some Cantor set $K \subset \mathbb{R}$:

$$ \varphi(x) := \int_{-\infty}^x d(t,K)^{\alpha} dt. $$
This function is $C^{1+}$ on $K$ but is not $C^2$ on $K$ (as it is not the restriction to $K$ of a $C^2$ function on $\mathbb{R}$). Still, we find $\varphi'=0$ on $K$, which is $C^\infty$ on $K$.\\

Our next goal is to show that the invariance of $E^s$ under the dynamics actually implies that $E^s$ is $C^2$. We will first show that $\partial_u E^s \in C^1$ implies the vanishing of a key dynamical quantity (the Anosov Cocycle), which is the main obstruction to $C^2$ regularity of the foliations.

\subsection{Vanishing of the Anosov cocycle.}

\begin{definition}
Let us denote, abusing slightly notations, $$(df_{\widehat{x}})_{(z,y)} =   \left( \begin{array}{lrcl}
  a_{\widehat{x}}(z,y) & b_{\widehat{x}}(z,y) \\
     c_{\widehat{x}}(z,y) & d_{\widehat{x}}(z,y)   \end{array} \right) $$
where $f_{\widehat{x}}$ is the dynamics seen in the almost linearizing coordinates $\iota_{\widehat{x}}$ introduced in Lemma \ref{lem:coordinates}. These functions are smooth. Moreover, we know that we have:
$$ a_{\widehat{x}}(z,0) = \lambda_x \quad , c_{\widehat{x}}(z,0)=0, \quad d_{\widehat{x}}(z,0)= \mu_x. $$
Define also $B_{\widehat{x}}(z) := b_{\widehat{x}}(z,0) e^{h(x)}$
where $\ln(\mu_x \lambda_x) = h\circ f(x) - h(x)$.
\end{definition}

\begin{lemma}
Suppose that $\partial_u E^s \in C^1$. Then we have vanishing of a cohomology class: 
$$ B_{\widehat{x}}''(0) \sim 0. $$
\end{lemma}

\begin{proof}
Let $\theta_x^s(z,y) \in \mathbb{R}$ be the unique real number such that $ \theta_x^s(z,y) \partial_z + \partial_y$ spans the stable direction in the coordinates $\iota_x$. The relation $f_*(E^s) = E^s$ yields, in coordinates:
$$ \theta_{\widehat{f}({\widehat{x}})}^s(f_{\widehat{x}}(z,y)) = \frac{a_{\widehat{x}}(z,y) \theta_{\widehat{x}}^s(z,y) + b_{\widehat{x}}(z,y)}{b_{\widehat{x}}(z,y) \theta_{\widehat{x}}^s(z,y) + d_{\widehat{x}}(z,y)} $$
Letting $y=0$ gives
$$ \theta_{\widehat{f}({\widehat{x}})}^s(\lambda_x z,0) = \frac{\lambda_x \theta_{\widehat{x}}^s(z,0) + b_{\widehat{x}}(z,0)}{\mu_x} .$$
Since $E^s \in C^{1+}$, we can take a derivative:
$$ (\partial_z\theta_{\widehat{f}({\widehat{x}})}^s)(\lambda_x z,0) = \frac{ (\partial_z \theta_{\widehat{x}}^s)(z,0)}{\mu_x} + \frac{\partial_z b_{\widehat{x}}(z,0)}{\mu_x}.  $$
Hence, for all nonzero $z \in \Omega_{\widehat{x}}^u$:
$$ \frac{(\partial_z \theta_{\widehat{f}({\widehat{x}})}^s)(\lambda_x z,0) - (\partial_z \theta_{\widehat{f}({\widehat{x}})}^s)( 0,0)}{z}  = \frac{ (\partial_z \theta_{\widehat{x}}^s)(z,0) - (\partial_z \theta_{\widehat{x}}^s)(0,0)}{z \mu_x} + \frac{\partial_z b_{\widehat{x}}(z,0) - \partial_z b_{\widehat{x}}(0,0)}{z \mu_x}. $$
The regularity hypothesis ensure that the limit $$ \Theta({\widehat{x}}) := \underset{z \in \Omega_{\widehat{x}}^u}{\underset{ z \rightarrow 0}{\lim}} \frac{(\partial_z \theta_{\widehat{x}}^s)(z,0) - (\partial_z \theta_{\widehat{x}}^s)(0,0)}{z} $$
exists. Taking the limit then yields
$$ \Theta(\widehat{f}({\widehat{x}})) \lambda_x = \frac{\Theta({\widehat{x}})}{\mu_x} + \frac{\partial_z^2 b_{\widehat{x}}(0,0)}{\mu_x} $$
That is, recalling that $\lambda_x \mu_x = e^{h(\widehat{f}({\widehat{x}}))-h(x)}$:
$$ \Theta(\widehat{f}({\widehat{x}})) e^{h(f(x))} - \Theta({\widehat{x}}) e^{h(x)} = B_{\widehat{x}}''(0). $$
\end{proof}

\subsection{$C^2$ regularity of the stable distribution}

\begin{lemma}
    Under the hypothesis $B_{\widehat{x}}''(0) \sim 0$, the stable distribution $E^s$ is $C^2$ along unstable laminations, on a set $\tilde{\Omega} \subset \Omega$ of full $\mu_x^u$ measure, in the sense that:
    $$ \forall x \in \Omega, \ \mu_x^u( W^u_{loc}(x) \setminus \tilde{\Omega})  ) = 0. $$
\end{lemma}

\begin{proof}
This will be a consequence of a standard procedure in hyperbolic dynamics: we will construct a sequence of $C^2$ distribution that $C^1$-approximate $E^s$ and show that the limit actually holds in the $C^2$ sense on some subset of $\Omega$. \\

Recall that $E^s \in C^{1+\alpha}$. Let us then choose some disctribution $E^s_{\text{false}} \in C^2$ that $C^{1+}$-approximate $E^s$. We then introduce a first sequence of distributions $ E^s_{false,n} := (f_*)^{-n}( E^s_{false}) $. Let us study the $C^2$-behavior of these distributions. \\

Let us define $\theta_{{\widehat{x}},n}^s(z,y)$ as the only real number such that $ \text{Span} (d \iota_{\widehat{x}})_{(z,y)}( \theta_{{\widehat{x}},n}^s(z,y) \partial_z + \partial_y ) = E^s_{false,n}(\iota_{\widehat{x}}(z,y)).$
The invariance relation $f_* E_{false,n}^s = E_{false,n-1}^s$
yields, in coordinates:
$$ \theta_{\widehat{f}({\widehat{x}}),n-1}^s(\lambda_x z,0) = \frac{\lambda_x \theta_{{\widehat{x}},n}^s(z,0) }{\mu_x} + \frac{b_{\widehat{x}}(z)}{\mu_x} $$
Notice that from these equations we recover the well known fact that $E^s_{false,n} \rightarrow E^s$ in the $C^1$ topology. Indeed, $\theta_{\widehat{x}}^s$ satisfies the same invariance relations, and taking the difference of the two equations yields:
$$ \frac{\mu_x}{\lambda_x}( \theta_{\widehat{f}({\widehat{x}}),n-1}^s(\lambda_x z,0) - \theta_{\widehat{f}({\widehat{x}})}^s(\lambda_x z,0)) =  \theta_{{\widehat{x}},n}^s(z,0) - \theta_{{\widehat{x}}}^s(z,0) ,$$
which gives
$$ |\theta_{{\widehat{x}},n}^s(0,0) - \theta_{{\widehat{x}}}^s(0,0)| \lesssim (\mu_x^{\langle n \rangle})^2 \quad, \quad |\partial_z \theta_{{\widehat{x}},n}^s(0,0) - \partial_z \theta_{{\widehat{x}}}^s(0,0)| \lesssim (\mu_x^{\langle n \rangle}). $$
In fact, one can adapt the argument to show $C^{2-}$ convergence along the unstable lamination. But this is not enough yet for $C^2$ convergence. Taking two times the $z$-derivative and letting $z=0$ gives us the relation:
$$ \partial_z^2\theta_{\widehat{f}({\widehat{x}}),n-1}^s( 0,0) \lambda_x \mu_x =  \partial_z^2 \theta_{{\widehat{x}},n}^s(0,0)  + b_{\widehat{x}}''(0) ,$$
which can be rewritten
$$ \partial_z^2\theta_{\widehat{f}({\widehat{x}}),n-1}^s( 0,0) e^{h(f(x))} - \partial_z^2\theta_{{\widehat{x}},n}^s(0,0) e^{h(x)} = B_{\widehat{x}}''(0). $$
Now since $B_{\widehat{x}}''(0) \sim 0$, there exists some Hölder regular map $\Theta:\Omega\rightarrow \mathbb{R}$ such that $B_{\widehat{x}}''(0) = \Theta(\widehat{f}({\widehat{x}})) - \Theta(\widehat{x})$. We thus find
$$ \partial_z^2\theta_{\widehat{f}({\widehat{x}}),n-1}^s( 0,0) e^{h(f(x))} - \partial_z^2\theta_{{\widehat{x}},n}^s(0,0) e^{h(x)} = \Theta(\widehat{f}({\widehat{x}})) - \Theta({\widehat{x}}), $$
from which we find:
$$  \partial_z^2 \theta_{{\widehat{x}},n}(0,0) e^{h(x)} = - \partial_z^2 \theta_{\widehat{f}^n({\widehat{x}}),0}(0,0) e^{h(f^n(x))} - \Theta(\widehat{f}^n(\widehat{x}))) + \Theta(\widehat{x}). $$
We see from this expression that $E^s_{fake,n}$ is thus a $C^2$-bounded sequence of distributions. \\

To improve the degree of convergence, we will consider a Cesaro average of this sequence. Of course, one can not sum distributions, so we need to be precise on what we mean. \\

As before, when we are given a distribution $X$ on $\Omega$ which approximate $E^s$, one can look at it along some $W^u_{loc}(x)$ using the coordinates $\iota_{\widehat{x}}$. In these coordinates, there exists a unique $\theta_{\widehat{x}}^X(z,0)$ such that $$ \text{Span} (d \iota_{\widehat{x}})_{(z,0)}( \theta_{\widehat{x}}^X(z,0) \partial_z + \partial_y ) = X. $$ 
Let us now choose another basepoint $x_1 \in W^u_{loc}(x)$. We then have the following change of coordinate formula (denoting $\overline{z} = (\Phi_{\widehat{x_1}}^u)^{-1} \circ \Phi_{\widehat{x}}^u(z)$):
$$ \text{Span} \ d(\iota_{\widehat{x_1}}^{-1} \circ \iota_{\widehat{x}})_{(z,0)}( \theta_{\widehat{x}}^X(z,0) \partial_z + \partial_y ) = \text{Span} \ (\theta_{\widehat{x_1}}^{X}(\overline{z},0) \partial_z + \partial_y) .$$
Notice that $\iota_{\widehat{x_1}}^{-1} \circ \iota_{\widehat{x}}(z,0) = ( (\Phi_{\widehat{x_1}}^u)^{-1} \circ \Phi_{\widehat{x}}^u(z) , 0 )$, which is affine by Lemma \ref{lem:aff}, so that the differential of this map is upper triangular on $(z,0)$. More precisely, we have a form:
$$ d(\iota_{\widehat{x_1}}^{-1} \circ \iota_{\widehat{x}})_{(z,0)} =   \left( \begin{array}{lrcl}
  \alpha_{{\widehat{x}},\widehat{x_1}} & \beta_{{\widehat{x}},\widehat{x_1}}(z) \\
     0 & \gamma_{\widehat{x},\widehat{x_1}}(z)   \end{array} \right). $$
And so finally we find that the change of coordinate relation is \emph{affine} along unstable curves:
$$ \alpha_{{\widehat{x}},\widehat{x_1}} \theta_{\widehat{x}}^X(z,0) + \beta_{{\widehat{x}},\widehat{x_1}}(z) =  \gamma_{\widehat{x},\widehat{x_1}}(z)    \theta_{\widehat{x_1}}^X(\overline{z},0)  .$$ 
From these computations it follows that the data of some functions $(\theta_{\widehat{x}}^X(\cdot,0))_{{\widehat{x}} \in \Omega}$ satisfying those change of coordinates relations is equivalent to the data of a distribution $X$ on $\Omega$. Moreover, the affine structure allows us to take averages of distributions in a meaningful way.
\\

This remark allows us to define rigorously what we mean by the Cesaro average of the sequence of distributions $E^s_{fake,n}$. Define $\widehat{E}^s_{fake,n}$ as the only continuous distribution on $\Omega$ such that, in coordinates, we have
$$ \theta_x^{\widehat{E}^s_{fake,n}}(z,0) = \frac{1}{n} \sum_{k=0}^{n-1} \theta_{x,n}^{s}(z,0). $$
This sequence of distribution is $C^2$ along the unstable direction. Moreover, we know that $\widehat{E}_{fake,n}$ converge to $E^s$ in the $C^{2-}$ topology and that it is a $C^2$-bounded sequence. But notice that, by construction:

$$ \partial_z^2\theta_{\widehat{x}}^{\widehat{E}^s_{fake,n}}(0,0) e^{h(x)} = \frac{1}{n} \sum_{k=0}^{n-1} \Big(- \partial_z^2 \theta_{\widehat{f}^n(\widehat{x}),0}(0,0) e^{h(f^n(\widehat{x}))} - \Theta(\widehat{f}^n(\widehat{x})) \Big) + \Theta(\widehat{x}), $$
which converge for $\mu$-almost every point $x$ by Birkhoff ergodic theorem to the Hölder-regular function
$$ \widehat{x} \mapsto \Theta(\widehat{x}) - \int_{\widehat{\Omega}} \partial_z^2 \theta_{x',0}(0,0) + \Theta(x') d\tilde{\mu}(x')$$ (where $\tilde{\mu}$ is a lift on $\widehat{\Omega}$ of $\mu$.) Fixing some local orientations, since we are considering Birkhoff sums under forward orbits, the set of points $x$ for which this converges contains stable manifolds. It follows that the sequence of functions $\varphi_n : x \in \Omega \mapsto \partial_z^2\theta_{x}^{\widehat{E}^s_{fake,n}}(0,0)$ converge pointwise on a set $\tilde{\Omega} \subset \Omega$ satisfying:
$$ \forall x \in \Omega, \ \mu_x^u( W^u_{loc}(x) \setminus \tilde{\Omega} ) = 0. $$
Egorov theorem then ensure that pointwise convergence in $\tilde{\Omega}$ of $\varphi_n$ to some limiting function $\varphi$ implies almost uniform convergence. Eventually reducing slightly $\tilde{\Omega}$, this shows that the sequence of distributions $\widehat{E}_{fake,n}^s$ 
converge in the $C^2$ topology along unstable laminations on $\tilde{\Omega}$. This proves that $E^s$ is $C^2$ along the unstable lamination on $\tilde{\Omega}$. \end{proof}

\subsection{From $C^2$ to $C^\infty$ regularity of the stable distribution}

The $C^2$ regularity of $E^s$ is enough to kickstart a smooth rigidity phenomenon, which we are about to describe now.

\begin{lemma}
The unstable distribution is smooth.
\end{lemma}

\begin{proof}
Let $x \in \tilde{\Omega}$, that is, $x$ is such that $E^s$ is $C^2$ along $W^u_{loc}(x)$ at $x$. Since $(df)(E^s(x))=E^s(f(x))$, $E^s$ is also $C^2$ along $W^u_{loc}(f^n(x))$ at $f^n(x)$ for all $n \in \mathbb{Z}$. The functions $z \in (-\rho,\rho) \cap \Omega_{\widehat{x}}^u \longmapsto \theta_{\widehat{f}^n(\widehat{x})}^u(z,0)$ are then $C^2$ at zero (uniformly in ${\widehat{x}}$). Let us then consider the maps:

$$ \delta^2 \theta_{\widehat{x}}^s(z,0) := \frac{\theta_{\widehat{x}}^s(z,0)-\theta_{\widehat{x}}^s(0,0)- \partial_z\theta_{\widehat{x}}^s(0,0)z - \partial_z^2 \theta_{\widehat{x}}^s(0,0)z^2/2}{z^2}. $$
These functions are continuous and vanish at $z=0$ (and this, uniformly in ${\widehat{x}}$). The invariance relation
$$ \theta_{\widehat{f}({\widehat{x}})}^s(\lambda_x z,0) = \frac{\lambda_x}{\mu_x} \theta_{{\widehat{x}}}^s(z,0) + \frac{b_{\widehat{x}}(z,0)}{\mu_x}$$
gives, when applying $\delta^2$:
$$ (\delta^2 \theta_{\widehat{f}({\widehat{x}})}^s)(\lambda_x z,0) = \frac{(\delta^2 \theta_{{\widehat{x}}}^s)(z,0)}{\lambda_x \mu_x} + \frac{\delta^2 b_{\widehat{x}}(z,0)}{\mu_x \lambda_x} ,$$
where $$\delta^2 b_{\widehat{x}}(z,0) = \frac{b_{\widehat{x}}(z,0)-b_{\widehat{x}}(0,0)- \partial_z b_{\widehat{x}}(0,0)z - \partial_z^2 b_{\widehat{x}}(0,0)z^2/2}{z^2},$$
which is smooth on $(-\rho,\rho)$ and vanish at zero. In particular, we have $\delta^2 b_{\widehat{x}}(z,0) \leq C|z|$. Now, since $\delta^2 \theta_{\widehat{x}}^s(z,0)$ vanish at $z=0$ and is continuous at $z=0$, we can iterate and sum this relation backward to find:
$$ \delta^2 \theta_{\widehat{x}}^s(z,0) = \sum_{n=1}^\infty \frac{\delta^2 b_{\widehat{f}^{-n}({\widehat{x}})}( \lambda_x^{\langle -n \rangle} z,0)}{\lambda_x^{\langle n \rangle} \mu_x^{\langle n \rangle}} .$$
This expression is smooth in $z$ on $(-\rho,\rho)$. It follows that
$$ \theta_{\widehat{x}}^s(z,0) = \theta_{\widehat{x}}^s(0,0) + z \partial_z\theta_{\widehat{x}}^s(0,0) + \frac{z^2}{2} \partial_z^2\theta_{\widehat{x}}^s(0,0) + z^2 \delta^2 \theta_{\widehat{x}}^s(z,0) $$
is the restriction to $\Omega_{\widehat{x}}^u$ of a smooth function on $(-\rho,\rho)$. 
Since the set of points $\tilde{\Omega}$ on which this applies have full $\mu_x^u$ measure, this implies that $E^s$ is smooth on all $\Omega$ along the unstable direction. Journé's lemma \cite{NT07,Jo88} then implies that $E^s$ is smooth.
\end{proof}

\begin{remark}
The same argument shows analytic rigidity of $E^s$ if $f$ is analytic.
\end{remark}


\begin{thebibliography}{9}
    






\bibitem[An67]{An67}
    {D. V. Anosov}
    \emph{Geodesic flows on closed Riemann manifolds with negative curvature.} Proceedings of the Steklov Institute of Mathematics, No. 90 (1967). Translated from the Russian by S. Feder American Mathematical Society, Providence, R.I. 1969 iv+235 pp.
    \href{hhttps://www.amazon.com/-/eng/manifolds-curvature-Proceedings-Institute-Mathematics/dp/0821818902}{ISBN: 978-0821818909}



\bibitem[BD17]{BD17}
     {J. Bourgain, S. Dyatlov}, 
     \emph{Fourier dimension and spectral gaps for hyperbolic surfaces},
     Geom. Funct. Anal. 27, 744–771 (2017),  \href{https://arxiv.org/abs/1704.02909}{arXiv:1704.02909} 


  
\bibitem[Bl96]{Bl96}
    {C. Bluhm},
    \emph{Random recursive construction of Salem sets}
    Ark. Mat. 34 (1996), 51-63.

\bibitem[Bo75]{Bo75}
    {R. Bowen},
    \emph{Equilibrium States and the Ergodic Theory of Anosov Diffeomorphisms}
    (New edition of) Lect. Notes in Math. 470, Springer, 1975. ISBN : 978-3-540-77605-5


\bibitem[Bo79]{Bo79}
    {R. Bowen},
    \emph{Hausdorff dimension of quasi-circles.}
    Publications Mathématiques de L’Institut des Hautes Scientifiques 50, 11–25 (1979). \href{https://doi.org/10.1007/BF02684767}{doi:10.1007/BF02684767}



\bibitem[Bo10]{Bo10}
    {J. Bourgain},
    \emph{The discretized sum-product and projection theorems} 
    JAMA 112, 193-236 (2010)  \href{https://doi.org/10.1007/s11854-010-0028-x}{doi:10.1007/s11854-010-0028-x}


\bibitem[BS02]{BS02}
    {M. Brin, G. Stuck},
    \emph{Introduction to Dynamical Systems.} 
     2002. Cambridge: Cambridge University Press. \href{https://doi.org/10.1017/CBO9780511755316}{doi:10.1017/CBO9780511755316}.

\bibitem[BS23]{BS23}
    {S. Baker, T. Sahlsten,}
    \emph{Spectral gaps and Fourier dimension for self-conformal sets with overlaps}
    \href{https://arxiv.org/abs/2306.01389}{arXiv:2306.01389}


\bibitem[Ch98]{Ch98}
    {N. I. Chernov}
    \emph{Markov Approximations and Decay of Correlations for Anosov Flows.}
    Annals of Mathematics, vol. 147, no. 2, 1998, pp. 269–324. JSTOR, \href{https://doi.org/10.2307/121010}{doi:10.2307/121010}


\bibitem[Cl20]{Cl20}
    {V. Climenhaga}
    \emph{SRB and equilibrium measures via dimension theory}
    to appear in A Vision for Dynamics in the 21st Century: The Legacy of Anatole Katok (2023), Cambridge. 
    \href{https://arxiv.org/abs/2009.09260}{arXiv:2009.09260}


\bibitem[DF22]{DF22}
    {D. Damanik, J. Fillman.}
    \emph{One-Dimensional Ergodic Schrödinger Operators: I. General Theory}
     Graduate Studies in Mathematics. Vol. 221, 2022. \href{https://www.ams.org/publications/authors/books/postpub/gsm-221}{ISBN: 978-1-4704-5606-1}
    \href{https://doi.org/10.1090/gsm/221}{doi:10.1090/gsm/221}



\bibitem[DF24]{DF24}
    {D. Damanik, J. Fillman.}
    \emph{One-Dimensional Ergodic Schrödinger Operators: II. Specific Classes}
     Graduate Studies in Mathematics. vol. 249, 2024. \href{https://bookstore.ams.org/GSM/249}{ISBN: 978-1-4704-6503-2}

\bibitem[DFY25]{DFY25}
    {D. Damanik, J. Fillman, G. Young}
    \emph{Optimal dispersion for discrete periodic Schrödinger operators}
    \href{https://arxiv.org/pdf/2505.14475}{arXiv:2505.14475}


\bibitem[DG09]{DG09}
    {D. Damanik, A. Gorodetski}
    \emph{Hyperbolicity of the trace map for the weakly coupled Fibonacci Hamiltonian}, 2009 Nonlinearity 22 123
    \href{https://iopscience.iop.org/article/10.1088/0951-7715/22/1/007}{doi:10.1088/0951-7715/22/1/007}, \href{https://arxiv.org/abs/0806.0645}{arXiv:0806.0645}

\bibitem[DG12]{DG12}
    {D. Damanik, A. Gorodetski}
    \emph{The Density of States Measure of the Weakly Coupled Fibonacci Hamiltonian.}
    Geom. Funct. Anal. 22, 976–989 (2012). \href{https://doi.org/10.1007/s00039-012-0173-8}{doi:10.1007/s00039-012-0173-8}



\bibitem[DGY16]{DGY16}
    {D. Damanik, A. Gorodetski, W. Yessen},
    \emph{The Fibonacci Hamiltonian.} \\
    Invent. math. 206, 629–692 (2016). \href{https://doi.org/10.1007/s00222-016-0660-x}{doi:10.1007/s00222-016-0660-x}

 
\bibitem[Do98]{Do98}
    {D. Dolgopyat},
    \emph{On decay of correlations in Anosov flows}, \\
    Ann. of Math. (2) 147 (2) (1998) 357–390.


\bibitem[Do00]{Do00}
    {D. Dolgopyat},
    \emph{Prevalence of rapid-mixing-II: topological prevalence.}
    Ergodic Theory and Dynamical Systems. 2000;20(4):1045-1059. \href{https://doi.org/10.1017/S0143385700000572}{doi:10.1017/S0143385700000572}


\bibitem[Do02]{Do02}
    {D. Dolgopyat},
    \emph{On mixing properties of compact group extensions of hyperbolic systems.}
    Isr. J. Math. 130, 157–205 (2002). \href{https://doi.org/10.1007/BF02764076}{doi:10.1007/BF02764076}.






\bibitem[FH23]{FH23}
    {R. Fraser, K. Hambrook },
    \emph{Explicit Salem sets in $\mathbb{R}^d$},
    Advances in Mathematics, Volume 416, (2023)
    \href{https://doi.org/10.1016/j.aim.2023.108901}{doi:10.1016/j.aim.2023.108901}, \href{https://arxiv.org/abs/1909.04581}{arXiv:1909.04581}.


\bibitem[Fro35]{Fro35}
    {O. Frostman}
    \emph{Potentiel d'équilibre et capacité des ensembles avec quelques applications à la théorie des functions},
    PhD thesis.

\bibitem[GM21]{GM21}
    {J. Griffin, J. Marklof}
    \emph{ Quantum Transport in a Crystal with Short-Range Interactions: The Boltzmann–Grad Limit.} J Stat Phys 184, 16 (2021). \href{https://doi.org/10.1007/s10955-021-02797-z}{doi:10.1007/s10955-021-02797-z}


\bibitem[Gr09]{Gr09}
    {B. Green}, 
    \emph{Sum-product phenomena in $\mathbb{F}_p$: a brief introduction} \\
    Notes written from a Cambridge course on Additive Combinatorics, 2009.  \href{https://arxiv.org/pdf/0904.2075.pdf}{arXiv:0904.2075}



\bibitem[Ha89]{Ha89}
    {B. Hasselblatt},
    \emph{Regularity of the Anosov splitting and a new description of the Margulis measure},
    PhD thesis.

\bibitem[Ha17]{Ha17}
    {K. Hambrook},
    \emph{Explicit Salem sets in $\mathbb{R}^2$}
    (2017) Advances in Mathematics 311(3):634-648
    \href{https://doi.org/10.1016/j.aim.2017.03.009}{doi:10.1016/j.aim.2017.03.009}


\bibitem[HdS22]{HdS22}
    {W. He, N. de Saxcé},
    \emph{Linear random walks on the torus}
    Duke Math. J. 171 (5) 1061 - 1133, 1 April 2022. \href{https://doi.org/10.1215/00127094-2021-0045}{doi:10.1215/00127094-2021-0045}

\bibitem[HK90]{HK90}
    {S. Hurder, A. Katok},
    \emph{Differentiability, rigidity and Godbillon-Vey classes for Anosov flows}
    Publications Mathématiques de l’Institut des Hautes Scientifiques 72, 5–61 (1990). \href{https://doi.org/10.1007/BF02699130}{doi: 10.1007/BF02699130}

\bibitem[Jo88]{Jo88}
    {J-L. Journé}, 
    \emph{A Regularity Lemma for Functions of Several Variables} 
    Rev. Mat. Iberoam. 4 (1988), no. 2, pp. 187–193 \href{https://doi.org/10.4171/rmi/69}{doi:10.4171/RMI/69}

\bibitem[JS16]{JS16}
    {T. Jordan, T. Sahlsten}
    \emph{ Fourier transforms of Gibbs measures for the Gauss map } \\
    Math. Ann., 364(3-4), 983-1023, 2016. 
    \href{ https://arxiv.org/abs/1312.3619 }{arXiv:1312.3619}


\bibitem[Ka80]{Ka80}
    {R. Kaufman},
    \emph{ Continued fractions and Fourier transforms.}
    (1980). Mathematika, 27(2), 262-267. \href{https://www.cambridge.org/core/journals/mathematika/article/abs/continued-fractions-and-fourier-transforms/2BF3B8A3312D7A5BD4CF7B3CC12C7C37}{doi:10.1112/S0025579300010147}


\bibitem[Ka81]{Ka81}
    {R. Kaufman},
    \emph{On the theorem of Jarník and Besicovitch.}
    Acta Arithmetica 39.3 (1981): 265-267. \href{http://eudml.org/doc/205767}{eudml.org/doc/205767}.



\bibitem[Ka66a]{Ka66a}
    {J.-P. Kahane},
    \emph{Images browniennes des ensembles parfaits}
    C. R. Acad. Sci. Paris Sér. A-B, 263:A613–A615, 1966.


\bibitem[Ka66b]{Ka66b}
    {J.-P. Kahane},
    \emph{Images d’ensembles parfaits par des séries de Fourier gaussiennes.}
    C. R. Acad. Sci. Paris Sér. A-B, 263:A678–A681, 1966.



\bibitem[Ka85]{Ka85}
    {J.-P. Kahane},
    \emph{Some Random series of functions}, 2nd edition (Cambridge University Press,1985).

\bibitem[KPST25]{KPST25}
    {J. Kerner, O. Post, M. Sabri, M. Täufer }
    \emph{The Curious Spectra and Dynamics of Non-locally Finite Crystals.}
    Commun. Math. Phys. 406, 169 (2025). \href{https://doi.org/10.1007/s00220-025-05346-x}{doi:10.1007/s00220-025-05346-x}

\bibitem[KH95]{KH95}
    {A. Katok, B. Hasselblatt},
    \emph{Introduction to the Modern Theory of Dynamical Systems.}
    Encyclopedia of Mathematics and its Applications. 1995. Cambridge: Cambridge University Press. \href{https://doi.org/10.1017/CBO9780511809187}{doi:10.1017/CBO9780511809187}

\bibitem[Kh23]{Kh23}
    {O. Khalil},
    \emph{Exponential Mixing Via Additive Combinatorics}
    Preprint, \href{https://arxiv.org/abs/2305.00527}{arXiv:2305.00527}


\bibitem[Ki90]{Ki90}
    {Y. Kiffer},
    \emph{Large deviations in dynamical systems and stochastic processes},
    Transactions of the American Mathematical Society, 321(1990),505-524.


\bibitem[KK07]{KK07}
    {B. Kalinin, A. Katok,}
    \emph{Measure rigidity beyond uniform hyperbolicity: invariant measures for cartan actions on tori}, Journal of Modern Dynamics, 2007, 1(1): 123-146.   \href{https://www.aimsciences.org/article/doi/10.3934/jmd.2007.1.123}{doi:10.3934/jmd.2007.1.123},  \href{https://arxiv.org/abs/math/0602176v1}{arXiv:math/0602176}


\bibitem[Kö92]{Kö92}
    {T.W. Körner},
    \emph{Sets of uniqueness} Cahiers du séminaire d'histoire des mathématiques 2 (1992): 51-63 \href{https://eudml.org/doc/91031}{eudml.org/doc/91031}


\bibitem[KS64]{KS64}
    {J.P. Kahane, R. Salem},
    \emph{ Ensembles Parfaits et Séries Trigonométriques}
    Canadian Mathematical Bulletin, Volume 7, Issue 3, September 1964, pp. 492 - 493 \href{https://www.cambridge.org/core/journals/canadian-mathematical-bulletin/article/ensembles-parfaits-et-series-trigonometriques-by-j-p-kahane-and-r-salem-hermann-paris-1963-paperback-188-pages-27-f/CD1EA7BBCCBDC467866D431347877BEA}{doi:10.1017/S0008439500032021}

\bibitem[La96]{La96}
    {Y. Last},
    \emph{Quantum Dynamics and Decompositions of Singular Continuous Spectra},
    Journal of Functional Analysis,
    Volume 142, Issue 2,
    1996,
    Pages 406-445,
    ISSN 0022-1236,
    \href{https://doi.org/10.1006/jfan.1996.0155}{doi:10.1006/jfan.1996.0155}.


\bibitem[Le00]{Le00}  
    {R. Leplaideur},
    \emph{Local product structure for equilibrium states},
    Trans. Amer. Math. Soc. 352 (2000), no. 4, 1889–1912. MR 1661262 
    \href{https://doi.org/10.1090/S0002-9947-99-02479-4}{doi:10.1090/S0002-9947-99-02479-4}
    
\bibitem[Le21]{Le21}
    {G. Leclerc},
    \emph{Julia sets of hyperbolic rational maps have positive Fourier dimension}, \\ Comm. Math. Phys. (2022) \href{https://arxiv.org/abs/2112.00701?context=math}{	arXiv:2112.00701},  \href{https://doi.org/10.1007/s00220-022-04496-6}{doi:10.1007/s00220-022-04496-6}


\bibitem[Le22]{Le22}
    {G. Leclerc},
    \emph{On oscillatory integrals with Hölder phases}, \\ 
    Discrete and Continuous Dynamical Systems, 2024, 44(1): 263-280. \href{https://www.aimsciences.org/article/doi/10.3934/dcds.2023103}{doi:10.3934/dcds.2023103}, \ 
    \href{https://arxiv.org/abs/2211.08088}{arXiv:2211.08088}


\bibitem[Le23a]{Le23a}
    {G. Leclerc},
    \emph{Fourier decay of equilibrium states for bunched attractors}, \\ 2024 Nonlinearity 37 095020
    \href{https://iopscience.iop.org/article/10.1088/1361-6544/ad6052}{doi:10.1088/1361-6544/ad6052}, 
    \href{https://arxiv.org/abs/2301.10623}{arXiv:2301.10623} 


\bibitem[Le23b]{Le23b}
    {G. Leclerc},
    \emph{Fourier decay of equilibrium states on hyperbolic surfaces}, \\
    Preprint (2023), \href{https://arxiv.org/abs/2307.10755}{arXiv:2307.10755}



\bibitem[Le24]{Le24}
    {G. Leclerc},
    \emph{Nonlinearity, Fractals, Fourier decay -- Harmonic analysis of equilibrium states for hyperbolic dynamical systems}, 
    PhD Thesis (2024), \href{https://arxiv.org/abs/2410.15476}{arXiv:2410.15476}


\bibitem[Li18]{Li18}
     {J. Li}, 
     \emph{Discretized sum-product and Fourier decay in }$\mathbb{R}^n$,
     JAMA 143, 763–800 (2021). \newline \href{https://doi.org/10.1007/s11854-021-0169-0}{doi:10.1007/s11854-021-0169-0},  \href{https://arxiv.org/abs/1811.06852v2}{arXiv:1811.06852} 

     

\bibitem[Li20]{Li20}
    {J. Li},
    \emph{Fourier decay, Renewal theorem and Spectral gaps for random walks on split semisimple Lie groups}
    Annales Scientifiques de l'ÉNS, Tome 55, Fasc.6, pp 1613-1686, 2022. \href{https://arxiv.org/abs/1811.06484v2}{	arXiv:1811.06484}.


\bibitem[Liv71]{Liv71}
    {A.N. Livshits},
    \emph{Homology properties of Y-systems.}
    Mathematical Notes of the Academy of Sciences of the USSR 10, 758–763 (1971). \href{https://doi.org/10.1007/BF01109040}{doi:10.1007/BF01109040}
    
\bibitem[LNP19]{LNP19}
    {J. Li, F. Naud, W. Pan},
    \emph{Kleinian Schottky groups, Patterson-Sullivan measures and Fourier decay} 
    Duke Math. J. 170 (4) 775 - 825, 15 March 2021. \\ \href{https://doi.org/10.1215/00127094-2020-0058}{doi:10.1215/00127094-2020-0058}, \href{https://arxiv.org/abs/1902.01103}{	arXiv:1902.01103}.  



\bibitem[ŁP09]{ŁP09}
    {I. Łaba, M. Pramanik},
    \emph{Arithmetic progressions in sets of fractional dimension.}
    Geom. Funct. Anal. 19, 429–456 (2009). \href{https://doi.org/10.1007/s00039-009-0003-9}{doi:10.1007/s00039-009-0003-9}


\bibitem[LPS25]{LPS25}
    {G. Leclerc, S. Paukkonen, T. Sahlsten},
    \emph{Fourier Dimension in $C^{1+\alpha}$ Parabolic Dynamics},
    Preprint (2025), \href{https://arxiv.org/abs/2505.15468}{arXiv:2505.15468}


\bibitem[LQZ03]{LQZ03}
    {P. Liu, M. Qian, Y. Zhao},
    \emph{Large deviations in Axiom A endomorphisms}
    Proceedings of the Royal Society of Edinburgh Section A: Mathematics , Volume 133 , Issue 6 , December 2003 , pp. 1379 - 1388
    \href{https://doi.org/10.1017/S0308210500002997}{doi:10.1017/S0308210500002997}



\bibitem[ŁW16]{ŁW16}
    {I. Łaba, H. Wang},
    \emph{Decoupling and near-optimal restriction estimates for Cantor sets}
    International Mathematics Research Notices, Volume 2018, Issue 9, May 2018, Pages 2944–2966, \href{https://doi.org/10.1093/imrn/rnw327}{doi:10.1093/imrn/rnw327},  \href{https://arxiv.org/abs/1607.08302}{arXiv:1607.08302}
  

\bibitem[Ma15]{Ma15}
     {P. Mattila},
     \emph{Fourier analysis and Hausdorff dimension},
     Cambridge University Press, 2015. \href{https://www.cambridge.org/core/books/fourier-analysis-and-hausdorff-dimension/78A25BED2C6E4F9B911971305DC928B0}{ISBN:9781316227619}

\bibitem[Mc34]{Mc34}
    {J. McShane},
    \emph{Extension of range of functions},
    Bull. Amer. Math. Soc. 40 (1934), 837–842.

\bibitem[McM83]{McM83}
    {McCluskey H, Manning A.}
    \emph{Hausdorff dimension for horseshoes.}
    Ergodic Theory and Dynamical Systems. 1983;3(2):251-260. \href{https://doi.org/10.1017/S0143385700001966}{doi:10.1017/S0143385700001966}

\bibitem[NT07]{NT07}
    {M. Nicol, A. Török}
    \emph{Whitney regularity for solutions to the coboundary equation on Cantor sets} Mathematical Physics Electronic Journal. Volume: 13, page Paper No. 6, 20 p.-Paper No. 6, 20 p.
    ISSN: 1086-6655


\bibitem[OdSS24]{OdSS24}
    {T. Orponen, N. de Saxcé, P. Shmerkin}
    \emph{On the fourier decay of multiplicative convolutions}, Preprint, 2024. \href{https://arxiv.org/abs/2309.03068}{arXiv:2309.03068}

\bibitem[PP90]{PP90}
    {W. Parry, M. Pollicott},
    \emph{Zeta Functions and the Periodic Orbit Structure of Hyperbolic Dynamics},
    Asterisque; 187-188, société mathématique de France, 1990.


\bibitem[PR02]{PR02}
    {A. Pinto, D. Rand},
    \emph{Smoothness of Holonomies for codimension 1 hyperbolic dynamics},
      Bulletin of the London Mathematical Society , Volume 34 , Issue 3 , May 2002 , pp. 341 - 352 \\ \href{https://doi.org/10.1112/S0024609301008670}{doi:10.1112/S0024609301008670}

\bibitem[PR05]{PR05}
    {A. Pinto, D. Rand}.
    \emph{Rigidity of hyperbolic sets on surfaces}
    Journal of the London Mathematical Society , Volume 71 , Issue 2 , April 2005 , pp. 481 - 502
    \href{https://doi.org/10.1112/S0024610704006052}{doi:10.1112/S0024610704006052}


\bibitem[QR03]{QR03}
    {M. Queffélec, O. Ramaré},
    \emph{Analyse de Fourier des fractions continues à quotients restreints}
    (2003) Enseign. Math. (2). 49. \href{https://www.e-periodica.ch/digbib/view?pid=ens-001:2003:49::251#684}{doi:10.5169/seals-66692}



\bibitem[RS78]{RS78}
    {M. Reed, B. Simon}, 
    \emph{Methods of Modern Mathematical Physics. IV.}
    Analysis of Operators. Academic Press, New York 1978



\bibitem[Ru78]{Ru78}
    {D. Ruelle}, 
    \emph{Thermodynamic Formalism}, 
    Second edition, Cambridge University Press 2004 


\bibitem[Sa43]{Sa43}
    {R. Salem}, \emph{Sets of uniqueness and sets of multiplicity}, 
   Transactions of the American Mathematical Society 54 (1943): 218-228. \href{https://www.semanticscholar.org/paper/Sets-of-uniqueness-and-sets-of-multiplicity-Salem/0737c31586a87737659eb2a547427a118122eaba}{doi:10.1090/S0002-9947-1943-0008428-8}


\bibitem[Sa51]{Sa51}
    {R. Salem}, \emph{On singular monotonic functions whose spectrum has a given Hausdorff dimension.} 
Ark. Mat., 1:353–365, 1951.

\bibitem[Sah23]{Sah23}
    {T. Sahlsten},
    \emph{Fourier transforms and iterated function systems},
      Recent Developments in Fractals and Related Fields. FARF 4 2022. Trends in Mathematics. Birkhäuser, Cham. 
      \href{https://doi.org/10.1007/978-3-031-80453-3_12}{doi:10.1007/978-3-031-80453-3\_12} \
      \href{https://arxiv.org/abs/2311.00585}{arXiv:2311.00585}


\bibitem[SS20]{SS20}
    {T. Sahlsten, C. Stevens},
    \emph{Fourier transform and expanding maps on Cantor sets} \\
    To be published in Amer. J. Math. (2022)
    \href{https://arxiv.org/abs/2009.01703v4}{arXiv:2009.01703 }


\bibitem[Su87]{Su87}
    {A. Sütő},
    \emph{The spectrum of a quasiperiodic Schrödinger operator.} 
    Commun. Math. Phys. 111, 409–415 (1987). \href{https://doi.org/10.1007/BF01238906}{doi:10.1007/BF01238906}



\bibitem[TZ20]{TZ20}
    {M. Tsujii, Z. Zhang},
    \emph{Smooth mixing Anosov flows in dimension three are exponentially mixing}
    Ann. of Math. (2) 197 (1) 65 - 158, January 2023. \href{https://doi.org/10.4007/annals.2023.197.1.2}{doi.org/10.4007/annals.2023.197.1.2} \ \href{https://arxiv.org/abs/2006.04293}{arXiv:2006.04293}



\bibitem[Yo90]{Yo90}
    {L.-S. Young.}
    \emph{Large deviations in dynamical systems.}
    Trans. Amer. Math. Soc., 318(2):525–543, 1990

\bibitem[Yo02]{Yo02}
    {L-S. Young},
    \emph{What Are SRB Measures, and Which Dynamical Systems Have Them ?}
     Journal of Statistical Physics. 108. 733-754. 
     \href{http://dx.doi.org/10.1023/A:1019762724717}{doi:10.1023/A:1019762724717}.


\bibitem[Si82]{Si82}
    {Barry Simon},
    \emph{Almost periodic Schrödinger operators: A Review} \\
    Advances in Applied Mathematics,
    Volume 3, Issue 4,
    1982,
    Pages 463-490,
    ISSN 0196-8858,
    \href{https://doi.org/10.1016/S0196-8858(82)80018-3}{doi:10.1016/S0196-8858(82)80018-3}


\end{thebibliography}
\end{document}